\documentclass[10pt]{article}

\usepackage[T1]{fontenc}
\RequirePackage[OT1]{fontenc}
\usepackage{amsthm,amsmath,amssymb,bbm}
\RequirePackage{bm}
\RequirePackage{url}

\RequirePackage{natbib}
\usepackage[usenames]{color}
\RequirePackage[colorlinks,linkcolor=mypurple,anchorcolor=mypurple,citecolor=mypurple]{hyperref}
\usepackage{appendix}
\usepackage{multirow}
\usepackage{natbib}
\usepackage{enumitem}
\usepackage{footmisc}
\usepackage{graphicx}
\usepackage{subfigure}
\usepackage{makecell}
\usepackage{booktabs}
\usepackage{booktabs}
\usepackage{array}
\usepackage{url}

\usepackage[ruled]{algorithm2e}
\usepackage{dsfont}
\usepackage{authblk}
\usepackage{hyperref}

\usepackage{tikz}
\usetikzlibrary{positioning}

\definecolor{myred}{HTML}{ae1908}
\definecolor{myblue}{HTML}{05348b}
\definecolor{myorange}{HTML}{ec813b}
\definecolor{mylightblue}{HTML}{9acdc4}
\definecolor{mypurple}{HTML}{743096}

\usepackage{fullpage}

\usepackage[font={footnotesize,it}]{caption}
\numberwithin{equation}{section}

\usepackage{smile}

\newcommand*{\supp}{\mathrm{supp}}
\newcommand{\Rom}[1]{\text{\uppercase\expandafter{\romannumeral #1\relax}}}

\begin{document}

\title{Factor Augmented Sparse Throughput Deep ReLU Neural Networks for High Dimensional Regression\thanks{Supported by NSF grants DMS-2210833, DMS-2053832, DMS-2052926 and ONR grant N00014-22-1-2340}}
\author{Jianqing Fan and Yihong Gu\\
Department of Operations Research and Financial Engineering\\
Princeton University}
\date{}
\maketitle

\begin{abstract}
This paper introduces a Factor Augmented Sparse Throughput (FAST) model that utilizes both latent factors and sparse idiosyncratic components for nonparametric regression.  The FAST model bridges factor models on one end and sparse nonparametric models on the other end.  It encompasses structured nonparametric models such as factor augmented additive models and sparse low-dimensional nonparametric interaction models and covers the cases where the covariates do not admit factor structures.  Via diversified projections as estimation of latent factor space, we employ truncated deep ReLU networks to nonparametric factor regression without regularization and to a more general FAST model using nonconvex regularization, resulting in factor augmented regression using neural network (FAR-NN) and FAST-NN estimators respectively. We show that FAR-NN and FAST-NN estimators adapt to the unknown low-dimensional structure using hierarchical composition models in nonasymptotic minimax rates. We also study statistical learning for the factor augmented sparse additive model using a more specific neural network architecture.  Our results are applicable to the weak dependent cases without factor structures. In proving the main technical result for FAST-NN, we establish a new deep ReLU network approximation result that contributes to the foundation of neural network theory.  Our theory and methods are further supported by simulation studies and an application to macroeconomic data.
\end{abstract} 

\section{Introduction}

Deep learning \citep{goodfellow2016deep} has achieved tremendous empirical success in computer vision \citep{voulodimos2018deep}, natural language processing \citep{otter2020survey} 
and other statistical prediction tasks \citep{lecun2015deep} due to its representation power. 
It has also been widely applied to cell subpopulation learning in single-cell RNA-seq \citep{tian2019clustering,li2020deep}, genetic association studies \citep{wang2019high, lin2018deep}, genomics \cite{zhu2021deeplink} and protein folding  \citep{senior2020improved} that has proven to be a remarkable success.   It is a scalable nonparametric technique with a great ability to balance bias and variance, adapting to unknown low-dimensional structure. See \cite{fan2020selective} for a selective overview.  Despite numerous high-profile empirical successes, there are limited theoretical understandings on neural networks.  This paper contributes to high-dimensional nonparametric model selection for both strong dependent and weak dependent covariates and hence to interpretable machine learning, particularly to the feature selections for neural networks, via introducing a FAST model and its associated new neural network architecture and algorithms. As a result, we provide additional understanding on the performance of neural networks for algorithmic nonparametric regression modeling, which learn low-dimensional structures with no or little supervision on the forms of functions.   

\subsection{Nonparametric regression}

Given covariate vector $\bx \in \mathbb{R}^p$ and the response variable $y$ with $(\bx, y)\sim\mu$, we are interested in estimating the regression function $m^*(\bx) = \mathbb{E}[y|\bx]$, which minimizes the \emph{population $L_2$ risk}: 
\begin{align*}
    \mathsf{R}(m) = \int |y - m(\bx)|^2\mu(d\bx, dy),
\end{align*}
based on an i.i.d. sample $\{(\bx_i,y_i)\}_{i=1}^n$ from $\mu$. There is a considerable literature on nonparametric regression methods. 
We refer the readers to \cite{fan1996local, gyorfi2002distribution, tsybakov2009nonparametric} for a comprehensive account on nonparametric regression. When the regression function $m^*$ is a $p$-variate $(\beta,C)$-smooth function, \cite{stone1982optimal} shows the minimax optimal convergence rate for $m^*(\bx)$ is of order $n^{-\frac{2\beta}{2\beta+p}}$, which does not converge to zero if $p \asymp \log n$. This implies when the dimension $p$ is relatively large compared with the degree of smoothness $\beta$, it requires a large amount of data to estimate the regression function well, which has been referred to as the ``curse of dimensionality''.

To alleviate the ``curse of dimensionality'', it is natural to impose some low-dimension structures on the regression function $m^*$. For example, \citep{stone1985additive} imposes an additive structure $m^*(\bx) = \sum_{j=1}^p m^*_j(x_j)$ and shows a faster convergence rate $n^{-2\beta/(2\beta+1)}$ is obtainable when the univariate functions are $(\beta,C)$-smooth. There is also a considerable literature on characterizing the intrinsic low-dimension structures and developing efficient statistical methods to achieve a faster convergence rate, including interaction models \citep{stone1994use}, single-index models \citep{hardle1989investigating}
, projection pursuit \citep{friedman1981projection}, to name a few. However, these methods are structural rather than algorithmic in that a structure of regression functions needs to be imposed before designing statistical methods.  This is where neural networks come to play, which can adapt well to unknown low-dimensional structures through algorithmic optimization.

\subsection{Neural networks}

The recent decades have witnessed the great success of deep learning \citep{lecun2015deep}. The key driving force behind such success is the use of neural networks. A neural network is a composition of simple functions parameterized by its \emph{weights}. For example, the fully connected deep ReLU neural network is a composition of linear transformation $\mathcal{L}(\bx) = \bW \bx + \bb$ with weights $(\bW, \bb)$ followed by elementwise ReLU nonlinear transformation $\sigma(x) = \max\{x, 0\}$. Intuitively, such composition nature not only allows it to approximate complex functions well but also endows it with the capability to automatically extract hierarchical features from raw data, for example, text or images.

From a statistical viewpoint, the success of neural networks can be attributed to their ability to approximate complex nonlinear functions effectively. The analysis of neural networks' approximation ability can date back to around 1990s. \cite{cybenko1989approximation, hornik1991approximation, barron1993universal} showed that a one-hidden-layer neural network can approximate any continuous function to an arbitrary degree of accuracy, which has been referred to as the \emph{universal approximation theorem}. Afterward, \cite{telgarsky2016benefits} illustrated the benefits of using a deep neural network compared to a shallow counterpart. There is a considerable literature on investigating the nonasymptotic approximation error of fully-connected deep ReLU neural networks over some specific function classes \citep{yarotsky2017error, approxyarotsky2018optimal, shen2019nonlinear, kohler2021rate, hanin2019universal, lu2021deep}. These explicit characterizations of the dependency of a neural network's approximation error on depth $L$ and width $N$ are instrumental to the study of the statistical properties of neural networks.

\cite{bauer2019deep, schmidt2020nonparametric, kohler2021rate} further demonstrated the benefits of using deep neural networks in nonparametric regression. In particular, \cite{schmidt2020nonparametric, kohler2021rate} showed that a deep ReLU neural network can be adaptive to the intrinsic low-dimension structure of the regression function, which empowers it to circumvent the curse of dimensionality. To be specific, if the regression function is a composition of $t$-variate $(\beta, C)$-smooth functions with $(\beta, t) \in \mathcal{P}$, then the deep ReLU neural network least squares estimator $\hat{m}$ with an appropriate choice of depth and width  will achieve a convergence rate of
\begin{align}
\label{eq:roc-nn-low-dim}
    \mathbb{E} \left[\int |\hat{m}(\bx) - m^*(\bx)|^2 \mu(d\bx) \right] \le C \cdot (\log n)^3 \cdot n^{-\frac{2\gamma^*}{2\gamma^*+1}} ~~~~~~ \text{with} ~~~~ \gamma^* = \inf_{(\beta,t)\in \mathcal{P}} \frac{\beta}{t}
\end{align} 
where $C$ is a constant independent of $n$. When all the dimension-adjusted smoothnesses $(\beta/t)$ are not small, i.e., each composition in $m^*$ either has a small dimension $t$ or has a large degree of smoothness $\beta$, the rate of convergence will be fast even when $p$ is large. Moreover, unlike the previous estimators that need to be aware of the hierarchical composition structure beforehand, such a neural network estimator only needs to know the constant $\gamma^*$ in advance. These results explain why neural network can outperform many other methods in a wide range of real-world applications because many laws in nature and human societies admit certain low-dimensional composition structures \citep{dahmen2022compositional}, for example, natural language \citep{partee1984compositionality}. 

\subsection{The problem under study}

Exciting though the above results for neural network, it is only applicable in the regime that the ambient dimension $p$ is fixed.
Things might be different in the regime that $p \gg \log n$ since the constant $C$ in \eqref{eq:roc-nn-low-dim} may depend polynomially or even exponentially on $p$. In the era of big data, there are more and more data with high dimensionality available. Moreover, there is also a surge in demand for making predictions based on a large number of variables \citep{fan2014challenges, wainwright2019high}. For example, the ImageNet dataset \citep{deng2009imagenet} contains $n\approx 1.5 \times 10^7$ high-resolution labelled images with size around $p\approx 2\times 10^5$, which can not be treated as a fixed constant. These facts indicate that it is a necessity to adapt the neural network estimator to the relative high dimension regime $p \gg \log n$.

Associated with high-dimensional features  is the dependence among variables \citep{fan2014challenges}.  We address this issue through by a linear factor model on covariate $\bx$, which admits
\begin{align}
\label{eq:intro-factor-model}
    \bx = \bB \bbf + \bu,
\end{align} 
where the latent factor $\bbf \in \mathbb{R}^r$ and the idiosyncratic component $\bu \in \mathbb{R}^p$ is unobserved, the factor loading matrix $\bB\in \mathbb{R}^{p\times r}$ is fixed but unknown. Our goal is to use neural network to estimate the Factor Augmented Sparse Throughput (FAST) regression function
\begin{align*}
    \mathbb{E}[y|\bbf, \bu] = m^*(\bbf, \bu_\mathcal{J}) ~~~~~~ \text{with} ~~~~ \mathcal{J} \subset \{1,\ldots, p\}
\end{align*} using i.i.d. sample $\{(\bx_i,y_i)\}_{i=1}^n$, where $\bx_i$ follows the model \eqref{eq:intro-factor-model}, and 
\begin{align} \label{eq1.3}
y_i = m^*(\bbf_i, \bu_{i, \mathcal{J}}) + \varepsilon_i, \qquad \mathbb{E}[\varepsilon_i|\bbf_i, \bu_i]=0
\end{align}
with i.i.d. noises $\varepsilon_1,\ldots, \varepsilon_n$.
We wish to estimate $m^*$ well in the regime that  $p$ is relatively high dimension  and the latent factor dimension $r$ and the number of important variables $|\mathcal{J}|$ is small. We refer to model \eqref{eq1.3} as the Factor Augmented Sparse Throughput (FAST) model. 
\begin{figure}[!t]
\centering
\begin{tikzpicture}[->, auto]
	\node[shape=rectangle,draw] (A) {FAST model};
	\node[shape=rectangle,draw] (B) [below left=of A]{Sparse model};
        \node[shape=rectangle,draw] (H) [below =of B, align=center]{High-dimensional\\additive models};
	\node[shape=rectangle,draw] (C) [below left=of B, align=center]{Sparse\\linear model};
	\node[shape=rectangle,draw] (D) [below =of A]{Factor regression};
	\node[shape=rectangle,draw] (E) [below =of D]{PCR};
	\node[shape=rectangle,draw] (F) [below right=of A]{Sparse model w/ dep};
	\node[shape=rectangle,draw] (G) [below =of F]{FARM};
	
	\path (A) edge    node[above=3pt]{$r=0$} (B)
		  (B) edge    node[above=3pt]{linear} (C)
		  (A) edge    node{$|\mathcal{J}|=0$} (D)
		  (D) edge    node{linear} (E)
		  (A) edge    node{$r\neq 0$} (F)
		  (F) edge    node{linear} (G)
            (B) edge    node{additive} (H);
\end{tikzpicture}
\caption{The versatility of FAST model. The arrow from model $x$ to model $y$ means $x$ includes $y$.}
\label{fig:versatility}
\end{figure}
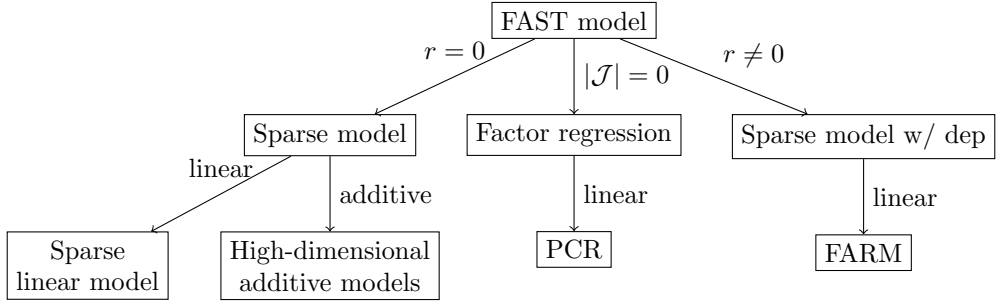

Model \eqref{eq1.3} encompasses many useful statistical models.  When $|\mathcal{J}|=0$, it reduces to nonparametric factor regression models with function form $m^*(\bbf)$.  One can specify further the structure on $m^*(\bbf)$ such as additive models.  On the other hand, when there is no factor structure ($\bx = \bu$), model \eqref{eq1.3} reduces to nonparametric sparse regression model $\mathbb{E}[y|\bx] = m^*(\bx_{\cJ})$.  Therefore, we bridge nonparametrically two seemingly unrelated models through a unified framework.  Further specification includes factor-augmented sparse additive models and sparse additive models. In particular, if the nonparametric function is specified as linear function, they are referred to as principal component regression and sparse linear regression in the literature.  The versatility of the FAST model is clear; refer to Figure \ref{fig:versatility} for particular instances of the FAST model.  In particular, strong dependence is allowed in both nonparametric model selection and sparse linear model selection (referred to as factor-augmented regularized model by \cite{fan2020factor}).

\subsection{Contributions of the paper}

In this work, we propose a Factor Augmented Sparse Throughput Neural Network (FAST-NN) estimator, which allows the neural network to be adaptive to the intrinsic low-dimension structure as described by FAST model \eqref{eq1.3}.  The estimator is designed to capture the latent factor information from the covariate $\bx$ and use only a small set of idiosyncratic component $\bu$ to explain further additional contributions. Specifically, we use a diversified projection matrix and a variable selection matrix ahead of the neural network's input layer. The diversified projection matrix is pre-trained and fixed, which is necessary for optimal learning, while the variable selection matrix is trained jointly with neural network weights via a penalized least squares objective to select variables. This result in a new neural network architecture for high-dimensional data. We now highlight our main contributions.

We introduce the FAST model, a versatile high-dimensional nonparametric regression model encompassing many useful statistical models.  It overcomes the unrealistic assumptions on finite dimensionality $p$ in the previous literature \citep{kohler2016nonparametric, bauer2019deep, schmidt2020nonparametric, kohler2021rate} about nonparametric regression using neural networks.  The model addresses the information loss in high-dimensional principal component regression via using sparse idiosyncratic components and at the same time deals with model selection inconsistency issue even for linear sparse model due to high correlation among covariates \citep{zhao2006model,fan2020factor}.   It allows us to conduct nonparametric model selection using neural networks under both strong and weak dependence covariates and hence permits us to select features for neural network prediction and contributes to the interpretable machine learning.

We propose a FAST-NN estimator, a structured neural network with provable pre-trained diversified projection (DP) matrix and penalized objective function that bridges neural network and high-dimensional statistics. The implementation of the DP matrix within the neural network handles the high dependence in high-dimensional data and contributes to nonparametric factor modeling. The idea of DP matrix is also applicable to other factor-augmented models in econometrics and statistics \citep{fan2022learning}.  Furthermore, the FAST-NN architecture contributes to feature selection in deep neural networks \citep{ismael2021lassonet, feng2017sparse} in a dual way: firstly, it leverages the sample efficiency of neural networks to facilitate nonparametric model selection; secondly, it offers an interpretable neural network prediction model for high-dimensional data. 

We develop theoretical justifications showing that our proposed FAST-NN estimator achieves the minimax optimal $L_2$ error and can avoid the curse of dimensionality in high-dimensional nonparametric regression under mild model structure. In particular, our method incurs minimal additional costs if the regression function is moderately over-parameterized (e.g. $r=0$ or $\mathcal{J}=\emptyset$, but we use more factors or attempts to select variables). This highlights our method's potential to algorithmically exploit the low-dimensional (latent) structure from high-dimensional covariate. We support the theoretical findings in numerical studies and show that stochastic gradient descent continues to apply to our specifically designed neural network estimators with no curse of dimension in implementation.

We also offer an in-depth analysis of the training protocol for the FAST model. A fundamental question is whether the weights for learning latent factors $\bbf$ should be jointly trained with the neural network weights for learning $m^*$ in \eqref{eq1.3}.  To address this issue, we look at the simpler model with $|\mathcal{J}|=0$ so that the problem is isolated from the task of model selection. We unveil, both theoretically and empirically,  that a joint training of the DP matrix and neural network weights will result in a sub-optimal performance. Conversely, we show that the DP matrix can be efficiently estimated using a tiny fraction of data without precisely determining the exact number of factors. This affirms the necessity of employing pre-trained weights in FAST-NN.

\noindent \textbf{Organization.} In Section \ref{sec:model} we introduce the Factor Augmented Sparse Throughput (FAST) model considered in this paper and the ReLU neural network class. We briefly review the idea of diversified projection matrix and propose the two estimators, the Factor Augmented Regression Neural Network (FAR-NN) estimator in the case where $\mathcal{J}=\emptyset$ and the Factor Augmented Sparse Throughput Neural Network (FAST-NN) estimator in the case where $|\mathcal{J}|=\mathcal{O}(1)$ in Section \ref{sec:method}. The theoretical analyses of the proposed two estimators are included in Section \ref{sec:theory:fa} and Section \ref{sec:theory:fast} respectively, followed by the minimax optimal lower bound for FAST model in Section \ref{sec:theory:lb}. The simulation studies are presented in Section \ref{sec:experiments}. The FANAM estimator, a more computationally efficient estimator when the regression function further admits a sparse additive structure, and its theory, an empirical application, some discussions and related works, and all the proofs are collected in the supplemental material.

\noindent \textbf{Notations.} Denote by $[m] = \{1, \cdots, m\}$. For $f(n)$ and $g(n)$, we use $f(n) \lesssim g(n)$ or $f(n) = \mathcal{O}(g(n))$ to represent that  $f(n) \le C\cdot g(n)$ for a constant $C$ independent of $n$; similarly, we use $f(n) \gtrsim g(n)$ or $f(n) = \Omega(g(n))$ to represent that there exists a constant $C>0$ independent of $n$ such that $f(n) \ge C \cdot g(n)$ for all the $n \ge 1$. We use $f(n) \asymp g(n)$ if $f(n) \lesssim g(n)$ and $f(n) \gtrsim g(n)$. We denote $f(n) \ll g(n)$ or $f(n) = o(g(n))$ or $g(n) \gg f(n)$ if $\limsup_{n\to\infty} [f(n)/g(n)] = 0$. Moreover, we let $a \lor b = \max\{a, b\}$ and $a \land b = \min\{a, b\}$. We let $\mathrm{Pdim}(\mathcal{H})$ be the Pseudo dimension \citep{pollard1990empirical} of the function class $\mathcal{H}$. We use bold lower case letter $\bx = (x_1,\ldots, x_d)^\top$ to represent a $d$-dimension vector, let $\|\bx\|_q = (\sum_{i=1}^d |x_i|^q)^{1/q}$ be its $\ell_q$ norm, and let $\|\bx\|_\infty = \max_{1\le i\le d} |x_i|$ be its $\ell_\infty$ norm. We use bold upper case $\bA = [A_{i,j}]_{i\in [n], j\in [m]}$ to denote a matrix. We define $\|\bA\| = \sup_{\bx \in \mathbb{R}^m, \|\bx\|_2=1} \| \bA \bx\|_2$, $\|\bA\|_{F} = \sqrt{\sum_{i,j} A_{i,j}^2}$, and $\|\bA\|_{\max}=\max_{i\in [n],j\in [m]} |A_{i,j}|$. Moreover, we use $\lambda_{\min}(\bA)$ and $\nu_{\min}(\bA)$ to denote its minimum eigenvalue and singular value, respectively.

\section{Model}
\label{sec:model}

\subsection{High-dimensional augmented sparse nonparametric regression}

Suppose we observe $n$ an i.i.d. sample $\{(\bx_i, y_i)\}_{i=1}^n$ from $(\bx,y)\sim \mu$, where $\bx \in \mathbb{R}^p$ is the $p$-dimensional covariate vector and $y \in \mathbb{R}$ is the response variable. Our target is to find a function $m: \mathbb{R}^p \to \mathbb{R}$ to minimize the following \emph{population $L_2$ risk}:
\begin{align*}
    \mathsf{R}(m) = \mathbb{E}(y - m(\bx))^2 = \int \left| m(\bx) - y\right|^2 \mu(d\bx, dy).
\end{align*}
Standard nonparametric regression focuses on the regime where $p$ is fixed and does not grow with $n$.  Let $m^*(\bx) = \mathbb{E}[y|\bx]$ be the population risk minimizer, the regression function. It is natural to evaluate the accuracy of the predictor $m$ via the \emph{excess risk}:
\begin{align*}
    \mathsf{R}(m) - \mathsf{R}(m^*) = \mathbb{E} [|m(\bx)-m^*(\bx)|^2]= \int \left|m(\bx)-m^*(\bx)\right|^2 \mu(d\bx),
\end{align*} 
which is mean squared error of $m(\bx)$.
The statistical rate of convergence depends on the class of functions $m^*$ lies in. 
\begin{definition}[$(\beta,C)$-smooth function]
    Let $\beta = r+s$ for some nonnegative integer $r$ and $0<s\le 1$, and $C>0$. A $d$-variate function $f$ is $(\beta, C)$-smooth if for every non-negative sequence $\balpha \in \mathbb{N}^d$ such that $\sum_{j=1}^d \alpha_j = r$, the partial derivative $(\partial f)/(\partial x_1^{\alpha_1}\cdots x_d^{\alpha_d})$ exists and satisfies
    \begin{align}
        \left|\frac{\partial^r f}{ \partial x_1^{\alpha_1}\cdots \partial x_d^{\alpha_d}}(\bx) - \frac{\partial^r f}{\partial x_1^{\alpha_1}\cdots \partial x_d^{\alpha_d}}(\bz) \right| \le C \|\bx-\bz\|_2^s.
    \end{align} 
    We use $\mathcal{F}_{d,\beta, C}$ to denote the set of all the $d$-variate $(\beta, C)$-smooth functions.
\end{definition}
It is well known the minimax optimal rate of convergence when $m^* \in \mathcal{F}_{p,\beta, C}$ is
\begin{align*}
    \inf_{\hat{m}} \sup_{m^*\in \mathcal{F}_{d,\beta,C}}  \mathbb{E} \left[\int \left|\hat{m}(\bx)-m^*(\bx)\right|^2 \mu(d\bx)\right] \asymp  C^{\frac{p}{2\beta+p}} n^{-\frac{2\beta}{2\beta+p}}
\end{align*} for all the $n \ge N_0$.
The result indicates that the rate of convergence is slow if $p$ is large relative to $\beta$, which is referred to as `\emph{curse of dimensionality}'. Such problem can be even more severe in mildly high dimensional regime that $p \gg \log n$. In this case, the MSE does not converge to zero for any finite $\beta$.  We therefore consider dimensionality reduction in two directions: factor structure in covariates and low-dimensional structure in  regression function.

Consider the factor model where  the covariate $\bx$ can be decomposed as
\begin{align}
\label{eq:dgp-lfm}
    \bx = \bB \bbf + \bu,
\end{align} where $\bB\in \mathbb{R}^{p\times r}$ is an unknown  loading matrix, $\bbf \in \mathbb{R}^r$ is the vector of latent factors, $\bu \in \mathbb{R}^p$ is the vector of the idiosyncratic component or throughput. Throughout the paper, we assume the covariates $\bx_1, \ldots, \bx_n$ are observable and their associated latent vectors $\{(\bbf_i, \bu_i)\}_{i=1}^n$ are i.i.d. copies of $(\bbf, \bu)$.

Given such latent factor structure, it is naturally to consider the factor augmented regression problem in which we use $\bbf$ and $\bu$, or, more precise, an estimate of $\bbf$ and $\bu$ from the observation $\bx$, as the regressor. This is the same as using $\bx$ and $\bbf$ as regressor, and hence the factor-augmented regression model, but the former representation makes variables much  weakly dependent. 
The regression function in the factor augmented space is $m^*(\bbf,\bu)=\mathbb{E}[y|\bbf,\bu]$. We assume further that 
\begin{align*}
    \mathbb{E}[y|\bbf,\bu] = m^*(\bbf, \bu_{\mathcal{J}}),
\end{align*} where $\mathcal{J} \subset \{1, 2, \ldots, p\}$ is an unknown subset of indexes. Defining the noise to be $\varepsilon = y-\mathbb{E}[y|\bbf,\bu]$, the data generating process can be summarized as
\begin{align}
\label{eq:dgp-samples}
    \bx_i = \bB \bbf_i + \bu_i \qquad \text{and} \qquad y_i = m^*(\bbf_i, \bu_{i,\mathcal{J}}) + \varepsilon_i.
\end{align}
This model will be referred to as the Factor Augmented Sparse Throughput (FAST) model.
The above framework bridges two specific nonparametric regression models.  When $\cJ = \emptyset$, it is a nonparametric factor regression model, whereas when there is no factor ($r=0$, $\bx_i = \bu_i$), it reduces to sparse nonparametric regression model. In particular, it includes sparse high dimensional additive modeling as a special case:
\begin{align*}
    y_i = \sum_{j\in \mathcal{J}} m_j^*(x_{i,j}) + \varepsilon_i \qquad \text{for} \qquad i=1,\ldots, n.
\end{align*}

\subsection{ReLU Neural Networks}

We  build our model using a fully-connected deep neural network with ReLU activation $\sigma(\cdot) = \max\{\cdot, 0\}$ due to its great empirical success, and we call it \emph{deep ReLU network} for short. Let $L$ be any positive integer and $\bd=(d_1,\ldots, d_{L+1}) \in \mathbb{N}^{L+1}$. A deep ReLU network is a function mapping from $\mathbb{R}^{d_{0}}$ to $\mathbb{R}^{d_{L+1}}$ which takes the form of
\begin{align}
\label{eq:nn-def}
    g(x) = \mathcal{L}_{L+1} \circ \bar{\sigma} \circ \mathcal{L}_{L} \circ \bar{\sigma} \circ \cdots \circ \mathcal{L}_2 \circ \bar{\sigma} \circ \mathcal{L}_1(x),
\end{align} 
where $\mathcal{L}_\ell(\bz) = \bW_\ell \bz + \bb_\ell$ is an affine transformation with the weight matrix $\bW_\ell \in \mathbb{R}^{d_\ell \times d_{\ell-1}}$ and bias vector $\bb_\ell \in \mathbb{R}^{d_\ell}$, and $\bar{\sigma}: \mathbb{R}^{d_\ell} \to \mathbb{R}^{d_\ell}$ applies the ReLU activation function  to each entry of a $d_\ell$-dimension vector. For simplicity, we refer to both $\bW_\ell$ and $\bb_\ell$ as \emph{weights of a deep ReLU network}. 

\begin{definition}[Deep ReLU network class]
\label{def:nn}
    For any $L\in \mathbb{N}$, $\bd \in \mathbb{N}^{L+1}$, $B, M\in \mathbb{R}^+ \cup \{\infty\}$, define the family of deep ReLU network truncated by $M$ with depth $L$, width parameter $\bd$, and weights bounded by $B$ as
    \begin{align*}
        \mathcal{G}(L, \bd, M, B) = \Big\{\tilde{g}(\bx) = \bar{T}_M(g(\bx)): g \text{ of form \eqref{eq:nn-def}} &\text{ with } \|\bW_\ell\|_{\max} \le B, \|\bb_\ell\|_{\max} \le B  \Big\}
    \end{align*} where $\bar{T}_M(\cdot)$ applies truncation operator at level $M$ to each entry of a $d_{L+1}$ dimension vector, i.e., $[\bar{T}_M(z)]_i = \text{sgn}(z_i) (|z_i| \land M)$. We denote it as $\mathcal{G}(L, d_{in}, d_{out}, N, M, B)$ if the width parameter $\bd=(d_{in}, N, N, \ldots, N, d_{out})$, which we referred as {deep ReLU network with depth $L$ and width $N$} for brevity.
\end{definition}

\section{Methodology}
\label{sec:method}

In this section, we describe our method of estimation: Factor Augmented Sparse Throughput Neural Network (FAST-NN). Compared with the standard neural network that directly uses the covariate $\bx$ as input, our method introduces two additional modules ahead of the input layer of a deep ReLU network: (1) a pre-defined \emph{diversified projection matrix} $\bW$ to estimate the factor $\bbf$, and (2) a sparse-constrained \emph{variable selection matrix} $\bTheta$ to extract information from important variables $\bx_\mathcal{J}$ (or $\bu_\mathcal{J}$).

\subsection{Factor estimation via diversified projection matrix}
\label{sec:method:diversified-weight}

Let $\overline{r}$ be an integer satisfying $\overline{r}\ge r$. We first introduce the idea of \emph{diversified projection matrix} $\bW \in \mathbb{R}^{p\times \overline{r}}$ proposed by \cite{fan2022learning}.
\begin{definition}[Diversified projection matrix]
\label{def:dpm}
    Let $\overline{r} \ge r$, and $c_1$ be a universal positive constant. A $p\times \overline{r}$ matrix $\bW$ is said a diversified projection matrix if it satisfies \\
    \noindent $~~~~$ (Boundedness) $\|\bW\|_{\max} \leq c_1$; ~~~~~~ (Exogeneity) $\bW$ is independent of $\bx_1,\ldots, \bx_n$ in \eqref{eq:dgp-samples}; \\
    \noindent $~~~~$ (Significance) The matrix $\bH=p^{-1} \bW^\top \bB\in \mathbb{R}^{\overline{r}\times r}$ satisfies $\nu_{\min}(\bH) \gg p^{-1/2}$. \\
    \noindent Each column of $\bW$ is called as \emph{diversified weight}, and  $\overline{r}$ is the number of diversified weights.
\end{definition}
The key idea of the \emph{diversified projection matrix} is that we can use 
\begin{align}
\label{eq:est-factor}
    \tilde{\bbf} = p^{-1} \bW^\top \bx
\end{align} as a surrogate of the factor $\bbf$ in the downstream prediction.  Throughout the paper, we assume that  $\overline{r} \geq r$, an over-parameterized regime. 
To appreciate this,  we substitute \eqref{eq:dgp-samples}  into \eqref{eq:est-factor} and obtain 
\begin{align}
\label{eq:est-factor-decomposition}
    \tilde{\bbf} = \bH \bbf + \bxi,  \quad \mbox{with } \bxi = p^{-1} \bW^\top \bu.
\end{align}
Decomposition \eqref{eq:est-factor-decomposition} reveals that $\tilde{\bbf}$ should estimate well an affine transformation of the latent factor $\bbf$ under mild conditions. The intuition is that when idiosyncratic components $\bu$ are weakly dependent with uniformly bounded second moments,  $\|\bxi\|=\mathcal{O}_\mathbb{P}(p^{-1/2})$ due to its order of variance (e.g. if components of $\bu$ are uncorrelated, $\Var(\bW_j^T \bu) = \mathcal{O}(1)$, where $\bW_j$ is the $j^{th}$ column of $\bW$). On the other hand, owing to the significance condition and non-degenerate factors, the signal
\begin{align*}
    \|\bH\bbf\|_2 \ge \nu_{\min}(\bH) \|\bbf\|_2 \gg \|\bxi\|_2, 
\end{align*} 
which means the first term $\bH \bbf$ will be the dominating term among the above decomposition \eqref{eq:est-factor-decomposition}.

As for the diversified projection matrix $\bW$, an intuitive explanation is that $\bW$ can be treated as an over-estimate of the factor loading $\bB$. The prefix `over-' indicates there is no need to accurately determine the number of factors $r$ in our framework; we can loosely choose some large $\overline{r}$ instead. Meanwhile, the  `significance' condition  requires that $\bW$ should be better than a random guess. For example, if $\bW$ is a $\bar{r}\times p$ matrix with i.i.d. standard Gaussian entries, and $\bar{r}\asymp 1$, it is easy to verify that $\nu_{\min}(\bH) \le \|\bH\| = \mathcal{O}_{\mathbb{P}}(p^{-1/2})$, which does not meet the ``significance'' condition in Definition \ref{def:dpm}. 

A natural question then is how to determine the diversified projection matrix in practice. Here we briefly illustrate two ways. The first is to use domain knowledge to construct each column of $\bW$. Examples for financial data are given in \cite{fan2022learning}, other examples in deep learning literature includes the usage of word embedding, e.g., word2vec \citep{mikolov2013efficient}, in natural language process. Another is a data-driven method that uses another $n'$ pre-trained $\bx_1,\ldots, \bx_{n'}$. To be specific, we can apply principal component analysis to get top-$\overline{r}$ eigenvectors $\hat{\bv}_1, \ldots, \hat{\bv}_{\overline{r}} \in \mathbb{R}^p$ of sample covariance matrix $\hat{\bSigma} = (n')^{-1} \sum_{i=1}^{n'} \bx_i \bx_i^\top$, and take $\bW$ as $\bW=\sqrt{p} \left[\hat{\bv}_1,\ldots, \hat{\bv}_{\overline{r}}\right]$. When only $n$ labelled samples are available, it reduces to standard sample splitting. However, in many applications, there are a lot of unlabelled data available, in which the second method is related to the semi-supervised learning \citep{van2020survey} in practice.

\subsection{Factor Augmented Regression using Neural Networks}

To understand why the pre-training works and is necessary, let us consider the specific case in which $|\cJ|= 0$, namely FAR-model.  Based on the above facts about the diversified projection matrix, we can run nonparametric factor regression model $y = g(\bbf) + \varepsilon$ by using its proxy $\tilde{\bbf}$ and deep ReLU networks.   Letting $\tilde{\bbf}_i = p^{-1} \bW^\top \bx_i$, run the following least squares
\begin{align}
\label{eq:fann-lse}
    \hat{g}(\cdot) 
    = \argmin_{g\in \mathcal{G}(L, \overline{r}, 1, N, M, \infty)} \frac{1}{n} \sum_{i=1}^n \left\{y_i - g(\tilde{\bbf}_i)\right\}^2.
\end{align} 
Then, the factor augmented regression using neural network (FAR-NN) is defined by
\begin{align}
\label{eq:fann-estimator}
    \hat{m}_{\mathtt{FAR}}(x) = \hat{g}\left(p^{-1} \bW^\top \bx\right).
\end{align} 
A visualization of the FAR-NN estimator is shown in Fig \ref{fig:method}(a).

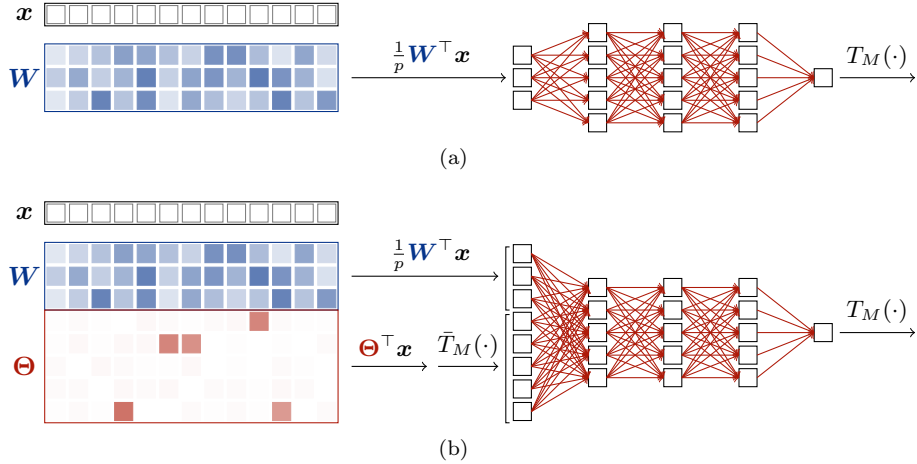
\begin{figure}
\centering

\subfigure[]{
\begin{tikzpicture}[scale=1]
\pgfmathsetseed{2}
\definecolor{myred}{HTML}{ae1908}
\definecolor{myblue}{HTML}{05348b}
\foreach \x in {-1, 0, 1}
{
	\draw[black] (0, \x*0.3) rectangle +(0.24, 0.24);
}
\foreach \l in {1, 2, 3}
\foreach \x in {-2, -1, 0, 1, 2}
{
	\draw[black] (\l, \x*0.3) rectangle +(0.24, 0.24);
}
\draw[black] (4, 0) rectangle+(0.24, 0.24);
\foreach \x in {-1, 0, 1}
	\foreach \y in {-2, -1, 0, 1, 2}
{
	\draw[->, myred] (0+0.24, \x*0.3+0.12) -- (1, \y * 0.3+0.12);
}

\foreach \l in {2,3}
	\foreach \x in {-2, -1, 0, 1, 2}
		\foreach \y in {-2, -1, 0, 1, 2}
		{
			\draw[->, myred] (\l-1+0.24, \x*0.3+0.12) -- (\l, \y * 0.3+0.12);
		}
		
\foreach \x in {-2, -1, 0, 1, 2}
	{
		\draw[->, myred] (4-1+0.24, \x*0.3+0.12) -- (4, 0 * 0.3+0.12);
	}
\draw[black] (-2.3 - 13 * 0.3 - 0.3, 0.3 + 0.24 + 0.3 - 0.06 + 0.18) node{$\bm{x}$};
\draw[black] (-2.3 - 13 * 0.3 - 0.03, 0.3 + 0.24 + 0.3 - 0.03) rectangle+(0.3 * 13, 0.3);
\foreach \y in {1, 2, 3, 4, 5, 6, 7, 8, 9, 10, 11, 12, 13}
	\draw[gray] (-2.3 - \y * 0.3, 0.3+0.24+0.3) rectangle+(0.24, 0.24);
\draw[myblue] (-2.3 - 13 * 0.3 - 0.3, -0.3 - 0.06 + 0.48) node{$\bm{W}$};
\draw[myblue] (-2.3 - 13 * 0.3 - 0.03, -0.3 - 0.03) rectangle+(0.3 * 13, 0.3*3);
\foreach \x in {-1, 0, 1}
	\foreach \y in {1, 2, 3, 4, 5, 6, 7, 8, 9, 10, 11, 12, 13}
	{
		\pgfmathrandominteger{\a}{10}{70}
		\filldraw[myblue!\a] (-2.3 -\y * 0.3, \x * 0.3) rectangle+ (0.24, 0.24);
	}
\draw[black] (-1.1, 0.4) node{\small $\frac{1}{p} {\color{myblue} \bm{W}}^\top \bm{x}$};
\draw[->] (-2.15, 0.12) -- (-0.1, 0.12);

\draw[black] (4+0.24+0.6, 0.4) node{\small $T_M(\cdot)$};
\draw[->] (4+0.24 + 0.1, 0.12) -- (4+0.24+1.1, 0.12);
\end{tikzpicture} 
}

\subfigure[]{
\begin{tikzpicture}[scale=1]
\pgfmathsetseed{2}
\definecolor{myred}{HTML}{ae1908}
\definecolor{myblue}{HTML}{05348b}
\foreach \x in {-4, -3, -2, -1, 0, 1, 2, 3}
{
	\draw[black] (0, \x*0.3 + 0.15) rectangle +(0.24, 0.24);
}
\foreach \l in {1, 2, 3}
\foreach \x in {-2, -1, 0, 1, 2}
{
	\draw[black] (\l, \x*0.3) rectangle +(0.24, 0.24);
}
\draw[black] (4, 0) rectangle+(0.24, 0.24);
\foreach \x in {-4, -3, -2, -1, 0, 1, 2, 3}
	\foreach \y in {-2, -1, 0, 1, 2}
{
	\draw[->, myred] (0+0.24, \x*0.3+0.12+0.15) -- (1, \y * 0.3+0.12);
}

\foreach \l in {2,3}
	\foreach \x in {-2, -1, 0, 1, 2}
		\foreach \y in {-2, -1, 0, 1, 2}
		{
			\draw[->, myred] (\l-1+0.24, \x*0.3+0.12) -- (\l, \y * 0.3+0.12);
		}
		
\foreach \x in {-2, -1, 0, 1, 2}
	{
		\draw[->, myred] (4-1+0.24, \x*0.3+0.12) -- (4, 0 * 0.3+0.12);
	}
\draw[black] (-2.3 - 13 * 0.3 - 0.3, 0.3*3 + 0.15 + 0.24 + 0.3 - 0.06 + 0.18) node{$\bm{x}$};
\draw[black] (-2.3 - 13 * 0.3 - 0.03, 0.3*3 + 0.15 + 0.24 + 0.3-0.03) rectangle+(0.3 * 13, 0.3);
\foreach \y in {1, 2, 3, 4, 5, 6, 7, 8, 9, 10, 11, 12, 13}
	\draw[gray] (-2.3 - \y * 0.3, 0.3*3+0.15+0.24+0.3) rectangle+(0.24, 0.24);
\draw[myblue] (-2.3 - 13 * 0.3 - 0.3, 0.3*2+0.15 -0.3 - 0.06 + 0.48) node{$\bm{W}$};
\draw[myblue] (-2.3 - 13 * 0.3 - 0.03, 0.3*2+0.15 -0.3 - 0.03) rectangle+(0.3 * 13, 0.3*3);
\foreach \x in {1, 2, 3}
	\foreach \y in {1, 2, 3, 4, 5, 6, 7, 8, 9, 10, 11, 12, 13}
	{
		\pgfmathrandominteger{\a}{10}{70}
		\filldraw[myblue!\a] (-2.3 -\y * 0.3, \x * 0.3 +0.15) rectangle+ (0.24, 0.24);
	}
\draw[black] (-1.1, 0.4+0.3*2+0.15) node{\small $\frac{1}{p} {\color{myblue} \bm{W}}^\top \bm{x}$};
\draw[->] (-2.15, 0.3*2+0.15+0.12) -- (-0.2, 0.3*2+0.15+0.12);
\draw[black] (-0.05, 1* 0.3+0.12) -- (-0.1, 1*0.3+0.12) -- (-0.1, 3*0.3+0.24+0.12) -- (-0.05, 3*0.3+0.24+0.12);

\draw[myred] (-2.3 - 13 * 0.3 - 0.3, -0.3 - 0.15 -0.3 - 0.06 + 0.48) node{$\bm{\Theta}$};
\draw[myred] (-2.3 - 13 * 0.3 - 0.03, -0.3*2 - 0.15 -0.3 - 0.03) rectangle+(0.3 * 13, 0.3*5);
\foreach \x in {-4, -3, -2, -1, 0}
	\foreach \y in {1, 2, 3, 4, 5, 6, 7, 8, 9, 10, 11, 12, 13}
	{
		\pgfmathrandominteger{\a}{0}{3};
		\filldraw[myred!\a] (-2.3 -\y * 0.3, \x * 0.3 +0.15) rectangle+ (0.24, 0.24);
	}
\foreach \nn in {1, 2, 3, 4, 5}
{
	\pgfmathrandominteger{\a}{40}{70};
	\pgfmathrandominteger{\x}{-4}{0};
	\pgfmathrandominteger{\y}{1}{13};
	\filldraw[myred!\a] (-2.3 -\y * 0.3, \x * 0.3 +0.15) rectangle+ (0.24, 0.24);
}
\draw[black] (-1.7, -2*0.3+0.4+0.15) node{\small ${\color{myred} \bm{\Theta}}^\top \bm{x}$};
\draw[black] (-0.6, -2*0.3+0.4+0.15) node{\small $\bar{T}_M(\cdot)$};
\draw[->] (-2.15, -2*0.3+0.15+0.12) -- (-1.15, -2*0.3+0.15+0.12);
\draw[->] (-1, -2*0.3+0.15+0.12) -- (-0.2, -2*0.3+0.15+0.12);
\draw[black] (-0.05, -4 * 0.3+0.12) -- (-0.1, -4*0.3+0.12) -- (-0.1, 0.24+0.12) -- (-0.05, 0.24+0.12); 
\draw[black] (4+0.24+0.6, 0.4) node{\small $T_M(\cdot)$};
\draw[->] (4+0.24 + 0.1, 0.12) -- (4+0.24+1.1, 0.12);
\end{tikzpicture} 
}
\caption{Visualization of (a) the FAR-NN estimator and (b) the FAST-NN estimator. The red color represents weights to be learned from data $(\bx_1,y_1),\ldots, (\bx_n, y_n)$, while blue color represents pre-defined fixed weights.}
\label{fig:method}

\end{figure}

\subsection{Fitting Factor Augmented Sparse Throughput Models}
\label{sec:method:full}

Define the following clipped-$L_1$ function with the clipping threshold $\tau>0$,
\begin{align*}
    \psi_{\tau}(x) = \frac{|x|}{\tau} \land 1.
\end{align*} This function can be seen as a continuous relaxation of the indicator function $1_{\{x\neq 0\}}(x)$ as $\tau$ is very small in our application. It was proposed in \cite{fan2001variable,zhang2010analysis} for high-dimensional linear regression, and was also used by \cite{ohn2022nonconvex} to learn neural network with sparse weights. One can also take any rescaled folded concave penalty \citep{fan2001variable}.  For example, let $p_\tau(\cdot)$ be the SCAD penalty and take $\psi_\tau(x) = \tau^{-1} p_\tau(x)$.

Now consider the FAST model \eqref{eq:dgp-samples}.  One natural method to estimate the throughput $\bu_i$ is the residuals of fitting $\{\bx_i\}_{i=1}^n$ on $\{\tilde \bbf_i\}_{i=1}^n$ via linear regression.  However, the residuals are estimated with errors in high dimension and is challenging to analyze the error propagation through the neural networks.  To avoid these technical challenges, we create sparse linear combinations to select a subset of idiosyncratic throughputs.
This leads us to the consideration of the following penalized least squares
\begin{align}
\label{eq:fastnn-lse}
    \hat{g}(\cdot), \hat{\bTheta} \in \argmin_{g\in \mathcal{G}(L, \overline{r}+N, 1, N, M, B), \Theta\in \mathbb{R}^{p\times N}} \frac{1}{n} \sum_{i=1}^n \left\{y_i - g\left(\left[\tilde{\bbf}_i, \bar{T}_M(\bTheta^\top \bx_i)\right]\right)\right\}^2 + \lambda \sum_{i,j} \psi_\tau(\bTheta_{i,j}).
\end{align} where $[\bx, \by]$ concatenate two vectors $\bx\in \mathbb{R}^{d_1}$ and $\by\in \mathbb{R}^{d_2}$ together to form a $(d_1+d_2)$-dimensional vector, $\bar{T}_M(\cdot)$ is the truncation operator defined in Definition \ref{def:nn}, $\lambda$ and $\tau$ are tuning hyper-parameters. The final FAST-NN estimator is 
\begin{align}
\label{eq:fastnn-estimator}
    \hat{m}_{\mathtt{FAST}}(\bx) = \hat{g}\left(\left[p^{-1}\bW^\top \bx, \bar{T}_M(\hat{\bTheta}^\top \bx)\right]\right)
\end{align}
Fig \ref{fig:method}(b) visualizes the network architecture of the FAST-NN estimator.

\section{Theory}

Before presenting our theoretical results, we first impose some regularity conditions.

\begin{condition}[Boundedness] \label{cond1}
For factor model \eqref{eq:dgp-lfm}, there exists universal constants $c_1$ and $b$ such that
\begin{itemize}[noitemsep, topsep=0pt]
    \item[1.] The factor loading matrix satisfies $\|\bB\|_{\max}  \le c_1$.
    \item[2.] For each factor, $f_k$ is zero-mean and bounded in $[-b, b]$ for all the $k\in [r]$.
    \item[3.] For each idiosyncratic component, $u_j$ is zero-mean and bounded in $[-b, b]$ for all the $j\in [p]$.
\end{itemize}
\end{condition}

\begin{condition}[Weak dependence]\label{cond3}
$\sum_{j,k\in \{1,\ldots, p\}, j\neq k} \left|\mathbb{E}[u_j u_k]\right| \le c_1 \cdot p$ for some universal constant $c_1$.
\end{condition}

These conditions are standard, except we replace the sub-Gaussian condition of  $\bbf$ and  $\bu$ by the uniform boundedness. We make such modifications to adapt to the setting of non-parametric regression, in which the covariate is usually assumed to be bounded in order to have sufficient local data. For the theories developed in this section, we only impose weak dependence conditions on the idiosyncratic component (or covariates when $r=0$). Such an assumption is weaker than what is imposed in some high-dimensional nonparametric regression literature, for example, the restricted strong convexity condition \eqref{eq:high-dim-rsc}.
As for the non-parametric regression, we impose the following standard assumptions.

\begin{condition}[Sub-Gaussian noise]
\label{cond4}
    There exist a universal constant $c_1$ such that $\mathbb{P}(|\varepsilon| \ge t|\bbf,\bu) \le 2 e^{-c_1 t^2}$ for all the $t> 0$ almost surely.
\end{condition}

\begin{condition}[Regression function]
\label{cond5}
    The regression function $m^*$ satisfies $\|m^*\|_\infty \le M^*$ and $m^*$ is $c_1$-Lipschitz for some universal constants $M^*$ and $c_1$. We further assume that $1\le M^* \le M \le c_2 M^*$ for some universal constant $c_2>1$.
\end{condition}

In this section, we wish to establish high probability bounds on the following out-of-sample mean squared error and in-sample mean squared error
\begin{align*}
    \|\hat{m} - m^*\|_2^2 = \int \left|\hat{m}(\bx)-m^*(\bbf, \bu)\right|^2 \mu(d\bbf,d\bu) \qquad \|\hat{m} - m^*\|_n^2 = \frac{1}{n}\sum_{i=1}^n \left|\hat{m}(\bx_i)-m^*(\bbf_i, \bu_i)\right|^2 
\end{align*} for some estimator $\hat{m}(\bx)$ that can only get access to the covariate $\bx$. See the justifications for the importance of deriving high probability error bound on both in-sample and out-of-sample $L_2$ error bound in \cite{farrell2021deep}.

\subsection{Factor Augmented Regression Neural Network Estimator}
\label{sec:theory:fa}

Let $\bW$ be the diversified projection matrix according to Definition \ref{def:dpm}. Define the empirical $L_2$ loss:
\begin{align*}
    \hat{\mathsf{R}}_{\mathtt{FAR}}(g) = \frac{1}{n} \sum_{i=1}^n \left\{y_i - g\left(p^{-1} \bW^\top \bx_i\right) \right\}^2.
\end{align*} 
For arbitrary given neural network hyper-parameters $L$ and $N$, we suppose that our FAR-NN estimator is an approximate empirical loss minimizer, i.e., 
\begin{align}
\label{eq:thm-fa:opt-error}
    \hat{m}_{\mathtt{FAR}}(\bx) = \hat{g}(p^{-1}\bW^\top \bx) ~~~~ \text{with} ~~~~ \hat{\mathsf{R}}_{\mathtt{FAR}}(\hat{g}) \le \inf_{g\in \mathcal{G}(L, \overline{r}, 1, N, M, \infty)}\hat{\mathsf{R}}_{\mathtt{FAR}}(g) + \delta_{\mathtt{opt}}
\end{align} with some optimization error $\delta_{\mathtt{opt}}$. We first present an oracle-type inequality for the error bound on the mean squared error of the FAR-NN estimator.

\begin{theorem}[Oracle-type inequality for FAR-NN estimator]
\label{thm:fa-oracle}
Assume Conditions \ref{cond1}--\ref{cond5} hold with $\|\bu\|_\infty \le b$ being replaced by $\max_{j\in [p]} \mathbb{E}|u_j|^2 \le c_1$ for a universal constant $c_1$. 
Consider the FAST model \eqref{eq:dgp-samples} with $\mathcal{J}=\emptyset$ (the FAR model) and FAR-NN estimator $\hat{m}_{\mathtt{FAR}}(\bx)$ in \eqref{eq:thm-fa:opt-error} with number of projections $\bar r \geq r$. 
Define 
\begin{align*}
    \delta_{\mathtt{a}} = \inf_{g\in \mathcal{G}(L, r, 1, N, M, \infty)} \|g - m^*\|_{\infty}^2, ~~~ \delta_{\mathtt{s}} = (N^2L^2+\overline{r}NL) \log\left(NL\overline{r}\right) \frac{\log n}{n}, ~~~ \delta_{\mathtt{f}} = \frac{\overline{r}}{p \cdot \nu^2_{\min}(\bH)}.
\end{align*}
Then with probability at least $1-3e^{-t}$, for $n$ large enough,
\begin{align*}
    \|\hat{m}_{\mathtt{FAR}}-m^*\|^2_2 + \|\hat{m}_{\mathtt{FAR}}-m^*\|^2_n \le c_2 \left\{\delta_{\mathtt{opt}} + \delta_{\mathtt{a}} + \delta_{\mathtt{s}} + \delta_{\mathtt{f}} + \frac{t}{n}\right\}
\end{align*} for a universal constant $c_2$ that only depends on $c_1$ and constants in Condition \ref{cond1}--\ref{cond5}.
\end{theorem}

Theorem \ref{thm:fa-oracle} establishes a high probability bound on both the out-of-sample mean squared error and the in-sample mean squared error, we use `$+$' here for a clear presentation. From Theorem \ref{thm:fa-oracle}, the error bound is composed of four terms: the optimization error $\delta_{\mathrm{opt}}$, the neural network approximation error $\delta_{\mathrm{a}}$ to the underlying function $m^*$, the stochastic error $\delta_{\mathrm{s}}$ scales linearly with $n^{-1} \log n$ and $(N^2L^2+\overline{r}NL) \log\left(NL\overline{r}\right)$, which is proportional to the Pseudo-dimension of the neural network class we used, and the error $\delta_{\mathrm{f}}$ related to inferring the latent factors $\bbf$ from the observations $\bx$ and it scales linearly with $\overline{r}$. Such an error bound is not applicable without specifying the network hyper-parameters $N$ and $L$. An optimal rate can be further obtained by choosing $N$ and $L$ to trade off the approximation error $\delta_{\mathtt{a}}$ and the stochastic error $\delta_{\mathtt{s}}$. Moreover, it should be noted that Theorem \ref{thm:fa-oracle} provides a generic result that
also works for other non-parametric estimators that use function class with finite Pseudo-dimension, for example, the spline method. To adapt to the error bound for those methods, the only change is to replace the stochastic error term by $n^{-1} \{\mathrm{Pdim}(\mathcal{G}) \log n\}$ for the function class $\mathcal{G}$ used in estimation.

It is known from \cite{schmidt2020nonparametric} and \cite{kohler2021rate} that deep ReLU networks can be adaptive to unknown hierarchical composition structures. The following corollary claims that in the high-dimensional nonparametric regression model \eqref{eq:dgp-samples}, our proposed FAR-NN estimator can be efficiently adaptive to the hierarchical composition structure of $m^*$ in the same way. We first introduce the concept of the hierarchical composition model, which is a composition of 
$(\beta, C)$-smooth functions for $(\beta,t)$ in a given finite set $\mathcal{P}$. 

\begin{definition}[Hierarchical composition model]
\label{hcm} The function class of hierarchical composition model $\mathcal{H}(d, l, \mathcal{P})$ \citep{kohler2021rate}, with $l, d \in \mathbb{N}^+$ and $\cP$, a subset of $[1,\infty) \times \mathbb{N}^+$ satisfying $\sup_{(\beta, t)\in \mathcal{P}} (\beta \lor t) < \infty$, is defined as follows. 
For $l=1$, 
\begin{align*}
\mathcal{H}(d, 1,\mathcal{P}) = \big\{  &h: \mathbb{R}^d \to \mathbb{R}: h(\bx) = g(x_{\pi(1)},...,x_{\pi(t)})\text{, where} \\
&~~~~~~~~g:\mathbb{R}^t \to \mathbb{R}\text{ is }(\beta, C)\text{-smooth for some } (\beta,t)\in \mathcal{P}  \text{ and }\pi: [t]\to [d] \big\}.
\end{align*} 
It consists of all $t$-variate functions with $(\beta, C)$ smoothness with a positive constant $C$.
For $l>1$, $\mathcal{H}(d, l,\mathcal{P})$ is defined recursively as
\begin{align*}
\mathcal{H}(d, l,\mathcal{P}) = \big\{&h: \mathbb{R}^d \to \mathbb{R}: h(\bx) = g(f_1(\bx),...,f_t(\bx))\text{, where} \\
&~~~~~g:\mathbb{R}^t \to \mathbb{R}\text{ is }(\beta, C)\text{-smooth for some } (\beta,t)\in \mathcal{P}  ~~\text{and}~~  f_i \in \mathcal{H}(d, l-1,\mathcal{P})  \big\}.
\end{align*}
\end{definition}

The minimax optimal estimation risk over $\mathcal{H}(d, l,\mathcal{P})$ is determined by the hardest component in the composition, which can be characterized via the following quantity
\begin{align}
    \label{eq:hcm:gamma-def}
    \gamma^* = \frac{\beta^*}{d^*} ~~~~\text{with}~~~~ (\beta^*,d^*)=\argmin_{(\beta, t)\in \mathcal{P}} \frac{\beta}{t}.
\end{align} Following \cite{kohler2021rate}, we restrict to the case where all the compositions has smoothness parameter $\beta\ge 1$ to simplify the presentation. The optimal rate for the FAR-NN estimator when $m^* \in \mathcal{H}(r, l,\mathcal{P})$ is stated as follows.

\begin{corollary}[Optimal rate for FAR-NN estimator]
\label{coro:roc-fa-nn}
Let $\delta_{n} = n^{-\frac{2\gamma^*}{2\gamma^*+1}} (\log n)^{\frac{12\gamma^*}{2\gamma^*+1}}$ with $\gamma^* = \frac{\beta^*}{d^*}$.  If we choose $r \le \bar{r} \lesssim 1$ and $NL \asymp n^{\frac{1}{4\gamma^*+2}} (\log n)^{\frac{4\gamma^*-1}{2\gamma^*+1}}$, then under the conditions in Theorem \ref{thm:fa-oracle}, the FAR-NN estimator with $\delta_{\mathtt{opt}} \lesssim \delta_n$ satisfies, with probability at least $1-3e^{-t}$, for $n$ large enough, 
    \begin{align*}
        \sup_{m^* \in \mathcal{H}(r,l,\mathcal{P})} \|\hat{m}_{\mathtt{FAR}}-m^*\|^2_2 + \|\hat{m}_{\mathtt{FAR}}-m^*\|^2_n \le c_1 \left(\delta_n + \delta_{\mathtt{f}} + \frac{t}{n}\right),
    \end{align*}
for some universal constant $c_1$ independent of $n, p, t$ and choice of $\bW$.
\end{corollary}

Note that $\delta_n$ is the optimal excess risk (except a poly-logarithm term) associated with the hardest component in $\mathcal{H}(d, l,\mathcal{P})$ when the latent factor $\bbf$ is observable and used as the input of the deep ReLU network. Therefore, Corollary \ref{coro:roc-fa-nn} asserts that if $\bx$ indeed admits linear factor model structure \eqref{eq:dgp-lfm}, and our choice of diversified projection matrix is near-optimal such that $\nu_{\min}(\bH) \asymp 1$, then
\begin{align}
\label{eq:roc-fa-nn-final}
    \int \left|\hat{m}_{\mathtt{FAR}}(\bx)-m^*(\bbf)\right|^2 \mu(d\bbf, d\bu) \lesssim \delta_n + \frac{1}{p} + \frac{t}{n},
\end{align} with probability at least $1-3e^{-t}$ (with $t=10 \log n$, say, the probability is at least $1-3n^{-10}$ and error contribution is only of order $(\log n)/n$ in \eqref{eq:roc-fa-nn-final}). We can see that the error bound is dominated by $\delta_n$ in the high-dimensional regime $p\asymp  n$. Hence, it's evident that the FAR-NN estimator can achieve an oracle convergence rate as if the factor $\bbf$ is observable and the composition structure is known. 

A natural question is whether we can choose the diversified projection matrix such that the condition $\nu_{\min}(\bH) \asymp 1$ is satisfied. The following proposition gives an affirmative answer under the additional assumptions on the distribution of $(\bbf, \bu)$.

\begin{condition}[Pervasiveness]
\label{cond:pervasiveness}
There exists a universal constant $c_1\ge 1$ such that $p / c_1< \lambda_{\min}(\bB^\top \bB) \le \lambda_{\max}(\bB^\top \bB) \le c_1 p$.
\end{condition}

\begin{condition}[Weak dependence between $\bbf$ and $\bu$]
\label{cond:dist-fu}
There exists a universal constant $c_1$ such that $\|\bB \bSigma_{\bbf, \bu}\|_{F} \le c_1 \sqrt{p}$ where $\bSigma_{\bbf, \bu} = \mathbb{E} [\bbf \bu^T] \in \mathbb{R}^{r\times p}$ is the covariance matrix between $\bbf$ and $\bu$.
\end{condition}

\begin{proposition}
\label{prop:choice-dpm}
Suppose $\bx_1,\ldots, \bx_n$ are i.i.d. copies of $\bx$ in model \eqref{eq:dgp-lfm}. Let $\hat{\bSigma}=\frac{1}{n} \sum_{i=1}^n \bx_i \bx_i^\top$, $\overline{r} \ge r$, and $\hat{\bv}_1,\ldots, \hat{\bv}_{\overline{r}}$ be the top-$\overline{r}$ eigenvectors of $\hat{\bSigma}$.
Then, under Conditions \ref{cond1}, \ref{cond3}, \ref{cond:pervasiveness} and \ref{cond:dist-fu},
with probability at least $1-3e^{-t}$, the matrix $\tilde{\bW}=\sqrt{p} \left[\hat{\bv}_1, \ldots, \hat{\bv}_{\overline{r}}\right] \in \mathbb{R}^{p\times \overline{r}}$ satisfies
\begin{align}
\label{eq:prop:choice-dpm:bd}
    c_1 - c_2 \left(r\sqrt{\frac{\log p + t}{n}} + r^2\sqrt{\frac{\log r + t}{n}} + \frac{1}{\sqrt{p}}\right)\le \nu_{\min}(p^{-1}\tilde{\bW}^\top \bB) \le \nu_{\max}(p^{-1}\tilde{\bW}^\top \bB) \le c_3 
\end{align} for some universal constants $c_1$--$c_3$ that independent of $n, p, t, r, \overline{r}$.
\end{proposition}

Combining Proposition \ref{prop:choice-dpm} and Corollary \ref{coro:roc-fa-nn}, when $r\asymp \overline {r}\asymp 1$, we can build a FAR-NN estimator from scratch which is capable of achieving the convergence rate \eqref{eq:roc-fa-nn-final} if $n \gg (\log p + t)$ and $p \gg 1$.  As noted before, $t$ is usually of order $\log n$, and hence a small fraction of the training sample suffice to learn good diversified weights with $\nu_{\min}(\bH) \asymp 1$.  More specifically, given the observations $(\bx_1,y_1), (\bx_2,y_2),\ldots, (\bx_n, y_n)$, we first divide the whole dataset into two subsets $\mathcal{D}_1 = \{(\bx_i,y_i)\}_{i=1}^{n_1}$ and $\mathcal{D}_2 = \{(\bx_i, y_i)\}_{i=n_1+1}^n$. We use a small set $\mathcal{D}_1$ to learn the diversified projection matrix $\bW=\tilde{\bW}$ using PCA as described in Proposition \ref{prop:choice-dpm} and use the big set $\mathcal{D}_2$ together with the diversified projection matrix $\bW$ to learn our final estimator $\hat{m}_{\mathtt{FAR}}$. By using Proposition \ref{prop:choice-dpm} and the fact $n_1 \gg (\log p + t)$ and $p\gg 1$, with probability at least $1-3e^{-t}$, the diversified projection matrix $\bW$ satisfies $\nu_{\min}(p^{-1} \bW^\top \bB) \asymp 1$. Having such a good diversified projection matrix in hand, Corollary \ref{coro:roc-fa-nn} claims that the FAR-NN estimator $\hat{m}_{\mathtt{FAR}}$ built from pre-trained $\bW$ with tolerable optimization error has an error bound of \eqref{eq:roc-fa-nn-final} with probability at least $1-3e^{-t}$. From the above discussion, we can assign $n_1\asymp \sqrt{n}$ because the requirement of diversified projection matrix is milder and we can use almost all the data to learn the regression function.

How about jointly estimating the diversified projection matrix and the neural network weights using a unified objective function? For example, consider the following optimization problem
\begin{align}
\label{eq:far-joint-opt}
    \hat{\bW}, \hat{g}(\cdot) = \argmin_{\bW\in \mathbb{R}^{p\times \overline{r}}, g\in \mathcal{G}(L, \overline{r}, 1, N, M, \infty)} \frac{1}{n} \sum_{i=1}^n \left\{y_i - g\left(p^{-1} \bW^\top \bx_i\right) \right\}^2,
\end{align}
can we achieve a comparable rate of convergence if we use $\hat{m}(\bx) = \hat{g}(p^{-1} \hat{\bW}^\top \bx)$? Now it is an under-determined system that has the number of freedoms (number of parameters) more than the number of constraints (number of data). It is not surprising that there might exist some unreliable solution $(\hat{\bW}, \hat{g})$ that can perfectly fit all the data. We claim that even with the help of some implicit regularizations such as the minimum $\ell_2$ norm solution, the above estimate will achieve, in some setups, a slower convergence rate compared with \eqref{eq:roc-fa-nn-final} for the FAR-NN estimator even for the most simple `linear neural network' class. To formally present the idea, we consider the following special case
\begin{align}
\label{eq:sub-optimality-ls-model}
    \bx_i = \bB \bbf_i + \bu_i ~~~~ \text{and} ~~~~ y_i = \varepsilon_i
\end{align} such that the factor $\bbf$ has i.i.d. $\text{Unif}[-\sqrt{3}, \sqrt{3}]$ entries, the idiosyncratic component vector $\bu$ has i.i.d. standard normal entries, and the noise $\varepsilon$ can be arbitrary bounded distribution satisfying $\mathbb{E}[\varepsilon|\bbf,\bu]=0$ and $\mathbb{E}[|\varepsilon|^2|\bbf,\bu] = \sigma^2$. The distribution of $(\bbf, \bu, \varepsilon)$ constructed above lies in the regime where Theorem \ref{thm:fa-oracle} is applicable, which allows us to compare the theoretical performance of FAR-NN estimator and the estimator \eqref{eq:far-joint-opt} under a same scenario. We consider the minimum $\ell_2$ norm estimator over the linear function class $\mathcal{F}_{\mathrm{linear}} = \{g(\bx) = \bbeta^\top \bx: \bbeta \in \mathbb{R}^p\}$, a special case of \eqref{eq:far-joint-opt} by observing that 
\begin{align*}
    \mathcal{F}_{\mathrm{linear}} = \left\{m(\bx) = g(p^{-1}\bW^\top \bx): \bW \in \mathbb{R}^{p\times \overline{r}}, g\in \mathcal{G}(0, \overline{r}, 1, N, \infty, \infty)\right\}
\end{align*} for any $\overline{r} \in \mathbb{N}^+$. We have the following proposition characterizing the lower bound on the excess risk of the minimum $\ell_2$ norm least squares estimator, which is defined as
\begin{align}
\label{eq:sub-optimality-ls-estimator}
    \hat{\bbeta}_{\mathtt{LS}} = \argmin\left\{\|\bbeta\|_2: \bx_i^\top \bbeta = y_i,~~ \forall i \in \{1,\ldots,n\}\right\}.
\end{align}

\begin{proposition}[Sub-optimality of minimum $\ell_2$ norm least squares in null case]
\label{thm:sub-optimality-least-squares}
Consider the model \eqref{eq:sub-optimality-ls-model} and $y_i = {\bm{0}}^T \bbf_i + \varepsilon_i$, and the the minimum $\ell_2$ norm least squares estimator $ \hat{\bbeta}_{\mathtt{LS}}$ in \eqref{eq:sub-optimality-ls-estimator}. There exist a universal constant $c_1$ such that if $p > n+r$ and $r<n/c_1$, then for $n$ large enough,
\begin{align*}
    \mathbb{E}_{(\bx_1,y_1),\ldots,(\bx_n,y_n)} \left[\int \left| \hat{\bbeta}_{\mathtt{LS}}^\top \bx - 0\right|^2 d\mu(\bbf,\bu)\right] \gtrsim \sigma^2\left(\frac{r}{n} + \frac{n}{p}\right).
\end{align*}
\end{proposition}

If $r\asymp 1$, the convergence rate for $\hat{\bbeta}_{\mathtt{LS}}$ in \eqref{eq:sub-optimality-ls-estimator} is lower bounded by $n^{-1} + p^{-1}n$, and the convergence rate for the FAR-NN estimator is upper bounded by $n^{-1} + p^{-1}$. We can see that when $n,p \to \infty$ and $n/p \to \gamma \in (0,1)$, the former is even not consistent, while the FAR-NN estimator can achieve an $n^{-1}$ rate. This demonstrates the necessity of including a pre-trained diversified projection matrix.

\subsection{Factor Augmented Sparse Throughput Neural Network Estimator}
\label{sec:theory:fast}

For a given diversified projection matrix $\bW$, a deep ReLU network $g(\cdot)\in \mathcal{G}(L, \overline{r}+N, 1, N, M, B)$, a variable selection matrix $\bTheta\in \mathbb{R}^{p\times N}$, let the prediction of the FAST-NN model be $m(\bx) = m(\bx;\bW,g,\bTheta) = g\left([p^{-1}\bW^\top \bx, \bar{T}_M(\bTheta^\top \bx)]\right)$, and define the associated the empirical loss as
\begin{align*}
    \hat{\mathsf{R}}_{\mathtt{FAST}}(m) = \frac{1}{n} \sum_{i=1}^n (m(\bx_i) - y_i)^2 + \lambda \sum_{i,j} \psi_\tau(\Theta_{i,j}).
\end{align*}

We suppose that our FAST-NN estimator is an approximate empirical loss minimizer, that is, $\hat{m}_{\mathtt{FAST}}(\bx) = m(\bx;\bW,\hat{g}, \hat{\bTheta})$, where $\hat{g}$ and $\hat{\bTheta}$ satisfies
\begin{align}
\label{eq:fast-nn-estimator}
    \hat{\mathsf{R}}_{\mathtt{FAST}}\Big(\hat{m}_{\mathtt{FAST}}(\cdot;\bW,\hat{g}, \hat{\bTheta})\Big) \le \inf_{\substack{\bTheta \in \mathbb{R}^{p\times N} \\ g\in \mathcal{G}(L, \overline{r}+N, 1, N, M, B) }} \hat{\mathsf{R}}_{\mathtt{FAST}}\Big(m(\cdot;\bW,g,\bTheta)\Big) + \delta_{\mathtt{opt}}
\end{align} for some optimization error $\delta_{\mathtt{opt}}$. The following theorem presents an oracle-type inequality for the error bound for the FAST-NN estimator.

\begin{theorem}[Oracle-type inequality for FAST-NN estimator]
\label{thm:fast-oracle} 
Assume Conditions \ref{cond1}--\ref{cond5} hold. Consider the FAST model \eqref{eq:dgp-samples} and the FAST-NN estimator solving \eqref{eq:fast-nn-estimator} with $N \ge 2(r+|\mathcal{J}|)$, $B \ge c_1 [\nu_{\min}(\bH)]^{-1} |\mathcal{J}|r$, 
    \begin{align}
    \label{eq:choice-lambda-tau}
        \lambda \ge c_2 \frac{\log (np(N+\overline{r})) + L\log (BN)}{n}, \qquad \tau^{-1} \ge c_3 (r+1) (BN)^{L+1} (N+\overline{r}) p n
    \end{align} for some universal constants $c_1$--$c_3$, and the number of diversified projections $\bar r \geq r$. Define $\delta_{\mathtt{a}} = \inf_{g\in \mathcal{G}(L-1, r+|\mathcal{J}|, 1, N, M, B)} \|g - m^*\|_\infty^2$,  $\delta_{\mathtt{s}} = (N^2L + N\overline{r}) \{L\log (BNn)\}/n + \lambda|\mathcal{J}|$, and $\delta_{\mathtt{f}} = |\mathcal{J}|r\cdot \overline{r} / \{\nu^2_{\min}(H) \cdot p\}$. Then, with probability at least $1-3e^{-t}$, the following holds, for $n$ large enough,
    \begin{align*}
        \|\hat{m}_{\mathtt{FAST}}-m^*\|^2_2 + \|\hat{m}_{\mathtt{FAST}}-m^*\|^2_n \le c_4 \left\{ \delta_{\mathtt{opt}} + \delta_{\mathtt{a}} + \delta_{\mathtt{s}} + \delta_{\mathtt{f}} + \frac{t}{n} \right\}
    \end{align*} where $c_4$ is a universal constant that depends only on the constants in Condition \ref{cond1}--\ref{cond5}.
\end{theorem}

The optimal rate for the FAST-NN estimator when the regression function $m^*(\bbf, \bu_\mathcal{J})$ admits a hierarchical composition structure, i.e., $m^*(\bbf, \bu_\mathcal{J}) \in \mathcal{H}(r+|\mathcal{J}|, l, \mathcal{P})$ can be then obtained under a careful choice of hyper-parameters $(N,L,\bar{r},B)$ summarized below.

\begin{condition}
\label{cond:fast-nn}
The following conditions on the deep ReLU network architecture hyper-parameters hold

\noindent  (1) $c_1 \le L \lesssim 1$; ~~~~ (2) $c_2 \left(n/\log n\right)^{\frac{1}{4\gamma^*+2}} \le N \lesssim \left(n/\log n\right)^{\frac{1}{4\gamma^*+2}}$ where $\gamma^*$ is given by \eqref{eq:hcm:gamma-def};
    
\noindent (3) $c_3 \{\log n \lor \log [\nu_{\min}(\bH)]^{-1}\}\le \log B \lesssim \log n$; ~~~~ (4) $r\le \overline{r} \lesssim r + 1$

\noindent for some universal constant $c_1$--$c_3$ which only depends on $l$ and $\mathcal{P}$ of $\mathcal{H}(r+|\mathcal{J}|, l,\mathcal{P})$.
\end{condition}

These requirements are mild. In conditions (1)-(3), the constants $c_1$--$c_3$ in the lower bound part are specified by our neural network approximation result depicted in Theorem \ref{thm:approx-smooth}. This ensures that as $n\to \infty$, the approximation error decays at the rate of\begin{align*}
    \sup_{m^* \in \mathcal{H}(r+|\mathcal{J}|, l,\mathcal{P})} \inf_{g\in \mathcal{G}} \|g - m^*\|_\infty^2 \lesssim (NL)^{-4\gamma^*} \lesssim \left(\frac{\log n}{n}\right)^{\frac{2\gamma^*}{2\gamma^*+1}},
\end{align*} while the upper bound of the hyper-parameters controls the stochastic error and allow it to decay at the same rate as $n\to \infty$. It is worth pointing out that because of our new, tighter neural network approximation result depicted in Theorem \ref{thm:approx-smooth}, we adopt $\mathcal{O}(1)$ depth architecture. We can obtain a faster convergence rate and milder choice of $\tau$ with an $\mathcal{O}(1)$ depth ReLU neural network from a theoretical perspective. One can also choose diverging depth and get a similar result to Theorem~\ref{thm:fast-roc} below by Theorem \ref{thm:fast-oracle}, but that would lead to a slower convergence rate (comparing the logarithmic factors) and sharper choice of $\tau$ according to \eqref{eq:choice-lambda-tau}. Moreover, condition (4) implies that we do not need to precisely determine the number of latent factors. We can incorporate more diversified projections instead. That will not affect the error bound if $\overline{r}$ has the same order as $r+1$, which includes the case where $r=0$ and we use constant order $\overline{r}$.

\begin{theorem}[Optimal rate for FAST-NN estimator]
\label{thm:fast-roc}
Assume Conditions \ref{cond1}--\ref{cond5} hold.
Consider the FAST model \eqref{eq:dgp-samples} with $r+ |\mathcal{J}|\le c_1$ for some universal constant $c_1$, and the FAST-NN estimator $\hat{m}_{\mathtt{FAST}}(\bx) = m(\bx;\bW,\hat{g}, \hat{\bTheta})$ that satisfies Condition \ref{cond:fast-nn} and solves \eqref{eq:fast-nn-estimator} with 
\begin{align*}
    \lambda \ge c_2 \frac{\log (pn)}{n} \qquad \text{and} \qquad \tau^{-1} \ge n^{c_3} p
\end{align*}
for some universal constants $c_2$--$c_3$ independent of $n$ and $p$.
Then with probability at least $1-3e^{-t}$, for large enough $n$ and any $m^* \in \mathcal{H}(r+|\mathcal{J}|,l,\mathcal{P})$,
    \begin{align*}
        \|\hat{m}_{\mathtt{FAST}}-m^*\|^2_2 + \|\hat{m}_{\mathtt{FAST}}-m^*\|^2_n \le c_4 \left\{ \delta_{\mathtt{opt}} + \left(\frac{\log n}{n}\right)^{\frac{2\gamma^*}{2\gamma^*+1}} + \lambda + \frac{1 \land r}{\nu^2_{\min}(\bH) \cdot p} + \frac{t}{n} \right\},
    \end{align*} 
for some constant $c_4$ only depending on $c_1$ and constants in Condition \ref{cond1}--\ref{cond5} and \ref{cond:fast-nn}.
\end{theorem}

With the optimal choice of the hyper-parameters $\lambda$ and $\tau$, the convergence rate for the FAST-NN estimator is determined by
\begin{align}
\label{eq:roc-fast}
    \delta_{\mathtt{FAST}} \asymp \left(\frac{\log n}{n}\right)^{\frac{2\gamma^*}{2\gamma^*+1}} + \frac{\log p}{n} + \frac{1 \land r}{\nu^2_{\min}(\bH) \cdot p},
\end{align} if the optimization error $\delta_{\mathtt{opt}} \lesssim \delta_{\mathtt{FAST}}$ and $t \asymp \log p$. In the next subsection, we will show $\delta_{\mathtt{FAST}}$ is also minimax optimal up to logarithmic factors of $n$. It consists of three terms in \eqref{eq:roc-fast}. The first one is the minimax risk of estimating a function with hierarchical composition structure $\mathcal{H}(r+|\mathcal{J}|,l,\mathcal{P})$ in a low dimension regime where both $r+|\mathcal{J}|$ and $l$ are fixed, i.e., do not grow with $n$, and the set $\mathcal J$ is known. The second term is a typical risk related to variable selection uncertainty, which cannot be improved even in a linear regression model. The third term is related to the factor estimation error and it is $0$ if there is no latent factor ($r=0$). Its minimax optimality is deferred to the next subsection.

When $r=0$, the FAST model is reduced to a nonparametric sparse regression model, we have the following corollary, in which the third term of the convergence rate disappears.
\begin{corollary} 
\label{coro:fast-r0}
Consider the FAST model \eqref{eq:dgp-samples} with $r=0$ and the FAST-NN estimator solving \eqref{eq:fast-nn-estimator}. Under the settings of Theorem \ref{thm:fast-roc}, with probability at least $1-3e^{-t}$, the following holds
\begin{align*}
        \|\hat{m}_{\mathtt{FAST}}-m^*\|^2_2 + \|\hat{m}_{\mathtt{FAST}}-m^*\|^2_n \lesssim \left\{ \delta_{\mathtt{opt}} + \left(\frac{\log n}{n}\right)^{\frac{2\gamma^*}{2\gamma^*+1}} + \lambda + \frac{t}{n} \right\}.
\end{align*}
\end{corollary}

It is worth comparing the first term of $\delta_{\mathtt{FAST}}$, i.e., $\left(\frac{\log n}{n}\right)^{\frac{2\gamma^*}{2\gamma^*+1}}$, with previous non-parametric regression rate using ReLU neural networks. We have the fastest rate regarding the logarithmic factor $\log n$. This is attributed to the finer neural network approximation capacity presented Theorem \ref{thm:approx-smooth}. 

The use of bounded weights is to control the sparsity. Here we allow the magnitude of weights to grow polynomially with $n$, which matches what it is in practice, for example, gradient descent with polynomial iterations. From a theoretical view, with a close inspection of the following Theorem \ref{thm:fast-oracle} and its proof, the choice of $\tau$ is to make sure
\begin{align}
\label{eq:reg-effect-tau}
    \left| \hat{m}(\bx;\bW,\bTheta,g) - \hat{m}(\bx;\bW,\bTheta+\bDelta,g) \right| \lesssim n^{-1} ~~~~~ \text{for any } \bx\text{ and } \|\bDelta\|_{\max} \le \tau,
\end{align} in other words, a perturbation of the variable selection matrix $\bTheta$ within the range of $\tau$ can only lead to a change of output prediction no more than $n^{-1}$. This explains why we need to use deep ReLU networks with bounded weights because it is impossible to control the change in prediction for deep ReLU networks with no constraints on their weights. When $p\asymp n^C$ for some constant $C$, we can choose $\log (\tau^{-1}) \asymp \log n$.  This is an improvement over previous results for clipped-$L_1$ penalty which requires  $\log(\tau^{-1}) \asymp (\log n)^2$ \citep{ohn2022nonconvex}.

The proof of Theorem \ref{thm:fast-roc} is based on Theorem \ref{thm:fast-oracle} and the following neural network approximation result Theorem \ref{thm:approx-smooth}; the novelties and improvements of Theorem \ref{thm:approx-smooth} are discussed in Section 
\ref{dis:approx}.

\begin{theorem}
\label{thm:approx-smooth}
Let $g$ be a $d$-variate, $(\beta, C)$-smooth function. There exists some universal constants $c_1$--$c_5$ depending only on $d, \beta, C$, such that for arbitrary $N\in \mathbb{N}^+ \setminus \{1\}$, there exists a deep ReLU network $g^\dagger \in \mathcal{G}(c_1, d, 1, c_2 N, \infty, c_3 N^{c_4})$ satisfying
\begin{align*}
    \|g^\dagger-g\|_{\infty,[0,1]^d} \le c_5 N^{-2\beta/d}.
\end{align*}
Furthermore, if $g\in \mathcal{H}(d, l,\mathcal{P})$ with $\sup_{(\beta, t)\in \mathcal{P}} (\beta \lor t) <\infty$ and $g$ is supported on $[-c_6, c_6]^d$ for some constant $c_6$. There also exists some universal constants $c_{7}$--$c_{11}$ such that for arbitrary $N\in \mathbb{N}^+ \setminus \{1\}$, there exists a deep ReLU network $g^\dagger \in \mathcal{G}(c_7, d, 1, c_8 N, \infty, c_9 N^{c_{10}})$ satisfying
\begin{align*}
    \|g^\dagger-g\|_{\infty,[-c_{6},c_6]^d} \le c_{11} N^{-\inf_{(\beta, t) \in \mathcal{P}}(2\beta/t)}.
\end{align*}
\end{theorem}

\subsection{Minimax optimal lower bound}
\label{sec:theory:lb}

The optimal rate of the FAST-NN estimator contains a $p^{-1}$ term when $\nu_{\min}(\bH) \asymp 1$. This term may be a dominating term when $p \ll (n/\log n)^{\frac{2\gamma^*}{2\gamma^*+1}}$. Such error term $p^{-1}$ arises from the fact that we use an estimate of $(\bbf, \bu)$ from the observation $\bx$ instead of directly getting access to the latent variables $(\bbf, \bu)$, and can be intuitively interpreted as the ``error of estimating the factors $\bbf$''. Is it possible to achieve a faster convergence rate by using a more sophisticated algorithm when $p^{-1}$ is the dominating term? The following Lemma provides an answer.

\begin{lemma}
\label{lemma:error-factor-lb}
For any given $\lambda \in \mathbb{R}^+$, $p, r\in \mathbb{N}^+$ with $p \ge (r\lor \lambda)$, we have
\begin{align}
\label{eq:error-factor-lb}
    \inf_{m: \mathbb{R}^p \to \mathbb{R}} \sup_{\substack{\mu \in \mathcal{P}(p,r,\lambda)\\ m^* \text{linear, } 1\text{-Lipschitz}} } \int |m(\bx) - m^*(\bbf)|^2 \mu(d\bbf,d\bu) \gtrsim \frac{1}{\lambda} 
\end{align} where $\mathcal{P}(p,r,\lambda)$ is a family of distributions of $(\bbf, \bu, \bx)$ defined as
\begin{align}
\label{eq:lb-factor-dist}
\begin{split}
    \mathcal{P}(p,r,\lambda) = \Big\{\mu(\bbf, \bu, \bx): &\supp(\bbf) \subset [-1,1]^r, \supp(\bu) \subset [-1,1]^p, \mathbb{E}[\bbf] = \bm{0}, \mathbb{E}[\bu] = \bm{0},\\
    &~~~~ \bbf \text{ and } \bu \text{ independent, both have independent components,}\\
    &~~~~ \bx = \bB \bbf + \bu \text{ with } \|\bB\|_{\max} \le 1 \text{ and }\lambda_{\min}(\bB^\top \bB) \ge \lambda \Big\}.
\end{split}
\end{align} 
\end{lemma}

Under the regime in which Conditions \ref{cond1}--\ref{cond5} hold, Lemma \ref{lemma:error-factor-lb} affirms that any estimator getting access to $\bx$ is unable to achieve an error rate faster than $\lambda_{\min}(\bB^\top \bB)^{-1}$. Such a lower bound matches the error term $\delta_{\mathrm{f}}$ in our analysis of the FAR-NN estimator and the FAST-NN estimator under the pervasiveness condition $\lambda_{\min}(\bB^\top \bB) \asymp p$. With the help of Lemma \ref{lemma:error-factor-lb}, we next show that $\delta_{\mathtt{FAST}}$ in \eqref{eq:roc-fast} is the minimax optimal lower bound up to logarithmic factors of $n$. 

\begin{theorem}
\label{thm:fast-lb}
Consider i.i.d. samples $(\bx_1,y_1),\ldots, (\bx_n, y_n)$ from the FAST model \eqref{eq:dgp-samples} with $\varepsilon_1,\ldots, \varepsilon_n \overset{i.i.d.}{\sim} \mathcal{N}(0,1)$. Suppose further that $d^* \le r+1$ with $d^*$ given by \eqref{eq:hcm:gamma-def}. Then, for $n$ large enough, 
\begin{align*}
    \inf_{\hat{m}} \sup_{\substack{\mu \in \mathcal{P}(p,r,\lambda), \mathcal{J} \subset [p], |\mathcal{J}| = 1 \\ m^*(\bbf, \bu_{\mathcal{J}}) \in \mathcal{H}(r+1,l,\mathcal{P})}} \mathbb{E} \left[ \int \big|\hat{m}(\bx) - m^*(\bbf, \bu_{\mathcal{J}})\big|^2 \mu(d\bbf, d\bu)\right] \ge c_1 \left\{n^{-\frac{2\gamma^*}{2\gamma^*+1}} + \frac{\log p}{n} + \frac{1}{\lambda} \right\},
\end{align*}
for a universal constant $c_1$ independent of $n, p$ and $\lambda$, where the infimum is taken over all possible estimators based on $n$ i.i.d. observations $(\bx_1,y_1),\ldots, (\bx_n,y_n)$ and $\gamma^* =  \beta^*/d^*$.
\end{theorem}

\section{Simulation Studies}
\label{sec:experiments}

In this section, we use simulated data to illustrate the finite sample performance of our proposed estimators, supplemented by an in-depth comparative analysis that supports our theoretical discoveries. A real data analysis is presented in Supplemental Material Section \ref{sec:application}.

\subsection{Finite Sample Performance of the FAR-NN Estimator}
\label{sec:exp:simluation-far}

The target of the experiments in this section is to show that (1) the FAR-NN estimator can achieve a near-oracle finite sample performance in the high-dimensional regime; (2) it will have sub-optimal performance if we jointly train the diversified projection matrix and network weights; (3) we can use a tiny subset of unlabelled data $\{\bx_i\}_{i=1}^{n_1}$ to estimate the diversified projection matrix in practice.

\noindent \textbf{Data Generating Process. } The covariate vector $\bx$ admits a linear factor model with the number of factors $r=5$. The factor loading matrix $\bB$ has i.i.d. $\mathrm{Unif}[-\sqrt{3}, \sqrt{3}]$ entries. The latent factor $\bbf$ and the idiosyncratic component $\bu$ are independent and have i.i.d. $\mathrm{Unif}[-1,1]$ entries, respectively. The regression function is $m^*(\bbf)=\sum_{j=1}^r m^*_j(f_j)$, where $m_j^*$ are selected randomly from $\{\cos(\pi x), \sin(x), (1-|x|)^2, 1/(1+e^{-x}), 2\sqrt{|x|}-1\}$ in each trial. The response variable $y$ is assigned to be $m^*(\bbf) + \varepsilon$, where the noise $\varepsilon \sim \mathcal{N}(0,0.3)$ is independent of $(\bbf, \bu)$. Throughout this section, we might vary the ambient dimension $p$, but will keep using $n_{\mathrm{train}}=500$ i.i.d. samples $\{(\bbf_i, \bu_i, y_i)\}_{i=1}^{n_{\mathrm{train}}}$ from the above data generating process to train our neural network. We also use other $n_{\mathrm{valid}}=150$ i.i.d. observations as a validation data set for model selection.

\noindent \textbf{Implementation.} We use fully connected ReLU neural networks with depth $L=4$ and width $N=300$ for all the estimators. The neural network weights are optimized using the Adam optimizer \citep{kingma2014adam} with a learning rate of $10^{-4}$ and batch size $64$ for 200 epochs. The total training time is about 30s using a CPU-only laptop.z
We do not use other regularization techniques except early stopping, which selects the model with a minimum $L_2$ loss on the validation set for evaluation. 
We adopt a data-driven method for the fixed diversified projection matrix $\bW$ in the FAR-NN estimator as described in Section \ref{sec:theory:fa}. 
Specifically, we let $\overline{r}=10$ to allow over-estimating the number of factors and apply PCA to newly generated $n_1=10\% \times  n_{\mathrm{train}} = 50$ unlabelled samples to pre-train $\bW$. The performance of the estimator $\hat{m}(\bx)$ is evaluated via the empirical mean squared error computed using another $n_{\mathrm{test}}=10^5$ i.i.d. samples, i.e., we use 
\begin{align}
\label{eq:exp:mse-far-nn}
    \hat{\mathtt{MSE}} = \frac{1}{n_{\mathrm{test}}} \sum_{i=1}^{n_{\mathrm{test}}} \left\{\hat{m}(\bx_i) - m^*(\bbf_i)\right\}^2 
\end{align} as an estimate of $\int |\hat{m}(\bx) - m^*(\bbf)|^2 \mu(d\bbf, d\bu)$ to evaluate its finite sample performance.

\noindent \textbf{Exp \uppercase\expandafter{\romannumeral1}. Finite Sample Performance of the Estimators. } We compare the performances of the six estimators below. Figure \ref{fig:exp1}(d) illustrates how training/validation $L_2$ losses change across epochs.

\begin{itemize}[noitemsep, topsep=0pt]
    \item[1.] Oracle-NN estimator. This estimator takes the exact latent factor $\bbf$ as input and directly regresses the response variable $y$ on the latent factor $\bbf$. It uses $\{(\bbf_i, y_i)\}_{i=1}^{n_{\mathrm{train}}}$ to estimate $\hat{m}$. Its performance is evaluated by $\hat{\mathtt{MSE}} = \frac{1}{n_{\mathrm{test}}} \sum_{i=1}^{n_{\mathrm{test}}} \left\{\hat{m}(\bbf_i) - m^*(\bbf_i)\right\}^2$. This can be seen as the benchmark of the mean squared error the neural network estimators can achieve.
    
    \item[2.] FAR-NN estimator. This is the estimator we proposed in Section \ref{sec:method}. We use $\{(\bx_i, y_i)\}_{i=1}^{n_{\mathrm{train}}}$ to estimate $\hat{m}$, and the performance is evaluated using \eqref{eq:exp:mse-far-nn}.

    \item[3.] FAST-NN estimator. We also evaluate the performance of the FAST-NN estimator when $\mathcal{J}=\emptyset$ (but the estimator is unaware of this condition). The implementation details are presented in Section \ref{subsec:fast-nn-exp}, the training and evaluation protocols are the same as that for FAR-NN.

    \item[4.] PCA-A-NN estimator. We examine a naive PCA Augmented NN implementation widely used: performing PCA on $\bm{x}$ and training a neural network with input $\bm{z} = [\mathrm{PCA}(\bm{x}), \bm{x}]$.
    
    \item[5.] Vanilla-NN estimator. The estimator uses a deep ReLU network with input dimension $p$, depth $L$, and width $N$ to estimate $\hat{m}$ based on $\{(\bx_i, y_i)\}_{i=1}^{n_{\mathrm{train}}}$. The performance is evaluated using \eqref{eq:exp:mse-far-nn}.
    
    \item[6.] NN-Joint estimator. It is the same as our proposed FAR-NN estimator, except that the matrix $\bW$ is jointed trained with the neural network weights using the training data set $\{(\bx_i, y_i)\}_{i=1}^{n_{\mathrm{train}}}$ rather than being fixed. It has two advantages over the Vanilla-NN. Firstly, it has some inductive bias towards the low-dimension structure because the first 
    layer only has $\overline{r}=10$ hidden units. Secondly, the weights in the first layer have a good initialization.
\end{itemize}

\begin{figure*}[t!]
\centering
\begin{tabular}{cc}
\subfigure[]{
    \includegraphics[scale=0.4]{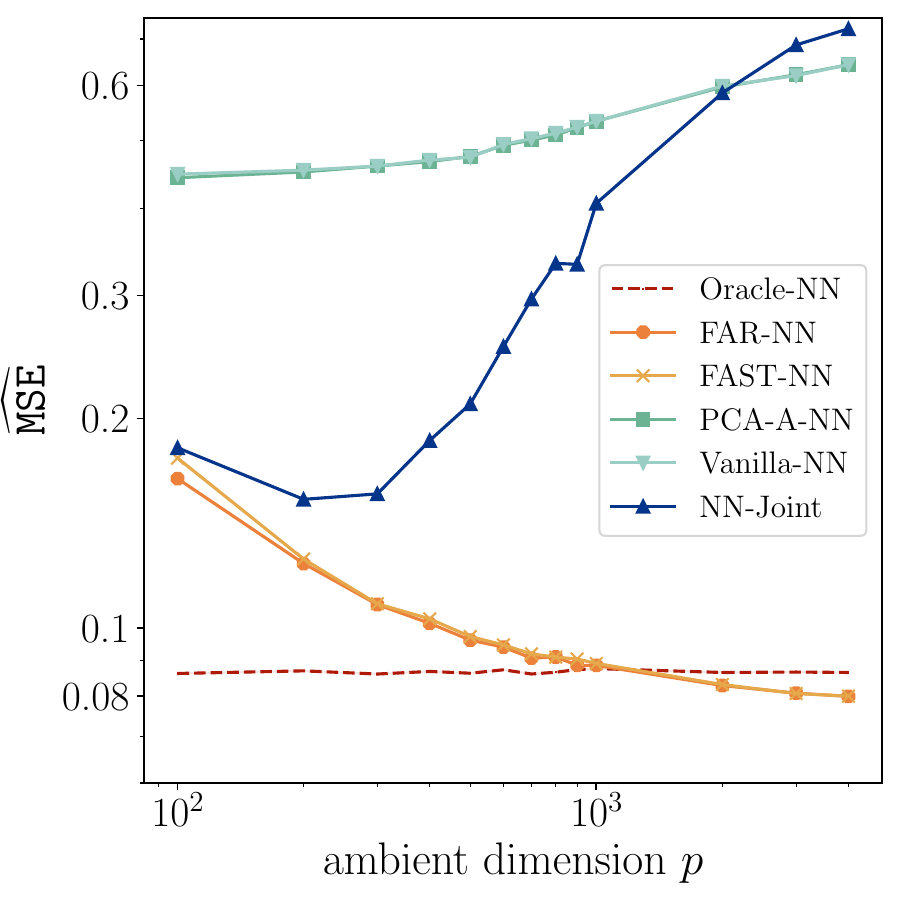}
} &
\subfigure[]{
    \includegraphics[scale=0.4]{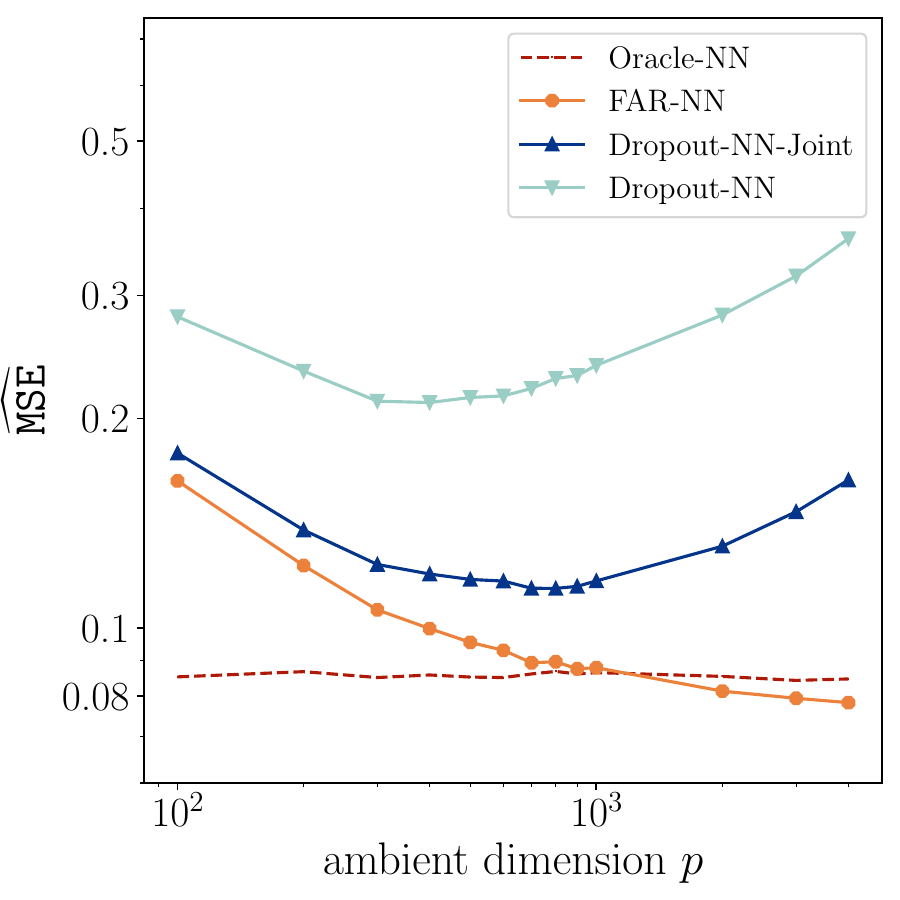}
} \\
\subfigure[]{
    \includegraphics[scale=0.4]{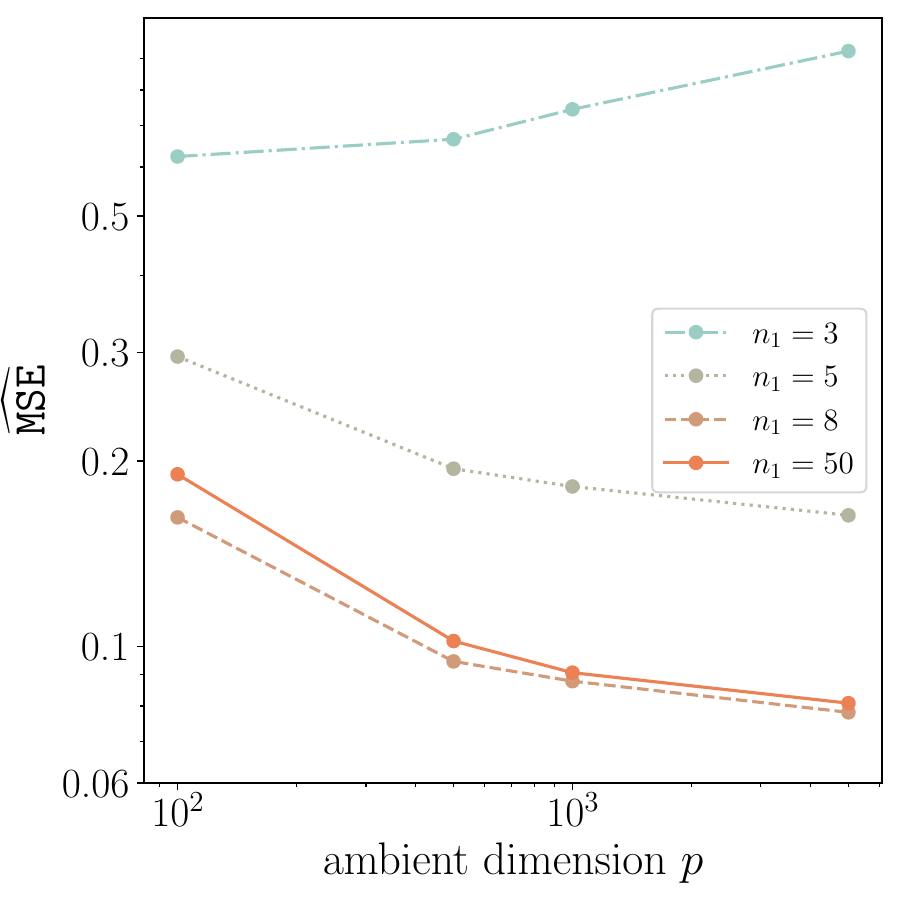}
} &
\subfigure[]{
    \includegraphics[scale=0.4]{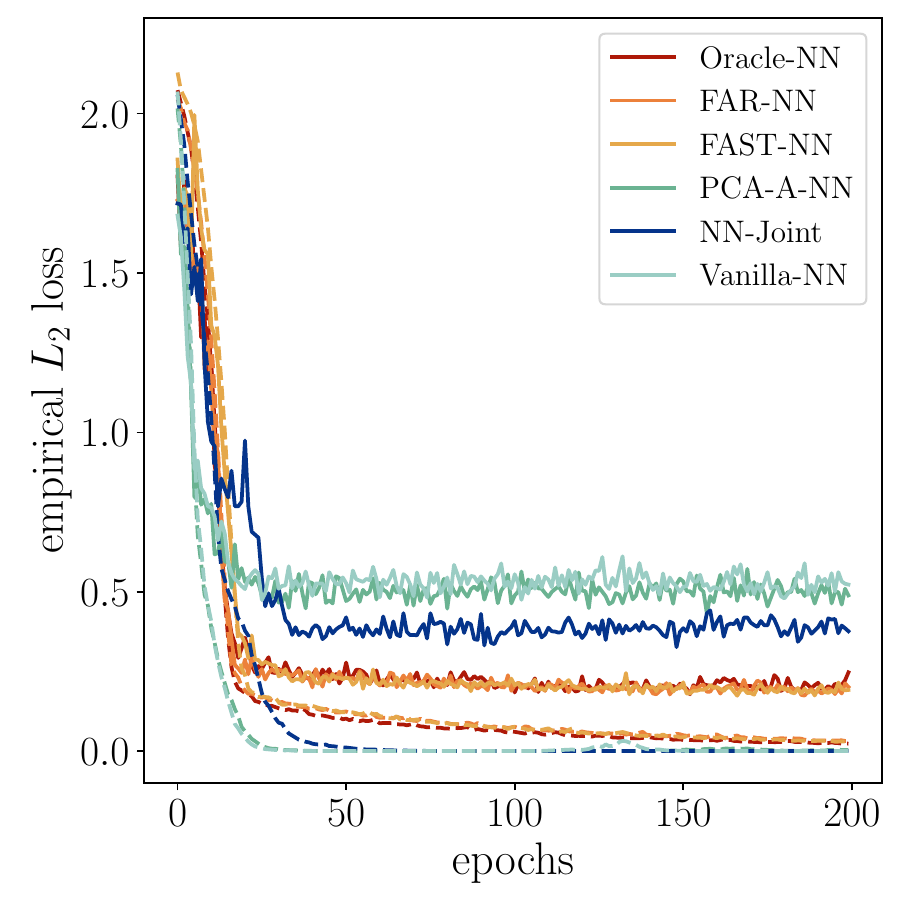}
}
\end{tabular}
\caption{(a)-(c) depict finite sample mean squared errors of the estimators for different ambient dimension $p$ in Exp \uppercase\expandafter{\romannumeral1}--\uppercase\expandafter{\romannumeral3}, respectively. For (a) and (b), the curves with different colors represent the performance of different estimators. We use dashed lines to emphasize the estimators with an `Oracle' prefix. For (c), the curves with different colors (shapes) represent the performance of FAR-NN estimators whose diversified projection matrix is estimated using PCA with other $n_1$ samples. (d) displays the changes in empirical $L_2$ losses $\hat{\mathsf{R}}(f)=|\mathcal{D}|^{-1}\sum_{(\bx,y)\in \mathcal{D}} \{f(\bx)-y\}^2$ on the training set (represented by dashed lines) and validation set (represented by solid lines) across epochs for various methods (represented by different colors) in a single trial with $p=1000$. 
}
\label{fig:exp1}
\end{figure*}

For each $p=\{100k: k\in [10]\} \cup \{2000, 3000, 4000\}$, we generate the data $200$ times and calculate the average of the empirical mean squared errors over the $200$ trials for all the estimators. The result is presented in Fig \ref{fig:exp1}(a). Firstly, we can see that as the ambient dimension $p$ grows, the FAR-NN estimator's performance improves. It is almost the same as the Oracle-NN estimator when $p=1000$. Interestingly, it is even better than the Oracle-NN estimator when $p\ge 2000$. We guess this might be attributed to some implicit regularization effects the diversified projection matrix introduces. Moreover, the negligible performance difference between the FAST-NN and FAR-NN estimators suggests that the FAST-NN estimator performs nearly identically to the FAR-NN estimator when $\mathcal{J}=\emptyset$. Consequently, users need not concern themselves with choosing between the FAR-NN and FAST-NN estimators and can confidently opt for the FAST-NN estimator in practice.

As a comparison, the Vanilla-NN estimator behaves uniformly badly, and the performance of the PCA-A-NN estimator is almost the same. The gap in the performance between the NN-Joint estimator and FAR-NN estimator is small when $p$ is small, but it becomes larger as $p$ grows. Such empirical findings show that without fixing the diversified projection matrix, the estimator might fail to recover the latent factor and suffer from the data's high dimensionality. Such empirical conclusion is consistent with our theoretical justification Proposition \ref{thm:sub-optimality-least-squares}.

\noindent \textbf{Exp \uppercase\expandafter{\romannumeral2}. Comparison with Dropout.} Dropout \citep{srivastava2014dropout} is a technique to prevent neural networks from over-fitting. Specifically, when the dropout with dropout rate $\rho \in [0,1]$ is applied to a particular neural network layer, it randomly sets $100\rho \%$ of the (hidden) units as zero in each iteration during training. Such a technique can also be applied to the input of the neural network, which is known as word dropout \citep{dai2015semi} in the natural language processing literature. In this case, it prevents a neural network from learning from a fixed subset of input features and encourages it to learn the common structure that is invariant among different choices of subsets of input features. We further compare our FAR-NN estimator with neural network estimators that uses the dropout in the input layer. To be specific, we consider applying the dropout technique to the input of the neural networks in the Vanilla-NN estimator and the NN-Joint estimator, resulting in the Dropout-NN estimator and the Dropout-NN-Joint estimator, respectively. For the two estimators, we try the dropout rate $\rho \in \{0, 0.3, 0.5,  0.6, 0.7, 0.8, 0.9\}$ in each trial and do model selection using the validation data set. 

The result is depicted in Fig \ref{fig:exp1}(b). Compared to Fig \ref{fig:exp1}(a), we can see that applying proper input dropout leads to significant improvements for both estimators. However, when $p\ge 1000$, both estimators' mean squared errors increase as $p$ grows, implying that the dropout regularization still fails to help these estimators consistently estimate the regression function in the high-dimensional regime.

\noindent \textbf{Exp \uppercase\expandafter{\romannumeral3}. When $n_1$ is large enough? } We also investigate how the choice of $n_1$ will affect the performance of the FAR-NN estimator. We repeat the same procedure for different choices of $n_1$ and $p\in \{100, 500, 1000, 5000\}$. The result is shown in Fig \ref{fig:exp1}(c). We can see that $n_1=8$ is enough for the FAR-NN estimator to make good predictions. The empirical finding supports our theoretical claim in Section \ref{sec:theory:fa} that the number of samples for the diversified projection matrix is negligible compared to the number of samples for estimating the regression function, i.e., $n_1 \ll (n - n_1)$.

\subsection{Finite Sample Performance of the FAST-NN estimator}
\label{subsec:fast-nn-exp}

In this section, we illustrate the finite sample performance of the FAST-NN estimator. 

\noindent \textbf{Data Generating Process. } The covariate vector $\bx$ admit a linear factor model with the number of factors $r=4$. The law of $(\bbf, \bu)$ together with the generation of the factor loading matrix is same with Section \ref{sec:exp:simluation-far}. The regression function only depends on $(\bbf, u_1,\ldots, u_5)$. Specifically, we consider the following two regression functions
\begin{align*}
    m^*_{\mathtt{fast},1}(\bbf, u_1,\ldots, u_5) &= \sum_{i=1}^4 (-1)^{i-1} f_i + \sum_{j=1}^5 (-1)^j u_j \\
    m^*_{\mathtt{fast},2}(\bbf, u_1,\ldots, u_5) &= f_1 f_2^2 - f_3 + \log\left(8 + f_4 + 4u_1 + e^{u_2 u_3 - 5u_1}\right) + \tan(u_4 + 0.1) + \sin(u_5)
\end{align*}
The first function is linear, and the second function is a nonlinear function admitting a hierarchical composition structure.  Similar to the simulation of the FAR-NN estimator, the response variable $y$ is set to be $m^*_{\mathtt{fast},k}(\bbf, u_1,\ldots, u_5)+\varepsilon$ for $k\in \{1,2\}$. We will vary the ambient dimension and use fixed $n_{\mathrm{train}}=1000$ i.i.d. samples from the above data generating process to train the neural network and the matrix $\bTheta$. We use other $n_{\mathrm{valid}}=300$ i.i.d. samples as a validation set for model selection.

\begin{figure*}[t!]
\centering
\begin{tabular}{cc}
\subfigure[]{
    \includegraphics[scale=0.4]{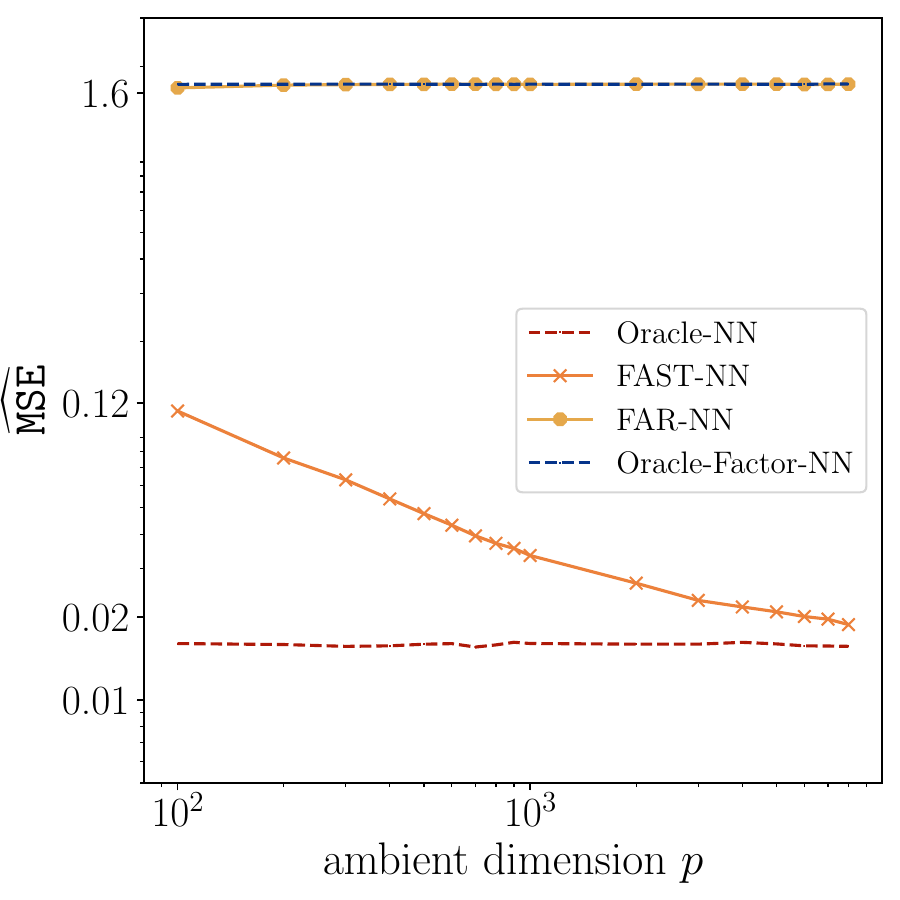}
} &
\subfigure[]{
    \includegraphics[scale=0.4]{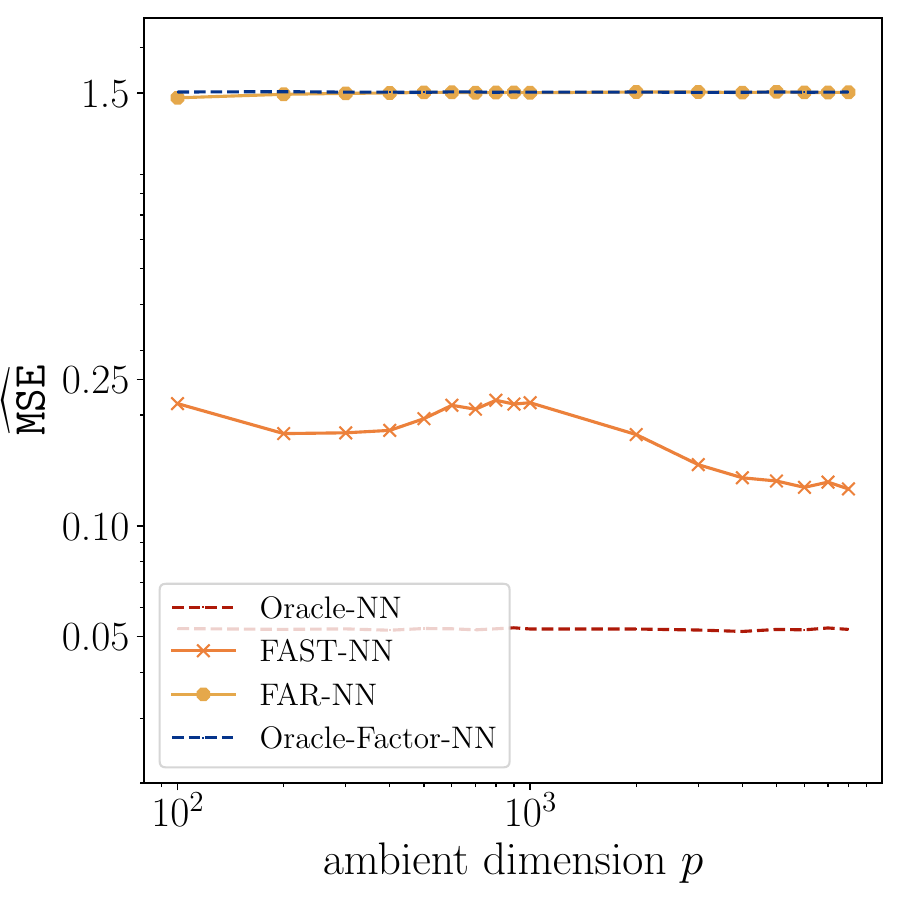}
}
\end{tabular}
\caption{Finite sample mean squared error of the estimators when $m^*$ is (a) $m^*_{\mathtt{fast}(1)}$ (b) $m^*_{\mathtt{fast}(2)}$ for different ambient dimension $p$. The curves with different colors represent the performance of different estimators. We use dashed lines to emphasize the estimators with an `Oracle' prefix in its name.}
\label{fig:exp4}
\end{figure*}

\begin{figure*}[t!]
\centering
\subfigure[]{
    \includegraphics[scale=0.45]{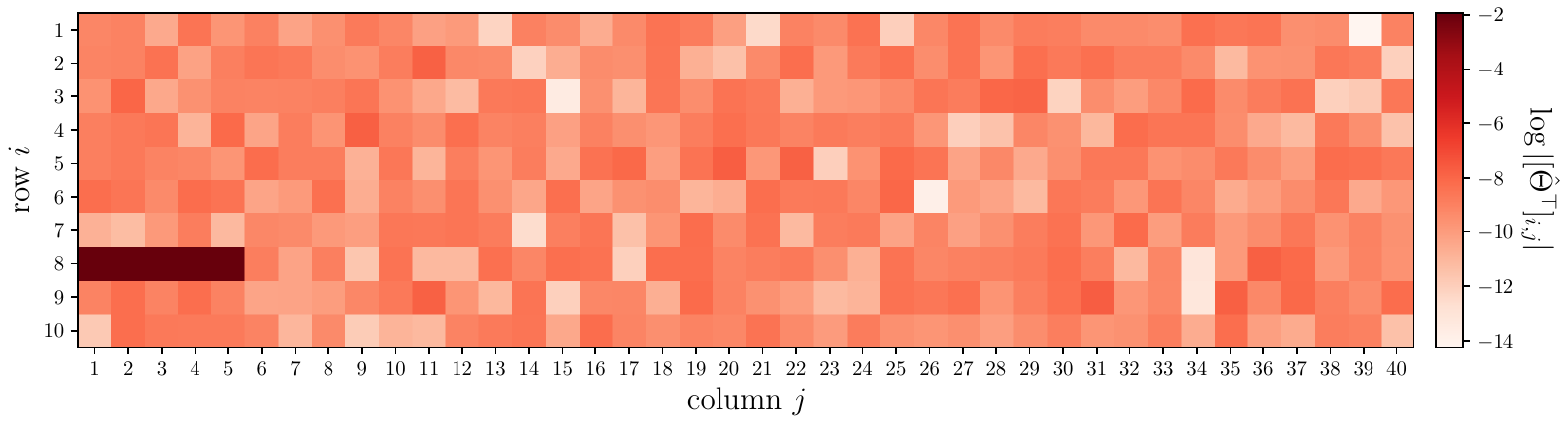}
} 
\subfigure[]{
    \includegraphics[scale=0.45]{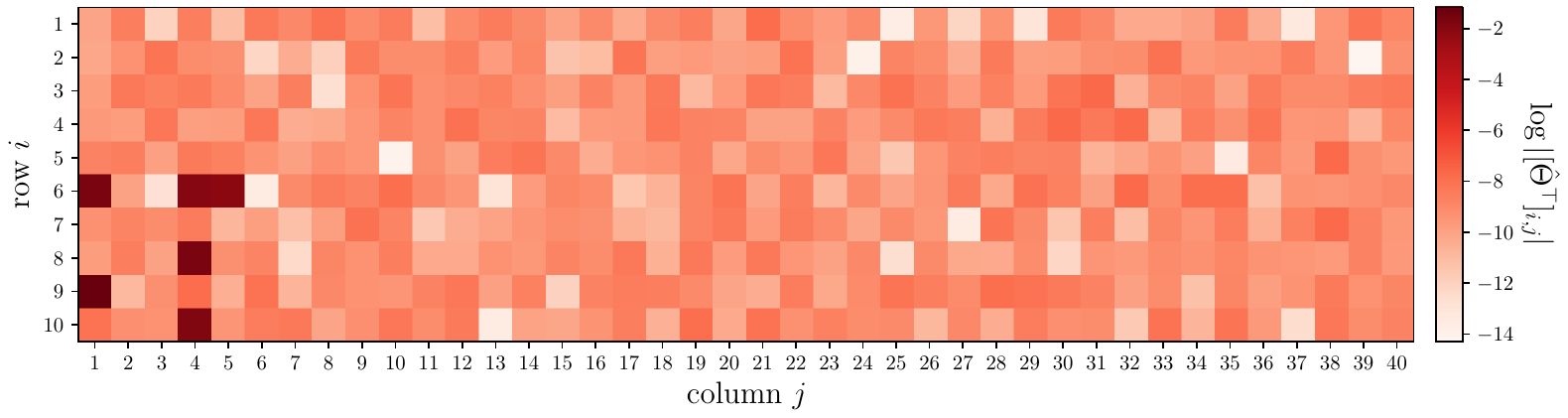}
}
\caption{The visualizations of the first $40$ columns of the transpose of the variable selection matrices $\hat{\Theta}^\top \in \mathbb{R}^{10\times p}$ in one trial when the regression function $m^*$ is (a) $m^*_{\mathtt{fast}(1)}$ (b) $m^*_{\mathtt{fast}(2)}$, respectively. The colors represent the logarithm of the magnitude of the entry, with darker colors implying greaterx magnitudes. We can see that all the dark colors appear in the first five columns, indicating that $\hat{\Theta}$ does not select variables $x_j$ with $j>5$ in both cases. }
\label{fig:exp5}
\end{figure*}

\noindent \textbf{Implementation. } We adopt the same neural network architecture and training configurations as that in Section \ref{sec:exp:simluation-far}, and set additional hyper-parameters to be $\tau=1.3(n_{\mathrm{train}})^{-1} \log p$ and $\tau=5\cdot 10^{-3}$. We reduce the number of columns of the variable selection matrix such that it is a $p\times 10$ matrix in practice. The performance of the estimator $\hat{m}(\bx)$ is evaluated using other $n_{\mathrm{test}}=10^5$ i.i.d. samples as
\begin{align*}
    \hat{\mathtt{MSE}} = \frac{1}{n_{\mathrm{test}}} \sum_{i=1}^{n_{\mathrm{test}}} \left\{\hat{m}(\bx_i) - m^*(\bbf_i, u_{i,1},\ldots, u_{i,5})\right\}^2.
\end{align*} 

We compare the performance of the FAST-NN estimator with the following three estimators:
\begin{itemize}[noitemsep, topsep=0pt]
    \item[1.] Oracle-NN estimator. The estimator takes all the important features $\bz=(\bbf,u_1,\ldots, u_5)$ as input and regresses $y$ on $\bz$. Its finite sample performance is the lower bound of the mean squared error the neural network estimators can attain.
    \item[2.] Oracle-Factor-NN estimator. The estimator takes the latent factor $\bbf$ as the input and regresses $y$ on $\bbf$. We compare its performance with that of the FAST-NN estimator to see whether the latter can learn some dependency of the response variable $y$ on the idiosyncratic component $\bu$. 
    \item[3.] FAR-NN estimator. The performance of the FAR-NN estimator is reported for completeness.
\end{itemize}

\noindent \textbf{Results. } For each function $m_{\mathtt{fast},k}^*$ with $k\in \{1,2\}$ and $p\in \{100k: k\in [10]\} \cup \{2000, 3000, 4000\}$, we generate the data $200$ times and calculate the average of the empirical mean squared error over the $200$ trials for the four estimators. The result is shown in Fig \ref{fig:exp4}. In the case of both regression functions, the Oracle-Factor-NN and FAR-NN exhibit comparable performances, while FAST-NN notably outperforms these two. Furthermore, the performance gaps between the FAST-NN and the Oracle-NN diminish as $p$ grows. These findings indicate that the FAST-NN can establish a non-trivial association between the response variable and some important idiosyncratic components. 

To see whether the FAST-NN learns a sparse representation, we consider visualizing the trained variable selection matrix $\hat{\bTheta}$ in one trial for the two regression functions. Figure \ref{fig:exp5} visualizes the first $40$ columns submatrix of the transpose of the variable selection matrix $\hat{\bTheta}^\top \in \mathbb{R}^{10 \times p}$ when $p=1000$ for two different regression functions. The rows are sorted according to on the maximum absolute value of the entries in each row. We can see that it indeed learns some sparse representations in a correct way. When the regression function is linear $m^*=m^*_{\mathtt{fast}(1)}$, it correctly selects all the important variables $u_1,\ldots, u_5$. It selects some of the variables that significantly influence the regression function when $m^*=m^*_{\mathtt{fast}(2)}$ while missing $u_2$ and $u_3$ due to weak signals.  This results in better prediction than FAR-NN but worse than Oracle-NN.

\newpage
\appendix
\begin{center}
\LARGE
\textbf{Supplemental Material for ``Factor Augmented Sparse Throughput Deep ReLU Neural Networks for High Dimensional Regression''}
\end{center}
\section{Organization of Supplemental Material}

The organization of the supplemental material is as follows
\begin{description}
    \item[Section \ref{sec:additive}] describes the Factor Augmented Neural Additive Model and presents the corresponding theoretical finding.
    \item[Section \ref{sec:application}] includes a real data analysis for the FAST-NN estimator.
    \item[Section \ref{sec:relatedworks}] provides a detailed related work.
    \item[Section \ref{sec:discussions}] collects some discussions omitted in the main text, including (1) the comparison and novelty of our developed neural network approximation theory and (2) an informal argument claiming that our method and theoretical analysis can be extended to more general setting.
    \item[Section \ref{sec:proof-far-nn-estimator}] contains all the proofs for the FAR-NN estimator in Section \ref{sec:theory:fa}.
    \item[Section \ref{sec:proof-fast-nn-estimator}] contains all the proofs for the FAST-NN estimator in Section \ref{sec:theory:fast}.
    \item[Section \ref{sec:proof-lb}] contains all the proofs of the lower bounds in Section \ref{sec:theory:lb}.
    \item[Section \ref{sec:proof-nn-approx}] contains the proof of our new ReLU neural network approximation result Theorem \ref{thm:approx-smooth}.
    \item[Section \ref{sec:proof-fanam}] contains the proof for the FANAM estimator in Section \ref{sec:additive}.
\end{description}

In the proof, we will also use the following notations in empirical process literature. We use $\mathcal{N}(\epsilon, \mathcal{H}, d)$ to denote the $\epsilon$-covering number of the function class $\mathcal{H}$ with respect to the metric $d$. For given $\bz_1^n=(\bz_1,\ldots, \bz_n)$, let the $\|\cdot\|_{L_p(\mu_n)}$ norm be that
\begin{align*}
    \|h\|_{L_p(\mu_n)} = \begin{cases}
        \left\{\frac{1}{n} \sum_{i=1}^n |h(\bz_i)|^p \right\}^{1/p} & ~~~~ p \in [1,\infty) \\
        \sup_{1\le i\le n} |h(\bz_i)| & ~~~~ p=\infty
    \end{cases}.
\end{align*}
We use $\mathcal{N}_p(\epsilon, \mathcal{H}, \bz_1^n)$ to denote the $\epsilon$-covering number of $\mathcal{H}$ with respect to $\|\cdot\|_{L_p(\mu_n)}$ norm.

\section{Factor Augmented Sparse Additive Neural Network Estimator}
\label{sec:additive}

When $\mathcal{J} \neq \emptyset$, the FAST-NN estimator proposed in Section \ref{sec:method:full} induces the clipped-$L_1$ function with small $\tau$, which will make the optimization hard to solve. Meanwhile, the theoretical analysis only focuses on the regime in which $|\mathcal{J}|$ is fixed and does not grow with $n$ and $p$, it is unclear how increasing $|\mathcal{J}|$ affects the rate of convergence for the proposed estimator. In this section, we propose Factor Augmented Neural Additive Model (FANAM), which partially resolves the above concerns in a special scenario where the regression function $m^*$ admits an additive structure.

The proposed model also use pre-defined fixed diversified projection matrix according to Defintion \ref{def:dpm}. Given the hyper-parameters $\overline{r}$, $L$, $N$ and $M$, let $g_0,\ldots, g_p$ be deep ReLU networks truncated by $M$ with depth $L$ and width $N$, particularly, 
\begin{align*}
    g_0 \in \mathcal{G}(L, \overline{r}, 1, N, M, \infty) ~~~~~~ \text{and} ~~~~~~ g_j \in \mathcal{G}(L, 1, 1, N, M, \infty) ~~ \text{for }  j\in [p],
\end{align*}
and $\bbeta = (\beta_1,\ldots, \beta_p) \in \mathbb{R}^p$, $\bV = [\bv_1,\ldots,\bv_p]^\top \in \mathbb{R}^{p\times \overline{r}}$ be other weights to be learned from data. Letting $\tilde{\bbf}(\bx) = p^{-1} \bW^\top \bx$, the prediction of our model is given by
\begin{align}
\label{eq:fanam-pred}
    m(\bx)=m\left(\bx;\bW, \bV, \{g_j\}_{j=0}^p, \bbeta\right) = g_0(\tilde{\bbf}(\bx)) + \sum_{j=1}^p \beta_j g_j(x_j - \bv_j^\top \tilde{\bbf}(\bx)).
\end{align} An architecture visualization of the proposed model is presented in Fig \ref{fig:fanam}. Given the i.i.d. sample $\{(\bx_i,y_i)\}_{i=1}^n$, the weights are determined via the following optimization problem
\begin{align}
\label{eq:fanam-est}
    \hat{m}_{\mathtt{FANAM}}(\bx) = \argmin_{m\left(\bx;\bW, \bV, \{g_j\}_{j=0}^p, \bbeta\right)} \frac{1}{n} \sum_{i=1}^n \left\{y_i - m(\bx_i)\right\}^2 + \lambda \|\bbeta\|_1.
\end{align}
with given $\bW$ and $\lambda$.  This results in the Factor-Augmented Neural  Additive Model Estimator (FANAM).

Suppose the regression function $m^*$ has the following additive structure.
\begin{condition}[Factor Augmented Sparse Additive Model]
\label{cond:faam}
    The regression function $m^*(\bbf, \bu)$ admits an sparse additive form as
    \begin{align}
        m^*(\bbf, \bu) = m^*_0(\bbf) + \sum_{j \in \mathcal{J}_u} m^*_j(u_j) + \sum_{j \in \mathcal{J}_x} m^*_j(x_j) ~~~~~~ \text{with}~~ \mathcal{J}_u \cap \mathcal{J}_x = \emptyset
    \end{align} where all the components $m^*_j$ satisfies $\|m^*_j\|_\infty \le M^*$, and $m^*_j$ is $c_1$-Lipschitz for some universal constants $M^*$ and $c_1$. Moreover, $m^*_0 \in \mathcal{H}(r, l, \mathcal{P})$ and $m^*_j \in \mathcal{H}(1, l, \mathcal{P})$. 
\end{condition}

\begin{figure}

\centering
\begin{tikzpicture}
\definecolor{myred}{HTML}{ae1908}
\definecolor{myblue}{HTML}{05348b}
\draw[black] (0 + 0.06 + 0.08, 0 - 0.3) node{$\bx$};
\draw[black] (0 - 0.06, 0 - 0.06) rectangle+(0.3+0.06, 0.3 * 13 +0.06);
\foreach \y in {0, 1, 2, 3, 4, 5, 6, 7, 8, 9, 10, 11, 12}
	\draw[gray] (0, \y * 0.3) rectangle+(0.24, 0.24);

\draw[myblue] (-0.3*5 - 0.06 + 0.48, 0 - 0.3) node{$\bW$};
\draw[myblue] (-0.3*5 - 0.06, 0 - 0.06) rectangle+(0.3*3+0.06, 0.3 * 13 +0.06);
\foreach \x in {-3, -4, -5}
	\foreach \y in {0, 1, 2, 3, 4, 5, 6, 7, 8, 9, 10, 11, 12}
	{
		\pgfmathrandominteger{\a}{10}{50}
		\filldraw[myblue!\a] (\x * 0.3, \y * 0.3) rectangle+ (0.24, 0.24);
	}
	
\draw[black] (-4*0.3+0.12, 13*0.3 + 0.15) -- (-4*0.3+0.12, 20*0.3+0.12) node [sloped, midway, above] {\footnotesize $\tilde{\bbf} = \frac{1}{p} {\color{myblue} \bW}^\top \bx$};
\draw[->, black] (-4*0.3+0.12, 20*0.3+0.12) -- (-0.15, 20*0.3+0.12);

\foreach \x in {-1, 0, 1}
	{
		\draw[black] (0, \x * 0.3 + 20*0.3) rectangle + (0.24, 0.24);
	}

\draw[->] (0.40, 20*0.3+0.12) -- (1.40, 20*0.3+0.12);
\foreach \x in {-1, 0, 1}
{
	\draw[black] (1.5, \x*0.3+ 20*0.3) rectangle +(0.24, 0.24);
}
\foreach \l in {1, 2, 3}
\foreach \x in {-2, -1, 0, 1, 2}
{
	\draw[black] (1.5+\l, \x*0.3+ 20*0.3) rectangle +(0.24, 0.24);
}
\draw[black] (1.5+4, 20*0.3) rectangle+(0.24, 0.24);
\foreach \x in {-1, 0, 1}
	\foreach \y in {-2, -1, 0, 1, 2}
{
	\draw[->, myred] (1.5+0.24, \x*0.3+0.12+20*0.3) -- (1.5+1, \y * 0.3+0.12+20*0.3);
}
\foreach \l in {2,3}
	\foreach \x in {-2, -1, 0, 1, 2}
		\foreach \y in {-2, -1, 0, 1, 2}
		{
			\draw[->, myred] (1.5+\l-1+0.24, \x*0.3+0.12+20*0.3) -- (1.5+\l, \y * 0.3+0.12+20*0.3);
		}
\foreach \x in {-2, -1, 0, 1, 2}
	{
		\draw[->, myred] (1.5+4-1+0.24, \x*0.3+0.12+20*0.3) -- (1.5+4, 0 * 0.3+0.12+20*0.3);
	}
\draw[black] (1.5+4+0.24+0.6, 0.4+20*0.3) node{\footnotesize $T_M(\cdot)$};
\draw[->] (1.5+4+0.24 + 0.1, 0.12+20*0.3) -- (1.5+4+0.24+1.1, 0.12+20*0.3);
\draw[->, myblue] (7, 0.12+20*0.3) -- (9, 8.5*0.3) node[sloped, midway, above] {\footnotesize $\color{myblue}{1}$};
\draw[black, rounded corners] (1.3, 17*0.3+0.15) -- (1.3, 24*0.3-0.15) -- (7, 24*0.3-0.15) -- (7, 17*0.3+0.15) -- cycle;
\draw[black] (6.2, 18.5*0.3) node{\small $g_0(\tilde{\bbf})$};
\draw[->, myblue] (0.40, 12*0.3+0.12) -- (1.40, 13*0.3+0.12);
\draw[myblue] (1.0, 12.5*0.3+0.12) node{\footnotesize $1$};
\draw[myred] (1.1, 17*0.3) node{$\bv_1$};
\foreach \x in {-1, 0, 1}
{
	\draw[->, myred] (0.40, 20*0.3+0.12+\x*0.3) -- (1.40, 13*0.3 + 0.12);
}

\draw[black] (1.5, 13*0.3) rectangle + (0.24, 0.24);
\foreach \l in {1, 2, 3}
\foreach \x in {-2, -1, 0, 1, 2}
{
	\draw[black] (1.5+\l, \x*0.3+ 13*0.3) rectangle +(0.24, 0.24);
}
\draw[black] (1.5+4, 13*0.3) rectangle+(0.24, 0.24);
\foreach \y in {-2, -1, 0, 1, 2}
{
	\draw[->, myred] (1.5+0.24, 0.12+13*0.3) -- (1.5+1, \y * 0.3+0.12+13*0.3);
}
\foreach \l in {2,3}
	\foreach \x in {-2, -1, 0, 1, 2}
		\foreach \y in {-2, -1, 0, 1, 2}
		{
			\draw[->, myred] (1.5+\l-1+0.24, \x*0.3+0.12+13*0.3) -- (1.5+\l, \y * 0.3+0.12+13*0.3);
		}
\foreach \x in {-2, -1, 0, 1, 2}
	{
		\draw[->, myred] (1.5+4-1+0.24, \x*0.3+0.12+13*0.3) -- (1.5+4, 0 * 0.3+0.12+13*0.3);
	}
\draw[black] (1.5+4+0.24+0.6, 0.4+13*0.3) node{\footnotesize $T_M(\cdot)$};
\draw[->] (1.5+4+0.24 + 0.1, 0.12+13*0.3) -- (1.5+4+0.24+1.1, 0.12+13*0.3);
\draw[->, myred] (7, 0.12+13*0.3) -- (9, 8.5*0.3) node[sloped, midway, above] {\footnotesize $\color{myred}{\beta_1}$};
\draw[black, rounded corners] (1.3, 10*0.3 + 0.15) -- (1.3, 17*0.3-0.15) -- (7, 17*0.3-0.15) -- (7, 10*0.3+0.15) -- cycle;
\draw[black] (6, 11.5*0.3) node{\small $g_1(x_1-\bv_1^\top \tilde{\bbf})$};

\draw[->, myblue] (0.40, 11*0.3+0.12) -- (1.40, 6*0.3+0.12);
\draw[myblue] (1.0, 7.5*0.3+0.12) node{\footnotesize $1$};
\draw[myred] (1.1, 12*0.3) node{$\bv_2$};
\foreach \x in {-1, 0, 1}
{
	\draw[->, myred] (0.40, 20*0.3+0.12+\x*0.3) -- (1.40, 6*0.3 + 0.12);
}
\draw[black] (1.5, 6*0.3) rectangle + (0.24, 0.24);
\foreach \l in {1, 2, 3}
\foreach \x in {-2, -1, 0, 1, 2}
{
	\draw[black] (1.5+\l, \x*0.3+ 6*0.3) rectangle +(0.24, 0.24);
}
\draw[black] (1.5+4, 6*0.3) rectangle+(0.24, 0.24);
\foreach \y in {-2, -1, 0, 1, 2}
{
	\draw[->, myred] (1.5+0.24, 0.12+6*0.3) -- (1.5+1, \y * 0.3+0.12+6*0.3);
}
\foreach \l in {2,3}
	\foreach \x in {-2, -1, 0, 1, 2}
		\foreach \y in {-2, -1, 0, 1, 2}
		{
			\draw[->, myred] (1.5+\l-1+0.24, \x*0.3+0.12+6*0.3) -- (1.5+\l, \y * 0.3+0.12+6*0.3);
		}
\foreach \x in {-2, -1, 0, 1, 2}
	{
		\draw[->, myred] (1.5+4-1+0.24, \x*0.3+0.12+6*0.3) -- (1.5+4, 0 * 0.3+0.12+6*0.3);
	}
\draw[black] (1.5+4+0.24+0.6, 0.4+6*0.3) node{\footnotesize $T_M(\cdot)$};
\draw[->] (1.5+4+0.24 + 0.1, 0.12+6*0.3) -- (1.5+4+0.24+1.1, 0.12+6*0.3);
\draw[->, myred] (7, 0.12+6*0.3) -- (9, 8.5*0.3) node[sloped, midway, above] {\footnotesize $\color{myred}{\beta_2}$};
\draw[black, rounded corners] (1.3, 3*0.3 + 0.15) -- (1.3, 10*0.3-0.15) -- (7, 10*0.3-0.15) -- (7, 3*0.3+0.15) -- cycle;
\draw[black] (6, 4.5*0.3) node{\small $g_2(x_2-\bv_2^\top \tilde{\bbf})$};

\draw (4, 0.3*1+0.12) node{$\vdots$};

\draw[->, myblue] (0.40, 0*0.3+0.12) -- (1.40, -5*0.3+0.12);
\draw[myblue] (1.0, -3.5*0.3+0.12) node{\footnotesize $1$};
\draw[myred] (1.1, -1*0.3) node{$\bv_p$};
\foreach \x in {-1, 0, 1}
{
	\draw[->, myred] (0.40, 20*0.3+0.12+\x*0.3) -- (1.40, -5*0.3 + 0.12);
}

\draw[black] (1.5, -5*0.3) rectangle + (0.24, 0.24);
\foreach \l in {1, 2, 3}
\foreach \x in {-2, -1, 0, 1, 2}
{
	\draw[black] (1.5+\l, \x*0.3+ -5*0.3) rectangle +(0.24, 0.24);
}
\draw[black] (1.5+4, -5*0.3) rectangle+(0.24, 0.24);
\foreach \y in {-2, -1, 0, 1, 2}
{
	\draw[->, myred] (1.5+0.24, 0.12-5*0.3) -- (1.5+1, \y * 0.3+0.12-5*0.3);
}
\foreach \l in {2,3}
	\foreach \x in {-2, -1, 0, 1, 2}
		\foreach \y in {-2, -1, 0, 1, 2}
		{
			\draw[->, myred] (1.5+\l-1+0.24, \x*0.3+0.12-5*0.3) -- (1.5+\l, \y * 0.3+0.12-5*0.3);
		}
\foreach \x in {-2, -1, 0, 1, 2}
	{
		\draw[->, myred] (1.5+4-1+0.24, \x*0.3+0.12-5*0.3) -- (1.5+4, 0 * 0.3+0.12-5*0.3);
	}
\draw[black] (1.5+4+0.24+0.6, 0.4-5*0.3) node{\footnotesize $T_M(\cdot)$};
\draw[->] (1.5+4+0.24 + 0.1, 0.12-5*0.3) -- (1.5+4+0.24+1.1, 0.12-5*0.3);
\draw[->, myred] (7, 0.12-5*0.3) -- (9, 8.5*0.3) node[sloped, midway, above] {\footnotesize $\color{myred}{\beta_p}$};
\draw[black, rounded corners] (1.3, -8*0.3 + 0.15) -- (1.3, -1*0.3-0.15) -- (7, -1*0.3-0.15) -- (7, -8*0.3+0.15) -- cycle;
\draw[black] (6, -6.5*0.3) node{\small $g_p(x_p-\bv_p^\top \tilde{\bbf})$};

\draw (9.3, 8.5*0.3) circle [radius=0.3];
\draw (9.3, 8.5*0.3) node{\small $\sum$};
\draw[->] (9.6, 8.5*0.3) -- (11, 8.5*0.3) node[sloped, midway, above] {\small $m(\bx)$};

\end{tikzpicture}

\caption{Visualization of a FANAM estimator. The red color represents weights to be learned from data, the blue color represents the pre-defined fixed weights. }
\label{fig:fanam}

\end{figure}
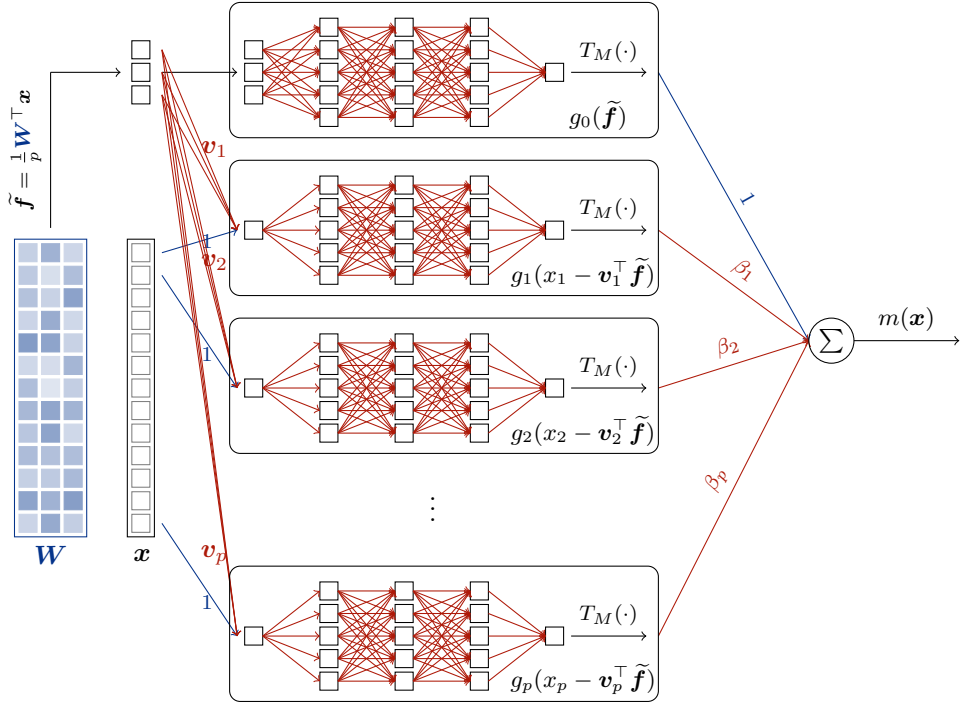

Let $s_u = |\mathcal{J}_u|$, $s_x = |\mathcal{J}_x|$ and $s=s_u + s_x + 1$. We are ready to present the error bound on the FANAM estimator in the regime that $s^4 \log p \ll n$.

\begin{theorem}
\label{thm:fanam}
   Suppose $(\bx_1,y_1), \ldots, (\bx_n, y_n)$ are i.i.d. samples from high-dimension nonparametric regression model that
    \begin{align*}
        \bx_i = \bB \bbf_i + \bu_i ~~~~ \text{and} ~~~~ y_i = m^*(\bbf_i, \bu_i) + \varepsilon_i
    \end{align*}with $m^*$ satisfies Condition \ref{cond:faam}. In addition, suppose  $\overline{r} \le NL$, $\log (NL) \lesssim \log n$, and $M^*\le M \lesssim 1$. Let  $c_1$--$c_3$ be some universal constants.
    If we choose 
    \begin{align*}
        \lambda \ge c_1\underbrace{\sqrt{\frac{(\log p) + (NL\log n)^2}{n}}}_{\delta_{\mathtt{s}}},
    \end{align*}
     then under Conditions \ref{cond1} -- \ref{cond4}, the FANAM estimator \eqref{eq:fanam-est} satisfies
    \begin{align*}
        \int |\hat{m}_{\mathtt{FANAM}}(\bx) - m^*(\bbf, \bu)|^2 \mu(d\bbf, d\bu) \le c_2 \left\{s\lambda + s\delta_{\mathtt{a}+\mathtt{f}} + s^2 \left(1+\lambda^{-1} \delta_{\mathtt{a}+\mathtt{f}} \right)^2 \delta_{\mathtt{s}} \right\}
    \end{align*} 
    with probability at least $1-e^{-t^*}$, where
    \begin{align*}
        \delta_{\mathtt{a}+\mathtt{f}} = s \left(\frac{NL}{\log^2 n}\right)^{-4\gamma^*} +  \frac{\overline{r} \cdot r(s_u + 1)}{\nu_{\min}^2(\bH) \cdot p} ~~~~\text{and}~~~~ t^* \asymp \sqrt{\frac{n}{s^4 \left(1+\lambda^{-1} \delta_{\mathtt{a}+\mathtt{f}}\right)^4}} \bigwedge \frac{n \delta_{\mathtt{a}+\mathtt{f}}}{s} \bigwedge (n\delta_{\mathtt{s}}^2).
    \end{align*} 
    In particular, with the optimal choice of $N$, $L$ and $\lambda$ such that
    \begin{align*}
        NL \asymp n^{\frac{1}{2(4\gamma^*+1)}} (\log n)^{\frac{8\gamma^*-1}{4\gamma^*+1}} ~~~~\text{and}~~~~ \lambda \asymp s \left\{\sqrt{\frac{\log p}{n}} + \left(\frac{\log^6 n}{n}\right)^{\frac{2\gamma^*}{4\gamma^*+1}}\right\},
    \end{align*} we have
    \begin{align}
    \label{eq:fanam-optimal-roc}
        \|\hat{m}_{\mathtt{FANAM}}(\bx) - m^*(\bbf, \bu)\|^2_2 \le c_3 s\left\{s \sqrt{\frac{\log p}{n}} + s\left(\frac{\log^6 n}{n}\right)^{\frac{2\gamma^*}{4\gamma^*+1}} + (s_u + 1) \frac{\overline{r} \cdot r}{\nu_{\min}^2(\bH)\cdot p}\right\}
    \end{align} with probability at least
    \begin{align*}
        1-\exp\left(-\frac{n}{s^4} \land \left(n^{\frac{1}{4\gamma^*+1}} (\log n)^{\frac{24\gamma^*}{4\gamma^*+1}} + \log p\right)\right).
    \end{align*}
\end{theorem}

As shown in \eqref{eq:fanam-optimal-roc}, the error bound consists of three terms as Theorem \ref{thm:fast-roc} do. The first term is related to the variable selection uncertainty, the second term is related to estimating a function with some low-dimension structure. The last term, which scales linearly with the number of additive functions of the idiosyncratic components $\bu$ and is exactly $0$ when $r=0$, is proportional to the cumulative error of estimating the latent factor structure from the observation $\bx$.

Compared to \cite{raskutti2012minimax}, our rate has a quadratic dependence on $s$ and suffers from a slower rate compared with $(\log p)/n$ and $n^{-\frac{2\gamma^*}{2\gamma^*+1}}$ as they do. This should be attributed to the lack of the nonparametric version of Restricted Strong Convexity condition \eqref{eq:high-dim-rsc}. It is still open whether such RSC condition holds in an estimated latent factor space, for example, whether the following condition holds
\begin{align*}
    \left\|f_0(\tilde{\bB}^\top \bx) + \sum_{j=1}^p f_j(x_j - \tilde{\bv}_j^\top \tilde{\bB}^\top \bx)\right\|_2^2 \asymp \left\|f_0(\tilde{\bB}^\top \bx)\right\|_2^2 + \sum_{j=1}^p \left\|f_j(x_j - \tilde{\bv}_j^\top \tilde{\bB}^\top \bx)\right\|_2^2,
\end{align*} for arbitrary $f_0,\cdots, f_p$ such that $\mathbb{E}[f_j]\equiv 0$, where $\tilde{\bB}\in \mathbb{R}^{p\times \overline{r}}$ and $\tilde{\bv}_j \in \mathbb{R}^{\overline{r}}$ are such that $\tilde{\bB}\bx$ is a fair estimate of $\bbf$ and $x_j - \tilde{\bv}_j^\top \tilde{\bB}^\top \bx$ is a fair estimate of $u_j$. If such condition is satisfied, we can take advantage of these fixed weights $\tilde{\bB}\in \mathbb{R}^{p\times r}$ and $\tilde{\bv}_j \in \mathbb{R}^p$ and modify slightly our proof to obtain a convergence rate of 
\begin{align*}
    s \left(\frac{\log p}{n} + n^{-\frac{2\gamma^*}{2\gamma^*+1}}\right) + (s_u + 1) \frac{\overline{r} \cdot r}{\nu_{\min}^2(\bH)\cdot p}
\end{align*} up to logarithmic factors of $n$ using a similar regularization term as that in \cite{raskutti2012minimax}.

\section{Empirical Applications}
\label{sec:application}

In this section, we compare our FAST-NN estimator with other high-dimension linear estimators using a macroeconomics dataset FRED-MD \citep{mccracken2016fred} to illustrate that our proposed estimator can find nonlinear associations in high-dimensional data. 

The FRED-MD dataset collects $p=134$ monthly U.S. macroeconomics variables starting from 1959/01 such as unemployment rate and real personal income. It is shown in \cite{mccracken2016fred} that the variables can be explained well by several latent factors. We consider predicting the variables $\mathtt{UEMP15T26}$, $\mathtt{TB3SMFFM}$, and $\mathtt{TB6SMFFM}$ using other variables. 
The variable $\mathtt{UEMP15T26}$ represents the civilians unemployed for 15-26 weeks. The variable $\mathtt{TB3SMFFM}$ ($\mathtt{TB6SMFFM}$) measures the 3-month (6-month) treasury bill rate minus the effective federal funds rate.
For each target response variable $y$ in $\{\mathtt{UEMP15T26}, \mathtt{TB3SMFFM}, \mathtt{TB6SMFFM}\}$, we use $\bx_{t}$ to predict $y_t$ for all the time index $t\in [T]$, where $\bx$ is the vector of all the other variables.

We slightly change the implementation of the FAST-NN estimator. We fix the hyper-parameter $\overline{r}=5$, $\tau=10^{-1}$, and use a neural network with depth $L=3$ and width $N=32$ because the sample size $n$ and the ambient dimension $p$ are not very large. The hyper-parameter $\lambda$ is determined via the validation set. We do not use other regularization techniques except early stopping, which is also determined by the validation set. For the competing estimators, we consider the following three high-dimensional linear models: Lasso \citep{tibshirani1997lasso}, Principal Component Regression (PCR) and Factor Augmented Regression Model \citep{fan2022latent} (FARM). For the implementation of PCR and FARM, the number of factors is estimated in a data-driven way as in \cite{fan2022latent}. Moreover, the hyper-parameter associated with the $\ell_1$ penalty is determined according to the validation set. 

We consider using the data from January 1980 to July 2022 and the data preprocessing done by \cite{mccracken2016fred}, after which the sample size is $n=330$. In particular, we use data from 1980/01 to 2009/9 ($60\%$ of the data, 200 months) for training and validation, and use data from 2009/10 to 2022/07 ($40\%$ of the data, denoted as $\mathcal{D}_{test}$, 130 months) for testing. For the former part of the data, we further use random sampled $70\%$ of them ($\mathcal{D}_{train}$) to train the model and use the rest $30\%$ of them ($\mathcal{D}_{valid}$) for validation and model selection. The performance of the estimator $\hat{m}$ is evaluated via the out-of-sample $R^2$, which is defined as
\begin{align*}
    R^2_{oos} = 1 - \frac{\sum_{(\bx,y)\in \mathcal{D}_{test}} (\hat{m}(\bx) - y)^2}{\sum_{(\bx, y) \in \mathcal{D}_{test}} (\bar{y}_{train}-y)^2} ~~~~~~ \text{with} ~~ \bar{y}_{train} = \sum_{(\bx, y)\in \mathcal{D}_{train}} y.
\end{align*} To alleviate the effects caused by the random split of the training and validation set and the algorithms' randomness, we repeat these splits $30$ times, and report the averaged out-of-sample $R^2$. 

\begin{table}[htb!]
    \centering
    \begin{tabular}{lllllll}
      \hline
      Data & FAST-NN & FARM & Lasso & PCR \\
      \hline
      $\mathtt{UEMP15T26}$ & 0.876 & 0.771 & 0.773 & 0.060 \\
      $\mathtt{TB3SMFFM}$ & 0.892 & 0.801 & 0.812 & 0.492 \\
      $\mathtt{TB6SMFFM}$ & 0.927 & 0.895 & 0.872 & 0.542 \\
      \hline
    \end{tabular}
	\caption{Out-of-sample $R^2$ for predicting the variables $\mathtt{UEMP15T26}$, $\mathtt{TB3SMFFM}$ and $\mathtt{TB6SMFFM}$ using different estimators. }
    \label{table:fred-md}
\end{table}

The results are presented in Table \ref{table:fred-md}. We can see that the FAST-NN estimator outperforms the best of these high-dimensional linear models by a large margin and reduce the $L_2$ risk by $55\%$, $43\%$ and $30\%$ compared with the best linear models for the  three response variables, respectively. The results provides stark evidence on the FAST-NN estimator's capacity to detect nonlinear relationships for high-dimensional data and demonstrate the importance of using idiosyncratic components to improve the performance of PCR.

\section{Detailed Related works}
\label{sec:relatedworks}

\noindent \textbf{Factor model.} A stylized feature of high dimensional data is that observed data are often dependence. One common model for such a dependence is the factor model \eqref{eq:intro-factor-model} \citep{forni2000generalized,bai2003inferential,hallin2007determining}, in which the dependence among explanatory variables is driven predominantly by linear combinations of common latent factors. Given that the latent factors $\bbf$ can be inferred from observation $\bx$ via various methods such as principal component analysis \citep{stock2002forecasting,bai2003inferential}, maximum likelihood estimation \citep{bai2012statistical, doz2012quasi}, low-rank estimation \citep{agarwal2012noisy}, and covariance estimation \citep{fan2013large}, there is a considerable literature on forecasting with the estimated factors. On one hand, from the motivation of dimension reduction, \cite{stock2002forecasting, bai2006confidence, bair2006prediction, bai2008forecasting} considered the factor regression (or factor-augmented regression) model in which there is a linear association between the response variable and latent factor, and use estimated factors to predict response variable $y$. \cite{bai2008forecasting} took quadratic factor into regressor and \cite{fan2017sufficient} generalized it to a nonlinear model with multiple indices. On the other hand, standard sparse linear regression models such as LASSO \citep{tibshirani1996regression} and SCAD \citep{fan2001variable} will suffer from the strong dependence among the explanatory variables. Hence, there are attempts to bridge factor regression and sparse linear regression together \citep{fan2021bridging, fan2022latent} to fully exploit the low-dimension structure in high-dimension linear regression.

We also notice some literature on factor models using neural networks \citep{chen2019deep, gu2021autoencoder, fan2022structural}. Their problems of study mainly focused on estimating the condition asset pricing models, i.e., predicting $\bx-\bu$. Hence, our work differs a lot from theirs. 

\medskip

\noindent \textbf{High-dimensional nonparametric regression.} There are several attempts to estimate the regression function $m^*$ in the regime $p \gg n$. The most well-developed case is the high dimension additive model \citep{ravikumar2009sparse, meier2009high, raskutti2012minimax, yuan2016minimax} that $m^*(\bx) = \sum_{j=1}^p \beta_j m_j^*(x_j)$ with sparsity constraint $\|\bbeta\|_q \le s$ for $q\in [0,1]$. Under a nonparametric version of the Restricted Strong Convexity (RSC) condition that 
\begin{align}
\label{eq:high-dim-rsc}
    \mathbb{E}\left[\left|\sum_{j=1}^p f_j(x_j)\right|^2\right] \asymp \sum_{j=1}^p \mathbb{E} \left[|f_j(x_j)|^2\right] ~~~~~~ \text{if} ~~~~ \forall j, ~~\mathbb{E}\big[f_j(x_j)\big] = 0,
\end{align} \cite{raskutti2012minimax} showed the minimax optimal $L_2$ excess risk when $q=0$ is $s(\frac{\log p}{n} + n^{-2\beta/(2\beta+1)})$ if the univariate functions have the same smoothness $\beta>0$, and such convergence rate can be attained via regression in reproducing kernel Hilbert spaces (RKHS) with smoothness and sparse penalties. However, such RSC condition will be violated in the presence of highly correlated explanatory variables. Moreover, \cite{yang2015minimax} extended the result to sparse interaction models, that $m^* = \sum_{k=1}^s m^*_k$ where each component function $m^*_s$ only depends on $d_k=o(\log n)$ variables, and derived corresponding minimax optimal $L_2$ excess risk, which can be achieved via Bayesian additive Gaussian process regression. There are also some works about variable selection in high-dimension nonparametric regression, for example, \cite{lafferty2008rodeo, fan2011nonparametric, comminges2012tight}. 

\medskip

\noindent \textbf{Neural networks with high-dimensional input.} The presence of high-dimension data also motivates the development of algorithms to extract (sparse) important variables when using a neural network estimator. Many of these works adopt the idea of applying a (group) Lasso type penalty on the input weights of the neural network \citep{scardapane2017group, feng2017sparse, ho2020consistent, lemhadri2021lassonet}. \cite{feng2017sparse} showed the excess risk converges at the rate $n^{-1} \log p$ for regression and classification when using the sparse-input neural network. However, such a result only applies to neural networks with fixed depth and width, and requires additional conditions for the optimal solution $\btheta^*$ in the neural network parameter space. Therefore, their result is not applicable in the general high-dimensional nonparametric regression scenario.

\section{More discussions}
\label{sec:discussions}

\subsection{Our new neural network approximation result}
\label{dis:approx}

\begin{table}[htb!]
    \centering \footnotesize
    \begin{tabular}{lllllll}
      \hline
       & depth & width & sparsity & $\|\bW\|_{\max}$ & approx & entropy \\
      \hline
      \cite{schmidt2020nonparametric} & $\log N$ & $N^2$ & Yes & 1 & $N^{-2\beta/d}$ & $\log \mathcal{N}_\infty(\epsilon) \lesssim N^2\log( \epsilon^{-1}N) \log N$ \\
      \cite{kohler2021rate} & $\log N$ & $N$ & No & $\infty$ & $N^{-2\beta/d}$ & $\mathrm{Pdim} \lesssim N^2 (\log N)^2$\\
      \cite{kohler2021rate} & $L$ & $1$ & No & $\infty$ & $L^{-2\beta/d}$ & $\mathrm{Pdim} \lesssim L^2$\\
      \cite{lu2021deep} & $L\log L$ & $N\log N$ & No & $\infty$ & $(NL)^{-2\beta/d}$ &  $\mathrm{Pdim} \lesssim (NL)^2 (\log N\log L)^2\log(NL)$ \\
      Ours & $1$ & $N$ & No & $\text{poly}(N)$ & $N^{-2\beta/d}$ & $\log \mathcal{N}_\infty(\epsilon) \lesssim N^2 \log (\epsilon^{-1} N)$ \\
      \hline
    \end{tabular}
	\caption{A summary of ReLU neural network approximation result for $d$-variate $(\beta,1)$-smooth function. We omit constants dependent on $\beta$ and $d$ for a clean presentation. The first four columns specify the deep ReLU network function class used for approximation in terms of the hyper-parameters $L$ and $N$ (but not necessarily refer to depth and width) if they are flexible to tune. The column `approx' presents the approximation error for the deep ReLU network class. The column `entropy' characterize the statistical complexity of the deep ReLU network class using two metrics: the Pseudo-dimension $\mathrm{Pdim}$ for those using unbounded weights, and the logarithm of covering number with respect to $\|\cdot\|_{\infty, [0,1]^d}$ norm $\log \mathcal{N}_\infty(\epsilon)$ for those with explicitly bounded weights. }
    \label{table:nn-approx}
\end{table}

We summarize our result and the previous ReLU network approximation results in Table \ref{table:nn-approx}. As a comparison, there are two main refinements in our results compared with the existing ones. Such refinements are necessary for our convergence rate analysis Theorem \ref{thm:fast-roc}.

\begin{itemize}
    \item[1.] Compared with \cite{kohler2021rate} and \cite{lu2021deep}, we develop a comparable ReLU network approximation error result with bounded weights constraints, which makes \eqref{eq:reg-effect-tau} possible by choosing appropriate $\tau$. To the best of our knowledge, Theorem \ref{thm:approx-smooth} is the first ReLU network approximation result that can achieve the optimal approximation error up to logarithmic factors under the constraints that (1) use polynomial order weight magnitude and (2) do not need to impose sparsity on network weights. As a comparison, \cite{lu2021deep} explicitly uses weights scales at least $e^{N+L}$, and \cite{schmidt2020nonparametric} imposes sparsity on network weights such that their total number of active parameters is $NL$ for depth $L$ and width $N$ deep ReLU network instead of $N^2L$ as we do. 
    \item[2.] We have some improvements on logarithmic factors. In other words, for fixed approximation error to achieve, the entropy of the deep ReLU network class $\mathcal{G}$ we used is the smallest. Specifically, we can use constant order of depth $L$. This will contribute to (1) faster convergence rate for estimating a single function with hierarchical composition structure, i.e., $(\log n/n)^{\frac{2\gamma^*}{2\gamma^*+1}}$, and (2) milder requirement of $\tau$, $\log(\tau^{-1}) \gtrsim \log n$ by a close inspection of the oracle-type inequality in Theorem \ref{thm:fast-oracle}.
\end{itemize}

\subsection{Practical utility and model misspecification}
\label{dis:extend-nonlinear}

In the main text, we proposed the FAR-NN estimator and the FAST-NN estimator. Some readers might wonder how to choose between these two methods empirically. In this section, we argue that the user can opt for the FAST-NN estimator directly and recommend incorporating FAST-NN into the frequently utilized toolkit for making predictions based on high-dimensional input. This is due to its adaptability and robustness against misspecification, which we will explain further below.

\paragraph{Regression function reduced to specific one. } The design of our proposed FAST-NN estimator assumes the existence of latent factor and the dependency of regression function on latent factors and a small set of idiosyncratic components. We argue that one need to pay little expense for such generality of design if the true regression function reduces to a specific one. 

\begin{description}
\item[$|\mathcal{J}|=0$.] It reduces to a (nonparametric) factor regression model, which is encompassed within the FAST model. Although our proposed FAR-NN leverages the inductive bias that only factor contributes to the prediction while FAST-NN does not, their empirical performance gap is negligible in practice. The simulations of Exp \uppercase\expandafter{\romannumeral1} in Section \ref{sec:exp:simluation-far} support such an argument; see the results in Figure \ref{fig:exp1} (a) and the corresponding discussion in Section \ref{sec:exp:simluation-far}. Moreover, the extra theoretical cost the FAST-NN estimator need to pay is $(\log p)/n$ according to a direct corollary of Theorem \ref{thm:fast-roc}, which is negligible when $p$ only grows at most polynomially with $n$, i.e., $p = O(n^C)$ for some $C>0$.
\item[$r=0$.] It now reduces to a nonparametric sparse regression model, and is also contained in the FAST model. Our Corollary \ref{coro:fast-r0} reveals that one doesn't need to pay any extra theoretical cost by falsely assuming the existence of latent factors in the design of FAST-NN estimator (for example, letting $\bar{r}=10$).
\end{description}

\paragraph{Model misspecification.} So far, we have assumed that the covariate $\bx$ admits a linear factor model \eqref{eq:intro-factor-model}. We claim that both our proposed method and corresponding theoretical justification are applicable in a more general setting. To be specific, we consider the case the linear factor structure \eqref{eq:intro-factor-model} does not hold, and it is relaxed by
\begin{align*}
    \bx = (x_1,\ldots, x_p) \text{ with } x_j = g_j(\bm{f}) + u_j
\end{align*} with unknown functions $g_1,\ldots, g_p: \mathbb{R}^p \to \mathbb{R}$, and the regression function assumption \eqref{eq1.3} still holds that
\begin{align*}
    \mathbb{E}[y|\bbf, \bu] = m^*(\bbf_i, \bu_{\mathcal{J}}).
\end{align*}
This model is highly generalized, with little restrictive assumptions, and our FAST model is a special case of the above model when $g_j(\bm{z}) = \bm{b}_j^\top \bm{z}$. We briefly outline the rationale behind achieving a similar theoretical result while leaving a comprehensive treatment for future work.
        
Let $\bW = [\bw_1,\ldots, \bw_{p}]^\top \in \mathbb{R}^{p\times \bar{r}}$ be the given fixed diversified projection matrix, and assume further that we are able to recover $\bbf$ given $\bar{\bbf}$ together with the map $\bar{\bbf} = G(\bbf) = \frac{1}{p} \sum_{j=1}^p \bw_j g_j(\bbf)$, that is, there exists some $G^{-1}: \mathbb{R}^{\overline{r}} \to \mathbb{R}^r$ satisfying
\begin{align*}
    G^{-1}(G(\bbf)) = G^{-1} \Bigg(\frac{1}{p} \sum_{j=1}^p \bw_j g_j(\bbf)\Bigg) = \bbf ~~~~~~~~ \forall \bbf \in \mathbb{R}^p, 
\end{align*} and $G^{-1}$ is $(\beta_*, C_p)$ smooth for some universal constant $C_p$ depends only on $p$. Then, following a similar proof strategy of Theorem \ref{thm:fast-oracle}, we can obtain a generalized result of Theorem \ref{thm:fast-roc}. Specifically, with properly chosen hyper-parameters, the following holds with probability at least $1-e^{-t}$,
\begin{align}
    \label{eq:general-nonlinear-factor-model}
    \|\hat{m}_{\mathtt{FAST}}-m^*\|_2^2 \lesssim \left(\frac{\log^6 n}{n}\right)^{\frac{2\gamma_\dagger}{2\gamma_\dagger+1}} + \frac{\log p}{n} + \Bigg(\frac{\overline{r}\cdot C_p^2}{p} \Bigg)^{\beta_*\land 1} + \frac{t}{n},
\end{align}
where $\gamma_\dagger = \gamma^* \land \frac{\beta_*}{\overline{r}}$. And our Theorem \ref{thm:fast-roc} is a special case of \eqref{eq:general-nonlinear-factor-model} with $G^{-1}(\bar{\bbf}) = \bH^+ \bar{\bbf}$, $\beta_* = \infty$ and $C_p\le \|\bH^+\|_2 \le \nu^{-1}_{\min}(\bH)$.

The above discussions demonstrate the potential capability of our proposed method to learn the latent factor structure in an algorithmic manner. Though it may still be worth exploring the interpretability of the factor structure of our proposed estimator extract, our estimator can be a potential candidate under situations other than pure prediction tasks. For example, it can serve as a role of estimating the non-parametric part in semi-parametric models; see what has already been done for the low-dimensional counterpart in \cite{farrell2021deep, zhong2022deep}.

\section{Proofs for the FAR-NN estimator in Section \ref{sec:theory:fa}}
\label{sec:proof-far-nn-estimator}

\subsection{Proof of Theorem \ref{thm:fa-oracle}}

We need the following technical lemma to establish a relationship between $\tilde{\bbf}$ and $\bbf$.

\begin{lemma}
\label{lemma:approx-factor}
Let $g: \mathbb{R}^r \to \mathbb{R}$ be a fixed $C$-Lipschitz function, $\bW$ be the diversified projection matrix in Definition \ref{def:dpm}, $\bH$ be the matrix in Definition \ref{def:dpm}, $\bH^+\in \mathbb{R}^{r\times \overline{r}}$ be the Pseudo-inverse of $\bH$, let $\tilde{g}: \mathbb{R}^{\overline{r}} \to \mathbb{R}$ be $\tilde{g}(\cdot)=g(\bH^+\cdot)$.  Then, under Conditions \ref{cond1}, \ref{cond3} and \ref{cond5} with $\max_{j\in \{1,\cdots, p\}} \mathbb{E} |u_j|^2 \le c_1$ for a universal constant $c_1$, there exists a universal constant $c_2$ such that
\begin{align}
    \mathbb{E} \left[\left|\tilde{g}(\tilde{\bbf}) - g(\bbf)\right|^2 \right] \le c_2 \left(\frac{C}{\nu^2_{\min}(\bH)} \cdot \frac{\overline{r}}{p} \right).
\end{align}
\end{lemma}

We also need the following lemmas about the empirical process. We first introduce some notations before presenting the lemmas. Let $\bz_1,\cdots, \bz_n$ be i.i.d. copies of $\bz \sim \mathcal{Z}$ from some distribution $\mu$, $\mathcal{H}$ be a real-valued function class defined on $\mathcal{Z}$. Define the empirical $L_2$ norm and population $L_2$ norm for each $h\in \mathcal{H}$ respectively as
\begin{align*}
    \|h\|_n = \left(\frac{1}{n} \sum_{i=1}^n h(\bz_i)^2\right)^{1/2} ~~~~\text{and}~~~~\|h\|_2 = \left(\mathbb{E}[h(\bz)^2]\right)^{1/2} = \left(\int h(\bz)^2 \mu(d\bz) \right)^{1/2}.
\end{align*}
Throughout the proof of the FAR-NN estimator, the choice of $\mathcal{Z}$ and $\mathcal{H}$ will vary in different contexts. However, we can define a unified empirical (population) $L_2$ norm on the function class $\mathcal{H} = \mathcal{H}_1 - \mathcal{H}_1$ over $\mathcal{Z}=\{\bz=(\bbf, \bu)\}$, where the difference between two sets is defined as $\mathcal{A}-\mathcal{B} = \{a-b: a\in \mathcal{A}, b\in \mathcal{B}\}$, and
\begin{align*}
    \mathcal{H}_1 = \left\{g(p^{-1}\bW^\top \bx) = g(p^{-1}\bW^\top (\bB\bbf+\bu)): g\in \mathcal{G}(L, \overline{r}, 1, N, M, \infty)\right\} \cup \left\{m^*(\bbf)\right\} \cup \{0\}
\end{align*} 

The first lemma bounds the difference between the population $L_2$ norm and empirical $L_2$ norm.
\begin{lemma}
\label{lemma:equivalence-population-empirical-l2}
Let $\bz_1,\ldots, \bz_n \in \mathcal{Z}$ be i.i.d. copies of $\bz$, $\mathcal{G}$ be a uniformly bounded function class with finite Pseudo-dimension defined on $\mathcal{Z}$.  Then, there exist some universal constants $c_1$--$c_3$ such that  for all  $t \ge \epsilon_n/(\eta(1-\eta))$ with $0<\eta<1$ and $\epsilon_n = c_1 \sqrt{\frac{\mathrm{Pdim}(\mathcal{G})\log n}{n}}$, we have
\begin{align*}
    \big| \|g\|_n^2 - \|g\|^2_2 \big| \le \eta \left(\|g\|_2^2 + t^2\right) \qquad \forall g\in \mathcal{G} 
\end{align*} with probability at least $1-c_{2} \exp(-{c_3} n\eta^2(1-\eta)^2 t^2)$.
\end{lemma}

The second lemma establishes the tail probability of the weighted empirical process.

\begin{lemma}
\label{lemma:fa-weight-empirical-process}
Let $\bz_1,\ldots, \bz_n$ be fixed, $\varepsilon_1,\ldots, \varepsilon_n$ be independent sub-Gaussian random variables with parameter $\sigma$. Suppose $\mathcal{G}$ is a uniformly bounded function class with finite Pseud-dimension defined on $\mathcal{Z}$, and $\tilde{g}$ is a fixed function in $\mathcal{G}$. If $n \ge 3$ and $\mathrm{Pdim}(\mathcal{G}) \ge 1$, then there exists universal constants $c_1$--$c_2$, such that the event
\begin{align*}
    \mathcal{B}_t(\bz_1^n) = \left\{ \forall g\in \mathcal{G}, \left|\frac{1}{n} \sum_{i=1}^n \varepsilon_i \left(g(\bz_i) - \tilde{g}(\bz_i)\right) \right|\le c_1 \left( \|g - \tilde{g}\|_n + \epsilon \right) \sqrt{v_n^2 + \frac{t}{n}}\right\},
\end{align*} occurs with probability at least $1-c_2 \log(1/\epsilon) e^{-t}$ for any $t>0$ and $\epsilon>0$, where $v_n = \sqrt{\frac{\mathrm{Pdim}(\mathcal{G})\log n}{n}}$.
\end{lemma}

Now we are ready to prove Theorem \ref{thm:fa-oracle}:

\begin{proof}[Proof of Theorem~\ref{thm:fa-oracle}]

\noindent {\sc Step 1. Find an Approximation of $m^*$}. 
By using Lemma \ref{lemma:approx-factor}, the function $\tilde{m}(\tilde{\bbf}) = m^*(\bH^+ \tilde{\bbf})$ satisfies
\begin{align}
\label{eq:proof:thm:fa-oracle:eq0}
    \mathbb{E} \left[\left|\tilde{m}(\tilde{\bbf}) - m^*(\bbf)\right|^2 \right] \le C_1 \delta_{\mathtt{f}}.
\end{align} According to Condition \ref{cond1}, suppose the factor $\bbf$ is supported on $[-b, b]^r$ for some $b > 0$. The definition of $\delta_{\mathtt{a}}$ implies that there exists a deep ReLU neural network $h\in \mathcal{G}(L, r, 1, N, M, \infty)$ such that 
\begin{align}
\label{eq:proof:thm:fa-oracle:approx-assump}
    \|h(\bbf) - m^*(\bbf) \|_\infty \le \sqrt{2 \cdot \delta_{\mathtt{a}}} \qquad \forall \bbf\in [-2b,2b]^r.
\end{align} Denote $\tilde{g}(\tilde{\bbf}) = h(\bH^+ \tilde{\bbf})$. It is easy to verify that $\tilde{g} \in \mathcal{G}(L, \overline{r}, 1, N, M, \infty)$. Our goal in this step is to show that 
\begin{align*}
    \mathbb{E} \left|\tilde{g}(\tilde{\bbf}) - m^*(\bbf)\right|^2 \lesssim \delta_{\mathtt{a}} + \delta_{\mathtt{f}}.
\end{align*} 
To this end, it follows from triangle inequality and Young's inequality that
\begin{align}
\label{eq:proof:thm:fa-oracle:eq1}
     \mathbb{E} \left|\tilde{g}(\tilde{\bbf}) - m^*(\bbf)\right|^2 \le 2 \left\{\mathbb{E} \left|\tilde{g}(\tilde{\bbf})-\tilde{m}(\tilde{\bbf})\right|^2 + \mathbb{E} \left|\tilde{m}(\tilde{\bbf}) - m^*(\bbf)\right|^2\right\} \lesssim \mathbb{E} \left|\tilde{g}(\tilde{\bbf})-\tilde{m}(\tilde{\bbf})\right|^2 + \delta_{\mathtt{f}}.
\end{align}

Applying the truncation argument yields 
\begin{align*}
    \mathbb{E} \left|\tilde{g}(\tilde{\bbf})-\tilde{m}(\tilde{\bbf})\right|^2 = \mathbb{E} \left|\tilde{g}(\tilde{\bbf})-\tilde{m}(\tilde{\bbf})\right|^2 1_{\{\bH^+\tilde{\bbf} \in [-2b, 2b]^r\}} + \mathbb{E} \left|\tilde{g}(\tilde{\bbf})-\tilde{m}(\tilde{\bbf})\right|^2 1_{\{\bH^+\tilde{\bbf} \notin [-2b, 2b]^r\}}.
\end{align*}
Combining the approximation assumption \eqref{eq:proof:thm:fa-oracle:approx-assump}, the definition of $\tilde{g}$, $\tilde{m}$ together with the condition $\bH^+\tilde{\bbf} \in [-2b, 2b]^r$, we have
\begin{align*}
    \mathbb{E} \left|\tilde{g}(\tilde{\bbf})-\tilde{m}(\tilde{\bbf})\right|^2 1_{\{\bH^+\tilde{\bbf} \in [-2b, 2b]^r\}} = \mathbb{E}  \left|h(\bH^+\tilde{\bbf})-m^*(\bH^+\tilde{\bbf})\right|^2 1_{\{\bH^+\tilde{\bbf} \in [-2b, 2b]^r\}} \le 2 \delta_{\mathtt{a}}
\end{align*}
Moreover, by the fact that $h$ and $m^*$ are all bounded by constants $M$ and $M^*$ respectively,
\begin{align*}
    \mathbb{E} \left|\tilde{g}(\tilde{\bbf})-\tilde{m}(\tilde{\bbf})\right|^2 1_{\{\bH^+\tilde{\bbf} \notin [-2b, 2b]^r\}} &\le (M+M^*)^2 \mathbb{P}\left[\bH^+\tilde{\bbf} \notin [-2b, 2b]^r\right] \\
    &\overset{(a)}{\lesssim} \mathbb{P}\left[\left\|p^{-1} \bH^+ \bW^\top \bu\right\|_2 \ge b\right] \\
    &\overset{(b)}{\lesssim} \frac{1}{b^2} \mathbb{E} \left\|p^{-1} \bH^+ \bW^\top \bu\right\|_2^2,
\end{align*} where (a) follows from the decomposition $\bH^+ \tilde{\bbf} = \bbf + p^{-1} \bH^+ \bW^\top \bu$, (b) follows from Markov inequality. As a by-product of Lemma \ref{lemma:approx-factor},  $\mathbb{E} \left\|p^{-1} \bH^+ \bW^\top \bu\right\|_2^2 \lesssim \delta_{\mathtt{f}}$. Putting these pieces together gives 
\begin{align}
\label{eq:proof:thm:fa-oracle:eq2}
    \mathbb{E} \left|\tilde{g}(\tilde{\bbf})-\tilde{m}(\tilde{\bbf})\right|^2 \le C_2(\delta_{\mathtt{f}} + \delta_{\mathtt{a}})
\end{align}
Plugging \eqref{eq:proof:thm:fa-oracle:eq0} and \eqref{eq:proof:thm:fa-oracle:eq2} into \eqref{eq:proof:thm:fa-oracle:eq1}, we can conclude that there exists a deep ReLU neural network $\tilde{g} \in \mathcal{G}(L,\overline{r}, 1, N, M, \infty)$ such that
\begin{align}
\label{eq:proof:thm-fa:approx-error-final}
    \mathbb{E} \left|\tilde{g}(\tilde{\bbf}) - m^*(\bbf) \right|^2 \le C_3(\delta_{\mathtt{f}} + \delta_{\mathtt{a}})
\end{align} for some universal constant $C_3>0$ as claimed.

\noindent {\sc Step 2. Derive Basic Inequality. } 
The empirical risk minimization objective \eqref{eq:thm-fa:opt-error} and the construction of $\tilde{g}$ in {\sc Step 1} implies
\begin{align*}
    \frac{1}{n} \sum_{i=1}^n \left(\hat{g}(\tilde{\bbf}_i) - y_i\right)^2 \le \frac{1}{n} \sum_{i=1}^n \left(\tilde{g}(\tilde{\bbf}_i) - y_i\right)^2 + \delta_{\mathtt{opt}}.
\end{align*}
Plugging in the data generating process of $y$ in \eqref{eq:dgp-samples}, and recall the definition of empirical $L_2$ norm $\|\cdot\|_n$, using simple algebra gives
\begin{align}
\label{eq:proof:thm-fa:basic-ineq0}
    \|\hat{g} - m^*\|_n^2 \le \|\tilde{g} - m^*\|^2_n + \frac{2}{n} \sum_{i=1}^n \varepsilon_i \left(\hat{g}(\tilde{\bbf}_i) - \tilde{g}(\tilde{\bbf}_i)\right) + \delta_{\mathtt{opt}}
\end{align} Moreover, note that $\|\cdot\|_n$ is a norm, this implies
\begin{align*}
    \|\hat{g} - \tilde{g}\|_n^2 \le 2\left(\|\hat{g} - m^*\|_n^2 + \|m^* - \tilde{g}\|_n^2\right)
\end{align*}
Combining with \eqref{eq:proof:thm-fa:basic-ineq0} gives the following inequality,
\begin{align}
\label{eq:thm-fa:basic-ineq}
    \|\hat{g} - \tilde{g}\|_n^2 \le 4 \|\tilde{g} - m^*\|_n^2 + \frac{4}{n} \sum_{i=1}^n \varepsilon_i \left(\hat{g}(\tilde{\bbf}_i) - \tilde{g}(\tilde{\bbf}_i)\right) + 2\delta_{\mathtt{opt}}.
\end{align} We refer to \eqref{eq:thm-fa:basic-ineq} as the basic inequality.

\noindent {\sc Step 3. Concentration for the Fixed Function. } Note that the choice of $m^*$ and $\tilde{g}$ is fixed, which implies the samples $z_i = \left|\tilde{g}(\tilde{\bbf}_i) - m^*(\bbf_i)\right|^2$ are i.i.d. because $\bW$ is independent of $(\bx_i)_{i=1}^n$. We use Bernstein inequality to establish a tail bound for $\left\|\tilde{g} - m^*\right\|_n^2$. To this end, the boundedness of $\tilde{g}$ and $m^*$ implies $|z_i| \le (M+M^*)^2$ and
\begin{align*}
    \text{Var}(z_i) \le \mathbb{E} [|z_i|^2] \le (M+M^*)^2 \mathbb{E}[|z_i|] \le C_4 \left(\delta_{\mathtt{f}} + \delta_{\mathtt{a}} \right),
\end{align*} where the last inequality follows from the result in {\sc Step 1}. It follows from Bernstein inequality (Proposition 2.10 \cite{wainwright2019high}) that 
\begin{align*}
    \frac{1}{n} \sum_{i=1}^n \left|\tilde{g}(\tilde{\bbf}_i) - m^*(\bbf_i)\right|^2 = \frac{1}{n} \sum_{i=1}^n z_i &\le \mathbb{E}[z_1] + C_5\left(\sqrt{\delta_{\mathtt{f}} + \delta_{\mathtt{a}}} \sqrt{\frac{u}{n}} + \frac{u}{n}\right) \\
    &\le C_6 \left(\delta_{\mathtt{f}} + \delta_{\mathtt{a}} + \frac{u}{n}\right)
\end{align*} 
for arbitrary $u>0$. Define the event 
\begin{align}
\label{eq:far-oracle-inequality:approx-l2-error}
    \mathcal{A}_t = \left\{\|\tilde{g} - m^*\|_n^2 \le C_6  \left(\delta_{\mathtt{f}} + \delta_{\mathtt{a}} + \frac{t}{n}\right)\right\}.
\end{align} We can conclude that $\mathbb{P}\left[\mathcal{A}_t\right] \ge 1-e^{-t}$.

\noindent {\sc Step 4. Concentration for Weighted Empirical Process.} Define $v_n = \sqrt{\frac{\mathrm{Pdim}(\mathcal{G}) \log n}{n}}$. Consider the event with $u$ to be determined
\begin{align*}
    \mathcal{B}_t = \left\{ \forall g\in \mathcal{G}, \left|\frac{1}{n} \sum_{i=1}^n \varepsilon_i \left(g(\tilde{\bbf}_i) - \tilde{g}(\tilde{\bbf}_i)\right) \right|\le C_7\left( \|g - \tilde{g}\|_n + v_n \right) \sqrt{v_n^2 + \frac{u}{n}}\right\},
\end{align*} 
where $C_7$ is the universal constant $c_1$ in the statement of Lemma \ref{lemma:fa-weight-empirical-process}. Applying Lemma \ref{lemma:fa-weight-empirical-process} with $\epsilon=v_n$ and $(\bz_1,\cdots, \bz_n) = (\tilde{\bbf}_1,\cdots, \tilde{\bbf}_n)$, we obtain
\begin{align*}
    \mathbb{P}(\mathcal{B}_t^c) = \mathbb{E}\left[\mathbb{P}\left(\mathcal{B}_t^c(\tilde{\bbf}_1^n)\big|\tilde{\bbf}_1^n\right)\right] \le e^{-u+C \log (\log n)},
\end{align*} provided $\varepsilon_1,\cdots, \varepsilon_n$ is sub-Gaussian with some constant parameter $C$ conditioned on fixed $\tilde{\bbf}_1,\cdots, \tilde{\bbf}_n$, which is validated by Condition \ref{cond4}. Letting $u=t + C\log (\log n)$. Under $\mathcal{B}_t$, we have
\begin{align*}
    \forall g\in \mathcal{G} ~~~~~~ \left|\frac{1}{n} \sum_{i=1}^n \varepsilon_i \left(g(\tilde{\bbf}_i) - \tilde{g}(\tilde{\bbf}_i)\right) \right|\lesssim \left( \|g - \tilde{g}\|_n + v_n \right) \sqrt{v_n^2 + \frac{t}{n}}
\end{align*} by the fact that $v_n^2 \gtrsim \frac{\log (\log n)}{n}$.

\noindent {\sc Step 5. Conclusion of the Proof.} Recall the basic inequality \eqref{eq:thm-fa:basic-ineq}. It follows from the definition of event $\mathcal{A}_t$ and $\mathcal{B}_t$ that we can write down the basic inequality as
\begin{align*}
    \|\hat{g} - \tilde{g}\|_n^2 \lesssim \delta_{\mathtt{opt}} + \delta_{\mathtt{a}} + \frac{t}{n} + \left(v_n + \|\hat{g} - \tilde{g}\|_n \right) \sqrt{v_n^2 + \frac{t}{n}},
\end{align*} under $\mathcal{A}_t \cap \mathcal{B}_t$, which implies
\begin{align}
\label{eq:proof:thm-fa:basic-ineq2}
    \|\hat{g} - \tilde{g} \|_n^2 \le C_{8} \left(\delta_{\mathtt{opt}} + \delta_{\mathtt{a}} + v_n^2 + \frac{t}{n}\right)
\end{align} for some universal constant $C_{8}$.

We next derive upper bound for $\|\hat{g} - \tilde{g} \|_2^2$ using Lemma \ref{lemma:equivalence-population-empirical-l2} with $\eta=\frac{1}{2}$. Because $\tilde{\bbf}_1,\cdots, \tilde{\bbf}_n$ are i.i.d. copies of $\tilde{\bbf}$, and $\bar{\mathcal{G}} = \{g-\tilde{g}: g\in \mathcal{G}\}$ is a function class with Pseudo-dimension $\mathrm{Pdim}(\mathcal{G})+1$, and for all the $g\in \bar{\mathcal{G}}$, $\|g\|_\infty \le 2M$, Define the event
\begin{align}
    \mathcal{C}_t = \left\{ \forall g\in \mathcal{G}, \text{ } \left|\left\|g - \tilde{g} \right\|_2^2 - \left\|g - \tilde{g} \right\|_n^2\right| \le \frac{1}{2} \left\|g - \tilde{g} \right\|_2^2 + \frac{C_{9}}{2} \left(v_n^2 + \frac{t}{n} \right)\right\}
\end{align} for some universal constant $C_{9}$, it follows from Lemma \ref{lemma:equivalence-population-empirical-l2} that $\mathbb{P}(\mathcal{C}_t) \ge 1-e^{-t}$. Therefore, under $\mathcal{C}_t$, we have
\begin{align*}
    \|\hat{g} - \tilde{g} \|_2^2 \le 2\|\hat{g} - \tilde{g} \|_n^2 + C_{9} \left(v_n^2 + \frac{t}{n}\right).
\end{align*} Combining with the inequality \eqref{eq:proof:thm-fa:basic-ineq2} under $\mathcal{A}_t \cap \mathcal{B}_t$, we can conclude that under $\mathcal{A}_t \cap \mathcal{B}_t \cap \mathcal{C}_t$, 
\begin{align}
    \|\hat{g} - \tilde{g} \|_2^2 \le C_{10} \left(\delta_{\mathtt{opt}} + \delta_{\mathtt{a}} + v_n^2 + \frac{t}{n}\right).
\end{align}

Moreover, note from \eqref{eq:proof:thm-fa:approx-error-final}, our construction of $\tilde{g}$ satisfies $\left\|\tilde{g} - m^*\right\|_2^2 \lesssim \delta_{\mathtt{opt}} + \delta_{\mathtt{a}}$, by triangle inequality and Young's inequality, 
\begin{align}
    \|\hat{g} - m^* \|_2^2 \le 2\|\hat{g} - \tilde{g} \|_2^2 + \|\tilde{g} - m^*\|_2^2 \le C_{11} \left(\delta_{\mathtt{opt}} + \delta_{\mathtt{a}} + v_n^2 + \frac{t}{n}\right) ~~~~ \text{under} ~~~~ \mathcal{A}_t \cap \mathcal{B}_t \cap \mathcal{C}_t,
\end{align} with $ v_n^2 = \frac{\mathrm{Pdim}(\mathcal{G}) \log n}{n}$ and $\mathbb{P}(\mathcal{A}_t \cap \mathcal{B}_t \cap \mathcal{C}_t) \ge 1-3e^{-t}$. The upper bound of the empirical $L_2$ error $\|\hat{g} - m^*\|_n^2$ follows a similar way by combining \eqref{eq:proof:thm-fa:basic-ineq2} and \eqref{eq:far-oracle-inequality:approx-l2-error} with triangle equality.

We conclude that proof via specifying the Pseudo-dimension of the deep ReLU network class used $\mathcal{G}(L,\overline{r}, 1, N, M, \infty)$. Applying further Theorem~7 of \cite{BHLM2019} yields the bound ${\rm Pdim}(\mathcal{G})  \lesssim W L \log(W)$, where $W$ is the number of parameters of the network $\mathcal{G}$. The input dimension of the neural network we used in $\mathcal{G}(L,\overline{r}, 1, N, M, \infty)$ is $\overline{r}$ rather than $p$, hence
\begin{align*}
    W = (L-1) (N^2 + N) + (N + 1) + (\overline{r} N + N) \lesssim LN^2 + \overline{r} N.
\end{align*} This completes the proof.

\end{proof}

\subsection{Proof of Technical Lemmas for Theorem \ref{thm:fa-oracle}}
\begin{proof}[Proof of Lemma~\ref{lemma:approx-factor}]
Recall the decomposition of $\tilde{\bbf}$, 
\begin{align*}
    \tilde{g}(\tilde{\bbf}) - g(\bbf) &= g(\bH^+ \tilde{\bbf}) - g(\bbf) \\
    &= g(\bH^+ \bH \bbf + p^{-1} \bH^+ \bW^\top \bu) - g(\bbf).
\end{align*}
By the definition of diversified projection matrix, $\text{rank}(\bH)=r$, this implies $\bH^+ \bH=\bI_r$. It follows from the Lipschitz property of $g$ that
\begin{align}
\label{eq:lemma:approx-factor-eq1}
    \left|\tilde{g}(\tilde{\bbf}) - g(\bbf)\right| \le C p^{-1} \left\|\bH^+ \bW^\top \bu\right\|_2 \le C p^{-1} \|\bH^+\| \left\|\bW^\top \bu\right\|_2
\end{align} 
Recall that $\bH^+=(\bH^\top \bH)^{-1} H^\top$. Let $\bH=\bU \bSigma \bV^\top$ be its singular value decomposition, where $\bU\in \mathbb{R}^{\overline{r}\times \overline{r}}$, $\bV\in \mathbb{R}^{r\times r}$ are both orthogonal matrices, $\bSigma \in \mathbb{R}^{\overline{r} \times r}$ is a rectangular diagonal matrix with diagonal entry $\Sigma_{i,i}=\nu_i(\bH)$, then 
\begin{align*}
    \bH^+ = (\bV \bSigma^\top \bU^\top \bU \Sigma \bV^\top)^{-1} \bV\bSigma^\top \bU^\top = \bV (\bSigma^\top \bSigma)^{-1} \bSigma^\top \bU^\top = \bV \bD \bU^\top .
\end{align*} where $\bD$ is also a rectangular diagonal matrix with $D_{i,i}=[\nu_i(\bH)]^{-1}$ for all the $i\le r$. Consequently,
\begin{align}
\label{eq:lemma:approx-factor-eq2}
    \|\bH^+\| \le \|\bD\| \le [\nu_{\min}(\bH)]^{-1}
\end{align}
Moreover, it follows from the linearity of expectation that
\begin{align*}
    \mathbb{E}\left\|\bW^\top \bu\right\|^2 &= \mathbb{E} \left[\sum_{k=1}^{\overline{r}} \left(\sum_{j=1}^p W_{j,k} u_j\right)^2 \right] \\
    &= \sum_{k=1}^{\overline{r}} \sum_{j=1}^p W_{j,k}^2 \mathbb{E} [u_j^2] + \sum_{j\neq j'} W_{j,k} W_{j',k} \mathbb{E}[u_ju_{j'}] \\
    &\le \overline{r} p \max_{j,k} |W_{j,k}| \max_{j} \mathbb{E} [u_j^2] + \overline{r} \max_{j,k} |W_{j,k}|^2 \sum_{j\neq j'} \left|\mathbb{E}[u_j u_{j'}]\right|.
\end{align*}
Applying Condition \ref{cond1} and \ref{cond3} gives
\begin{align}
\label{eq:lemma:approx-factor-eq3}
    \mathbb{E} \left\|\bW^\top \bu\right\|^2 \lesssim \overline{r}{p}.
\end{align} It concludes by applying square followed by taking expectation to both sides of \eqref{eq:lemma:approx-factor-eq1} and substituting \eqref{eq:lemma:approx-factor-eq2} and \eqref{eq:lemma:approx-factor-eq3} in.
\end{proof}

\begin{proof}[Proof of Lemma \ref{lemma:equivalence-population-empirical-l2}]
This is a direct consequence of Theorem 19.3 in \cite{gyorfi2002distribution}. Consider the function class $\mathcal{H} = \{h=g^2: g\in \mathcal{G}\}$. The uniform boundedness of $\mathcal{G}$ yields
\begin{align*}
    h(\bz) \le C_1^2 \qquad \text{and} \qquad \mathbb{E} [h(\bz)^2] \le C_1^2 \mathbb{E} [h(\bz)].
\end{align*} 
provided $\|g\|_\infty \le C_1$ for all the $g\in \mathcal{G}$.

Following the same notations as those in Theorem 19.3, we choose $\epsilon = \eta$ and $\alpha$ satisfying
\begin{align*}
    \alpha \ge \frac{288 \cdot \left(2C_1^2 \lor \sqrt{2} C_1\right)}{n\eta^2 (1-\eta)} \lor \frac{384^2 \times \mathrm{Pdim}(\mathcal{G}) }{n\eta^2(1-\eta)^2} \log \left(\frac{n^2\eta C_1}{288 \cdot 2C_1^2 \lor \sqrt{2}C_1}\right).
\end{align*} 

The uniform boundedness condition also implies that any $\epsilon$-net of $\mathcal{G}$ is also an $(C_1 \epsilon)$-net of $\mathcal{H}$. Therefore, for any $u \in (0, C_1^2ne)$, 
\begin{align*}
    \log \mathcal{N}_2 \left(u, \left\{h\in \mathcal{H}, \frac{1}{n}\sum_{i=1}^n h^2(z_i) \le 16\delta \right\}, z_{1}^n\right) \le \log \mathcal{N}_\infty \left(\frac{u}{C_1}, \mathcal{G}, z_{1}^n\right) \le \mathrm{Pdim}(\mathcal{G}) \log \left(\frac{C_1^2ne}{u}\right).
\end{align*} Then our choice of $\alpha$ guarantees
\begin{align*}
    \int_{\delta(1-\eta)\eta/(32C_1^2)}^{\sqrt{\delta}} \sqrt{\log \mathcal{N}_2 \left(u, \left\{h\in \mathcal{H}, \frac{1}{n}\sum_{i=1}^n h^2(\bz_i) \le 16\delta \right\}, \bz_{1}^n\right)} \mathrm{d}u &\le \sqrt{\delta \mathrm{Pdim}(\mathcal{G}) \log \frac{n^2\eta C_1}{288 \cdot 2C_1^2 \lor \sqrt{2}C_1}} \\ &\le \frac{\sqrt{n}\eta(1-\eta)\delta}{96\sqrt{2} 2C_1^2},
\end{align*} for all the $\delta \ge \alpha/8$ and $\bz_1,\cdots, \bz_n \in \mathcal{Z}$. Applying Theorem 19.3 in \cite{gyorfi2002distribution}, we have
\begin{align*}
    \mathbb{P}\left[\sup_{h\in \mathcal{H}} \frac{|\mathbb{E}[h(Z)] - \frac{1}{n} \sum_{i=1}^n h(Z_i)|}{\alpha + \mathbb{E}[h(Z)]} > \eta \right] \le 60 \exp\left(-\frac{n\alpha \eta^2(1-\eta)}{128 \cdot 2304 (C_1^4 \lor C_1^2)}\right).
\end{align*}

With a change of variable $\alpha = t^2$, we can conclude that the following event
\begin{align*}
    \forall g\in \mathcal{G}, \qquad \left|\|g\|_2^2 - \|g\|_n^2\right| \le \eta (\alpha + \|g\|_2^2)
\end{align*} occurs with probability at least $1-C \exp(-\frac{nt^2\eta^2(1-\eta)}{C b^2\lor b^4})$ as long as $t \ge C\sqrt{\frac{\mathrm{Pdim}(\mathcal{G})\log n}{n\eta^2(1-\eta)^2}}$ for some large enough constant $C>0$.
\end{proof}

\begin{proof}[Proof of Lemma~\ref{lemma:fa-weight-empirical-process}]
For any fixed $u>0$ and $\delta>0$, denote
\begin{align}
\label{eq:proof:thm-fa:b-delta-event}
    \mathcal{B}_u(\delta, \bz_1^n) = \left\{ \sup_{g\in \mathcal{G}, \|g - \tilde{g}\|_n\le\delta } \left|\frac{1}{n} \sum_{i=1}^n \varepsilon_i \left(g({\bz}_i) - \tilde{g}({\bz}_i)\right) \right|\le C_1 \delta \left(v_n + \sqrt{\frac{u}{n}}\right)\right\}
\end{align} for some universal constant $C_1$ to be specified, we first show $\mathbb{P}\left[\mathcal{B}_u(\delta, \bz_1^n)^c\right] \lesssim e^{-u}$ using the generic chaining result (Theorem 2.2.27 in \cite{talagrand2014upper}). 

Following the same notations as Theorem 2.2.27 in \cite{talagrand2014upper}, let $X_g=\frac{1}{n}\sum_{i=1}^n \varepsilon_i \left(g(\bz_i) - \tilde{g}(\bz_i)\right)$. The uniform sub-Gaussian assumption on $\varepsilon_1,\ldots, \varepsilon_n$ implies,
\begin{align*}
    \forall g, g'\in \mathcal{G}~~~~~~\mathbb{P}\left[|X_g - X_{g'}| \ge u\big| \bz_1,\ldots \bz_n \right] &\le 2\exp\left(-\frac{u^2}{(\frac{\sigma}{n})^2 \sum_{i=1}^n (g({\bz}_i)-g'({\bz}_i))^2}\right) \\
    &= 2\exp\left(-\frac{u^2}{d(g,g')^2}\right),
\end{align*} where $d(g, g') = {\sigma} n^{-1/2}\|g - g'\|_n$. Denote $T=\{g\in \mathcal{G}: \|g - \tilde{g}\|_n \le \delta\}$. It is easy to see that $\Delta(T)$, the diameter of $T$, is bounded by $2\sigma\delta n^{-1/2}$. 

It remains to bound $\gamma_2(T,d)$ before applying Theorem 2.2.27. Recall that 
\begin{align*}
    \gamma_2(T, d) = \inf_{(T_k)_{k=0}^\infty: T_0=\{t_0\}, |T_k| \le 2^{2^k}} \sup_{t\in T} \sum_{k=1}^\infty 2^{k/2} \inf_{s\in T_k} d(s, t).
\end{align*} 

It follows from Theorem 12.2 of \cite{AB1999} that
\begin{align*}
    \log \mathcal{N}_\infty(\varepsilon, \mathcal{G}, \bz_1^n) \le \mathrm{Pdim}(\mathcal{G})\log\left(\frac{ebn}{\epsilon}\right).
\end{align*} By definition, for any $0<\epsilon \le bn$, there exists $g^{(1)},\cdots g^{(N)} \in \mathcal{G}$ with $N\le \left(\frac{ebn}{\epsilon}\right)^{\mathrm{Pdim}(\mathcal{G})}$ such that
\begin{align*}
    \sup_{g\in \mathcal{G}} \inf_{1\le k\le N} \sup_{1\le i\le n} |g^{{(k)}}(\bz_i) - g(\bz_i)| \le \epsilon.
\end{align*} Because $T\subset \mathcal{G}$, letting $T_k$ be the above set with $N \le 2^{2^k}$ and smallest $\varepsilon$, we have
\begin{align*}
    \sup_{g\in T} \inf_{g' \in T_k} d(g,g') \le \sup_{g\in T} \inf_{g' \in T_k} \frac{\sigma}{\sqrt{n}} \|g-g'\|_n \lesssim \frac{\sigma}{\sqrt{n}} (bne) e^{-2^k/\mathrm{Pdim}(\mathcal{G})} := \epsilon_k.
\end{align*} 
Then it follows from the fact $\sup_{g\in T} \inf_{g' \in T_k} d(g,g') \le \Delta(T)$ that, 
\begin{align}
\label{eq:lemma:local-weighted:eq1}
    \gamma_2(T,d) \le \sum_{k=1}^\infty 2^{k/2} \left(\epsilon_k \land \frac{2\sigma \delta}{\sqrt{n}}\right).
\end{align} 

Let $k_0$ be the largest $k$ satisfying $\epsilon_k > 2\sigma \delta/\sqrt{n}$. Then there exists some constant $C_2\ge 2$ such that $2^{k_0} < C_2 (\log n) \mathrm{Pdim}(\mathcal{G})$  and  $2^{k_0 + 1} \ge C_2 (\log n) \mathrm{Pdim}(\mathcal{G})$, thus implying
\begin{align}
\label{eq:lemma:local-weighted:eq2}
    \sum_{k=1}^{k_0} 2^{k/2} \left(\epsilon_k \land \frac{2\sigma \delta}{\sqrt{n}}\right) \le \frac{\sqrt{2}^{k_0+1}-1}{\sqrt{2}-1} \frac{2\sigma \delta}{\sqrt{n}} \le C_{3} \sigma \delta \sqrt{\frac{\mathrm{Pdim}(\mathcal{G})\log n}{n}}.
\end{align} At the same time, for any $k \ge k_0+1$, 
\begin{align*}
    \frac{2^{(k+1)/2} \epsilon_{k+1}}{2^{k/2} \epsilon_{k}} = \sqrt{2} \exp\left(-\frac{2^k}{\mathrm{Pdim}(\mathcal{G})}\right) \le \sqrt{2} \exp\left(-\frac{2^{k_0+1}}{\mathrm{Pdim}(\mathcal{G})}\right) \le \sqrt{2} \exp(-2) \le \frac{1}{2},
\end{align*} this further implies that
\begin{align}
\label{eq:lemma:local-weighted:eq3}
    2^{k/2} \epsilon_{k} \le \left(\frac{1}{2}\right)^{k-(k_0+1)} 2^{(k_0+1)/2} \epsilon_{k_0+1} \le 2^{-k+k_0+1} C_{4} \sigma \delta \sqrt{\frac{d\log n}{n}}
\end{align} Plug \eqref{eq:lemma:local-weighted:eq2} and \eqref{eq:lemma:local-weighted:eq3} back into \eqref{eq:lemma:local-weighted:eq1}, we find
\begin{align*}
    \gamma_2(T,d) &\le \sum_{k=1}^{k_0} \left(\epsilon_k \land \frac{2\sigma \delta}{\sqrt{n}}\right) + \sum_{k=k_0+1}^\infty \left(\epsilon_k \land \frac{2\sigma \delta}{\sqrt{n}}\right) \\
    &\le C_{3} \sigma \delta \sqrt{\frac{\mathrm{Pdim}(\mathcal{G})\log n}{n}} + \sum_{k=k_0+1}^\infty 2^{-k+k_0+1} C_{4} \sigma \delta \sqrt{\frac{\mathrm{Pdim}(\mathcal{G})\log n}{n}} \\
    &= (C_{3} + 2C_{4}) \sigma \delta \sqrt{\frac{\mathrm{Pdim}(\mathcal{G})\log n}{n}}.
\end{align*} Applying the generic chaining bound gives,
\begin{align*}
\mathbb{P}\left[\sup_{g, g'\in T} |X_g - X_{g'}| \ge C_{5}(\gamma_2(T,d) + \sqrt{u}\Delta(T))\Big| \bz_1,\cdots, \bz_n\right] \le \exp(-u),
\end{align*} for any $u>0$ provided $n\ge 3$. Therefore, the following event
\begin{align*}
    \sup_{g\in \mathcal{G}, \|g-\tilde{g}\|_n\le \delta} \left|\frac{1}{n}\sum_{i=1}^n \varepsilon_i \left(g(\bz_i) - \tilde{g}(\bz_i)\right)\right| &\le \sup_{g, g'\in T} |X_g - X_{g'}| + \left|\frac{1}{n} \sum_{i=1}^n \varepsilon_i \left(\tilde{g}(\bz_i) - \tilde{g}(\bz_i)\right) \right| \\
    &\le C_{5}\sigma \delta  \left(\sqrt{\frac{\mathrm{Pdim}(\mathcal{G}) \log n}{n}} + \sqrt{\frac{u}{{n}}} \right) 
\end{align*} occurs with probability at least $1-e^{-u}$. This completes the proof of the claim that the event $B_u(\delta)$ defined \eqref{eq:proof:thm-fa:b-delta-event} occurs with probability at least $1- e^{-u}$ with constant $C_1=C_{5}\sigma$.

Now we are ready to establish tail probability for the event $\mathcal{B}_t$ with $c_1=4C_1$ using the peeling device. To this end, define the sets $\mathcal{S}_{\ell} = \{g \in \mathcal{G}: \alpha_{\ell-1} \epsilon < \|g-\tilde{g}\|_n \le \alpha_\ell \epsilon\}$ for $\ell\ge 0$, where $\alpha_\ell = 2^\ell$ for $\ell\ge 0$ and $\alpha_{-1} = 0$. We have $\alpha_{\ell} \le 2 \alpha_{\ell - 1} + 1$.

For each $\ell \ge 0$, define the event $\mathcal{E}_\ell$ as
\begin{align*}
    \mathcal{E}_\ell = \left\{ \forall g\in \mathcal{S}_{\ell}, ~~\left|\frac{1}{n}\sum_{i=1}^n \varepsilon_i \left(g(\bz_i) - \tilde{g}(\bz_i)\right)\right| \le 4C_1 \left(\epsilon + \|g - \tilde{g}\|_n \right) \sqrt{v_n^2 + \frac{t}{n}} \right\}.
\end{align*} 

It follows from the uniform boundedness of the function class that $\|g-\tilde{g}\|_n \le C_6$ for some universal constant $C_6$. This further implies
\begin{align*}
    \mathcal{B}_t^c \subset \bigcup_{\ell=0}^{\lceil\log_2(C_6/\epsilon)\rceil} \mathcal{E}_\ell^c.
\end{align*}

We claim that $\mathcal{B}_{t} (\alpha_\ell \epsilon) \subset \mathcal{E}_\ell$. To see this, if the event $\mathcal{B}_{t}(\alpha_\ell \epsilon)$ occurs, then for any $g\in \mathcal{S}_\ell$, we have
\begin{align*}
    \left|\frac{1}{n}\sum_{i=1}^n \varepsilon_i \left(g(\bz_i) - \tilde{g}(\bz_i)\right)\right| &\overset{(a)}{\le} C_1 \alpha_\ell \epsilon \left(v_n + \sqrt{\frac{t}{n}}\right) \\
    &\overset{(b)}{\le} C_1 (2\alpha_{\ell-1} \epsilon + \epsilon) \left(v_n + \sqrt{\frac{t}{n}}\right) \\
    &\overset{(c)}{\le} 2C_1 \left(\|g - \tilde{g}\|_n + \epsilon \right) \left(v_n + \sqrt{\frac{t}{n}}\right).
\end{align*} where (a) follows from the definition of $\mathcal{B}_{t}(\alpha_\ell \epsilon)$ and $\mathcal{S}_\ell$, (b) follows from the relationship between $\alpha_{\ell-1}$ and $\alpha_\ell$ above, (c) follows from the definition of $\mathcal{S}_\ell$ that $\|g-\tilde{g}\|_n \ge \alpha_{\ell-1} \epsilon$. 

Putting these pieces together, we can conclude that
\begin{align*}
    \mathbb{P}\left[\mathcal{B}_t^c\right] \le \sum_{\ell=0}^{\lceil\log_2(C_6/\epsilon)\rceil} \mathbb{P}\left[\mathcal{E}_\ell^c\right] \le \sum_{\ell=0}^{\lceil\log_2(C_6/\epsilon)\rceil} \mathbb{P}\left[\left(\mathcal{B}_{t+nv_n^2}(\alpha_\ell v_n)\right)^c\right] \le \log (1/\epsilon) e^{-t}.
\end{align*} This completes the proof.
\end{proof}

\subsection{Proof of Corollary \ref{coro:roc-fa-nn}}

\begin{proof}[Proof of Corollary \ref{coro:roc-fa-nn}]
It follows directly from Theorem \ref{thm:fa-oracle} and Proposition 3.4 in \cite{fan2022noise}.
\end{proof}

\subsection{Proof of Proposition \ref{prop:choice-dpm}}

We need the following technical lemma to bound the Frobenius norm of the matrix $\hat{\bSigma}-\bB\bB^\top = \hat{\bSigma} - \bS$.

\begin{lemma}
\label{lemma:frobenius-norm}
There exists a universal constant $c_1$ such that under Condition \ref{cond1}, \ref{cond3} and \ref{cond:dist-fu}, we have
\begin{align}
    \|\hat{\bSigma} - \bS\|_F \le c_1 p \underbrace{\left(r \sqrt{\frac{\log p + t}{n}} + r^2 \sqrt{\frac{(\log r + t)}{n}} + \frac{1}{\sqrt{p}}\right)}_{\delta}
\end{align} with probability at least $1-3e^{-t}$ for any $t>0$.
\end{lemma}

\begin{proof}[Proof of Proposition~\ref{prop:choice-dpm}]

Let $\hat{\bV}  = \left[\hat{\bv}_1,\cdots, \hat{\bv}_{\overline{r}}\right]\in \mathbb{R}^{p\times \overline{r}}$ and thus $\tilde{\bW}=\sqrt{p} \hat{\bV}$. Let $\bS = \bB\bB^\top$. Since $\bS$ is a symmetric matrix, it has the following eigen-decomposition $\bS = \bV \bLambda \bV^\top$, where $\bLambda$ is an $r\times r$ diagonal matrix, $\bV\in \mathbb{R}^{p\times r}$ satisfies $\bV^\top \bV=I_r$. It follows from the identification condition that $\bB=\bV\bLambda^{1/2}$. 

The pervasiveness condition (Condition \ref{cond:pervasiveness}) implies
\begin{align}
\label{eq:prop:choice-dpm:proof:pervasiveness}
    \Lambda_{i,i} \asymp p \qquad \text{ for all }\qquad i\in\{1,\cdots, r\}.
\end{align}

We first prove the upper bound. Note
\begin{align*}
    \nu_{\max}(p^{-1} \tilde{\bW}^\top \bB) = \nu_{\max}(p^{-1} \bB^\top \tilde{\bW}) = \sup_{\bu \in \mathbb{S}^{\overline{r}-1}} \|p^{-1} \bB^\top \tilde{\bW} \bu\|_2 \le p^{-1/2} \|\bLambda^{1/2}\|_2 \sup_{u \in \mathbb{S}^{\overline{r}-1}} \|\bV^\top \hat{\bV} \bu\|_2.
\end{align*}
Hence it follows from \eqref{eq:prop:choice-dpm:proof:pervasiveness} that,
\begin{align*}
    \nu_{\max}(p^{-1} \tilde{\bW}^\top \bB) \le C_3  \sup_{u \in \mathbb{S}^{\overline{r}-1}} \|V^\top \hat{V} u\|_2 \le C_3  \sup_{u \in \mathbb{S}^{\overline{r}-1}} \|u\|_2 \le C_3,
\end{align*} for some universal constant $C_3>0$.

It remains to prove the lower bound. Let $\hat{\bV}_r = \left[\hat{\bv}_1,\cdots, \hat{\bv}_r\right] \in \mathbb{R}^{p\times r}$, we have
\begin{align}
\label{eq:prop:choice-dpm:proof-lb}
    \nu_{\min}(p^{-1} \tilde{\bW}^\top \bB) \ge \nu_{\min}(p^{-1/2} \hat{\bV}_r^\top \bB) \ge p^{-1/2} \nu_{\min}(\hat{\bV}_r^\top \bV) \nu_{\min}(\bSigma^{1/2}) \ge C_4 \nu_{\min}(\hat{\bV}_r^\top \bV).
\end{align} where the last inequality follow from \eqref{eq:prop:choice-dpm:proof:pervasiveness}. Therefore, it suffices to bound $\nu_{\min}(\hat{\bV}_r^\top \bV)$ from below.

It follows from Lemma \ref{lemma:frobenius-norm} that the event \begin{align}
\label{eq:prop:choice-dpm:proof:perturbation-claim}
    \mathcal{A}_t = \left\{\|\hat{\bSigma} - \bS\|_F \lesssim p \left(r \sqrt{\frac{\log p + t}{n}} + r^2 \sqrt{\frac{(\log r + t)}{n}} + \frac{1}{\sqrt{p}}\right) = p\delta\right\}
\end{align} occurs with probability at least $1-3e^{-t}$. The remaining proof proceeds conditioned on $\mathcal{A}_t$.

Let $\lambda_1,\cdots, \lambda_p$ denote the eigenvalues of the symmetric matrix $\bS$ in a non-increasing order, and $\hat{\lambda}_1,\cdots, \hat{\lambda}_p$ denote the eigenvalues of the symmetric matrix $\hat{\bSigma}$ in a non-increasing order. It is easy to see that
\begin{align}
\label{eq:prop:choice-dpm:proof:popu-lambda}
    \lambda_r \ge C_4 p \qquad \text{and} \qquad \lambda_{r+1} = 0.
\end{align} It follows from the Weyl's Theorem \cite{von1909beschrankte} that
\begin{align}
\label{eq:prop:choice-dpm:proof:emp-lambda}
    |\hat{\lambda}_{i} - \lambda_i| \le \|\hat{\bSigma} - S\|_2 \le \|\hat{\bSigma} - S\|_F \le C_5 p\delta. \qquad \text{ for all } i\in \{1,\cdots, p\}
\end{align} for some constant $C_5$.

We argue that
\begin{align}
\label{eq:prop:choice-dpm:proof-lb-inter}
    \nu_{\min}(\hat{\bV}_r^\top \bV) \ge 1 - \frac{2C_5}{C_4} \delta.
\end{align}

It suffices to prove the inequality \eqref{eq:prop:choice-dpm:proof-lb-inter} when $\delta \le \frac{C_4}{2C_5}$, otherwise the inequality is trivial since $\nu_{\min} \ge 0$. In this case, combining the above two claims \eqref{eq:prop:choice-dpm:proof:popu-lambda} and \eqref{eq:prop:choice-dpm:proof:emp-lambda} about the eigenvalues of $\hat{\bSigma}$ and $\bS$, we can bound the eigen-gap from below as
\begin{align*}
    \tilde{\delta} &= \inf \left\{|\lambda - \hat{\lambda}|: \lambda \in [\lambda_1,\lambda_{r}], \hat{\lambda} \in (-\infty, \hat{\lambda}_{r+1}]\right\} = \lambda_r - \hat{\lambda}_{r+1} = \lambda_r - \lambda_{r+1} + \lambda_{r+1} - \hat{\lambda}_{r+1} \\
    &\ge C_4 p - |\lambda_{r+1} - \hat{\lambda}_{r+1}| \ge C_4 p - C_5\delta p \ge \frac{C_4}{2} p.
\end{align*} Recall that $\hat{\bV}_r$ and $\bV$ are the top-$r$ eigenvectors of $\hat{\bSigma}$ and $\bS$ respectively, it follows from the $\sin\Theta$ Theorem (Theorem V.3.6 in \cite{stewart1990matrix}) that,
\begin{align*}
    \|\sin \Theta (\hat{\bV}_r, \bV) \|_F \le \frac{\|\hat{\bSigma} - \bS\|_F}{\tilde{\delta}} \le \frac{2C_5}{C_4} \delta.
\end{align*} Here $\sin \Theta (\hat{\bV}_r, \bV)$ is a $r\times r$ diagonal matrix satisfying
\begin{align*}
    [\sin \Theta (\hat{\bV}_r, \bV)]^2 + [\cos \Theta (\hat{\bV}_r, \bV)]^2 = \bI_r
\end{align*} where $\cos \Theta (\hat{\bV}_r, \bV)$ is also a $r\times r$ diagonal matrix of singular values of $\hat{\bV}_r^\top \bV$. Consequently,
\begin{align*}
    \nu_{\min}^2(\hat{\bV}_r^\top \bV) \ge 1 - \|\sin \Theta(\hat{\bV}_r, \bV)\|_2^2 \ge 1 - \|\sin \Theta(\hat{\bV}_r, \bV)\|_F^2 \ge 1 - \left( \frac{2C_5}{C_4}\delta\right)^2.
\end{align*}
This completes the proof of the claim \eqref{eq:prop:choice-dpm:proof-lb-inter} by the fact $\sqrt{1-x^2}\ge 1-x$ for $x\in [0,1]$. Plugging \eqref{eq:prop:choice-dpm:proof-lb-inter} back into \eqref{eq:prop:choice-dpm:proof-lb} completes the proof of the lower bound in \eqref{eq:prop:choice-dpm:bd}.

\end{proof}

\begin{proof}[Proof of Lemma~\ref{lemma:frobenius-norm}] Using the decomposition representation of the observation $\bx$ in \eqref{eq:dgp-lfm}, we can decompose the difference between $\hat{\bSigma}$ and $\bS$ as
\begin{align}
\label{eq:prop:choice-dpm:proof:delta-decomposition}
\begin{split}
    \hat{\bSigma} - \bS &= \frac{1}{n} \sum_{i=1}^n(B \bbf_i + \bu_i) (\bbf_i^\top \bB^\top + \bu_i^\top) - \bB \bB^\top \\
                     &= \underbrace{\bB \left(\frac{1}{n} \sum_{i=1}^n \bbf_i \bbf_i^\top - \bI\right) \bB^\top}_{\bS^{(1)}} + \underbrace{\frac{1}{n} \sum_{i=1}^n \bu_i \bu_i^\top - \bSigma_{\bu}}_{\bS^{(2)}} + \underbrace{\frac{1}{n} \sum_{i=1}^n \bB\bbf_i \bu_i^\top}_{\bS^{(3)}} + (\bS^{(3)})^\top + \bSigma_{\bu}
\end{split}
\end{align} where $\bSigma_{\bu} \in \mathbb{R}^{p\times p}$ is the covariance matrix of the idiosyncratic component $\bu \in \mathbb{R}^p$. We establish the upper bound on the Frobenius norm of each matrix respectively. 

Let $\bbf_i = (f_{i,1},\cdots f_{i,r})$ and $\bu_i = (u_{i,1},\cdots,u_{i,p})$. Define the event $\mathcal{A}_t = \cap_{i=1}^3 \mathcal{A}_{i,t}$, where
\begin{align*}
    \mathcal{A}_{1,t} &= \left\{\sup_{\ell, k\in \{1,\cdots, r\}} \left|\frac{1}{n} \sum_{i=1}^n f_{i,\ell} f_{i,k} - I_{\ell,k} \right|\le C_1 \sqrt{\frac{\log r + t}{n}}\right\}, \\
    \mathcal{A}_{2,t} &= \left\{\sup_{j,j'\in \{1,\cdots, p\}} \left|\frac{1}{n} \sum_{i=1}^n u_{i,j} u_{i,j'} - [\bSigma_{\bu}]_{j,j'} \right|\le C_1 \sqrt{\frac{\log p + t}{n}}\right\}, \\
    \mathcal{A}_{3,t} &= \left\{\sup_{\ell, k\in \{1,\cdots, r\}} \left|\frac{1}{n} \sum_{i=1}^n f_{i,k} u_{i,j} - \mathbb{E} [f_k u_j]\right| \le C_1 \sqrt{\frac{\log p + t}{n}}\right\}.
\end{align*} 

Because each entry of the factor $\bbf$ and the idiosyncratic component $\bu$ is a bounded random variable, it follows from the maximal inequality for the sub-Gaussian random variable that
\begin{align*}
    \mathbb{P}\left[\mathcal{A}_{i,t}\right] \ge 1-e^{-t} \qquad  \text{ for all } \qquad t > 0, i\in\{1,2,3\}.
\end{align*} 
Hence using union bound gives $\mathcal{P}(\mathcal{A}_t) \ge 1-3e^{-t}$. The remaining proof proceeds conditioned on $\mathcal{A}_t$.

For $\bS^{(1)}$, it follows from the boundedness of $\|\bB\|_{\max}$ and $\mathcal{A}_{1,t}$ that
\begin{align*}
    \left|S^{(1)}_{j,j'}\right| = \left|\sum_{k=1}^r\sum_{\ell=1}^r B_{j,k} B_{\ell,j'} \left(\frac{1}{n} \sum_{i=1}^n f_{i,k} f_{i,\ell} - (\Sigma_F)_{k,\ell}\right)\right| \le C_2 r^2 \sqrt{\frac{\log r + t}{n}},
\end{align*} which implies
\begin{align}
\label{eq:prop:choice-dpm:proof:s1f}
    \|\bS^{(1)}\|_F = \sqrt{\sum_{j,j'\in \{1,\cdots, p\}} |S^{(1)}_{j,j'}|^2} \le C_2 pr^2 \sqrt{\frac{\log r + t}{n}}.
\end{align}
Moreover, for $\bS^{(2)}$, we find that
\begin{align}
\label{eq:prop:choice-dpm:proof:s2f}
    \| \bS^{(2)} \|_F = \sqrt{\sum_{j,j'\in \{1,\cdots, p\}} |S^{(2)}_{j,j}|^2} \le \sqrt{p^2 \sup_{j,j' \in \{1,\cdots, p\}} \left|\frac{1}{n} \sum_{i=1}^n u_{i,j} u_{i,j'} - [\bSigma_{\bu}]_{j,j'} \right|^2} \le C_1 p \sqrt{\frac{\log p + t}{n}}.
\end{align}
For $\bS^{(3)}$, the weak dependence between $\bbf$ and $\bu$ Condition \ref{cond:dist-fu} implies 
\begin{align}
\begin{split}
\label{eq:prop:choice-dpm:proof:s3f}
    \|\bS^{(3)}\|_F &\le \|\bB \bSigma_{\bbf, \bu}\|_F + \|\bS^{(3)} - \bB \bSigma_{\bbf, \bu}\|_F \\
    &\lesssim \sqrt{p} + \sqrt{\sum_{j,j'} \left|\sum_{k=1}^r B_{j,k} \left(\frac{1}{n} \sum_{i=1}^n f_{i,k} u_{i,j} - \mathbb{E} [f_k u_j]\right)\right|^2}
    \lesssim \sqrt{p} + r p\sqrt{\frac{\log p + t}{n}}
\end{split}
\end{align} For $\bSigma_U$, it follows from the boundedness of idiosyncratic error $U$ in Condition \ref{cond1} and weak dependency Condition \ref{cond3} that,
\begin{align}
\label{eq:prop:choice-dpm:proof:sigmau}
\begin{split}
    \|\bSigma_U\|_F &= \sqrt{\sum_{j=1}^p (\mathbb{E}|u_j|^2)^2 + \sum_{j\neq j} \mathbb{E}[|u_j u_j'|] \times \mathbb{E}[|u_j u_j'|]} \\
                   &\lesssim \sqrt{\sum_{j=1}^p \mathbb{E}|u_j|^2 + \sum_{j\neq j} \mathbb{E}[|u_j u_j'|]} \lesssim \sqrt{p}
\end{split}
\end{align}
Plugging \eqref{eq:prop:choice-dpm:proof:s1f}, \eqref{eq:prop:choice-dpm:proof:s2f}, \eqref{eq:prop:choice-dpm:proof:s3f} and \eqref{eq:prop:choice-dpm:proof:sigmau} into \eqref{eq:prop:choice-dpm:proof:delta-decomposition} completes the proof.

\end{proof}

\subsection{Proof of Proposition \ref{thm:sub-optimality-least-squares}}

We will use the following technical lemma from \cite{bartlett2020benign}.
\begin{lemma}[Restatement of Lemma 14 in \cite{bartlett2020benign}]
\label{lemma:concentration:benign-overfitting}
Suppose $\bz_1,\cdots, \bz_p$ are independent $n$-dimension sub-Gaussian random vectors, i.e., $\mathbb{E}[e^{\xi^\top \bz_i}] \le e^{\frac{1}{2} c_1\|\xi\|_2^2}$ for some universal constant $c_1>0$, and $\lambda_1 \ge \cdots \ge \lambda_p>0$ are positive constants. There exists an universal constant $c_2>0$ such that for any $j \in \{1,\cdots, p\}$, $k\le n/c_2$, with probability at least $1-5e^{-n/c_2}$, 
\begin{align*}
    \frac{\lambda_j^2 \bz_j \bA_{-j}^{-2} \bz_j}{(1+\lambda_j \bz_j^\top \bA_{-j}^{-1} \bz_j)^2} \ge \frac{1}{c_2 n }\left(1+\frac{\sum_{j>k} \lambda_j + n\lambda_{k+1}}{n\lambda_i}\right)^{-2}.
\end{align*} where $\bA_{-j} = \sum_{\ell \neq j} \bz_\ell \bz_\ell^\top$.
\end{lemma}
Now we are ready to prove Proposition \ref{thm:sub-optimality-least-squares},
\begin{proof}[Proof of Proposition~\ref{thm:sub-optimality-least-squares}]
Denote
\begin{align*}
    \bX = \begin{bmatrix}
        -\bx_1^\top- \\
        -\bx_2^\top- \\
        \vdots \\
        -\bx_n^\top-
    \end{bmatrix} \in \mathbb{R}^{n\times p} \qquad \text{and} \qquad \by = \begin{bmatrix}
    y_1\\
    y_2\\
    \vdots\\
    y_n
    \end{bmatrix}\in \mathbb{R}^n.
\end{align*} The minimum $\ell_2$ norm estimator is
\begin{align*}
    \hat{\bbeta} = \bX^\top (\bX\bX^\top)^{-1} \by,
\end{align*} provided $p>n$, and the associated $L_2$ risk can be written as
\begin{align}
\label{eq:benign-overfitting-l2-risk}
    \int |\bx^\top \hat{\bbeta}|^2 \mu(d\bbf, d\bu) = (\hat{\bbeta})^\top \bSigma \hat{\bbeta},
\end{align} where $\bSigma$ is the covariance matrix of the covariate $\bx$. Under Condition \ref{cond:pervasiveness}, the matrix $\bB\bB^\top$ has the following eigen-decomposition
\begin{align*}
    \bB\bB^\top = \sum_{j=1}^r \lambda_i \bv_i \bv_i^\top ~~~~ \text{with} ~~~~ \lambda_i \asymp p ~~\text{and}~~ \text{orthogonal vectors } \bv_1,\ldots, \bv_r.
\end{align*} Let $\lambda_{r+1}=\lambda_{r+2} = \cdots = \lambda_p = 0$, $\bLambda$ be a $p\times p$ diagonal matrix with $\Lambda_{j,j}=\lambda_j$ for $j \in \{1,\cdots, p\}$, $\bv_{r+1}, \ldots, \bv_p$ be orthogonal component of $\bv_1,\cdots, \bv_r$. Therefore, the covariance matrix $\bSigma$ can be written as
\begin{align*}
    \bSigma = \bV (\bLambda + \bI_p) \bV^\top.
\end{align*} Denote
\begin{align*}
    [\bz_1,\ldots, \bz_p] = \bZ = \bX \bV (\bLambda + \bI_p)^{-1/2} = \begin{bmatrix}
        \bx_1^\top \bv_1/\sqrt{\lambda_1+1} &\cdots& \bx_1^\top \bv_p/\sqrt{\lambda_p+1}\\
        \vdots & \ddots & \vdots\\
        \bx_n^\top \bv_1/\sqrt{\lambda_1+1}&\cdots& \bx_n^\top \bv_p/\sqrt{\lambda_p+1}\\
    \end{bmatrix}
\end{align*}
Note each vector $\bz_i$ has independent entries because of the fact that $\bx_1,\ldots, \bx_n$ are i.i.d. copies. Moreover, we can write $i$-th entry of the vector $\bz_j$ as
\begin{align*}
    z_{j,i} = \begin{cases}
        f_{i,j} \sqrt{\lambda_j} + \bu_i^\top \bv_j & j\le r\\
        \bu_i^\top \bv_j & j> r\\
    \end{cases}
\end{align*} Because $\bu_i$ has i.i.d. standard normal entries, it is easy to verify that $\bz_1,\cdots, \bz_p$ are independent because (1) the factor $\bbf$ has independent entries (2) $\bbf$ and $\bu$ are independent (3) $\bv_j^\top \bv_{j'}=0$ for arbitrary $j\neq j'$, and uncorrelatedness is equivalent to independence for random variables with joint normal distribution.

Taking expectation on both sides of \eqref{eq:benign-overfitting-l2-risk},
\begin{align*}
    \mathbb{E}_{(\bx_1,y_1)\cdots, (\bx_n,y_n)} \left[\int |\bx^\top \hat{\bbeta}|^2 \mu(d\bbf, d\bu)\right] &= \mathbb{E}_{\bx_1,\cdots,\bx_n} \mathbb{E}_{\by} \left[\by^\top (\bX\bX^\top)^{-1} \bX \bSigma \bX^\top (\bX\bX^\top)^{-1} \by\right]\\
    &= \sigma^2 \mathbb{E}_{\bx_1,\cdots,\bx_n}\left[\tr\left((\bX\bX^\top)^{-1} \bX \bSigma \bX^\top (\bX\bX^\top)^{-1}\right)\right] \\
    &= \sigma^2 \mathbb{E}_{\bz_1,\cdots,\bz_p}\left[\tr\left(\sum_{j=1}^p (\lambda_j+1)^2 \bz_j^\top \left(\sum_{\ell=1}^p (\lambda_\ell+1) \bz_\ell \bz_\ell^\top \right)^{-2} \bz_j\right)\right]\\
    &= \sigma^2 \sum_{j=1}^p \mathbb{E} \left[(\lambda_j+1)^2 \bz_j^\top \left(\sum_{\ell=1}^p (\lambda_\ell+1) \bz_\ell \bz_\ell^\top \right)^{-2} \bz_j\right].
\end{align*}
When $p>n+r$, it is easy to verify that the matrix $\bA_{-j} = \sum_{\ell \neq j} (\lambda_\ell+1) \bz_\ell \bz_\ell^\top$ is invertible almost surely. Applying Lemma 20 in \cite{bartlett2020benign} together with Lemma \ref{lemma:concentration:benign-overfitting} with $k=r$, we have
\begin{align*}
    \mathbb{E}\left[(\lambda_j+1)^2 \bz_j^\top \left(\sum_{\ell=1}^p (\lambda_\ell+1) \bz_\ell \bz_\ell^\top \right)^{-2} \bz_j\right] &= \mathbb{E}\left[ \frac{(\lambda_j+1)^2 \bz_j \bA_{-j}^{-2} \bz_j}{(1+(\lambda_j+1) \bz_j^\top \bA_{-j}^{-1} \bz_j)^2}\right] \\
    &\ge  \frac{1-5e^{-n/c_2}}{c_2 n }\left(1+\frac{\sum_{\ell>k} (\lambda_\ell+1) + n(\lambda_{k+1}+1)}{n(\lambda_j+1)}\right)^{-2} \\
    &\overset{(a)}{\gtrsim} 1_{\{j \le r\}}  n^{-1} +1_{\{j > r\}} np^{-2} ,
\end{align*} provided $r<n/c_2$ for large enough $n$, where (a) follow from the previous conditions of eigenvalues, i.e., $\lambda_i \asymp p$ for $i\le r$ and $\lambda_i=0$ for $i>r$. Therefore, we can conclude that
\begin{align*}
     \mathbb{E}_{(\bx_1,y_1)\cdots, (\bx_n,y_n)} \left[\int |\bx^\top \hat{\bbeta}|^2 \mu(d\bbf, d\bu)\right] \gtrsim \frac{r}{n} + \frac{n}{p}.
\end{align*}

\end{proof}

\section{Proofs for the FAST-NN estimator in Section \ref{sec:theory:fast}}
\label{sec:proof-fast-nn-estimator}

We first introduce some additional notations used throughout the proof. Define the function class $\mathcal{G}_m$,
\begin{align*}
    \mathcal{G}_m = \left\{ m(x;\bW,g, \bTheta): g\in \mathcal{G}(L, \overline{r}+N, 1, N, M, B), \bTheta\in \mathbb{R}^{p\times N}, \|\bTheta\|_{\max} \le B\right\}
\end{align*} and
\begin{align*}
    \mathcal{G}_{m,s} = \left\{m \in \mathcal{G}_m, \sum_{i,j}\psi_\tau(\Theta_{i,j}) \le s\right\}.
\end{align*} Similar to the representation of deep ReLU networks, we can write
\begin{align*}
    m(x) = \mathcal{L}_{L+1} \circ \bar{\sigma} \circ \mathcal{L}_{L} \circ \bar{\sigma} \circ \cdots \circ \mathcal{L}_2 \circ \bar{\sigma} \circ \mathcal{L}_1 \circ \phi \circ \mathcal{L}_0(x)
\end{align*} for each $m\in \mathcal{G}_m$, where the definition of $\mathcal{L}_{\ell}$ with $\ell\in \{1,\cdots, L+1\}$ and $\bar{\sigma}$ is same as that in Definition \ref{def:nn}, $\phi: \mathbb{R}^{\overline{r} + N} \to \mathbb{R}^{\overline{r}+N}$ satisfies
\begin{align*}
    [\phi(\bv)]_i = \begin{cases}
        v_i & \qquad i \le \overline{r} \\
        T_M(v_i) & \qquad i > \overline{r}
    \end{cases},
\end{align*} and $\mathcal{L}_0(\bx) = (p^{-1} \bx^\top \bW, \bx^\top \bTheta)^\top$. 

Let $\bz_1,\cdots, \bz_n$ be i.i.d. copies of $\bz \sim \mathcal{Z}$ from some distribution $\mu$, $\mathcal{H}$ be a real-valued function class defined on $\mathcal{Z}$. Define the empirical $L_2$ norm and population $L_2$ norm for each $h\in \mathcal{H}$ respectively as
\begin{align}
\label{eq:l2-ln-norm-fast-def}
    \|h\|_n = \left(\frac{1}{n} \sum_{i=1}^n h(\bz_i)^2\right)^{1/2} ~~~~\text{and}~~~~\|h\|_2 = \left(\mathbb{E}[h(\bz)^2]\right)^{1/2} = \left(\int h(\bz)^2 \mu(d\bz) \right)^{1/2}.
\end{align}
Throughout the proof of the FAST-NN estimator, the choice of $\mathcal{Z}$ and $\mathcal{H}$ will vary in different contexts. However, we can also define a unified empirical (population) $L_2$ norm on the function class $\mathcal{H} = \mathcal{H}_2 - \mathcal{H}_2$ over $\mathcal{Z}=\{\bz=(\bbf, \bu)\}$, where the difference between two sets is defined as $\mathcal{A}-\mathcal{B} = \{a-b: a\in \mathcal{A}, b\in \mathcal{B}\}$, and
\begin{align*}
    \mathcal{H}_2 = \left\{m(\bx) = m(\bB\bbf+\bu): m\in \mathcal{G}_{m,s} \right\} \cup \left\{m^*(\bbf, \bu_{\mathcal{J}})\right\} \cup \{0\}.
\end{align*} Therefore, we can use triangle inequality of such unified empirical (population) $L_2$ norm. 

Finally, we define the following two quantities of interest,
\begin{align}
\label{eq:fast-rate-v_n}
    v_n &= \left(N^2L + N\overline{r}\right) \frac{L\log (BNn)}{n}, \\
\label{eq:fast-rate-varrho_n}
\varrho_n &= \frac{\log (np(N+\overline{r})) + L\log (BN)}{n}.
\end{align} 

\subsection{Covering number of the class $\mathcal{G}_{m,s}$}

We need the following fact about the $\epsilon$-net covering number of the function class $\mathcal{G}_{m,s}$.

\begin{lemma}
\label{lemma:g-infty-cover-number}
    There exists a universal constant $c_1$ such that for any $\delta > 2\tau KB^{L+1}N^{L+1}(N+\overline{r})p$ and $N, L \ge 2$, 
    \begin{align*}
        \log \mathcal{N}(\delta, \mathcal{G}_{m,s}, \|\cdot\|_{\infty, [-K,K]^p}) \le c_1 \Bigg\{ &\left(N^2L + N\overline{r} \right) \left[L \log BN + \log \left(\frac{M\lor K\|W\|_{\max}}{\delta} \lor 1\right)\right] \\
        &~~~~~ + s \left[L \log (BN) + \log p + \log\left(\frac{K(N+\overline{r})}{\delta} \lor 1\right)\right] \Bigg\}.
    \end{align*}
\end{lemma}

Our proof of Lemma \ref{lemma:g-infty-cover-number} is based on the following Lemma,
\begin{lemma}
\label{lemma:m-lip-weight}
Let $\btheta(m) = \left\{\bTheta, \left(\bW_\ell, \bb_\ell\right)_{\ell=1}^{L+1}\right\}$ be the set of parameters for function $m\in \mathcal{G}_{m,s}$, $\btheta(\breve{m}) = \left\{\breve{\bTheta}, (\breve{\bW}_\ell, \breve{\bb}_\ell)_{\ell=1}^{L+1} \right\}$ be the set of parameters for function $\breve{m} \in \mathcal{G}_{m,s}$. Define
\begin{align*}
    \|\btheta_g(m) - \btheta_g(\breve{m})\|_\infty = \max_{1\le \ell\le L+1} \left( \|\bb_\ell - \breve{\bb}_{\ell}\|_\infty \lor \|\bW_\ell - \breve{\bW}_\ell \|_{\max}\right).
\end{align*}
If $M\ge 1$, then the following holds
\begin{align}
\label{eq:m-lip-weight}
\begin{split}
    \| m(\bx) - \tilde{m}(\bx) \|_{\infty, [-K,K]^p} \le &~ (M \lor K\|W\|_{\max}) (L+1)B^{L}(N+1)^{L+1} \|\btheta_g(m) - \btheta_g(\breve{m})\|_\infty \\
    &~~~~~ + KB^{L+1}N^L(N+\overline{r})p \|\bTheta-\breve{\bTheta}\|_{\max}.
\end{split}
\end{align}
\end{lemma}

We are ready to prove Lemma \ref{lemma:g-infty-cover-number}.
\begin{proof}[Proof of Lemma~\ref{lemma:g-infty-cover-number}]
    For any $m\in \mathcal{G}_{m,s}$ with parameter $\btheta(m) = \left\{\bTheta, \left(\bW_\ell, \bb_\ell\right)_{\ell=1}^{L+1}\right\}$, let $\bTheta^{(\tau)} \in \mathbb{R}^{p\times N}$ be that $[\bTheta^{(\tau)}]_{i,j} = \Theta_{i,j} 1_{\{|\Theta_{i,j}|\ge \tau\}}$, and 
    \begin{align*}
        m^{(\tau)}(\bx) = m(\bx; \bW, \bTheta^{(\tau)}, g).
    \end{align*}
    \noindent {\sc Step 1. Construct an $\delta$-Set.} With the help of the set $\mathcal{G}_{m,s}^{(\tau)}=\{m^{(\tau)}: m \in \mathcal{G}_{m,s}\}$, we first construct a $\delta$-cover of the set $\mathcal{G}_{m,s}$ with respect to $\|\cdot\|_{\infty, [-K,K]^p}$ norm. For any $S\subset \{1,\cdots, p\} \times \{1,\cdots, N\}$ with $|S|\le \lfloor s \rfloor$, let 
    \begin{align*}
        \mathcal{G}_{m,s}^{(\tau)}(\delta, S) = \Bigg\{ m(\bx;\bW, \bTheta, g): &[\bW_\ell]_{i,j}, [\bb_{\ell}]_{j} \in \left\{-B, -B+\epsilon_1,\cdots, -B+\epsilon_1 \cdot \left\lceil \frac{2B}{\epsilon_1} \right\rceil \right\},  \\
        &~~~~~ \bTheta_{S} \in \left\{-B, -B+\epsilon_2,\cdots, -B+\epsilon_2 \cdot \left\lceil \frac{2B}{\epsilon_2} \right\rceil \right\}^{|S|}, \Theta_{S^c} = 0 \Bigg\}
    \end{align*}
    with
    \begin{align}
    \label{eq:proof:cover-number:def-eps1-eps2}
        \epsilon_1 = \frac{\delta}{4 (M \lor K\|W\|_{\max}) (L+1)B^{L}(N+1)^{L+1}} \qquad \text{and} \qquad \epsilon_2 = \frac{\delta}{4 KB^{L+1}N^L(N+\overline{r})p}.
    \end{align}  Now define 
    \begin{align*}
        \mathcal{G}^{(\tau)}_{m,s}(\delta) = \bigcup_{|S| \le s} \mathcal{G}_{m,s}^{(\tau)}(\delta, S).
    \end{align*} 
    
    We first claim that $\mathcal{G}^{(\tau)}_{m,s}(\delta)$ is a $\delta$-cover of $\mathcal{G}_{m,s}$ with respect to $\|\cdot\|_{\infty, [-K,K]^p}$ norm as long as $\delta>2\tau KB^{L+1}N^L(N+\overline{r})p$. For any $m\in \mathcal{G}_{m,s}$, it follows from the construction of $m^{(\tau)}$ and Lemma \ref{lemma:m-lip-weight} that
    \begin{align*}
        \|m^{(\tau)} - m \|_{\infty, [-K,K]^p} \le KB^{L+1}N^L(N+\overline{r})p \|\Theta-\Theta^{(\tau)}\|_{\max} \le KB^{L+1}N^L(N+\overline{r})p \tau \le \frac{\delta}{2}.
    \end{align*}
    Denote $S_m=\{(i,j), \Theta^{(\tau)}_{i,j}\neq 0\}$. Using the definition of clipped $L_1$ norm and $\mathcal{G}_{m,s}$ set yields
    \begin{align*}
        \|\Theta^{(\tau)}\|_{0} = \sum_{i,j} 1_{\{|\Theta^{(\tau)}_{i,j}| \ge \tau \}} = \sum_{i,j} \psi_{\tau}(\Theta_{i,j}^{(\tau)}) \le \sum_{i,j} \psi_{\tau}(\Theta_{i,j}) \le s,
    \end{align*}
    which implies $|S_m| \le s$. By the construction of $\mathcal{G}_{m,s}^{(\tau)}(\delta, S)$, there exists some $\breve{m}^{(\tau)} \in \mathcal{G}_{m,s}^{(\tau)}(\delta, S_m) \subset \mathcal{G}_{m,s}^{(\tau)}(\delta)$ such that
    \begin{align*}
        \|\btheta_g(m^{(\tau)}) - \btheta_g(\breve{m}^{(\tau)})\|_\infty \le \epsilon_1 \qquad \text{and} \qquad \|\bTheta^{(\tau)} - \breve{\bTheta}^{(\tau)}\|_{\max} \le \epsilon_2.
    \end{align*} 
    
    Recall the definition of $\epsilon_1$ and $\epsilon_2$ in \eqref{eq:proof:cover-number:def-eps1-eps2}, applying Lemma \ref{lemma:m-lip-weight} again yields
    \begin{align*}
        \|m^{(\tau)} - \breve{m}^{(\tau)} \|_{\infty, [-K,K]^p} \le \epsilon_1 (M \lor K\|W\|_{\max}) (L+1)B^{L}(N+1)^{L+1} + \epsilon_2 KB^{L+1}N^L(N+\overline{r})p \le \frac{\delta}{2}.
    \end{align*}
    Therefore we can conclude that $\mathcal{G}^{(\tau)}_{m,s}(\delta)$ is a $\delta$-cover of $\mathcal{G}_{m,s}$ with respect to $\|\cdot\|_{\infty, [-K,K]^p}$ norm since for any $m\in \mathcal{G}_{m,s}$, we can find a corresponding $\breve{m}^{(\tau)}\in \mathcal{G}_{m,s}^{(\tau)}(\delta)$ such that
    \begin{align*}
        \|m - \breve{m}^{(\tau)} \|_{\infty, [-K,K]^p} \le \|m - m^{(\tau)} \|_{\infty, [-K,K]^p} + \|m^{(\tau)} - \breve{m}^{(\tau)} \|_{\infty, [-K,K]^p} \le \delta.
    \end{align*}
    
    \noindent {\sc Step 2. Bound the Cardinality of the Constructed Set.} In this step, we establish a bound on the cardinality of the function class $\mathcal{G}_{m,s}^{(\tau)}(\delta)$. We begin by bounding the cardinality of each set $\mathcal{G}_{m,s}^{(\tau)}(\delta, S)$. The total number of parameters in ReLU network $g$ can be bounded by
    \begin{align*}
        W = \sum_{\ell=1}^{L+1} d_\ell \times (d_{\ell-1} + 1) &= (N+1) + \sum_{\ell=2}^L N(N+1) + N(N+\overline{r}+1) \\
        &= LN^2+(L+1)N + N\overline{r} + 1 \le 2LN^2 + N\overline{r},
    \end{align*} hence it is easy to verify that
    \begin{align*}
        |\mathcal{G}_{m,s}^{(\tau)}(\delta, S)| \le \left\lceil \frac{2B}{\epsilon_1} \right\rceil^{W} \times \left\lceil \frac{2B}{\epsilon_2} \right\rceil^{|S|} \le \left(\frac{2B+\epsilon_1}{\epsilon_1}\right)^W \left(\frac{2B+\epsilon_2}{\epsilon_2}\right)^{|S|},
    \end{align*}
    which implies
    \begin{align}
    \label{eq:proof-cover-size1}
        |\mathcal{G}^{(\tau)}_{m,s}(\delta)| \le \sum_{k=0}^{\lfloor s \rfloor} \binom{Np}{k} \left(\frac{2B+\epsilon_1}{\epsilon_1}\right)^W \left(\frac{2B+\epsilon_2}{\epsilon_2}\right)^{k} \le \left(\frac{2B+\epsilon_1}{\epsilon_1}\right)^W \sum_{k=0}^{\lfloor s \rfloor} \binom{Np}{k}\left(\frac{2B+\epsilon_2}{\epsilon_2}\right)^{k}.
    \end{align}
    Note for any $K \le Np$, $\alpha\ge 1$, we have
    \begin{align*}
        \left(\frac{K}{Np\alpha}\right)^K \sum_{k=0}^K \binom{Np}{k} \alpha^k \le \sum_{k=0}^K \left(\frac{K}{Np}\right)^k \binom{Np}{k}\le\sum_{k=0}^p \left(\frac{K}{Np}\right)^k \binom{Np}{k} = \left(1+\frac{K}{Np}\right)^{Np} \le e^K,
    \end{align*} dividing both sides by $\left(\frac{K}{Np\alpha}\right)^K$ yields
    \begin{align}
    \label{eq:binomial-ineq1}
        \sum_{k=0}^K \binom{Np}{k} \alpha^k \le \left(\frac{eNp\alpha}{K}\right)^K.
    \end{align} Taking $\alpha=(2B+\epsilon_2)/\epsilon_2$ and $K=\lfloor s\rfloor$ in the inequality \eqref{eq:binomial-ineq1}, it follow from the above inequality \eqref{eq:binomial-ineq1} that \eqref{eq:proof-cover-size1} can be further bounded by
    \begin{align*}
        |\mathcal{G}^{(\tau)}_{m,s}(\delta)| \le \left(\frac{2B+\epsilon_1}{\epsilon_1}\right)^W \left(\frac{eNp(2B+\epsilon_2)}{s\epsilon_2}\right)^s.
    \end{align*} 
    
    \noindent {\sc Step 3. Conclusion of the Proof.} Combining Step 1 and Step 2 gives 
    \begin{align*}
        \log \mathcal{N}(\delta, \mathcal{G}_{m,s}, \|\cdot\|_{\infty, [-K,K]^p}) \le W \log\left(1 \lor \frac{4B}{\epsilon_1}\right) + s \log \left(\frac{4e NpB}{\epsilon_2} \lor 1\right)
    \end{align*} By the fact that $\log(xy \lor 1) \le \log x + \log (y\lor 1)$ as long as $x \ge 1$, we can bound the two terms as
    \begin{align*}
        W \log\left(1 \lor \frac{4B}{\epsilon_1}\right) &\le \left(2LN^2+N\overline{r}\right) \log \left(1 \lor \frac{4(M \lor K\|W\|_{\max}) (L+1)B^{L}(N+1)^{L+1}}{\delta} \right) \\
        &\lesssim \left(N^2L + N\overline{r} \right) \left[L \log BN + \log \left(\frac{M\lor K\|W\|_{\max}}{\delta} \lor 1\right)\right],
    \end{align*}and
    \begin{align*}
        s \log \left(\frac{4e NpB}{\epsilon_2} \lor 1\right) &\le s \log \left(\frac{16KB^{L+1}N^{L+1}(N+\overline{r})p^2}{\delta} \lor 1\right) \\
        &\lesssim s \left[L \log (BN) + \log p + \log\left(\frac{K(N+\overline{r})}{\delta} \lor 1\right)\right],
    \end{align*} which completes the proof.
\end{proof}

\begin{proof}[Proof of Lemma~\ref{lemma:m-lip-weight}]
For $m \in \mathcal{G}_{m,s}$, we first recursively define
\begin{align*}
    \bbm^{(\ell)}_+(\bx) = \begin{cases}
        \phi \circ \mathcal{L}_0(\bx) &\qquad \ell = 1 \\
        \mathcal{L}_1 \circ \bbm^{(\ell-1)}(\bx) &\qquad \ell = 2 \\
        \mathcal{L}_{\ell-1} \circ \bar{\sigma} \circ \bbm^{(\ell-1)}(\bx) & \qquad \ell \in \{3,\cdots, L+1\} 
    \end{cases},
\end{align*} and
\begin{align*}
    m^{(\ell)}_{-}(\bz) = \begin{cases}
        \mathcal{L}_{L+1} \circ \bar{\sigma}(\bz) &\qquad \ell = L+1\\
        m^{(\ell+1)} \circ \mathcal{L}_{\ell} \circ \bar{\sigma}(\bz) &\qquad \ell \in \{2,\cdots, L\} \\
        m^{(\ell+1)} \circ \mathcal{L}_{\ell}(\bz) & \qquad \ell =1 
    \end{cases}.
\end{align*}
Similarly, we can also defined corresponding $\breve{\bbm}^{(\ell)}_+(\bx): \mathbb{R}^p\to \mathbb{R}^{d_{\ell-1}}$ and $\breve{m}^{(\ell)}_-(\bz): \mathbb{R}^{d_{\ell-1}} \to \mathbb{R}$ for $\breve{m}$ that 
\begin{align*}
    \breve{m}(\bx) = \breve{\mathcal{L}}_{L+1} \circ \bar{\sigma} \circ\breve{\mathcal{L}}_{L} \circ \bar{\sigma} \circ \cdots \circ \breve{\mathcal{L}}_2 \circ \bar{\sigma} \circ\breve{\mathcal{L}}_1 \circ \phi \circ \breve{\mathcal{L}}_0(\bx).
\end{align*}

The above construction of $\bbm^{(\ell)}_+$ and $m^{(\ell)}_-$ implies
\begin{align*}
    m(\bx) = m^{(\ell)}_- \circ \bbm^{(\ell)}_+(\bx) =  m^{(\ell)}_- \circ \mathcal{L}_{\ell-1} \circ \bar{\sigma} \circ \bbm^{(\ell-1)}_+(\bx)
\end{align*}
Consequently, it follows from triangle inequality that
\begin{align}
\label{eq:proof:lemma:m-lip-weight:telescope}
\begin{split}
    |m(\bx) - \breve{m}(\bx)| = &\left|\mathcal{L}_{L+1} \circ \bar{\sigma}  \circ \bbm_+^{(L+1)}(\bx) - \breve{\mathcal{L}}_{L+1} \circ \bar{\sigma} \circ \bbm_+^{(L+1)}(\bx)\right| \\
    &~~~~~ + \sum_{\ell=1}^{L} \left|\breve{m}_-^{(\ell+1)} \circ \mathcal{L}_\ell \circ \bar{\sigma} \circ \bbm^{(\ell)}_+(\bx) - \breve{m}_-^{(\ell+1)} \circ \breve{\mathcal{L}}_\ell \circ \bar{\sigma} \circ \bbm^{(\ell)}_+(\bx)  \right|\\
    &~~~~~ + \left|\breve{m}_-^{(1)} \circ \phi \circ \mathcal{L}_0(\bx) - \breve{m}_-^{(1)} \circ \phi \circ \breve{\mathcal{L}}_0(\bx)\right|
\end{split}
\end{align}
We claim that
\begin{align}
\label{eq:m-lip-weight-claim1}
    \|m_-^{(\ell)}(\bu) - m_-^{(\ell)}(\bv)\|_\infty \le \begin{cases}
    (BN)^{L+2-\ell} \|\bu - \bv\|_\infty & \qquad \ell \in \{2,\cdots, L+1\} \\
    B^{L+1} N^{L} (N+\overline{r}) \|\bu - \bv\|_\infty & \qquad \ell = 1
    \end{cases},
\end{align} and
\begin{align}
\label{eq:m-lip-weight-claim2}
    \|\bbm_+^{(\ell)}(\bx)\|_\infty \le (M \lor K) (B(N+1))^{\ell-1} \qquad \forall \ell \in \{1,\cdots, L+1\}.
\end{align}

We first prove the argument \eqref{eq:m-lip-weight} using the above two claims \eqref{eq:m-lip-weight-claim1} and \eqref{eq:m-lip-weight-claim2}. Using claim \eqref{eq:m-lip-weight-claim1} with $\ell=1$ gives
\begin{align*}
    \left|\breve{m}_-^{(1)} \circ \phi \circ \mathcal{L}_0(\bx) - \breve{m}_-^{(1)} \circ \phi \circ \breve{\mathcal{L}}_0(\bx) \right| &\le B^{L+1} N^{L} (N+\overline{r}) \|\phi \circ \mathcal{L}_0(\bx) - \phi \circ \breve{\mathcal{L}}_0(\bx)\|_\infty \\
    &\le B^{L+1} N^{L} (N+\overline{r}) \|\mathcal{L}_0(\bx) - \breve{\mathcal{L}}_0(\bx)\|_\infty.
\end{align*}
Recall $[\mathcal{L}_0(\bx)]_i = [\breve{\mathcal{L}}_0(\bx)]_i = p^{-1} [\bW]_{:,i}^\top \bx$ for $i \in \{1,\cdots, \overline{r}\}$, then for any $i\in \{1,\cdots, N\}$,
\begin{align}
\label{eq:proof:lemma:m-lip-weight:derivation1}
\begin{split}
    \left|[\mathcal{L}_0(\bx)]_{i+\overline{r}} - [\breve{\mathcal{L}}_0(\bx)]_{i+\overline{r}} \right| &= \left|[\bTheta]_{:,i}^\top \bx - [\breve{\bTheta}]_{:,i}^\top \bx\right| \\
    &\le \|\bx\|_\infty \left\|[\bTheta]_{:,i} - [\breve{\bTheta}]_{:,i}\right\|_1 \le K p \|\bTheta - \breve{\bTheta}\|_{\max},
\end{split}
\end{align} provided $\bx \in [-K,K]^p$. This implies
\begin{align}
\label{eq:proof:lemma:m-lip-weight:decomp1}
    \left|\breve{m}_-^{(1)} \circ \phi \circ \mathcal{L}_0(\bx) - \breve{m}_-^{(1)} \circ \phi \circ \breve{\mathcal{L}}_0(\bx) \right| \le K B^{L+1} N^{L} (N+\overline{r}) p\|\bTheta - \breve{\bTheta}\|_{\max}.
\end{align} 

For any $\ell \in \{1,\cdots, L+1\}$, it follows from a similar argument that
\begin{align*}
    \|\mathcal{L}_\ell \circ \bar{\sigma}(\bz) - \breve{\mathcal{L}}_\ell \circ \bar{\sigma}(\bz)\|_\infty &= \left\|\bW_\ell \bar{\sigma}(\bz) - \breve{\bW}_\ell \bar{\sigma}(\bz) + \bb_\ell - \breve{\bb}_\ell\right\|_\infty \\
    &\le \left\|\bW_\ell \bar{\sigma}(\bz) - \breve{\bW}_\ell \bar{\sigma}(\bz)\right\|_\infty + \|\bb_\ell - \breve{\bb}_\ell\|_\infty \\
    &= \max_{i\in \{1,\cdots,d_\ell\}} \left|[\bW_\ell]_{i,:} \bar{\sigma}(z) - [\breve{\bW}_\ell]_{i,:} \bar{\sigma}(\bz) \right| + \|\bb_\ell - \breve{\bb}_\ell\|_\infty \\
    &\le \max_{i\in \{1,\cdots,d_\ell\}} \|[\bW_\ell]_{i,:} - [\breve{\bW}_\ell]_{i,:} \|_1 \|\bar{\sigma}(\bz)\|_\infty + \|\bb_\ell - \breve{\bb}_\ell\|_\infty \\
    &\le \|\bW_\ell - \breve{\bW}_\ell\|_{\max} d_{\ell-1} \|\bz\|_\infty + \|\bb_\ell - \breve{\bb}_\ell\|_\infty,
\end{align*} which implies
\begin{align}
\label{eq:m-lip-weight-lell-sigma-bound}
    \|\mathcal{L}_\ell \circ \bar{\sigma}(\bz) - \breve{\mathcal{L}_\ell} \circ \bar{\sigma}(\bz)\|_\infty \le \|\btheta_g(m) - \btheta_g(\breve{m})\|_\infty (1 + d_{\ell-1} \|\bz\|_\infty).
\end{align}
It follows from \eqref{eq:m-lip-weight-lell-sigma-bound} and \eqref{eq:m-lip-weight-claim2} with $\ell=L+1$ that
\begin{align}
\label{eq:proof:lemma:m-lip-weight:decomp2}
\begin{split}
    &\left|\mathcal{L}_{L+1} \circ \bar{\sigma}  \circ \bbm_+^{(L+1)}(\bx) - \breve{\mathcal{L}}_{L+1} \circ \bar{\sigma} \circ \bbm_+^{(L+1)}(\bx)\right| \\
    &~~~~ \le \|\btheta_g(m) - \btheta_g(\breve{m})\|_\infty (1 + N \|m_+^{(L+1)}(x)\|_\infty) \\
    &~~~~ \le (M\lor K\|\bW\|_{\max}) B^L(N+1)^{L+1} \|\btheta_g(m) - \btheta_g(\breve{m})\|_\infty,
\end{split}
\end{align} provided $(B(N+1))^L (M\lor K\|W\|_{\max}) \ge 1$.

Moreover, for any $\ell\in \{1,\cdots, L\}$, Combing \eqref{eq:m-lip-weight-claim1}, \eqref{eq:m-lip-weight-lell-sigma-bound} and \eqref{eq:m-lip-weight-claim2} gives
\begin{align}
\label{eq:proof:lemma:m-lip-weight:decomp3}
\begin{split}
    \Big|\breve{m}_-^{(\ell+1)} \circ &\mathcal{L}_\ell \circ \bar{\sigma} \circ \bbm^{(\ell)}_+(\bx) - \breve{m}_-^{(\ell+1)} \circ \breve{\mathcal{L}}_\ell \circ \bar{\sigma} \circ \bbm^{(\ell)}_+(\bx) \Big| \\
    &\le (BN)^{L+1-\ell} \|\mathcal{L}_\ell \circ \bar{\sigma} \circ \bbm^{(\ell)}_+(\bx) - \breve{\mathcal{L}}_\ell \circ \bar{\sigma} \circ \bbm^{(\ell)}_+(\bx)\|_\infty \\
    &\le (BN)^{L+1-\ell} \|\btheta_g(m) - \btheta_g(\breve{m})\|_\infty (1 + d_{\ell-1} \|\bbm^{(\ell)}_+(\bx)\|_\infty) \\
    &\le (BN)^{L+1-\ell} \|\btheta_g(m) - \btheta_g(\breve{m})\|_\infty \left(1 + N(M \lor K\|\bW\|_{\max}) (B(N+1))^{\ell-1}\right) \\
    &\le (M\lor K\|\bW\|_{\max}) B^L (N+1)^L \|\btheta_g(m) - \btheta_g(\breve{m})\|_\infty.
\end{split}
\end{align} 
Plugging \eqref{eq:proof:lemma:m-lip-weight:decomp1}, \eqref{eq:proof:lemma:m-lip-weight:decomp2} and \eqref{eq:proof:lemma:m-lip-weight:decomp3} into \eqref{eq:proof:lemma:m-lip-weight:telescope} completes the proof of our main argument \eqref{eq:m-lip-weight}. 

It suffices to prove the two claims \eqref{eq:m-lip-weight-claim1} and \eqref{eq:m-lip-weight-claim2}.

\noindent {\sc Proof of Claim \eqref{eq:m-lip-weight-claim1}. } For any $\ell \in \{1,\cdots, L\}$,
\begin{align}
\label{eq:m-lip-weight-claim1:proof-induction-ineq}
    \|\mathcal{L}_{\ell}(\bu) - \mathcal{L}_{\ell}(\bv)\|_\infty \le \|\bW_\ell^\top (\bu - \bv)\|_\infty \le d_{\ell-1} B \|\bu - \bv\|_\infty,
\end{align} provided $\|\bW_\ell\|_{\max} \le B$. We prove the claim \eqref{eq:m-lip-weight-claim1} by induction. When $\ell=L+1$, this is a direct consequence of \eqref{eq:m-lip-weight-claim1:proof-induction-ineq} that
\begin{align*}
    \|m^{(\ell)}_{-}(\bu) - m^{(\ell)}_{-}(\bv) \|_\infty = \left\|\mathcal{L}_{L+1} \left(\bar{\sigma}(\bu) - \bar{\sigma}(\bv)\right)\right\|_\infty \le d_L B \|\bu-\bv\|_\infty = BN \|\bu-\bv\|_\infty.
\end{align*} 

If the claim \eqref{eq:m-lip-weight-claim1} holds for $\ell+1$ with $\ell \in \{2, 3,\cdots, L\}$, then we further have
\begin{align*}
    \|m^{(\ell)}_{-}(\bu) - m^{(\ell)}_{-}(\bv)\|_\infty &= \left\|m^{(\ell+1)}_{-}(\bu) \circ \mathcal{L}_{\ell} \circ \bar{\sigma}(\bu) - m^{(\ell+1)}_{-}(\bu) \circ \mathcal{L}_{\ell} \circ \bar{\sigma}(\bv)\right\|_\infty \\
    &\overset{(a)}{\le} (BN)^{L+1-\ell} \|\mathcal{L}_{\ell} \circ \bar{\sigma}(\bu) - \mathcal{L}_{\ell} \circ \bar{\sigma}(\bv)\|_\infty \\
    &\overset{(b)}{\le} (BN)^{L+1-\ell} (BN) \|\bar{\sigma}(\bu) - \bar{\sigma}(\bv)\|_\infty \\
    &\overset{(c)}{\le} (BN)^{L+2-\ell} \|\bu - \bv\|_\infty.
\end{align*} Here (a) follows from induction, (b) follows from \eqref{eq:m-lip-weight-claim1:proof-induction-ineq}, (c) follows from the fact that $|\sigma(x)| \le |x|$. The case where $\ell=1$ is very similar, so it concludes.

\noindent {\sc Proof of Claim \eqref{eq:m-lip-weight-claim2}.} Our proof will use the fact that
\begin{align}
\label{eq:m-lip-weight-claim2:proof-induction-ineq}
    \|\mathcal{L}(\bx)\|_\infty = \|\bW_\ell \bx + \bb_\ell\|_\infty \le B + d_{\ell-1}B \|\bx\|_\infty
\end{align} provided $\ell\in \{1,\cdots, L\}$ repeatedly. We also prove \eqref{eq:m-lip-weight-claim2} by induction. When $\ell=1$, it follows from the definition of $\mathcal{L}_0$ and $\phi$ that,
\begin{align*}
    \| \bbm^{(\ell)}_+(\bx) \|_\infty \le K \|\bW\|_{\max} \lor M.
\end{align*} 

If \eqref{eq:m-lip-weight-claim2} holds for $\ell-1$ with $\ell \in \{2, \cdots, L+1\}$, then combining \eqref{eq:m-lip-weight-claim2:proof-induction-ineq} and the fact $|\sigma(x)| \le |x|$ gives
\begin{align*}
    \|\bbm_+^{(\ell)}(\bx)\|_\infty \le B + d_{\ell-1} B \|\bbm^{(\ell-1)}_+(\bx)\|_\infty &\le B + (BN) (B(N+1))^{\ell-2} (K \|\bW\|_{\max} \lor M) \\&\le (K \|\bW\|_{\max} \lor M) (B(N+1))^{\ell-1}.
\end{align*} which completes the proof of claim \eqref{eq:m-lip-weight-claim2}.
\end{proof}

\subsection{Some Facts about the Empirical Process}

Similar to the proof of oracle-type inequality for the FAR-NN estimator, we also need two technical lemmas to (1) establish a bound on the empirical $L_2$ norm from below in terms of the population $L_2$ norm, and (2) derive a bound on the weighted empirical process via the empirical $L_2$ norm \emph{both in the presence of the regularizer term}. Note that the following results are only applicable to ReLU network classes with bounded weights instead of generic function classes with finite Pseudo-dimension.

\begin{lemma}
\label{lemma:equivalence-population-empirical-l2-regularized}
Let $\varrho_n$, $v_n$ be the quantity defined in \eqref{eq:fast-rate-v_n} and \eqref{eq:fast-rate-varrho_n} respectively. Suppose $\tilde{m}$ is a fixed function in $\mathcal{G}_m$. Under the conditions of Theorem \ref{thm:fast-oracle}, there exists some universal constants $c_1$--$c_3$ such the event \begin{align*}
    \mathcal{C}_t = \left\{\forall m\in \mathcal{G}_m, ~~ \frac{1}{2}\|m - \tilde{m}\|^2_2 \le \|m - \tilde{m}\|_n^2 + 2\lambda \sum_{i,j} \psi_\tau({\Theta}_{i,j}) + c_1 \left(v_n + \rho_n + \frac{t}{n}\right)\right\}
\end{align*} satisfies $\mathbb{P}\left[\mathcal{C}_{t}\right] \ge 1-e^{-t}$ for any $t>0$ as long as $\lambda \ge c_2\varrho_n$ and $\tau^{-1} \ge c_3 (r+1) b (BN)^{L+1} (N+\overline{r}) p n$.
\end{lemma}

\begin{lemma}
\label{lemma:weighted-empirical-process-regularized}
Let $\varrho_n$, $v_n$ be the quantity defined in \eqref{eq:fast-rate-v_n} and \eqref{eq:fast-rate-varrho_n} respectively. Suppose $\tilde{m}$ is a fixed function in $\mathcal{G}_m$. Under the conditions of Theorem \ref{thm:fast-oracle}, there exists some universal constants $c_1$--$c_2$ such that for any fixed $\epsilon \in (0,1)$, the event 
\begin{align*}
    \mathcal{B}_{t,\epsilon} = \left\{\forall m\in \mathcal{G}_m, \frac{4}{n} \sum_{i=1}^n \varepsilon_i (m(\bx_i) - \tilde{m}(\bx_i)) - \lambda \sum_{i,j} \psi_\tau({\Theta}_{i,j}) \le \epsilon \|m - \tilde{m}\|_n^2 + \frac{c_1}{\epsilon} \left(v_n + \varrho_n + \frac{t}{n} \right)\right\}
\end{align*} occurs with probability at least $1-e^{-t}$ for any $t>0$ as long as $\lambda \ge \frac{c_2 \varrho_n}{\epsilon}$ and $\tau^{-1} \ge 4(r+1)b (BN)^{L+1} (N+\overline{r}) p n$.
\end{lemma}

We also need the following variant of standard chaining result to prove Lemma \ref{lemma:weighted-empirical-process-regularized}.
\begin{lemma}[Chaining]
\label{lemma:chaining}
Let $\{X_t, t\in \mathbb{T}\}$ be a zero-mean sub-Gaussian process with respect to the distance $d$ such that
\begin{align}
\label{eq:inc-cond}
\forall s, t\in \mathbb{T}, u>0, ~~~ \mathbb{P}\left[|X_t - X_s| \ge u\right] \le \exp\left(-\frac{u^2}{2d(s,t)^2}\right)
\end{align} Denote $\Delta(\mathbb{T}) = \sup_{s, t \in \mathbb{T}} d(s,t)$. There exists a universal constant $c_1$ such that for any $\epsilon \in [0, \Delta(\mathbb{T})]$ and $u>0$, if
\begin{align*}
    \mathbb{P} \left[ \sup_{s,t\in \mathbb{T}, d(s,t)\le \epsilon} |X_t - X_s| \le \alpha(\epsilon, u)\right] \ge 1-e^{-u},
\end{align*} then 
\begin{align*}
    \mathbb{P}\left[\sup_{s,t\in \mathbb{T}} |X_t - X_s| \le c_1\left\{\int_{\epsilon/4}^{\Delta(\mathbb{T})} \sqrt{\log \mathcal{N}(\omega, \mathbb{T}, d) } \mathrm{d}\omega + \left(u^{1/2}+1\right) \Delta(\mathbb{T}) + \alpha(\epsilon, u)\right\}\right] \ge 1-2e^{-u}.
\end{align*}
\end{lemma}
\begin{proof}[Proof of Lemma~\ref{lemma:chaining}] We apply the standard chaining technique, but carefully deal with the tail probability term $u$. Let $m \in \mathbb{N^+}$ be such that \begin{align*}
    \epsilon/2 < 2^{-m} \Delta(\mathbb{T}) \le \epsilon.
\end{align*} Let $T_0, \cdots, T_m\subset \mathbb{T}$ such that $T_k$ is a $2^{-k} \Delta(\mathbb{T})$-cover of $\mathbb{T}$ with respect to the distance $d$, thus $|T_k| = \mathcal{N}(2^{-k}\Delta(\mathbb{T}), \mathbb{T}, d)$. By this construction, for any $t\in \mathbb{T}$, there exists some $\pi_k(t) \in T_k$ such that $d(\pi_k(t), t) \le 2^{-k} \Delta(\mathbb{T})$. In particular, $|T_0|=1$, which implies $\pi_0(t)=\pi_0(s)$ for any $s, t \in \mathbb{T}$. Therefore, 
\begin{align}
\label{eq:chaining:telescope}
    \sup_{s,t\in \mathbb{T}} |X_t - X_s| \le |X_t - X_{\pi_m(t)}| + \sum_{k=1}^m |X_{\pi_k(t)} - X_{\pi_{k-1}(t)}| + \sum_{k=1}^m |X_{\pi_{k-1}(s)} - X_{\pi_{k}(s)}| + |X_{\pi_m(s)} - X_s|.
\end{align} By the increment condition \eqref{eq:inc-cond}, for fixed $t \in \mathbb{T}$,
\begin{align*}
    \mathbb{P}\left[|X_{\pi_k(t)} - X_{\pi_{k-1}(t)}| \ge d(\pi_k(t), \pi_{k-1}(t)) \sqrt{2(u + \eta_k)} \right] \le \exp(-u-\eta_k),
\end{align*} for arbitrary $\eta_k>0$. Set $\eta_k=k+2\log\mathcal{N}(2^{-k}\Delta(\mathbb{T}), \mathbb{T}, d)$. Define the event $\Omega_k$ as
\begin{align*}
    \forall t\in \mathbb{T}, ~~ |X_{\pi_k(t)} - X_{\pi_{k-1}(t)}| \le d(\pi_k(t), \pi_{k-1}(t)) \sqrt{2\left(u+\eta_k\right)}.
\end{align*} 
It follows from the union bound that
\begin{align*}
    \mathbb{P}\left[\Omega^c_k\right] \le |T_k| |T_{k-1}| e^{-u-\eta_k} = \exp\left(-u-k+\log |T_k| + \log |T_{k-1}| - 2\mathcal{N}(2^{-k}\Delta(\mathbb{T}), \mathbb{T}, d) \right) \le e^{-u-k}.
\end{align*} Let
\begin{align*}
    \Omega = \bigcap_{k=1}^m \Omega_k \cap \left\{\sup_{s,t\in \mathbb{T}, d(s,t)\le \epsilon} |X_t - X_s| \le \alpha(\epsilon, u)\right\},
\end{align*} using the union bound gives
\begin{align*}
    \mathbb{P}(\Omega^c) &\le \sum_{k=1}^m \mathbb{P}(\Omega_k^c) + \mathbb{P}\left[\sup_{s,t\in \mathbb{T}, d(s,t)\le \epsilon} |X_t - X_s| > \alpha(\epsilon, u)\right]\\
    &\le \sum_{k=1}^m \exp(-u-k) + e^{-u} \le e^{-u}\sum_{k=1}^m 2^{-k} + e^{-u} \le 2e^{-u}.
\end{align*} Notice that $d(\pi_k(t), \pi_{k-1}(t)) \le d(\pi_k(t), t) + d(\pi_{k-1}(t), t) \le 4\cdot 2^{-k} \Delta(\mathbb{T})$. Hence, under $\Omega$, the telescope decomposition \eqref{eq:chaining:telescope} gives
\begin{align*}
    \sup_{s,t\in \mathbb{T}} |X_t - X_s| &\le 2\alpha(\epsilon, u) + 2\sum_{k=1}^m d(\pi_k(t), \pi_{k-1}(t)) \sqrt{2(u+\eta_k)} \\
    &\lesssim \alpha(\epsilon, u) + \sum_{k=1}^m 2^{-k} \Delta(\mathbb{T})\left(\sqrt{k} + \sqrt{u} + \sqrt{\log \mathcal{N}(2^{-k}\Delta(\mathbb{T}), \mathbb{T}, d)}\right)
\end{align*}
We have
\begin{align*}
    \sum_{k=1}^{m} \sqrt{k}2^{-k} \le \sum_{k=1}^\infty \sqrt{k} \sqrt{2}^{-k} \sqrt{2}^{-k} \le \sum_{k=1}^m \sqrt{2}^{-k} \lesssim 1,
\end{align*} and
\begin{align*}
    \sum_{k=1}^m 2^{-k} \Delta(\mathbb{T}) \sqrt{\log \mathcal{N}(2^{-k}\Delta(\mathbb{T}), \mathbb{T}, d)} &\le 2\sum_{k=1}^m \int_{2^{-(k+1)}\Delta(\mathbb{T})}^{2^{-k}\Delta(\mathbb{T})} \sqrt{\log \mathcal{N}(\omega, \mathbb{T}, d)} \mathrm{d}\omega \\
    &\le 2\int_{2^{-(k+1)} \Delta(\mathbb{T})}^{\Delta(\mathbb{T})} \sqrt{\log \mathcal{N}(\omega, \mathbb{T}, d)} \mathrm{d}\omega \\
    &\le \int_{\epsilon/4}^{\Delta(\mathbb{T})} \sqrt{\log \mathcal{N}(\omega, \mathbb{T}, d)} \mathrm{d}\omega.
\end{align*} Putting these pieces together, we can conclude that under $\Omega$, which occurs with probability at least $1-2e^{-u}$, the following inequality
\begin{align*}
\sup_{s,t\in \mathbb{T}} |X_t - X_s| \lesssim \alpha(\epsilon, u) + \Delta(\mathbb{T}) \left(1 + \sqrt{u}\right) + \int_{\epsilon/4}^{\Delta(\mathbb{T})} \sqrt{\log \mathcal{N}(\omega, \mathbb{T}, d)} \mathrm{d}\omega,
\end{align*} holds, which completes the proof.
\end{proof}

\begin{proof}[Proof of Lemma \ref{lemma:equivalence-population-empirical-l2-regularized}]
    Define the event
    \begin{align*}
        \mathcal{C}_{t}(s) = \left\{ \forall m\in \mathcal{G}_{m,s}, ~~\frac{1}{2} \|m - \tilde{m}\|_2^2 \le \|m - \tilde{m}\|_n^2 + C_1\left(s\varrho_n + v_n + \frac{t}{n} \right)\right\}
    \end{align*} for some constant $C_1>0$ to be determined. We first use Theorem 19.6 in \cite{gyorfi2002distribution} with $\epsilon=0.5$ to show that $\mathbb{P}(\mathcal{C}_t(s)) \ge 1-e^{-t}$. To this end, define the function class $\mathcal{H} = \{h = (m-\tilde{m})^2: g\in \mathcal{G}_{m,s}\}$, then
    \begin{align*}
        h(\bx) \le (2M)^2=K_1 \qquad \text{and} \qquad \mathbb{E}|h(X)|^2 \le (2M)^2 \mathbb{E}[h(\bx)]=K_2\mathbb{E}[h(\bx)].
    \end{align*} Suppose $u \in \mathbb{R}^+$ is fixed and satisfying
    \begin{align*}
        u \ge \left(576 \cdot (8M^2 \lor 2\sqrt{2} M)\right)^2 \cdot \frac{2}{n}.
    \end{align*} It follows from the uniform boundedness of $m, \tilde{m}$ that
    \begin{align*}
        \log \mathcal{N}_2(\omega, \mathcal{H}, \bx_1^n) &\le \log \mathcal{N}_\infty(\omega, \mathcal{H}, \bx_1^n) \\
        &\le \log \mathcal{N}_\infty(\omega/M, \mathcal{G}_{m,s}-\tilde{g}, \bx_1^n) \\
        &\overset{(a)}{\le} \log \left(\mathcal{N}(\omega/M, \mathcal{G}_{m,s}, \|\cdot\|_{\infty, [-(r+1)b,(r+1)b]^p}) + 1\right),
    \end{align*} for arbitrary $\bx_1^n$ provided $\omega \ge \frac{0.5^2 u}{16 \cdot 8\cdot 8M^2}$,
    where (a) is implied by the uniform boundedness in Condition \ref{cond1}. Therefore, if 
    \begin{align*}
        \tau^{-1} \ge 2 (r+1)b (BN)^{L+1} (N+\overline{r}) p \left(\frac{0.5^2 u}{16 \cdot 8\cdot 8M^2}\right)^{-1} \gtrsim (r+1) (BN)^{L+1} (N+\overline{r}) pn,
    \end{align*} then it follows from Lemma \ref{lemma:g-infty-cover-number} that
    \begin{align*}
        \log \mathcal{N}_2(\omega, \mathcal{H}, \bx_1^n) \lesssim v_n + s\varrho_n + \log(\omega^{-1} \lor 1) \lesssim nv_n + ns\varrho_n + \log n.
    \end{align*} 
    
    For any $\delta\ge u/8$, we have
    \begin{align*}
        \mathcal{J}(\delta)&=\sup_{x_1^n} \int^{\sqrt{\delta}}_{\delta/512M^2} \sqrt{\log \mathcal{N}_2 \left(\omega, \left\{h\in \mathcal{H}: \frac{1}{n}\sum_{i=1}^n h(\bx_i)^2 \le 16\delta\right\}, \bx_{1}^n\right)} \mathrm{d}\omega \\
        &\le C_2 \sqrt{\delta} \sqrt{nv_n + ns\varrho_n + \log n}.
    \end{align*} for some universal constant $C_2>0$. If we further impose $u \ge (C_2 12288 M^2)^2 (v_n + s\varrho_n + \frac{\log n}{n})=\underline{u}$, then
    \begin{align*}
        \mathcal{J}(\delta) \le \frac{\sqrt{\delta} \sqrt{n\cdot \underline{u}}}{12288 \times M^2} \le \frac{\sqrt{n} \delta}{3072\sqrt{2} M^2} = \frac{\epsilon(1-\epsilon)\sqrt{n} \delta}{96\sqrt{2} (K_1\lor 2K_2)}.
    \end{align*} 
    
    Putting these pieces together, applying Theorem 19.3 in \cite{gyorfi2002distribution} gives
    \begin{align*}
        \mathbb{P}\left[\sup_{h\in \mathcal{H}} \frac{|\mathbb{E}[h(X)] - \frac{1}{n} \sum_{i=1}^n h(X_i)|}{u + \mathbb{E} h(X)} \ge \frac{1}{2} \right] \le C_3 \exp(-C_4nu),
    \end{align*} for some constants $C_3, C_4$. This implies the event
    \begin{align*}
        \mathcal{C}^\dagger_u(s) = \left\{\forall m\in \mathcal{G}_{m,s}, ~~ \left|\|m-\tilde{m}\|_n^2 - \|m-\tilde{m}\|_2 \right| \le \frac{1}{2} \|m-\tilde{m}\|_2 + \frac{1}{2} \underline{u} + \frac{1}{2} (u - \underline{u}) \right\}
    \end{align*} occurs with probability at least $1-C_3\exp(-C_4 n(u - \underline{u}))$. This completes the proof of the claim $\mathbb{P}[\mathcal{C}_t(s)] \ge 1-e^{-t}$ because $\frac{\log n}{n} \lesssim v_n$.
    
    Now we are ready to show that $\mathbb{P}(\mathcal{C}_t) \ge 1-e^{-t}$ by peeling device. To this end, let 
    \begin{align*}
        \mathcal{G}_m(k) = \left\{m\in \mathcal{G}_m: \alpha_{k-1} \le \sum_{i,j}\psi_\tau(\Theta_{i,j}) \le \alpha_k \right\}
    \end{align*} for $k=0,\cdots, \lceil \log_2 (pN)\rceil$, where $\alpha_k = 2^k$ for $k=0,1\cdots$, and $\alpha_{-1}=0$. It's easy to verify that $\mathcal{G}_m = \bigcup_{k=0}^{\lceil \log_2 (pN) \rceil} \mathcal{G}_m(k)$ and $\alpha_k \le 2\alpha_{k-1} + 1$ for all the $k\ge 0$. If the event $\mathcal{C}_t(\alpha_k)$ occurs, then for any $m\in \mathcal{G}_m(k)$, 
    \begin{align*}
        \frac{1}{2} \|m - \tilde{m}\|_2^2 &\le \|m-\tilde{m}\|_n^2 + C_1 \left(\alpha_k \varrho_n + v_n + \frac{t}{n}\right) \\
                                          &\le \|m-\tilde{m}\|_n^2 + C_1 2\alpha_{k-1}\varrho_n + C_1 \left(\varrho_n + v_n + \frac{t}{n}\right)\\
                                          &\le \|m-\tilde{m}\|_n^2 + 2(C_1\varrho_n)\sum_{i,j}\psi_\tau(\Theta_{i,j}) + C_1 \left(\varrho_n + v_n + \frac{t}{n}\right).
    \end{align*} Therefore, if $\lambda\ge C_1\varrho_n$, letting $c_1=(C_5+1)C_1$ for some constant $C_5>0$, we have
    \begin{align*}
        \mathbb{P}(\mathcal{C}_t^c) &= \sum_{k=0}^{\lceil \log_2 (pN)\rceil} \mathbb{P}\left[ \exists m \in \mathcal{G}_m(k), \frac{1}{2}\|m - \tilde{m}\|^2_2 > \|m - \tilde{m}\|_n^2 + 2\lambda \sum_{i,j} \psi_\tau({\Theta}_{i,j}) + C_1(1+C_5)\left(v_n + \varrho_n + \frac{t}{n}\right)\right]\\
        &\le \sum_{k=0}^{\lceil \log_2 (pN)\rceil} \mathbb{P}\left[\left(\mathcal{C}_{t+C_5n(v_n+\varrho_n)}(\alpha_k)\right)^c\right] \le C_6 \log (pN) e^{-t - C_5 n(v_n+\varrho_n)} = e^{-t} e^{\log (C_6\log (pN)) - C_5 n(v_n + \varrho_n)}
    \end{align*} 
    Note that $\log (\log (pN)) \lesssim n\varrho_n$, we can choose some large $C_5$ such that $e^{\log (C_6\log (pN)) - C_5 n(v_n + \varrho_n)}\le 1$, which completes the proof.
    
\end{proof}

\begin{proof}[Proof of Lemma~\ref{lemma:weighted-empirical-process-regularized}]
\noindent {\sc Step 1. Application of Chaining.} Condition on fixed $\bx_1,\cdots, \bx_n$, define the random variable
\begin{align*}
    Z_{s, \delta} := \sup_{m\in \mathcal{G}_{m,s}, \|m - \tilde{m}\|_n \le \delta} \left|\frac{1}{n} \sum_{i=1}^n \varepsilon_i \left(m(\bx_i) - \tilde{m}(\bx_i)\right)\right|.
\end{align*} We claim that
\begin{align}
\label{eq:claim-step4.1}
\mathbb{P}\left(\mathcal{B}_{t}(s,\delta)\right) \ge 1-e^{-t}, ~~\text{where}~~ \mathcal{B}_{t}(s,\delta) = \left\{Z_{s, \delta} \le C_1 \left[\delta \left( \sqrt{\frac{u + 1}{n}} + \sqrt{v_n} + \sqrt{s \varrho_n} \right) + \frac{u+1}{n} \right]\right\},
\end{align} for some constant $C_1 > 0$. To this end, we use Lemma \ref{lemma:chaining}. Let 
\begin{align*}
    d(g, h) = n^{-1/2}\sqrt{\frac{\sigma^2}{n} \sum_{i=1}^n (g(\bx_i) - h(\bx_i))^2}, 
\end{align*} and $\mathbb{T} = \{m\in \mathcal{G}_{m,s}, d(m, \tilde{m}) \le \sigma n^{-1/2} \delta\}$, then it is easy to verify that $\sup_{\theta,\tilde{\theta}\in \mathbb{T}} d(\theta,\tilde{\theta}) \le 2\sigma n^{-1/2}\delta$. Let $Z_\theta = \frac{1}{n}\sum_{i=1}^n \varepsilon_i (\theta(\bx_i) - \tilde{m}(\bx_i))$, thus $Z_{s,\delta} = \sup_{\theta\in \mathbb{T}} Z_\theta$. For any $\theta, \tilde{\theta} \in \mathbb{T}$, the uniform sub-Gaussian condition and independence of $\varepsilon_1,\cdots, \varepsilon_n$ implies
\begin{align*}
    \mathbb{E} \left[\exp\left(t (Z_\theta - Z_{\tilde{\theta}})\right)\right] = \prod_{i=1}^n \mathbb{E} \left[\exp\left(t\varepsilon_i n^{-1} (\theta(\bx_i)-\tilde{\theta}(\bx_i))\right)\right] &\le \prod_{i=1}^n \exp\left(\frac{\sigma^2t^2}{2n^2}(\theta(\bx_i)-\tilde{\theta}(\bx_i))^2\right) \\
    &= \exp\left(\frac{t^2}{2} d^2(\theta, \tilde{\theta}) \right).
\end{align*}
The above moment generating function condition validates the increment condition \eqref{eq:inc-cond}, hence it follows from Lemma \ref{lemma:chaining} that
\begin{align}
\label{eq:general-bound-step4.1}
    \sup_{\theta, \tilde{\theta} \in \mathbb{T}} (Z_\theta - Z_{\tilde{\theta}}) &\lesssim \alpha(\epsilon, u) + \int_{\epsilon/4}^{2\sigma n^{-1/2}\delta} \sqrt{\log \mathcal{N}(\omega,\mathbb{T},d)} \mathrm{d}\omega + \left(\sqrt{u} + 1\right)\frac{2\sigma\delta}{\sqrt{n}},
\end{align} with probability at least $1-2e^{-u}$ for arbitrary $\epsilon>0$. Note that any $(\sqrt{n}\omega/\sigma)$-cover of $\mathcal{G}_{m,s}$ with respect to $\|\cdot\|_\infty$ norm is also a $u$-cover of $\mathbb{T}$ with respect to $d(\cdot, \cdot)$ distance. Because $\|\bx\|_\infty \le (r+1)b$ by Condition \ref{cond1}, applying Lemma \ref{lemma:g-infty-cover-number} with $K=(r+1)b$ gives 
\begin{align*}
    \log \mathcal{N}(\omega,\mathbb{T},d) &\le \log \mathcal{N}\left(\sqrt{n}\omega/\sigma,\mathcal{G}_{m,s},\|\cdot\|_\infty\right) \\
    &\le c_1 \Bigg\{ \left(N^2L + N\overline{r} \right) \left[L \log BN + \log \left(\frac{\sigma (M\lor (r+1) b\|\bW\|_{\max})}{\omega\sqrt{n}} \lor 1\right)\right] \\
    &    ~~~~~~~~~~ + s \left[L \log (BN) + \log p + \log\left(\frac{\sigma (r+1)b (N+\overline{r})}{\omega \sqrt{n}} \lor 1\right)\right] \Bigg\},
\end{align*} provided $\omega \ge 2\tau (r+1)b (BN)^{L+1} (N+\overline{r}) p$. 

Set $\epsilon = 8\sigma n^{-1.5}$. If $\tau$ satisfies $2\tau (r+1)b (BN)^{L+1} (N+\overline{r}) p \le \epsilon/4\cdot \sqrt{n}/\sigma$, i.e.,
\begin{align*}
    \tau^{-1} \ge 4(r+1)b (BN)^{L+1} (N+\overline{r}) p n,
\end{align*} then 
\begin{align}
\label{eq:dudley-integral}
    \int_{\epsilon/4}^{2\sigma n^{-1/2}\delta} \sqrt{\log \mathcal{N}(\omega,\mathbb{T},d)} \mathrm{d}\omega &\lesssim \delta \left( \sqrt{v_n} +  \sqrt{s \varrho_n}\right)
\end{align} provided $2\sigma\delta n^{-1/2}> \epsilon/4$.

To calculate $\alpha(\epsilon, u)$, it follows from the Cauchy-Schwarz inequality that
\begin{align*}
    \sup_{\theta, \tilde{\theta}\in \mathbb{T}, d(\theta,\tilde{\theta})\le \epsilon} |Z_\theta - Z_{\tilde{\theta}}| \le \sqrt{\frac{1}{n}\sum_{i=1}^n \varepsilon_i^2} \sqrt{\frac{1}{n} \sum_{i=1}^n (\theta(\bx_i) - \tilde{\theta}(\bx_i))^2} \le \epsilon \sqrt{n} \sqrt{\frac{1}{n} \sum_{i=1}^n\varepsilon_i^2}.
\end{align*} Because $\varepsilon$ is a sub-Gaussian random variable, $\varepsilon^2$ is a sub-exponential random variable, this implies the following holds
\begin{align*}
    \frac{1}{n} \sum_{i=1}^n\varepsilon_i^2 \lesssim 1 + \sqrt{\frac{u}{n}} + \frac{u}{n}.
\end{align*} with probability at least $1-e^{-u}$ for any $u>0$. Plugging the choice of $\epsilon$ yields
\begin{align}
\label{eq:small-local-tail}
    \alpha(8\sigma n^{-1.5}, u) \lesssim n^{-1/2} n^{-1/2} \sqrt{\frac{1}{n} \sum_{i=1}^n\varepsilon_i^2} \lesssim \frac{1}{n} + \frac{1}{n} \left( 1 + \sqrt{\frac{u}{n}} + \frac{u}{n} \right) \lesssim \frac{1}{n} + \frac{u}{n}.
\end{align} 

Moreover, for any fixed $\bar{\theta} \in \mathbb{T}$, standard sub-Gaussian concentration inequality gives
\begin{align}
\label{eq:step4.1-single-concentration}
    \mathbb{P}\left[ |Z_{\bar{\theta}}| \ge \sigma \sqrt{\frac{u}{n}} \delta\right] \le e^{-u},
\end{align} for all the $u > 0$.

Combining \eqref{eq:step4.1-single-concentration} and \eqref{eq:general-bound-step4.1} with the bound \eqref{eq:small-local-tail} and \eqref{eq:dudley-integral}, we have
\begin{align*}
    Z_{s,\delta} = \sup_{\theta \in \mathbb{T}} \left|Z_{\theta} \right| &\le \sup_{\theta,\tilde{\theta} \in \mathbb{T}} \left|Z_{\theta} - Z_{\tilde{\theta}}\right| + |Z_{\bar{\theta}}| \\
    &\lesssim \frac{u+1}{n} + \delta \left( \sqrt{\frac{u + 1}{n}} + \sqrt{v_n} + \sqrt{s \varrho_n} \right),
\end{align*} with probability at least $1-3e^{-u}$, which completes the proof of the claim \eqref{eq:claim-step4.1}.

\noindent {\sc Step 2. Application of Peeling. } We will peel the function class $\mathcal{G}_m$ twice. To be specific, define
\begin{align*}
    \mathcal{G}_{m,k,\ell} = \left\{m\in \mathcal{G}_m, \alpha_{k-1} \le \sum_{i,j} \psi_\tau(\Theta_{i,j}) \le \alpha_k, \alpha_{\ell-1} \sqrt{v_n} < \|m - \tilde{m}\|_n \le \alpha_{\ell} \sqrt{v_n} \right\}
\end{align*} with $\alpha_i = 2^i$ for $i\ge 0$ and $\alpha_{i} = 0$ for $i=-1$. Then it's obvious that
\begin{align*}
    \mathcal{G}_m = \bigcup_{k=0}^{\lceil \log_2 (pN)\rceil} \bigcup_{\ell=0}^{\lceil \log_2 (2M/\sqrt{v_n})\rceil} \mathcal{G}_{m,k,\ell}.
\end{align*} 

We first establish the bound on the probability of the event
\begin{align*}
    \mathcal{B}_{t,\epsilon}(k,\ell) = \left\{\forall m\in \mathcal{G}_{m,k,\ell}, \frac{4}{n} \sum_{i=1}^n \varepsilon_i (m(\bx_i) - \tilde{m}(\bx_i)) - \lambda \sum_{i,j} \psi_\tau({\Theta}_{i,j}) \le \epsilon \|m - \tilde{m}\|_n^2 + \frac{C_2}{\epsilon} \left(v_n + \varrho_n + \frac{t}{n} \right) \right\}
\end{align*} from below for some constant $C_2$.

We will use the following simple fact that
\begin{align}
\label{eq:relationship-peeling}
    \alpha_{i} \le 2\alpha_{i-1} + 1 ~~ \text{and} ~~ \alpha_{i}^2 \le 4\alpha_{i-1}^2 + 1\qquad \forall i=-1, 0, 1, \cdots.
\end{align}

Note $m\in \mathcal{G}_{m,k,\ell}$ implies $\|m-\tilde{m}\|_n > \alpha_{\ell-1} \sqrt{v_n}$, and $\sum_{i,j} \psi_\tau({\Theta}_{i,j}) \ge \alpha_{k-1}$ by the definition of the set $\mathcal{G}_{m,k,\ell}$. If the event $\mathcal{B}_{t}(\alpha_k, \alpha_\ell \sqrt{v_n})$ occurs, then for all the $m\in \mathcal{G}_{m,k,\ell}$, 
\begin{align*}
    &\frac{4}{n} \sum_{i=1}^n \varepsilon_i (m(\bx_i) - \tilde{m}(\bx_i)) - \lambda \sum_{i,j} \psi_\tau({\Theta}_{i,j}) \\
    \le& 4\sup_{m\in \mathcal{G}_{m,\alpha_k}, \|m-\tilde{m}\|_n  
    \le \alpha_\ell \sqrt{v_n}} \left|\frac{1}{n} \sum_{i=1}^n \varepsilon_i (m(\bx_i) - \tilde{m}(\bx_i)) \right| - \lambda \alpha_{k-1} \\
    \overset{(a)}{\le}& \frac{t+1}{n} + 4C_1 \alpha_\ell \sqrt{v_n} \left(\sqrt{v_n} + \sqrt{\alpha_k \varrho_n} + \sqrt{\frac{t+1}{n}}\right) - \lambda \alpha_{k-1} \\
    \overset{(b)}{\le}& \frac{t+1}{n} + \frac{\epsilon}{4} (\alpha_\ell^2 v_n) + \frac{64C_1^2}{\epsilon} \left( v_n + \alpha_k \varrho_n + \frac{t+1}{n} \right) - \lambda \alpha_{k-1} \\
    \overset{(c)}{\le}& \frac{t+1}{n} + \frac{\epsilon}{4}(4 \alpha_{\ell-1}^2 + 1) v_n + \frac{64C_1^2}{\epsilon} (v_n + \varrho_n  + \frac{t+1}{n}) + \left(\frac{128C_1^2}{\epsilon}\varrho_n -\lambda\right) \alpha_{k-1} \\
    \overset{(d)}{\le}& \epsilon \|m-\tilde{m}\|_n^2 + \frac{C_2}{\epsilon} \left(v_n + \varrho_n + \frac{t}{n}\right) + \left(\frac{128C_1^2}{\epsilon}\varrho_n -\lambda\right) \alpha_{k-1}.
\end{align*} Here (a) follows from the definition of the event $\mathcal{B}_{t}(\alpha_k, \alpha_\ell \sqrt{v_n})$, (b) follows from Young's inequality, (c) uses the relationship \eqref{eq:relationship-peeling} between $\alpha_{i-1}$ and $\alpha_i$, (d) follows from the fact that $\|m-\tilde{m}\|_n > \alpha_{\ell-1} \sqrt{v_n}$ for $m\in \mathcal{G}_{m,k,\ell}$. Hence we have
\begin{align*}
    \mathbb{P}\left[\big(\mathcal{B}_{t,\epsilon}(k,\ell)\big)^c\right] \le \mathbb{P}\left[ \big(\mathcal{B}_t(\alpha_k, \alpha_\ell \sqrt{v_n})\big)^c\right]\le e^{-t}. 
\end{align*} provided $\lambda \ge \frac{128C_1^2}{\epsilon} \varrho_n$.

It follows from the union bound that
\begin{align*}
    &\mathbb{P}\left[\exists m\in \mathcal{G}_m, \frac{4}{n} \sum_{i=1}^n \varepsilon_i (m(\bx_i) - \tilde{m}(\bx_i)) - \lambda \sum_{i,j} \psi_\tau({\Theta}_{i,j}) > \epsilon \|m - \tilde{m}\|_n^2 + \frac{C_2}{\epsilon} \left(v_n + \varrho_n + \frac{u}{n} \right) \right]\\
    \le &\sum_{k=0}^{\lceil \log_2 (pN)\rceil} \sum_{\ell=0}^{\lceil \log_2 (2M/\sqrt{v_n})\rceil} \mathbb{P}\left[\big(\mathcal{B}_{t,\epsilon}(k,\ell)\big)^c\right] \le C_3 \log(pN) \log (2Mn) e^{-u},  
\end{align*} for any $u>0$. Letting $u=t+\log \left(C_3 \log(pN) \log (2Mn)\right)$ and $c_1=C_2(1+C_4)$ for some large $C_4$ satisfying $\log \left(C_3 \log(pN) \log (2Mn)\right) - C_4 n\varrho_n \le 0$, we have
\begin{align*}
    \mathbb{P}\left[\forall m\in \mathcal{G}_m, \frac{4}{n} \sum_{i=1}^n \varepsilon_i (m(\bx_i) - \tilde{m}(\bx_i)) - \lambda \sum_{i,j} \psi_\tau({\Theta}_{i,j}) \le \epsilon \|m - \tilde{m}\|_n^2 + \frac{c_1}{\epsilon} \left(v_n + \varrho_n + \frac{t}{n} \right) \right] \ge 1- e^{-t}
\end{align*} for arbitrary fixed $\bx_1,\ldots, \bx_n$.
This completes the proof of our main argument by the law of total probability.

\end{proof}

\subsection{Proof of Theorem \ref{thm:fast-oracle}}

\begin{proof}[Proof of Theorem~\ref{thm:fast-oracle}] 

\noindent {\sc Step 1. Find An Approximation of $m^*$. } The goal of this step is to find some $\tilde{m}(\bx;\bW,\tilde{\bTheta},\tilde{g})\in \mathcal{G}_m$ with $\|\tilde{\bTheta}\|_0 \le |\mathcal{J}|$ such that
\begin{align*}
    \int |\tilde{m}(\bx) - m^*(\bbf, \bu_{\mathcal{J}})|^2\mu(d\bbf, d\bu) \lesssim \delta_{\mathtt{f}} + \delta_{\mathtt{a}}.
\end{align*} Note that $\delta_{\mathtt{f}}=0$ when $r=0$, so we divide it into two cases.

\noindent \emph{Case 1. $r\ge 1$.} We first consider the case where $r\ge 1$. Let \begin{align*}
    \tilde{m}^*(\bx) = m^*(\bH^+\tilde{\bbf}, \bx_{\mathcal{J}} - [\bB]_{\mathcal{J}, :} \bH^+ \tilde{\bbf}).
\end{align*} 

We first show $\mathbb{E}|\tilde{m}^*(\bx)-m^*(\bbf,\bu_\mathcal{J})|^2 \lesssim \delta_{\mathtt{f}}$. Denote
\begin{align*}
    \bxi = p^{-1} \bW^\top \bu,
\end{align*} For fixed $\bbf, \bu$, it follows from the decomposition of the estimated factor $\tilde{\bbf}$ in \eqref{eq:est-factor-decomposition} that
\begin{align*}
    \tilde{m}^*(\bx) = m^*\left(\bbf + \bH^+ \bxi, \bx_{\mathcal{J}} - [\bB]_{\mathcal{J},:} \bbf - [\bB]_{\mathcal{J},:} \bH^+ \bxi\right),
\end{align*} 
which implies
\begin{align*}
    |\tilde{m}^*(\bx) - m^*(\bbf,\bu_{\mathcal{J}})| \lesssim \|\bH^+ \bxi\|_2 + \|[\bB]_{\mathcal{J},:} \bH^+ \bxi\| \le \left(\|[\bB]_{\mathcal{J},:}\|_2 + 1\right) \|\bH^+\|_2 \|\bxi\|_2,
\end{align*} by the Lipschitz condition of $m^*$ in Condition \ref{cond5}.

The boundedness of $\|\bB\|_{\max}$ in Condition \ref{cond1} gives $\|[\bB]_{\mathcal{J},:}\|_2 \le \|[\bB]_{\mathcal{J},:}\|_F \lesssim \sqrt{|\mathcal{J}| \times r}$. As the by-products of Lemma \ref{lemma:approx-factor}, we have
\begin{align*}
    \|\bH^+\|_2 \le [\nu_{\min}(\bH)]^{-1} ~~~~ \text{and} ~~~~ \mathbb{E}[\|\bxi\|_2^2] \lesssim \overline{r}/p
\end{align*} as long as Condition \ref{cond3} holds. Putting these pieces together yields
\begin{align}
\label{eq:proof:fast-oracle:approxeq1}
    \mathbb{E}|\tilde{m}^*(\bx) - m^*(\bbf, \bu_{\mathcal{J}})|^2 \lesssim \frac{(|\mathcal{J}|r)\times \overline{r}}{(\nu_{\min}(\bH))^2p}.
\end{align}

So it remains to find some $\tilde{m} \in \mathcal{G}_{m}$ to approximate $\tilde{m}^*$ well. It follows from the definition of $\delta_{\mathtt{a}}$ that there exists a ReLU neural network $h \in \mathcal{G}(L - 1, r+|\mathcal{J}|, N, M, 1, B)$ such that 
\begin{align}
\label{eq:proof:thm:fast-oracle:approx-assump}
    |h(\bbf, \bu_{\mathcal{J}}) - m^*(\bbf, \bu_{\mathcal{J}})| \le \sqrt{2\delta_{\mathtt{a}}} \qquad \forall (\bbf^\top, \bu_{\mathcal{J}}^\top)^\top \in [-2b, 2b]^{r+|\mathcal{J}|}.
\end{align} 

According to the definition of ReLU networks (Definition \ref{def:nn}), we can write ${h}$ as
\begin{align*}
    h(\bbf, \bu_{\mathcal{J}}) = {\mathcal{L}}^h_{L+1} \circ \bar{\sigma} \circ {\mathcal{L}}^h_{L} \circ \bar{\sigma} \circ \cdots \circ \bar{\sigma} \circ {\mathcal{L}}^h_3 \circ \bar{\sigma} \circ {\mathcal{L}}^h_2 ((\bbf^\top, \bu_{\mathcal{J}}^\top)^\top)
\end{align*} where ${\mathcal{L}}^h_2 \in \mathbb{R}^{r+|J|} \to \mathbb{R}^N$, ${\mathcal{L}}^h_{L+1} \in \mathbb{R}^N \to \mathbb{R}$, ${\mathcal{L}}^h_{\ell}: \mathbb{R}^N \to \mathbb{R}^N$ with $\ell \in \{3,\cdots, L\}$ are all affine transformations. Denote $\mathcal{J}=\{l_1,\ldots, l_{|\mathcal{J}|}\} \subset \{1,\ldots, p\}$. Consider the new function
\begin{align*}
    \tilde{m}(\bx) = \tilde{\mathcal{L}}_{L+1} \circ \bar{\sigma} \circ \tilde{\mathcal{L}}_{L} \circ \bar{\sigma} \circ \cdots \circ \bar{\sigma} \circ \tilde{\mathcal{L}}_3 \circ \bar{\sigma} \circ \tilde{\mathcal{L}}_2 \circ \bar{\sigma} \circ \tilde{\mathcal{L}}_1 \circ \phi \circ \tilde{\mathcal{L}}_0 (\bx).
\end{align*} Here for $\tilde{\mathcal{L}}_{\ell}$ with $\ell \in \{3, \ldots, L+1\}$, we directly copy weights from $h$, i.e., $\tilde{\mathcal{L}}_{\ell} = {\mathcal{L}}_{\ell}^h$  when $\ell \in \{3, \ldots, L+1\}$. Moreover, we construct $\tilde{\mathcal{L}}_0$, $\tilde{\mathcal{L}}_1$ and $\tilde{\mathcal{L}}_2$ as follows:

(1) For $\tilde{\mathcal{L}}_0 \in \mathbb{R}^{p} \to \mathbb{R}^{\overline{r} + |J|}$, 
\begin{align*}
    \tilde{\mathcal{L}}_0(x) = (p^{-1} \bx^\top \bW, \bx^\top \tilde{\bTheta})^\top ~~~~ \text{with} ~~~~ \tilde{\Theta}_{i,j} = 1\{i\le |\mathcal{J}|, j=\ell_i\}. 
\end{align*} 

Notice that $\tilde{\mathcal{L}}_0(\bx) = (p^{-1} \bx^\top \bW, \bx_{\mathcal{J}}^\top)^\top = (\tilde{\bbf}^\top, \bx_{\mathcal{J}}^\top)^\top$. So we further have $\phi\circ \tilde{\mathcal{L}}_0(\bx) = (\tilde{\bbf}^\top, \bx_{\mathcal{J}}^\top)^\top$ provided $M \ge r(b+1) \ge \|\bx_{\mathcal{J}}\|_\infty$.

(2) For $\tilde{\mathcal{L}}_1 \in \mathbb{R}^{\overline{r}+|J|} \to \mathbb{R}^{2(r+|J|)}$, let
\begin{align*}
    \tilde{\mathcal{L}}_1 \begin{bmatrix}
        \tilde{\bbf} \\
        \bx_{\mathcal{J}}
    \end{bmatrix} = \underbrace{\begin{bmatrix}
        \bH^+ & 0\\
        -[\bB]_{J,:} \bH^+ & \bI \\
        -\bH^+ & 0\\
        [\bB]_{\mathcal{J},:} \bH^+ & -\bI
    \end{bmatrix}}_{\tilde{\bW}_1} \begin{bmatrix}
        \tilde{\bbf} \\
        \bx_{\mathcal{J}}
    \end{bmatrix} + \underbrace{0}_{\tilde{\bb}_1}= \begin{bmatrix}
        \bH^+ \tilde{\bbf} \\
        \bx_{\mathcal{J}} - [\bB]_{\mathcal{J},:}\bH^+ \tilde{\bbf} \\
        -\bH^+ \tilde{\bbf} \\
        -(\bx_{\mathcal{J}} - [\bB]_{\mathcal{J},:}\bH^+ \tilde{\bbf})
    \end{bmatrix}
\end{align*}

(3) For $\tilde{\mathcal{L}}_2 \in \mathbb{R}^{2(r+|J|)} \to \mathbb{R}^N$, given $\bu,\bv\in \mathbb{R}^{r+|J|}$, let
\begin{align*}
    \tilde{\mathcal{L}}_2 \begin{bmatrix}
        \bu \\
        \bv
    \end{bmatrix} = \underbrace{\begin{bmatrix}
        \bW^h_2 & -\bW^h_2
    \end{bmatrix}}_{\tilde{\bW}_2} \begin{bmatrix}
        \bu \\
        \bv
    \end{bmatrix} + \underbrace{\bb^h_2}_{\tilde{\bb}_2}.
\end{align*}

It follows from the above construction of $\tilde{\mathcal{L}}_1, \tilde{\mathcal{L}}_2$ that
\begin{align*}
    \tilde{m}(\bx) &= h\left(\sigma(\bH^+\tilde{\bbf}) - \sigma(-\bH^+\tilde{\bbf}), \sigma(\bx_{\mathcal{J}} - [\bB]_{\mathcal{J},:}\bH^+ \tilde{\bbf}) - \sigma(-(\bx_{\mathcal{J}} - [\bB]_{\mathcal{J},:}\bH^+ \tilde{\bbf}))\right)\\
                 &\overset{(a)}{=} h(\bH^+\tilde{\bbf}, \bx_{\mathcal{J}} - [\bB]_{\mathcal{J},:}\bH^+ \tilde{\bbf})
\end{align*} where (a) follows from the fact that $\sigma(x) - \sigma(-x) = x$. Moreover, all the weights of the ReLU network
\begin{align}
\label{eq:proof:fast-oracle:tilde-g-def}
    \tilde{g} = \tilde{\mathcal{L}}_{L+1} \circ \bar{\sigma} \circ \tilde{\mathcal{L}}_{L} \circ \bar{\sigma} \circ \cdots \circ \bar{\sigma} \circ \tilde{\mathcal{L}}_3 \circ \bar{\sigma} \circ \tilde{\mathcal{L}}_2 \circ \bar{\sigma} \circ \tilde{\mathcal{L}}_1
\end{align} is bounded by $B\lor C_1 (|J|r [\nu_{\min}(H)]^{-1})$ for some constant $C_1$ since $\|\cdot\|_{\max}\le \|\cdot\|_2$.

We apply the truncation argument to derive the bound on $\mathbb{E} \left|\tilde{m}(\bx)-\tilde{m}^*(\bx)\right|^2$. To be specific, define the event
\begin{align*}
    \mathcal{E} = \left\{\bH^+\tilde{\bbf} \in [-2b, 2b]^r, \bx_{\mathcal{J}} - [\bB]_{\mathcal{J},:}\bH^+ \tilde{\bbf} \in [-2b, 2b]^{|\mathcal{J}|}\right\},
\end{align*} thus
\begin{align*}
    \mathbb{E} \left|\tilde{m}(\bx)-\tilde{m}^*(\bx)\right|^2 = \mathbb{E} \left|\tilde{m}(\bx)-\tilde{m}^*(\bx)\right|^2 1_{\mathcal{E}} + \mathbb{E} \left|\tilde{m}(\bx)-\tilde{m}^*(\bx)\right|^2 1_{\mathcal{E}^c}.
\end{align*}

On one hand, it follows from the approximation assumption \eqref{eq:proof:thm:fast-oracle:approx-assump}, the definition of $m$, $h$, $\tilde{m}$ and the condition of event $\mathcal{E}$ that, 
\begin{align*}
    \mathbb{E} \left|\tilde{m}(\bx)-\tilde{m}^*(\bx)\right|^2 1_{\mathcal{E}} = \mathbb{E}  \left|h(\bH^+\tilde{\bbf}, \bx_{\mathcal{J}} - [\bB]_{\mathcal{J},:}\bH^+ \tilde{\bbf})-{m}^*(\bH^+\tilde{\bbf}, \bx_{\mathcal{J}} - [\bB]_{\mathcal{J},:}\bH^+ \tilde{\bbf})\right|^2 1_{\mathcal{E}} \le 2 \delta_{\mathtt{a}}
\end{align*}

On the other hand, note $h$ and $m^*$ are all bounded by constants $M$ and $M^*$ respectively, and each component of $\bbf, \bu$ is bounded by $b$. Then it follows from the Markov inequality that
\begin{align*}
    \mathbb{E} \left|\tilde{m}(\bx)-\tilde{m}^*(\bx)\right|^2 1_{\mathcal{E}^c} &\le (M+M^*)^2 \mathbb{P}\left[\mathcal{E}^c\right] \\
    &\lesssim \mathbb{P}\left[\sqrt{\left\|\bH^+ \bxi\right\|_2^2 + \left\|[\bB]_{\mathcal{J},:} \bH^+ \bxi\right\|_2^2} \ge b\right] \\
    &\lesssim \frac{1}{b^2} \mathbb{E} \left[\left\|\bH^+ \bxi\right\|_2^2 + \left\|[\bB]_{\mathcal{J},:} \bH^+ \bxi\right\|_2^2 \right] \lesssim \delta_{\mathtt{f}}.
\end{align*} Putting these pieces together, our constructed $\tilde{m} \in \mathcal{G}_m$ satisfies
\begin{align}
\label{eq:proof:fast-oracle:approxeq2}
    \mathbb{E} \left|\tilde{m}(\bx) - \tilde{m}^*(\bx) \right|^2 \lesssim \delta_{\mathtt{f}} + \delta_{\mathtt{a}}
\end{align} 

Therefore, combining \eqref{eq:proof:fast-oracle:approxeq1} and \eqref{eq:proof:fast-oracle:approxeq2} with triangle inequality and Young's inequality, our constructed $\tilde{m}$ satisfies
\begin{align}
\label{eq:proof:fast-oracle:approx}
\begin{split}
    \|\tilde{m} - m^*\|_2^2=\mathbb{E}|\tilde{m}(\bx)-m^*(\bbf,\bu_{\mathcal{J}})|^2 &\le 2\Big(\mathbb{E}|\tilde{m}(\bx)-\tilde{m}^*(\bx)|^2 + \mathbb{E}|\tilde{m}^*(\bx)-m^*(\bbf,\bu_{\mathcal{J}})|^2\Big) \\
    &\lesssim \delta_{\mathtt{f}} + \delta_{\mathtt{a}}
\end{split}
\end{align} 

Using the padding argument in Section \ref{sec:approx:notation}, when $2(|\mathcal{J}|+r)\le N$, and $B\ge C_1\{|\mathcal{J}|r [\nu_{\min}(\bH)]^{-1}\}$, the ReLU network $\tilde{g}$ in \eqref{eq:proof:fast-oracle:tilde-g-def} satisfies $\tilde{g} \in \mathcal{G}(L, \overline{r}+N, 1, N, M, B)$. Moreover, the variable selection matrix $\tilde{\bTheta}$ satisfies $\|\tilde{\bTheta}\|_0 = |\mathcal{J}|$. This completes the proof of Case 1.

\noindent \emph{Case 2. $r=0$. } In this case, we have $\delta_{\mathtt{f}}=0$ and $\bx=\bu$. It follows from the approximation result that there exists a ReLU network $h\in \mathcal{G}(L-1,|\mathcal{J}|, N, M, 1, B)$ such that
\begin{align*}
    |h(\bx_\mathcal{J}) - m^*(\bx_{\mathcal{J}})| \le \sqrt{2\delta_{\mathtt{a}}} ~~~~~~ \forall \bx_{\mathcal{J}} \in [-2b,2b]^{|\mathcal{J}|}.
\end{align*}

The construction of $\tilde{m}$ proceeds in a similar way. Let $\mathcal{J}=\{l_1,\ldots, l_{|\mathcal{J}|}\}$. If $|\mathcal{J}| \le N$, it follows from the padding argument that there exists a ReLU network $\tilde{g} \in \mathcal{G}(L, \overline{r} + N, 1, N, M, B)$ such that
\begin{align*}
    \tilde{g}(\tilde{\bbf}, \tilde{\bTheta}^\top \bx) = h(\bx_{\mathcal{J}}) ~~~~ \text{with} ~~~~ \tilde{\Theta}_{i,j} = 1\{i\le |\mathcal{J}|, j=l_i\}.
\end{align*} Therefore, we have $\tilde{m}(\bx;\bW,\tilde{\bTheta},\tilde{g})$ satisfies $\tilde{m}(\bx) = h(\bx_{\mathcal{J}})$, thus implies
\begin{align*}
    \int |\tilde{m}(\bx) - m^*(\bx_{\mathcal{J}})|^2 \mu(d\bu) \le 2\delta_{\mathtt{a}},
\end{align*} which completes the proof of Case 2.

\noindent {\sc Step 2. Derive Basic Inequality. } It follows from \eqref{eq:fast-nn-estimator} and our construction of $\tilde{g}$ in {\sc Step 1} that 
\begin{align*}
    \frac{1}{n} \sum_{i=1}^n \left(y_i - \hat{m}(\bx_i) \right)^2 + \lambda \sum_{i,j} \psi_\tau(\hat{\Theta}_{i,j}) \le \frac{1}{n} \sum_{i=1}^n \left(y_i - \tilde{m}(\bx_i) \right)^2 + \lambda |J| + \delta_{\mathtt{opt}}.
\end{align*} 
Plugging in the representation of $y_i$ in \eqref{eq:dgp-samples}, we find
\begin{align*}
    \|\hat{m}-m^*\|_n^2 + \lambda \sum_{i,j} \psi_\tau(\hat{\Theta}_{i,j}) \le \|\tilde{m}-m^*\|_n^2 + \frac{2}{n} \sum_{i=1}^n \varepsilon_i \left(\hat{m}(\bx_i) - \tilde{m}(\bx_i)\right) + \lambda |J| + \delta_{\mathtt{opt}}.
\end{align*} It follows from the triangle inequality and Young's inequality that
\begin{align*}
    \|\hat{m} - \tilde{m}\|_n^2 \le 2\|\hat{m} - m^*\|_n^2 + 2\|\tilde{m} - m^*\|_n^2.
\end{align*} Therefore, we have
\begin{align}
\label{eq:fast:proof:basic-ineq}
    \|\hat{m}-\tilde{m}\|_n^2 + 2\lambda \sum_{i,j} \psi_\tau(\hat{\Theta}_{i,j}) \le 4\|\tilde{m}-m^*\|_n^2 + \frac{4}{n} \sum_{i=1}^n \varepsilon_i \left(\hat{m}(\bx_i) - \tilde{m}(\bx_i)\right) + 2\lambda |J| + 2\delta_{\mathtt{opt}},
\end{align} which we refer to as the basic inequality.

\noindent {\sc Step 3. Concentration for Fixed Function. } Let $z_i = |\tilde{m}(\bx_i)-m^*(\bbf_i,\bu_{i,\mathcal{J}})|$. It is easy to see that $z_i$ are i.i.d. samples satisfying 
\begin{align*}
    |z_i| \le (M+M^*)^2 \qquad \text{and} \qquad \mathrm{Var}(z_i) \le \mathbb{E}[|z_i|^2] \le (M+M^*)^2 \|\tilde{m} - m^*\|_2^2
\end{align*} provided that $\bW$ is independent of $\bx_1,\ldots, \bx_n$ and thus $\tilde{m}(\bx_1),\cdots, \tilde{m}(\bx_n)$ are independent. It follows from Bernstein inequality and the conclusion \eqref{eq:proof:fast-oracle:approx} in {\sc Step 1} that
\begin{align*}
    \|\tilde{m} - m^*\|_n^2 = \frac{1}{n} \sum_{i=1}^n z_i &\le \frac{1}{n} \sum_{i=1}^n \mathbb{E}[z_i] + C_2 \left((M+M^*) \|\tilde{m} - m^*\|_2 \sqrt{\frac{u}{n}} + \frac{u}{n}\right)\\
    &\le C_3 \left(\delta_{\mathtt{f}} + \delta_{\mathtt{a}} + \frac{u}{n}\right)
\end{align*} for any $u>0$.
Define the event \begin{align}
\label{eq:fast-oracle-ineq:approx-l2-error}
    \mathcal{A}_t = \left\{\|\tilde{m} - m^*\|_n^2 \le C_3 \left(\delta_{\mathtt{f}} + \delta_{\mathtt{a}} + \frac{t}{n}\right)\right\},
\end{align} then $\mathbb{P}(\mathcal{A}_t) \ge 1-e^{-t}$.

\noindent {\sc Step 4. Conclusion of the Proof.} 
Define the event
\begin{align*}
    \mathcal{B}_{t,1/2} = \left\{\forall m\in \mathcal{G}_m, \frac{4}{n} \sum_{i=1}^n \varepsilon_i (m(\bx_i) - \tilde{m}(\bx_i)) - \lambda \sum_{i,j} \psi_\tau({\Theta}_{i,j}) \le \frac{1}{2} \|m - \tilde{m}\|_n^2 + C_4 \left(v_n + \varrho_n + \frac{t}{n} \right)\right\}
\end{align*} with $C_4=2c_1$, where $c_1$ is the universal constant in Lemma \ref{lemma:weighted-empirical-process-regularized}. It follows from Lemma \ref{lemma:weighted-empirical-process-regularized} that $\mathbb{P}(\mathcal{B}_{t,1/2})\ge 1-e^{-t}$ provided $\lambda \ge C_5 \varrho_n$ for some universal constant $C_5$.

Recall the basic inequality \eqref{eq:fast:proof:basic-ineq}, under the events $\mathcal{A}_t$ and $\mathcal{B}_{t,1/2}$, we have
\begin{align*}
    \|\hat{m} - \tilde{m}\|_n^2 + \lambda \sum_{i,j} \psi_\tau(\hat{\Theta}_{i,j}) \le 4C_3(\delta_{\mathtt{f}} + \delta_{\mathtt{a}} + \frac{t}{n}) + \frac{1}{2} \|\hat{m} - \tilde{m}\|_n^2 + C_4 \left(v_n + \varrho_n + \frac{t}{n}\right) + 2\lambda|J| + 2\delta_{\mathtt{opt}}
\end{align*} provided $\lambda \ge C_5 \varrho_n$, which implies
\begin{align}
\label{eq:fast:basic-ineq-final}
    \|\hat{m} - \tilde{m}\|_n^2 + 2\lambda \sum_{i,j} \psi_\tau(\hat{\Theta}_{i,j}) \le C_6 \left(\delta_{\mathtt{f}} +\delta_{\mathtt{a}} + v_n + \lambda |J| + \delta_{\mathtt{opt}} + \frac{t}{n} \right).
\end{align}

Combining \eqref{eq:fast:basic-ineq-final} and \eqref{eq:fast-oracle-ineq:approx-l2-error} yields an upper bound on the empirical $L_2$ error $\|\hat{m} - m^*\|_n^2$. So it suffices to derive an upper bound on the population $L_2$ error $\|\hat{m}-m^*\|_2^2$. To this end, define the event \begin{align*}
    \mathcal{C}_t = \left\{\forall m\in \mathcal{G}_m, ~~ \frac{1}{2}\|m - \tilde{m}\|^2_2 \le \|m - \tilde{m}\|_n^2 + 2\lambda \sum_{i,j} \psi_\tau({\Theta}_{i,j}) + C_{7} \left(v_n + \rho_n + \frac{t}{n}\right)\right\}
\end{align*} for some constant $C_{7}$. Applying Lemma \ref{lemma:equivalence-population-empirical-l2-regularized} yields that $\mathbb{P}\left[\mathcal{C}_{t}\right] \ge 1-e^{-t}$ as long as $\lambda \ge C_{8}\varrho_n$ for some universal constant $C_{8}$. 

Putting these pieces together, we have
\begin{align*}
    \|\hat{m} - \tilde{m}\|^2_2 \lesssim \delta_{\mathtt{f}} +\delta_{\mathtt{a}} + v_n + \lambda |J| + \delta_{\mathtt{opt}} + \frac{t}{n}
\end{align*} conditioned on the event $\mathcal{A}_t \cap \mathcal{B}_t \cap \mathcal{C}_t$, which occurs with probability at least $1-3e^{-t}$.

Combing with our approximation result in {\sc Step 1} that $\|\tilde{m} - m^*\|_2^2 \lesssim \delta_{\mathtt{f}} +\delta_{\mathtt{a}}$, we can conclude that
\begin{align*}
    \|\hat{m} - m^*\|_2^2 \lesssim \delta_{\mathtt{f}} +\delta_{\mathtt{a}} + v_n + \lambda |J| + \delta_{\mathtt{opt}} + \frac{t}{n}
\end{align*} provided $\lambda \ge (C_{5} \lor C_8) \varrho_n$, which completes the proof.

\end{proof}

\subsection{Proof of Theorem \ref{thm:fast-roc}}
\begin{proof}[Proof of Theorem~\ref{thm:fast-roc}]
It follows from Theorem \ref{thm:fast-oracle} and Theorem \ref{thm:approx-smooth} by plugging in our choice of ReLU network hyper-parameter $N$ and $L$.
\end{proof}
\section{Proof of the lower bound result in Section \ref{sec:theory:lb}}
\label{sec:proof-lb}

\subsection{Proof of Lemma \ref{lemma:error-factor-lb}}
\label{proof-lemma-factor-lb}

We first provide a sketch to gain some intuition why the lower bound scales linearly with the inverse of the minimum eigenvalue of $\bB^\top \bB$. For simplicity, let $r=1$, $m^*(\bbf) = f_1$, and $\bB$ be a matrix with all entries being equal to $1$. Then, $\lambda_{\min}(\bB^\top \bB)=p$. By using the property of conditional expectation,
\begin{align*}
    \int |m(\bx) - m^*(\bbf)|^2 \mu(d\bbf, d\bu) \ge \int |\mathbb{E}[f_1|\bx] - f_1|^2 \mu(df_1, d\bx) = \inf_{\theta(\bx)} \mathbb{E}_{x_1,\ldots, x_p, f_1} \left[|\theta(\bx) - f_1|^2\right],
\end{align*} where the last inequality follows from the variational representation of conditional expectation that $\mathbb{E}[f_1|\bx] = \argmin_h \mathbb{E} |h(\bx)-f_1|^2$. If we further let each component of $\bu$ be identically distributed zero-mean random variable, then we can relate the lower bound on $\int |m(\bx) - m^*(\bbf)|^2 \mu(d\bbf, d\bu)$ to the \emph{minimax optimal lower bound} on estimating the mean parameter of a uniformly bounded distribution family since $f_1$ and $u_1$ are supposed to be bounded in $[-1,1]$ and $x_j=f_1 + u_j$ are i.i.d. samples conditioned on $f_1$. To further obtain a lower bound on $\inf_{\theta(\bx)} \mathbb{E}_{x_1,\ldots, x_p, f_1} \left[|\theta(\bx) - f_1|^2\right]$, we let $f_1$ be a random variable taking values in $\{-\delta, \delta\}$ with equal probability, and $u_1,\ldots, u_p$ be i.i.d. random variables with probabilistic density distribution $p(u) \propto   (1-|u|)_+^2$. By some calculations, we can bound the total variation distance between $\mu(d\bx|f_1=\delta)$ and $\mu(d\bx|f_1=-\delta)$ by $C\sqrt{p} \delta$ for some constant $C$. Combining this with a similar argument to Le Cam's (two points) method \citep{lecam1973convergence}, we can conclude that
\begin{align*}
    \inf_{\hat \theta(\bx)} \mathbb{E}_{x_1,\ldots, x_p, f_1} \left[  |\hat \theta - f_1|^2\right] \ge \sup_{\delta} \frac{\delta^2}{2} \Big(1 - \|\mu(d\bx|f_1=\delta) - \mu(d\bx|f_1=-\delta)\|_{\mathrm{TV}}\Big) \gtrsim \frac{1}{p}
\end{align*} by choosing $\delta \asymp p^{-1/2}$.

\begin{proof}[Proof of Lemma~\ref{lemma:error-factor-lb}]

We first reduce the original lower bound problem to the problem of obtaining lower bounds for a testing problem by choosing a specific $m^*$, matrix $\bB$, and distribution of $\bbf$. For arbitrary $m$ and $m^*$, we have the following decomposition
\begin{align*}
    \int |m(\bx) - m^*(\bbf)|^2 \mu(d\bbf, d\bu) =& \int \left|m(\bx) - \mathbb{E}[m^*(\bbf)|\bx] + \mathbb{E}[m^*(\bbf)|\bx] - m^*(\bbf)\right|^2 \mu(d\bbf, d\bu) \\
    =& \int \left|m(\bx) - \mathbb{E}[m^*(\bbf)|\bx]\right|^2 \mu(d\bx) + \int \left|\mathbb{E}[m^*(\bbf)|\bx] - m^*(\bbf)\right|^2 \mu(d\bbf, d\bu) \\
    &~~~~~~ + \int \left(m(\bx) - \mathbb{E}[m^*(\bbf)|\bx]\right) \left(\mathbb{E}[m^*(\bbf)|\bx] - m^*(\bbf)\right) \mu(d\bbf, d\bu).
\end{align*} 

It follows from the tower rule of conditional expectation that
\begin{align*}
    &\int \left(m(\bx) - \mathbb{E}[m^*(\bbf)|\bx]\right) \left(\mathbb{E}[m^*(\bbf)|\bx] - m^*(\bbf)\right) \mu(d\bbf, d\bu) \\
    =& \int \left(m(\bx) - \mathbb{E}[m^*(\bbf)|\bx]\right) \left(\int \left(\mathbb{E}[m^*(\bbf)|\bx] - m^*(\bbf)\right) \mu(d\bbf|\bx)\right) \mu(d\bx) \\
    =& \int \left(m(\bx) - \mathbb{E}[m^*(\bbf)|\bx]\right) \Big(\mathbb{E}[m^*(\bbf)|\bx] - \mathbb{E}[m^*(\bbf)|\bx]\Big) \mu(d\bx) = 0.
\end{align*} Therefore, we have
\begin{align*}
    \int |m(\bx) - m^*(\bbf)|^2 \mu(d\bbf, d\bu) \ge \int \left|\mathbb{E}[m^*(\bbf)|\bx] - m^*(\bbf)\right|^2 \mu(d\bbf, d\bu)
\end{align*} for any $m$ and $m^*$.

Without loss of generality, suppose $\lambda$ is a positive integer, otherwise, let it be $\lceil \lambda \rceil$. Let $m^*(\bbf) = f_1$, $\bV = [\bv_1,\cdots, \bv_r]\in \mathbb{R}^{p\times r}$, where column vectors satisfy $\bv_1 = (1/\sqrt{\lambda}, \cdots, 1/\sqrt{\lambda}, 0, \cdots, 0)^\top$ and $\bv_1,\cdots, \bv_r$ are orthogonal unit vectors, $\bLambda \in \mathbb{R}^{r\times r}$ be a diagonal matrix with $\bLambda_{1,1}=\lambda$ and $\bLambda_{i,i}=p$ for $i>1$. It is easy to verify that the matrix $\bB = \bV \bLambda \in \mathbb{R}^p$ satisfies  
\begin{align*}
    \lambda_{\min}(\bB^\top \bB) \ge \lambda.
\end{align*} Denote $\bbf_{-1}=(f_2,\cdots, f_r)^\top$, plugging in our choice of $\bB$ and $m^*$ yields
\begin{align*}
    \int \left|\mathbb{E}[m^*(\bbf)|\bx] - m^*(\bbf)\right|^2 \mu(d\bbf, d\bu) =& \int \big|\mathbb{E}[f_1|\bx] - f_1\big|^2 \mu(d f_1) \mu(\bbf_{-1}, d\bu) \\
    =& \int \left(\int \big|\mathbb{E}[f_1|\bx] - f_1\big|^2 \mu(d f_1|\bbf_{-1}, \bx) \right) \mu(d \bbf_{-1}, d\bx) \\
    \ge& \int \left(\int \big|\mathbb{E}[f_1|\bx,\bbf_{-1}] - f_1\big|^2 \mu(d f_1|\bbf_{-1}, \bx) \right) \mu(d \bbf_{-1}, d\bx) \\
    =& \int \left(\int \big|\mathbb{E}[f_1|\bx,\bbf_{-1}] - f_1\big|^2 \mu(d\bbf_{-1}) \mu(d \bx|f_1) \right)\mu(df_1).
\end{align*} 

Recall our construction of $\bB$, we can write down each component of $\bx = (x_1,\cdots, x_p)$ as
\begin{align*}
    x_j = 1\{j \le \lambda\} f_1 + \sum_{k=2}^r B_{j,k} f_k + u_j.
\end{align*} When $\bbf_{-1}$ and $f_1$ are independent, letting $\bu$ have i.i.d. components that are independent of $\bbf$, i.e., $u_1,\cdots, u_p \overset{i.i.d.}{\sim} \mu_u$, we have
\begin{align*}
    \mathbb{E}[f_1|\bx, \bbf_{-1}] &= \mathbb{E}[f_1|x_1,\cdots, x_{\lambda}, \bbf_{-1}] \\
    &= \mathbb{E}\left[f_1\Big|x_1 - \sum_{k=2}^r B_{1,k} f_k,\cdots, x_{\lambda} - \sum_{k=2}^r B_{\lambda,k} f_k, \bbf_{-1}\right] \\
    &= \mathbb{E}\left[f_1\Big|x_1 - \sum_{k=2}^r B_{1,k} f_k,\cdots, x_{\lambda} - \sum_{k=2}^r B_{\lambda,k} f_k\right].
\end{align*} Denote $\tilde{\bx} = \{\tilde{x}_1,\cdots, \tilde{x}_\lambda\} \in \mathbb{R}^\lambda$ with $\tilde{x}_j = f_1 + u_j$ for $j\in \{1,\cdots, \lambda\}$, the above discussion implies 
\begin{align*}
    \mathbb{E}[f_1|\bx, \bbf_{-1}] = \mathbb{E}[f_1|\tilde{\bx}(\bx, \bbf_{-1})]
\end{align*}

Using the variational representation of conditional expectation, that
\begin{align*}
    \mathbb{E}[y|\bx] = \argmin_g \int |y-g(\bx)|^2 \mu(\bx),
\end{align*} we find
\begin{align*}
    \int \left|\mathbb{E}[m^*(\bbf)|\bx] - m^*(\bbf)\right|^2 \mu(d\bbf, d\bu) \ge& \int \left( \int \left|\mathbb{E}[f_1|\tilde{\bx}] - f_1\right|^2 \mu(d\tilde{\bx} | f_1)\right) \mu(df_1) \\
    \ge& \inf_{\theta(\tilde{\bx})} \int \left|\theta(\tilde{\bx}) - f_1\right|^2 \mu(df_1,\tilde{\bx}).
\end{align*}
where the infimum is over all measurable function $\theta: \mathbb{R}^\lambda \to \mathbb{R}$. Let $\delta \in (0,1/2)$ be arbitrary, and $f_1$ be a random variable that takes value in $\{-\delta, \delta\}$ with equal probability. For any function $\theta: \mathbb{R}^\lambda \to \mathbb{R}$, we can define a corresponding classifier
\begin{align*}
    \psi(\tilde{\bx}) = 1\{\theta(\tilde{\bx}) \ge 0\}.
\end{align*} Let $v$ be a random variable such that $v=1\{f_1>0\}$, it is clear that
\begin{align*}
   \{\psi(\tilde{\bx}) \neq v\} \subset \{|\theta(\tilde{\bx}) - f_1| \ge \delta\},
\end{align*} which implies
\begin{align*}
    \int \left|\mathbb{E}[m^*(\bbf)|\bx] - m^*(\bbf)\right|^2 \mu(d\bbf, d\bu) \ge& \inf_{\theta(\tilde{\bx})} \int \left|\theta(\tilde{\bx}) - f_1\right|^2 \mu(df_1,d\tilde{\bx})\\
    \ge& \delta^2 \inf_{\theta(\tilde{\bx})} \int 1\{|\theta(\tilde{\bx}) - f_1| \ge \delta\} \mu(df_1,d\tilde{\bx})\\
    \ge& \delta^2 \inf_{\psi} \int 1\{\psi(\tilde{\bx}) \neq v\} \mu(df_1,d\tilde{\bx})\\
    \ge& \delta^2 \inf_{\psi} \mathbb{P}_\mu(\psi(\tilde{\bx}) \neq v).
\end{align*}

So far, we have
\begin{align*}
    \int |m(\bx) - m^*(\bbf)|^2 \mu(d\bbf, d\bu) \ge \delta^2 \inf_{\psi} \mathbb{P}_\mu[\psi(\tilde{\bx}) \neq v]
\end{align*} for arbitrary $\delta \in (0,1/2)$ by specifying $\bB$, $m^*$ and distribution of $\bbf$ that depends on (an arbitrary) $\delta$. It remains to lower bound the right hand side by specifying a particular distribution of $\bu$. Note $\psi$ induces a set in $\mathbb{R}^\lambda$, let $\mu_\delta$ and $\mu_{-\delta}$ be the distribution of $\tilde{\bx}$ when $\bbf=\delta$ and $\bbf=-\delta$ respectively, we obtain
\begin{align*}
    \sup_\psi \mathbb{P}_\mu[\psi(\tilde{\bx}) = v] &= \frac{1}{2} \sup_{\mathcal{A} \subset \mathbb{R}^\lambda} \mathbb{P}_\mu[ \tilde{\bx} \in \mathcal{A}|f_1=-\delta] + \mathbb{P}_\mu[ \tilde{\bx} \in \mathcal{A}|f_1=\delta] \\
    &= \frac{1}{2} \left( 1 + \sup_{\mathcal{A} \subset \mathbb{R}^\lambda} \mathbb{P}_{\mu_\delta}(\mathcal{A}) - \mathbb{P}_{\mu_{-\delta}}(\mathcal{A}) \right)\\
    &= \frac{1}{2} \Big(1 + \|\mu_\delta - \mu_{-\delta}\|_{\mathrm{TV}}\Big),
\end{align*} where $\|\mu-\nu\|_{\mathrm{TV}}$ is the total variation distance between the two distribution $\mu$ and $\nu$, plugging it back to the previous inequality yields
\begin{align*}
     \int |m(\bx) - m^*(\bbf)|^2 \mu(d\bbf, d\bu) \ge \frac{\delta^2}{2} \Big(1-\|\mu_\delta - \mu_{-\delta}\|_{\mathrm{TV}}\Big).
\end{align*}

Now we choose the distribution of $\bu$ to be that $u_1,\cdots, u_p \overset{i.i.d.}{\sim} \mu_u$, and the probabilistic density function for the distribution $\mu_u$ is
\begin{align*}
    \mu_u(du) = \frac{3}{2} (1 - |u|)^2 1\{|u| \le 1\} du.
\end{align*} It follows from the Le Cam inequality (Lemma 15.3 in \cite{wainwright2019high}) and the property of squared Hellinger distance that
\begin{align*}
    \|\mu_\delta - \mu_{-\delta}\|_{\mathrm{TV}}^2 \le H^2(\mu_\delta \| \mu_{-\delta}) \le \lambda H^2(\mu_{\tilde{x}_1,\delta}\|\mu_{\tilde{x}_1,-\delta}),
\end{align*} provided $\tilde{x}_1,\cdots, \tilde{x}_\lambda$ are i.i.d. copies, where $\mu_{\tilde{x}_1,v}$ is the distribution of $\tilde{x}_1 = f_1 + u_1$ when $f_1=v$, and $H^2(\mu\|\nu)$ is the squared Hellinger distance. Our choice of distribution of $\bu$ and the fact that $\delta \in (0,\frac{1}{2})$ together implies
\begin{align*}
    \frac{2}{3} H^2(\mu_{\tilde{x}_1,\delta}\|\mu_{\tilde{x}_1,-\delta}) = &\int_{-\delta-1}^{\delta-1} \left(\sqrt{(1-|u+\delta|)^2} - 0\right)^2 du \\
    &~~~~~~~~+ \int_{\delta-1}^{1-\delta} \left(\sqrt{(1-|u+\delta|)^2} - \sqrt{(1-|u-\delta|)^2}\right)^2 du\\
    &~~~~~~~~+ \int_{1-\delta}^{1+\delta} \left(0 - \sqrt{(1-|u-\delta|)^2}\right)^2 du \\
    \le& 2 \int_{0}^{2\delta} \omega^2 d\omega + \int_{\delta-1}^{1-\delta} (2\delta)^2 du \\
    \le& \frac{16}{3} \delta^3 + 8\delta^2 \le \frac{32}{3} \delta^2, 
\end{align*}
Combining all the pieces together, we have
\begin{align*}
    \int |m(\bx) - m^*(\bbf)|^2 \mu(d\bbf, d\bu) \ge& \frac{\delta^2}{2} \Big(1-\|\mu_\delta - \mu_{-\delta}\|_{\mathrm{TV}}\Big) \\
    \ge&\frac{\delta^2}{2} \left(1 - \sqrt{\lambda H^2(\mu_{\tilde{x}_1,\delta}\|\mu_{\tilde{x}_1,-\delta})}\right) \\
    \ge&\frac{\delta^2}{2} \left(1 - \sqrt{\frac{32}{3} \lambda\delta^2}\right).
\end{align*} for arbitrary $\delta \in (0,\frac{1}{2})$. Letting $\delta = 4^{-1} \lambda^{-1/2}$, we have
\begin{align*}
    \int |m(\bx) - m^*(\bbf)|^2 \mu(d\bbf, d\bu) \ge \frac{1}{32 \lambda} \cdot \frac{1}{6} = \frac{1}{192} \cdot \frac{1}{\lambda},
\end{align*} this completes the proof.
\end{proof}

\subsection{Proof of Theorem \ref{thm:fast-lb}}

\begin{proof}[Proof of Theorem~\ref{thm:fast-lb}]
\noindent {\sc Step 1. Reduction to the Oracle Case.} By Lemma \ref{lemma:error-factor-lb}, we argue that it suffices to show that
\begin{align*}
    \inf_{\breve{m}} \sup_{m^* \in \mathcal{H}(l,r+|\mathcal{J}|,\mathcal{P}), \mu \in \mathcal{P}(p,r,\lambda)} \mathbb{E} \left[ \int \big|\breve{m}(\bbf, \bu) - m^*(\bbf, \bu_{\mathcal{J}})\big|^2 \mu(d\bbf, d\bu)\right] \gtrsim n^{-\frac{2\gamma^*}{2\gamma^*+1}} + \frac{\log p}{n},
\end{align*} where the infimum is over all the estimator $\breve{m}(\bbf, \bu)$ based on $\{(\bbf_i,\bu_i, y_i)\}_{i=1}^n$ and the matrix $\bB$. To see this, for any estimator $\hat{m}(\bx)$ based on $\{(\bx_i, y_i)\}_{i=1}^n$, with the access to the latent factor structure $(\bbf, \bu)$ and the factor loading matrix $\bB$, we can choose a corresponding estimator $\breve{m}(\bbf, \bu)$ that
\begin{align*}
    \breve{m}(\bbf, \bu) = \hat{m}(\bB \bbf + \bu),
\end{align*} where $\hat{m}$ is estimated via the samples $\{(\bB \bbf_i + \bu_i, y_i)\}_{i=1}$. It is easy to verify that $\breve{m}$ can achieve the same risk as $\hat{m}(\bx)$,
\begin{align*}
    \int \big|\breve{m}(\bbf, \bu) - m^*(\bbf, \bu_{\mathcal{J}})\big|^2 \mu(d\bbf, d\bu) = \int \big|\hat{m}(\bx) - m^*(\bbf, \bu_{\mathcal{J}})\big|^2 \mu(d\bbf, d\bu),
\end{align*} which implies
\begin{align*}
     &\inf_{\hat{m}} \sup_{\mu \in \mathcal{P}(p,r,\lambda), m^*} \mathbb{E} \left[ \int \big|\hat{m}(\bx) - m^*(\bbf, \bu_{\mathcal{J}})\big|^2 \mu(d\bbf, d\bu)\right] \\
     &~~~~~~ \ge \inf_{\breve{m}} \sup_{\mu \in \mathcal{P}(p,r,\lambda), m^*} \mathbb{E} \left[ \int \big|\breve{m}(\bbf, \bu) - m^*(\bbf, \bu_{\mathcal{J}})\big|^2 \mu(d\bbf, d\bu)\right].
\end{align*}

\noindent {\sc Step 2. Lower Bound on Estimating a Function.} Let $\beta^*, d^*$ be that in \eqref{eq:hcm:gamma-def} and $\gamma^*=\beta^*/d^*$. When $d^* \le r+1$, $\mathcal{F}_{d^*, \beta^*, 1} \subset \mathcal{H}(r+1,l,\mathcal{P})$, it follows from the well-known minimax optimal lower bound for estimating $d$-variate $(\beta, 1)$-smooth function (e.g. Theorem 3.2 in \cite{gyorfi2002distribution}) that
\begin{align*}
    &\inf_{\breve{m}} \sup_{m^* \in \mathcal{H}(l,r+|\mathcal{J}|,\mathcal{P}), \mu \in \mathcal{P}(p,r,\lambda)} \mathbb{E} \left[ \int \big|\breve{m}(\bbf, \bu) - m^*(\bbf, \bu_{\mathcal{J}})\big|^2 \mu(d\bbf, d\bu)\right]\\
    \ge & \inf_{\breve{m}} \sup_{m^* \in \mathcal{F}_{d^*, \beta^*, 1}, \bx \sim \mathrm{unif}[-1,1]^{d^*}} \mathbb{E} \left[ \int \big|\breve{m}(\bx) - m^*(\bx)\big|^2 \mu(d\bx) \right] \gtrsim n^{-\frac{2\beta^*}{2\beta^*+d^*}}.
\end{align*}

\noindent {\sc Step 3. Lower Bound on Variable Selection Uncertainty.} For given $\delta \in (0,1)$, let $\mathbb{P}_j$ ($0 \le j \le p$) be the law of the i.i.d. data $(\bbf_1, \bu_1, y_1),\cdots, (\bbf_n, \bu_n, y_n)$ such that each component of $\bbf$ and $\bu$ are independent uniform distribution on $[-1, 1]$, and
\begin{align*}
    y_i = m^{(j)}(\bu_i) + \varepsilon_i ~~~~ \text{with} ~~~~ m^{(j)}(\bu) = \begin{cases} \delta \cdot u_j & j\ge 1\\
    0 & j = 0\end{cases},
\end{align*} where $\varepsilon_1,\ldots,\varepsilon_n$ are i.i.d. standard Gaussian random variables that are also independent of $\{(\bbf_i, \bu_i)\}_{i=1}^n$. It follows from the KL-divergence of two Gaussian random variables that
\begin{align*}
    KL(\mathbb{P}_0 \| \mathbb{P}_j) = \frac{1}{2} n \mathbb{E} \left|m^{(0)}(\bu) - m^{(j)}(\bu) \right|^2 \le \frac{1}{6} n\delta^2.
\end{align*}  At the same time, we have
\begin{align*}
    \|m^{(j)} - m^{(k)}\|_2 = \sqrt{\mathbb{E}|m^{(j)} - m^{(k)}|^2} = \sqrt{\frac{2}{3}} \delta 
\end{align*} for arbitrary $j\neq k$. 
Letting $\delta = \sqrt{(\log p)/(2n)}$, we have $m^{(j)} \in \mathcal{H}(l, r+|\mathcal{J}|, \mathcal{P})$, and $\|m^{(j)} - m^{(k)}\|_2 \ge  \sqrt{\frac{\log p}{3n}} $ as long as $j\neq k$. Moreover,
\begin{align*}
    \frac{1}{p} \sum_{j=1}^p KL(\mathbb{P}_0 \| \mathbb{P}_j) \le \frac{1}{6} n \cdot \frac{\log p}{2n} < \frac{1}{8} \log p.
\end{align*} Then it follows from Theorem 2.7 in \cite{tsybakov2009nonparametric} with $M=p$ that
\begin{align*}
    \inf_{\breve{m}} \sup_{\mu \in \mathcal{P}(p,r,\lambda), m^*} \mathbb{E} \left[ \int \big|\breve{m}(\bbf, \bu) - m^*(\bbf, \bu_{\mathcal{J}})\big|^2 \mu(d\bbf, d\bu)\right] \gtrsim \delta^2 \gtrsim \frac{\log p}{n},
\end{align*} which completes the proof.

\end{proof}

\section{Proof of the neural network approximation result}
\label{sec:proof-nn-approx}

In this section, we prove the neural network approximation result Theorem \ref{thm:approx-smooth}.

\subsection{Notations about the construction of neural network}

\label{sec:approx:notation}

In this subsection, we introduce several notations and simple facts on the construction of neural networks that might be helpful if we want to make constructive proofs of the neural network approximation result. Throughout this section, we fix $M=\infty$, that is, remove the truncation operator at the output.

\medskip
\noindent \textbf{Neural network padding.} If $f$ is a neural network with depth between 1 and $L$, and at most $N$ hidden nodes at each layer, then there exists some neural network $g$ with depth $L$ and $N$ hidden nodes at each layer such that $f(x)=g(x)$ for all the input $x$. We refer to this construction as \emph{neural network padding}. The padding with respect to width is trivial. For the padding with respect to depth, assume that the neural network has $L'\ge 1$ hidden layers. We can apply the identity map together with the activation function $L-L'$ times between the first hidden layer and the layer next to it. This will not change $f(x)$, but will increase the number of layers by $L-L'$. Hence $\mathcal{G}(L, d, o, N, B, \infty)$ can also be seen as the set of all neural networks with input dimension $d$, output dimension $o$, depth $L$ and width $N$. From the above discussion, we also have that $\mathcal{G}(L, d, o, N, B, \infty) \subset \mathcal{G}(L', d, o, N', B', \infty)$ if $L'\ge L, N'\ge N$ and $B'\ge B$.

\medskip
\noindent \textbf{Network composition.} Suppose $f \in \mathcal{G}(L_1, d_1, d_2, N_1, B, \infty)$ and $g \in \mathcal{G}(L_2, d_2, d_3, N_2, B, \infty)$, we use $h = g \circ f$ to denote the neural network which uses the input of $g$ as the output of $f$. It should be noted that $h$ is a neural network with width $N_1\lor N_2 \lor d_2$, depth $L_1 + L_2$ (instead of depth $L_1+L_2+1$), and weights bounded by $B \lor \tilde{B}$, where
\begin{align*}
    \tilde{B} = (d_2+1) \left(\|\bW_1^g\|_{\max} \lor \|\bb_1^g\|_{\max}\right)\left(\|\bW_{L_1+1}^f\|_{\max}\lor \|\bb_{L_1+1}^f\|_{\max}\right).
\end{align*} This is because we can combine the weight connecting the final hidden layer and the output layer of $g$ and the weight connecting the input layer and the first hidden layer of $f$ as a single weight, i.e. $\bW_2 (\bW_1 x + \bb_1) + \bb_2 = \bW' x + \bb'$. 

\medskip
\noindent \textbf{Network parallelization.} Suppose $f_i\in \mathcal{F}(L_i, d_i, o_i, N_i, B,\infty)$ for $i\in \{1,\ldots, k\}$. We use $h=(f_1, \ldots, f_k)$ to denote the neural network that takes $x\in \mathbb{R}^{\sum_{i=1}^k d_i}$ as the input, feeds the entries $x^{(i)}=\big(x_{\sum_{j=1}^{i-1} d_j + 1}, \cdots, x_{\sum_{j=1}^{i} d_j}\big)$ to the $i$-th sub-network $f_i$ that returns $y^{(i)}$, and combines these $y^{(i)}$ as the output. Such an $h$ is a neural network with input dimension $\sum_{i=1}^k d_i$, output dimension $\sum_{i=1}^k o_i$, depth at most $\max_{1\le i\le d} L_i$ and width at most $\sum_{i=1}^d N_i$, i.e., $(f_1, \cdots, f_k) \in \mathcal{G}(\max L_i, \sum d_i, \sum o_i, \sum N_i, B,\infty)$.

Suppose $d_i \le d$, we also use the notation $h=(f_1(x^{(1)}), \ldots, f_k(x^{(k)}))$ to denote the neural network that takes $x\in \mathbb{R}^d$ as the input, and feeds some of its entries $x^{(i)} = ((x)_{j_1},\ldots, (x)_{j_{d_i}})$ as input to the $i$-th subnetwork $f_i$, followed by the same procedure as above. Similarly, we conclude that $h$ is a neural network with input dimension $d$, output dimension $\sum_{i=1}^k o_i$, depth at most $\max_{1\le i\le d} L_i$ and width at most $\sum_{i=1}^d N_i$, i.e., $(f_1(x^{(1)}), \ldots, f_k(x^{(k)})) \in \mathcal{G}(\max L_i, d, \sum o_i, \sum N_i, B,\infty)$.

\medskip
\noindent \textbf{Simple functions.} At last, we introduce some simple functions that can be parameterized using ReLU neural networks:

\begin{lemma}[Identity, Absolute value, Min, Max]
\label{lemma:nn-simple}
For any $x, y\in \mathbb{R}$, the following properties hold:
\begin{itemize}
	\item[(1)] $x \in \mathcal{F}(1, 1, 1, 2, 1, \infty)$;
	\item[(2)] $|x| \in \mathcal{F}(1, 1, 1, 2, 1, \infty)$;
	\item[(3)] $\min(x,y) \in \mathcal{F}(1, 2, 1, 4, 1, \infty)$;
	\item[(4)] $\max(x,y) \in \mathcal{F}(1, 2, 1, 4, 1, \infty)$.
\end{itemize}
\end{lemma}
\begin{proof}[Proof of Lemma \ref{lemma:nn-simple}]
	For claims (1) and (2), recall that $\sigma(x) = (x)_+$, we thus have $x=\sigma(x)-\sigma(-x)$, $|x|=\sigma(x)+\sigma(-x)$. For claims (3) and (4), note that $\min(x,y)=\frac{1}{2}(x+y - |x-y|)$ and $\max(x,y)=\frac{1}{2}(x+y+|x-y|)$. It follows that \begin{align*}
		\min(x,y) &= \frac{1}{2} \big(x+y - |x-y|\big) \\
				  &= \frac{1}{2} \big(\sigma(x+y) - \sigma(-x-y) - \sigma(x-y) - \sigma(y-x)\big),
	\end{align*} 
	hence proving claim (3). Claim (4) can be similarly proved.
\end{proof}

\begin{lemma}[Piecewise linear function]
\label{lemma:nn-piecewise-linear}
	For fixed $N \in \mathbb{N}$ and $\delta>0$, let $\{(x_i, y_i)\}_{i=0}^N$ be $N+1$ points such that $x_i \in [0,1]$ with $x_i > x_{i-1}$. Then there exists a ReLU network $g^\dagger$ with depth $1$, width $N$, weights bounded by $2\max_{1\le i\le N} \frac{y_i - y_{i-1}}{x_i - x_{i-1}}$, and the linear map $\mathcal{L}_1$ only depends on $\{x_0,\ldots, x_N\}$, such that the following holds,
	\begin{itemize}
		\item[1.] $g^\dagger(x_i)=y_i$ holds for all $i=\{0,\ldots, N\}$.
		\item[2.] $g^\dagger(x)$ is linear on $[x_i, x_{i-1}]$ for all $i\in\{1,\ldots, N\}$.
	\end{itemize}
\end{lemma}

\begin{proof}
	We consider constructing our neural network as follows
	\begin{align*}
		g(x) = y_0 + \sum_{j=1}^N w_j \cdot \sigma(1 \cdot x-x_{j-1}),
	\end{align*} where the weights $w_i$ are defined as
	\begin{align*}
		w_j = \frac{y_j - y_{j-1}}{x_j - x_{j-1}} - 1\{j>1\} \cdot \frac{y_{j-1} - y_{j-2}}{x_{j-1} - x_{j-2}}.
	\end{align*} It is obvious that $f(x)$ is linear on each interval $[x_i, x_{i-1}]$, so it remains to show that $g(x_i)=y_i$ holds for all $i$, and to bound $|w_i|$. The bound of the magnitude of $|w_i|$ is also obvious. For the first part, we prove by induction. It is easy to verify that $g(x_0) = y_0$. If it holds for $i \in \{0,1,\ldots, N-1\}$, that is, $g(x_i) =y_i$, then
	\begin{align*}
		g(x_{i+1}) &= y_0 + \sum_{j=1}^{i+1} w_j \sigma(x_{i+1} - x_{j-1}) \\
		&= y_0 + \sum_{j=1}^{i} w_j \sigma(x_i - x_{j-1}) + \sum_{j=1}^{i+1} w_i (x_{i+1} - x_{i}) \\
	    &= y_i + (x_{i+1}-x_i) \sum_{j=1}^{i+1} w_j = y_{i+1}.
	\end{align*} so it concludes.
\end{proof}

We will also use the following notations. For given $d\in \mathbb{N}^+$, $K \in \mathbb{N}^+$ and $\Delta \in (0,1/3K]$, let $\mathcal{L}=\{0,\ldots, K-1\}^d$ be an index set. Define 
\begin{align}
\label{eq:approxmation-ql-delta}
    \mathcal{Q}_{\bl}(d, K, \Delta) = \left\{\bx=(x_1,\ldots, x_d): \frac{l_j}{K} \le x_j \le \frac{l_j+1}{K} - 1_{\{l_j+1<K\}} \Delta\right\}
\end{align} for any $\bl\in \mathcal{L}$ and 
\begin{align}
\label{eq:approxmation-q}
    \mathcal{Q}(d, K,\Delta) = \bigcup_{\bl \in \mathcal{L}} \mathcal{Q}_{\bl}(d, K, \Delta).
\end{align}

\subsection{Technical Lemmas}

In the proof of Theorem \ref{thm:approx-smooth}, we will use the following technical lemmas that build sub-network for different purposes. We provide detailed proofs of Lemma \ref{lemma:point-fitting-dim-1} and \ref{lemma:point-fitting-dim-d}, which play a key role in explicitly controlling the magnitude of the weights in Theorem \ref{thm:approx-smooth}. Lemma \ref{lemma:polynomial-fitting} -- \ref{lemma:midfunc} restate the results of \cite{lu2021deep} while explicitly presenting the weights' magnitude used in their construction.

\begin{lemma}[1-D Point fitting]
\label{lemma:point-fitting-dim-1}
	Let $N_1, N_2\in \mathbb{N}^+ \setminus \{1\}$ and $\delta>0$ be arbitrary. Then, for any set of points $\{(x_i, y_i)\}_{i=0}^{N_1N_2-1}$ with $x_i \ge x_{i-1} + \delta$ and $x_i, y_i\in [0,1]$, there exists a 3-hidden-layer ReLU network $g^\dagger$ with width parameter $\bd=(1, 2N_1-1, 4(N_2-2)+2, 8(N_2-2)+2, 1)$ and weights bounded by $4/\delta^2$ such that the following properties holds
	\begin{itemize}
	\item[1.] $g^\dagger(x_i) = y_i$ for all the $i\in \{0,\ldots, N_1N_2-1\}$.
	\item[2.] $g^\dagger(x)$ is linear on $x\in [x_{i-1}, x_i]$ if $i \notin \{jN_2, j=1,\ldots, N_1-1\}$.
	\end{itemize}
	Moreover, the weights in the first layer and the last layer are all bounded by $1$, i.e., $\|\bW_1\|_{\max} \lor \|\bb_1\|_{\max} \lor \|\bW_4\|_{\max} \lor \|\bb_4\|_{\max} \le 1$.
\end{lemma}

\begin{lemma}[Index creating in $d$-dimensional unit cube]
\label{lemma:point-fitting-dim-d}
Let $d, N\in \mathbb{N}^+$ be arbitrary. Define $K=\lfloor N^{1/d} \rfloor^2$, and $\mathcal{L}=\{0,\ldots, K-1\}^d$ as an index set. For any tolerance parameter $\Delta \in (0,1/(3K)]$, let $\mathcal{Q}_{\bl}(d, K, \Delta)$ be that in \eqref{eq:approxmation-ql-delta}. Then, there exists a deep ReLU network $\bg^\dagger: \mathbb{R}^d\to \mathbb{R}^d$ with depth $3$, width $16Nd$, and weights bounded by $4 \Delta^{-2}$ such that
\begin{align*}
    \bg^\dagger(\bx) = \bl/K ~~~~ \text{for any } \bx\in \mathcal{Q}_{\bl}(d, K, \Delta).
\end{align*} Moreover, the weights in the first layer and the last layer are all bounded by $1$.
\end{lemma}

\begin{lemma}[Polynomial fitting, restatement of Theorem 4.1 in \cite{lu2021deep}]
\label{lemma:polynomial-fitting}
Assume $P(\bx) = x_1^{\alpha_1} x_2^{\alpha_2}\cdots x_d^{\alpha_d}$ for $\alpha \in \mathbb{N}^d$ with $\|\alpha\|_1 \le k\in \mathbb{N}^+$. For any $N$, $L$, there exists a function $\phi$ implemented by a ReLU network with depth $7k^2L$, width $9(N+1) + k-1$, and weights bounded by $N^2$ such that
\begin{align*}
    |\phi(\bx) - P(\bx)| \le 9k(N+1)^{-7kL}  ~~~~ \text{for any } \bx\in [0,1]^d.
\end{align*}
\end{lemma}
\begin{lemma}[Multiplication, restatement of Lemma 4.2 in \cite{lu2021deep}]
\label{lemma:poly2}
For any $N, L\in \mathbb{N}^+$, and $a,b\in \mathbb{R}$ with $a<b$, there exists a function $\phi$ implemented by a ReLU network with depth $L$, width $9N+1$, and weights bounded by $3N^2[(|a|+|b|)^2 \lor 1]$ such that
\begin{align*}
    |\phi(x,y) - xy| \le 6(b-a)^2 N^{-L}  ~~~~ \text{for any } x,y\in [a,b]
\end{align*}
\end{lemma}

\begin{lemma}[Mid function, restatement of Lemma 3.1 in \cite{lu2021deep}]
\label{lemma:midfunc}
The middle value function $\mathrm{mid}(x_1, x_2, x_3)$ can be implemented by a ReLU network with depth 2, width 14, and weights bounded by 1.
\end{lemma}
\begin{lemma}[Lemma 3.4 in \cite{lu2021deep}]
\label{lemma:approx-extend} 
Given any $\epsilon>0$, $K\in \mathbb{N}^+$, and $\delta \in (0,1/3K]$, assume $f$ is continuous function on $[0,1]^d$ and $g$ is a general function with
\begin{align*}
    |f(\bx) - g(\bx)| \le \epsilon ~~~~\text{for any } x\in \mathcal{Q}(d, K,\Delta).
\end{align*}
Then 
\begin{align*}
    |\phi(\bx) - f(\bx)| \le \epsilon + d \sup_{\bx,\by \in [0,1]^d, \|\bx-\by\|_2\le \Delta} |f(\bx)-f(\by)|, ~~~~\text{for any } x\in [0,1]^d,
\end{align*} where $\phi:=\phi_d$ is defined through
\begin{align*}
    \phi_{j+1}(\bx) = \mathrm{mid}(\phi_j(\bx - \Delta \be_{j+1}), \phi_i(\bx), \phi_j(\bx + \Delta \be_{j+1})), ~~~~ \text{for } j=0,1,\ldots, d-1,
\end{align*} where $\phi_0=g$ and $\{\be_j\}_{j=1}^d$ is the standard basis in $\mathbb{R}^d$.
\end{lemma}

\subsection{Proof of Theorem \ref{thm:approx-smooth}}

We only prove the claim of the $d$-variate $(\beta,C)$-smooth function $g$. For the hierarchical composition model, it follows from the same idea of Proposition 3.4 in \cite{fan2022noise}. The key idea of the proof, which uses point-fitting sub-networks and polynomial fitting sub-networks to approximate the $r$-th order Taylor expansion of the target function $g$, is very similar to that of \cite{lu2021deep}. Yet, the proof differs from that in \cite{lu2021deep} in three aspects. (1) We use a point-fitting sub-network with carefully controlled bounded weights whose magnitude scales with $\mathcal{O}(N^c)$ for some constant $c$. Though not explicitly characterized in their construction, the point-fitting sub-network in \cite{lu2021deep, shen2019nonlinear} uses weights scales with $\Omega(e^N)$. This will lead to a huge difference when it is applied to statistical analyses. (2) In \cite{lu2021deep}, they allow the depth hyper-parameter to be tunable. We adopt a fixed depth instead to achieve a faster convergence rate in the $L_2$ norm when applying the approximation result to statistical analysis. (3) We allow the smoothness parameter $\beta$ to be in $\mathbb{R}^+$, and \cite{lu2021deep} requires $\beta \in \mathbb{N}^+$.

\begin{proof}[Proof of Theorem~\ref{thm:approx-smooth}]
\noindent {\sc Step 1. Establish Taylor Expansion Approximation for Function $g$. } Let $K=\lfloor N^{1/d} \rfloor^2$ and $0<\Delta<1/3K$ to be determined. For any $\bl \in \mathcal{L}=\{0,\ldots, K-1\}^d$, let $\mathcal{Q}_{\bl}(d, K, \Delta)$ be that in \eqref{eq:approxmation-ql-delta}, and $\bx_{\bl} = \bl/K$, we have $\bx_{\bl} \in \mathcal{Q}_{\bl}(d, K, \Delta)$ for any $\Delta\in [0,1/K)$. Define 
\begin{align}
    \bphi(\bx) = \bx_{\bl} ~~~~ \text{for} ~~~~ \bx \in \mathcal{Q}_{\bl}(\Delta),
\end{align} and $\bh = \bx - \bphi(\bx)$. Recall $\beta = r+s$ for $r\in \mathbb{N}$ and $s\in (0,1]$. Consider the following Taylor expansion approximation of $g$ as
\begin{align*}
    \bar{g}(\bx) = \sum_{\balpha \in \mathbb{N}^d, \|\balpha\|_1\le r} \frac{\partial f}{\partial {\bx}^{\balpha}}(\bphi(\bx)) \frac{{\bh}^{\balpha}}{\balpha!}
\end{align*} 

By Taylor's expansion at the point $\bphi(\bx)$ for $\bx$, we have for some $\xi \in (0,1)$ such that
\begin{align*}
    g(\bx) = \sum_{\balpha \in \mathbb{N}^d, \|\balpha\|_1 \le r-1} \frac{\partial f}{\partial \bx^{\balpha}}(\bphi(\bx)) \frac{{\bh}^{\balpha}}{\balpha!}+\sum_{\balpha \in \mathbb{N}^d, \|\balpha\|_1 = r} \frac{\partial f}{\partial {\bx}^{\balpha}}(\bphi(\bx) + \xi \bh) \frac{{\bh}^{\balpha}}{\balpha!},
\end{align*} then it follows from the definition of $(\beta, C)$-smooth function that
\begin{align*}
    |\bar{g}(\bx) - g(\bx)| &= \sum_{\balpha \in \mathbb{N}^d, \|\balpha\|_1 = r} \frac{\bh^{\balpha}}{\balpha!} \left|\frac{\partial f}{\partial {\bx}^{\balpha}}(\bphi(\bx) + \xi \bh) - \frac{\partial f}{\partial \bx^{\balpha}}(\bphi(\bx))\right| \\
    &\le \sum_{\balpha \in \mathbb{N}^d, \|\balpha\|_1 = r} \frac{\bh^{\balpha}}{\balpha!} \|\xi \bh\|_2^s \le \sum_{\balpha \in \mathbb{N}^d, \|\balpha\|_1 = r} \frac{\|\bh\|_\infty^{\|\balpha\|_1}}{\balpha!} \sqrt{d}^s \|\bh\|_\infty^s \\
    &\le C_{1} K^{-(r+s)} \le C_2 N^{-2\beta/d},
\end{align*} where $C_1, C_2$ are constants that only depend on $d$, $r$ and $s$. Moreover, the definition of $(\beta, C)$-smooth function also implies that 
\begin{align*}
    \left| \frac{1}{\balpha!} \cdot \frac{\partial f}{\partial {\bx}^{\balpha}} (\bx)\right| \le C_3 ~~~~ \text{for any } \bx \in [0,1]^d \text{ and } \balpha \in \mathbb{N}^d \text{ with } \|\balpha\|_1 \le r,
\end{align*} for some constant $C_3>9(r+1)+1$.

\noindent {\sc Step 2. Build Neural Network to Approximate $\bar{g}$.} We find some $g_{\Delta}^\dagger$ to approximate $\bar{g}$. Let $I(\bl) = \sum_{j=1}^d K^{-j} l_j$. It is easy to show that $I(\bl)$ is a bijective map between
\begin{align*}
    \Big\{\bl \in \mathcal{L}\Big\} ~~~~\text{and}~~~~ \left\{k/K^d: k\in \{0,\ldots, K^d-1\}\right\}.
\end{align*}

It follows from Lemma \ref{lemma:point-fitting-dim-d} that there exists a ReLU network $\bphi^\dagger: $ with depth $3$, width $16Nd$ and weights bounded by $8 \Delta^{-2}$ such that
\begin{align}
\label{eq:proof:nn-approx-indexing}
    \bphi^\dagger(\bx) = \bphi(\bx) ~~~~ \text{for any } \bx \in \bigcup_{\bl \in \mathcal{L}} \mathcal{Q}_{\bl}(\Delta) = \mathcal{Q}(\Delta).
\end{align}

For each $\balpha \in \mathbb{N}^d$ with $\|\balpha\|_1\le r$, we first apply Lemma \ref{lemma:point-fitting-dim-1} with $N_1 = N_2 = \lfloor N^{1/d} \rfloor^d$, and the point set
\begin{align*}
    \left\{\left(I(\bl), \frac{(\balpha!)^{-1}(\partial f/\partial {\bx}^{\balpha})(\bx_{\bl})+C_3}{2C_3}\right)\right\}_{\bl \in \mathcal{L}},
\end{align*} then there exists a ReLU network $v^\dagger_{\balpha}: \mathbb{R}\to \mathbb{R}$ with depth $3$, width less than $8N$, and weights bounded by $4/K^2 \le 4N^2$ such that
\begin{align}
\label{eq:proof:nn-approx-gradient-fitting}
    2C_3 v^\dagger_{\balpha}(I(\bl)) - C_3 = \frac{1}{\balpha!} \cdot \frac{\partial f}{\partial {\bx}^{\balpha}}(\bx_{\bl}) ~~~~ \text{for any } \bl \in \mathcal{L}.
\end{align} 

Applying Lemma \ref{lemma:polynomial-fitting} with $k=r+1$, we find that there also exists a ReLU network $p_{\balpha}^\dagger: \mathbb{R}^d\to \mathbb{R}$ with depth $7(r+1)^2$, width $9N+r$, and weights bounded by $N^2$ such that
\begin{align}
\label{eq:proof:nn-approx-polynomial-fitting}
    |p^\dagger_{\balpha}(\bx) - \bx^{\balpha}| \le (C_3-1) N^{-7(r+1)} ~~~~\text{for any } \bx\in [0,1]^d.
\end{align} 

Finally, we apply Lemma \ref{lemma:poly2} with $a=-C_3$ and $b=C_3$, then there exists a ReLU network $\varphi^\dagger: \mathbb{R}^2\to\mathbb{R}$ with depth $2(r+1)$, width $9N+1$, and weights bounded by $12(C_3)N^2$ such that
\begin{align}
\label{eq:proof:nn-approx-mult-fitting}
    |\varphi^\dagger(x, y) - xy| \le 24C_3^2 N^{-2(r+1)} ~~~~\text{for any }(x,y)\in [-C_3, C_3]^2.
\end{align}
Consider the function
\begin{align*}
    g^\dagger_\Delta(\bx) = \sum_{\alpha\in \mathbb{N}^d, \|\balpha\|_1 \le r} \varphi \Big(2C_2v^\dagger_{\balpha}\big(I(K\bphi^\dagger(\bx))\big)-C_2, p_{\balpha}^\dagger\big(\bx-\bphi^\dagger(\bx)\big)\Big).
\end{align*} 

In the remaining of {\sc Step 2}, we will show that it can be realized by a ReLU network with depth $\mathcal{O}(1)$, width $\mathcal{O}(N)$ and weights bounded by $\mathcal{O}(N^2 \lor \Delta^{-2})$, i.e., $g^\dagger_\Delta \in \mathcal{G}\Big(\mathcal{O}(1), d, 1, \mathcal{O}(N), \infty, \mathcal{O}(N^2 \lor \Delta^{-2})\Big)$ via repeatedly applying parallelization and composition argument. To this end, it follows from the parallelization argument that
\begin{align*}
    \bg_1 = (\text{Ident}(x_1),\cdots, \text{Ident}(x_d), \bphi^\dagger(\bx))
\end{align*} is a ReLU network $\mathbb{R}^d \to \mathbb{R}^{2d}$ with depth $3$, width $16dN+2d$ and weights bounded by $4\Delta^{-2}$. All the weights in the first layer and the last layer of $\bg_1$ are bounded by $1$. Secondly, note the mapping $(\bu, \bv) \to \bu - \bv$ and $\bu \to I(K\bu)$ are all linear transformations given $\bu, \bv\in \mathbb{R}^d$, applying the parallelization argument again for any $\balpha$ yields
\begin{align*}
    \bg_{2,\balpha}(\bu, \bv) = (2C_3 v_{\balpha}^\dagger(I(K \bv) - C_3, p^\dagger_{\balpha}(\bu - \bv)).
\end{align*} is a ReLU network $\mathbb{R}^{2d} \to \mathbb{R}^{2}$ with depth $3\lor 7(r+1)^2=7(r+1)^2$, width $(17N+r)$ and weights bounded by $4N^2 \lor d \lor (2C_3)$. Moreover, the weights in the first layer are bounded by $N^2 \lor d$ and the weights in the last layer is bounded by $N^2 \lor (2C_3)$. Then it follows from the composition argument that
\begin{align*}
    \bg_{3,\balpha}(\bu, \bv) = \varphi \circ \bg_{2,\balpha}
\end{align*} is a ReLU network $\mathbb{R}^d\to \mathbb{R}$ with depth $7(r+1)^2+2(r+1)$, width $17N+r$ and weights bounded by $12C_3N^4 \lor d$. All the weights in the first layer and last layer are bounded by $12C_2 N^2 \lor d$. 

Let \begin{align*}
    S = \sum_{\balpha \in \mathbb{N}^d, \|\balpha\|_1 \le r} 1 = \sum_{k=0}^r d^k \le d^{r+1}.
\end{align*} Note the function $h: (x_1,\cdots, x_S) \to \sum_{\ell=1}^S x_{\ell}$ is a linear function. Hence by composition and parallelization arguments, the function
\begin{align*}
    g_4(\bu, \bv) = h \circ \left((\bg_{3,\balpha}(\bu, \bv))_{\balpha \in \mathbb{N}^d, \|\balpha\|_1 \le r}\right)
\end{align*} can be realized as a ReLU network $\mathbb{R}^{2d}\to \mathbb{R}$ with depth $7(r+1)^2+2(r+1)$, width $S(17N+r)$ and weights bounded by $12C_3N^4 \lor d$. All the weights in the first layer are bounded by $12C_3N^2 \lor d$. 

Now we conclude the construction by letting $g^\dagger_\Delta(\bx) = \bg_4 \circ \bg_1$. It follows from the composition argument that $g^\dagger_\Delta(\bx)$ is a ReLU network $\mathbb{R}^d\to\mathbb{R}$ with depth $L^\dagger_\Delta = 7(r+1)^2 + 2(r+1)+3$, width $N^\dagger_\Delta = d^{r+1}(17N+r) \lor 16dN$ and weights bounded by $B^\dagger_\Delta = 12C_3N^4 \lor \Delta^{-2} \lor d$.

\noindent {\sc Step 3. Bound the Approximation Error in ``Good Region''.} The goal of this step is to show that
\begin{align}
\label{eq:proof:nn-approx-error-claim}
    |g^\dagger_\Delta(\bx) - g(\bx)| \lesssim N^{-2\beta/d} ~~~~ \text{for any } \bx \in \mathcal{Q}(d, K, \Delta).
\end{align} We have $|g(\bx) - \bar{g}(\bx)| \lesssim (NL)^{-2\beta/d}$ from {\sc Step 1}. Applying triangle inequality, it suffices to show that
\begin{align*}
    |g^\dagger_\Delta(\bx) - \bar{g}(\bx)| \lesssim N^{-2\beta/d} ~~~~ \text{for any } \bx \in \mathcal{Q}(d, K, \Delta).
\end{align*} 

Recall the definition of $\mathcal{Q}(d, K, \Delta)$, and the equivalence between the neural network indexing $\bphi^\dagger(\bx)$ and the index $\bphi(\bx)$ used to define $\bar{g}(\bx)$ in $\mathcal{Q}(d, K, \Delta)$ in \eqref{eq:proof:nn-approx-indexing}. When $\bx\in \mathcal{Q}(d, K, \Delta)$, we have
\begin{align*}
    |g^\dagger_\Delta(\bx) - \bar{g}(\bx)| &= \left|\sum_{\balpha \in \mathbb{N}^d, \|\balpha\|_1\le r} \frac{\partial f}{\partial {\bx}^{\balpha}}(\bphi(\bx)) \frac{{\bh}^{\balpha}}{\balpha!} - \varphi^\dagger \Big(2C_2v^\dagger_{\balpha}\big(I(K\bphi(\bx))\big)-C_2, p_{\balpha}^\dagger\big(\bx-\bphi(\bx)\big)\Big)\right| \\
    &\overset{(a)}{=} \left|\sum_{\balpha \in \mathbb{N}^d, \|\balpha\|_1\le r} \frac{\partial f}{\partial {\bx}^{\balpha}}(\bphi(\bx)) \frac{{\bh}^{\balpha}}{\balpha!} - \varphi^\dagger \left(\frac{1}{{\balpha!}} \cdot \frac{\partial f}{\partial {\bx}^{\balpha}}(\bphi(\bx)), p_{\balpha}^\dagger\big(\bx-\bphi(\bx)\big)\right)\right| \\
    &\overset{(b)}{\le} \sum_{\balpha \in \mathbb{N}^d, \|\balpha\|_1\le r} \left|\frac{1}{\balpha!}\cdot\frac{\partial f}{\partial {\bx}^{\balpha}}(\bphi(\bx))  \times {\bh}^{\balpha}- \frac{1}{\balpha!}\cdot \frac{\partial f}{\partial {\bx}^{\balpha}}(\bphi(\bx)) \times p_{\balpha}^\dagger(\bh)\right| \\
    &~~~~~~~~~~~~~~~~~~~~~~~~ + \left|\frac{1}{\balpha!}\cdot \frac{\partial f}{\partial {\bx}^{\balpha}}(\bphi(\bx)) \times p_{\balpha}^\dagger(\bh) - \varphi^\dagger \left(\frac{1}{{\balpha!}} \cdot \frac{\partial f}{\partial {\bx}^{\balpha}}(\bphi(\bx)), p_{\balpha}^\dagger\big(\bh\big)\right) \right| \\
    &\overset{(c)}{\le} \sum_{\balpha \in \mathbb{N}^d, \|\balpha\|_1\le r} \sup_{\bx} \left|\frac{1}{\balpha!}\cdot \frac{\partial f}{\partial {\bx}^{\balpha}}(\bphi(\bx))\right| \left|\bh^{\balpha} - p_{\balpha}^\dagger(\bh) \right| + 24C_3^2 N^{-2(r+1)} \\
    &\overset{(d)}{\le} C_3(C_3-1) N^{-7(r+1)} + 24C_3^2 N^{-2(r+1)}.
\end{align*} Here (a) follows from the point fitting result \eqref{eq:proof:nn-approx-gradient-fitting}, (b) follows from the triangle inequality and the fact that $\bh = \bx - \bphi(\bx)$, (c) follows from the bound on multiplication approximation \eqref{eq:proof:nn-approx-mult-fitting} and the fact that $\frac{1}{\balpha!}\cdot \frac{\partial f}{\partial {\bx}^{\balpha}}(\bphi(\bx)) \le C_3$ and $|p^\dagger_{\balpha}(\bh)| \le 9(r+1) N^{-7(r+1)} + |\bx^{\balpha}| \le 9(r+1) + 1 \le C_3$, (d) follows from the bound on polynomial approximation \eqref{eq:proof:nn-approx-polynomial-fitting}. This completes the proof of claim \eqref{eq:proof:nn-approx-error-claim} by the fact that $(r+1) \ge \beta \ge \beta/d$.

\noindent {\sc Step 4. Conclusion of the Proof by Generalizing to Unit Cube. } It follows from the above {\sc Step 1} -- {\sc Step 3} that for arbitrary $\Delta \in (0,1/(3K)]$, we can find a ReLU network $g^\dagger_{\Delta}$ with width $N^\dagger_\Delta \lesssim N$, depth $L^\dagger_\Delta \lesssim 1$ and weights bounded by $B^\dagger_\Delta \lesssim N^4 \lor \Delta^{-2}$ such that
\begin{align*}
    |g^\dagger_{\Delta}(\bx) - g(\bx)| \lesssim (NL)^{-2\beta/d} ~~~~\text{for any }\bx\in \mathcal{Q}(d, K, \Delta).
\end{align*}

In this step, we consider using Lemma \ref{lemma:approx-extend} and \ref{lemma:midfunc} to build $g^\dagger$ from $g^\dagger_\Delta$ with some suitable $\Delta$ such that the above approximation error bound holds for all the $\bx \in [0,1]^d$. It follows from the definition of $(\beta, C)$-smooth function that
\begin{align*}
    |g(\bx) - g(\by)| \le C_4 \|\bx - \by\|_2^{\beta \land 1}.
\end{align*} 

Setting $\Delta=N^{-2(\beta\lor 1)/d} \land (3K)^{-1}$, we have
\begin{align*}
    \sup_{\bx,\by\in [0,1]^d \|\bx-\by\|_2 \le \Delta} |g(\bx)-g(\by)| \le C_4 N^{-2\beta/d}.
\end{align*} This implies that we can use Lemma \ref{lemma:approx-extend} with $f=g$, $g=g^\dagger_\Delta$, $\Delta$ and $\epsilon\lesssim N^{-2\beta/d}$ to construct a function $\phi: \mathbb{R}^d \to \mathbb{R}$ based on $g^\dagger_\Delta$ satisfying
\begin{align*}
    |\phi(\bx) - g(\bx)| \le \epsilon + d\sup_{\bx,\by\in [0,1]^d \|\bx-\by\|_2 \le \Delta} |g(\bx)-g(\by)| \lesssim N^{-2\beta/d}.
\end{align*}

It remains to implement $\phi$ using the ReLU network and specify the parameter $L, N$, and $B$. We prove this by induction. To be specific, for fixed $\Delta$, we argue that $\phi_j(\bx)$ can be implemented via a ReLU network $\phi_j^\dagger(\bx)$ with depth $L^\dagger_\Delta+2d$, width $3^d N^\dagger_\Delta$ and weights bounded by $B^\dagger_\Delta$. To start with, when $j=1$, consider the construction that
\begin{align*}
    \phi^\dagger_1(\bx) = \mathrm{mid} \circ (g^\dagger_\Delta(\bx), g^\dagger_\Delta(\bx - \Delta\cdot \be_1), g^\dagger_\Delta(\bx + \Delta \cdot\be_1)).
\end{align*}

Note that all the weights of the ReLU network implementation of $\mathrm{mid}$ are bounded by $1$ in Lemma \ref{lemma:midfunc}. Hence it follows from the parallelization and composition arguments together with Lemma \ref{lemma:midfunc} that $\phi_1(\bx)$ can be implemented via a ReLU network with depth $L^\dagger_\Delta+2$, width $3N^\dagger_\Delta$, and weights bounded by $B^\dagger_\Delta$. 

If it holds for $j$, that is, we can implement $\phi_j(\bx)$ using a ReLU network $\phi^\dagger_j(\bx)$ with width depth $L^\dagger_\Delta+2j$, width $3^j N^\dagger_\Delta $ and weights bounded by $B^\dagger_\Delta$, then we can also implement $\phi_j(\bx-\Delta \be_{j+1})$ and $\phi_j(\bx+\Delta \be_{j+1})$ using ReLU network $\phi^\dagger_j(\bx-\Delta \be_{j+1})$ and $\phi^\dagger_j(\bx+\Delta \be_{j+1})$ respectively with the same network architecture configurations. Therefore, consider the construction that
\begin{align*}
    \phi^\dagger_{j+1}(\bx) = \mathrm{mid} \circ (\phi^\dagger_j(\bx), \phi^\dagger_j(\bx - \Delta\cdot \be_{j+1}), \phi^\dagger_j(\bx + \Delta \cdot\be_{j+1})).
\end{align*} It follows from the parallelization and composition arguments, and Lemma \ref{lemma:midfunc} that the claim holds for $j+1$.

Plugging in the choice of $\Delta$, we conclude that there exists a ReLU network $g^\dagger = \phi_d^\dagger(\bx)$ with depth $L^\dagger_\Delta +2d\lesssim 1$, width $3^d N^\dagger_\Delta \lesssim N$, and weights bounded by
\begin{align*}
    B^\dagger_\Delta \lesssim N^4 \lor \Delta^{-2} \lesssim N^4 \lor N^{4(\beta\lor 1)/d},
\end{align*} such that
\begin{align*}
    |g^\dagger(\bx) - g(\bx)| \lesssim N^{-2\beta/d} ~~~~ \text{for any } \bx\in [0,1]^d.
\end{align*}
\end{proof}

\subsection{Proof of Lemma \ref{lemma:point-fitting-dim-1}}

\begin{proof}[Proof of Lemma~\ref{lemma:point-fitting-dim-1}]

\noindent {\sc Step 1. Set the Basis.} We first set the parameters in the first hidden layer. Let $\mathcal{A}_L=\{jN_2: j\in \{0,\cdots, N_1-1\}\}$, $\mathcal{A}_U=\{(j+1)N_2-1: j\in\{0,\cdots, N_1-1\}\}$, and $\mathcal{A}=\mathcal{A}_L\cup \mathcal{A}_U$. Let $l_i$ with $i\in \{1,\cdots, 2N_1\}$ be the $i$-th smallest elements in $\mathcal{A}$, and define $\tilde{x}_i=x_{l_i}$ for $i\in \{1,\cdots, 2N_1\}$. In the remaining of the proof, we will use the piece-wise linear function result from Lemma \ref{lemma:nn-piecewise-linear} repeatedly with $\{(\tilde{x_i}, \tilde{y}_i)\}_{i=1}^{2N_1}$ where $(\tilde{y}_i)_{i=1}^{2N_1}$ are values to be specified by different hidden units. Therefore, we specify the first hidden layer as
\begin{align*}
	h_j(x) = \sigma(x - x_{l_j}) \qquad \text{ for } j = \{1,\ldots, 2N_1\}.
\end{align*}

\noindent {\sc Step 2. Setup of the First Unit in the Second Hidden Layer.} We will iteratively construct our neural network by induction. Firstly, applying Lemma \ref{lemma:nn-piecewise-linear} with point set $\{(x_i, y_i), i \in \mathcal{A}\}$, the following function $g_0(x)$
\begin{align*}
	g_0(x) = \sum_{j'=1}^{2N_1-1} w_{0,j'} h_{j'}(x) + b_0  	
\end{align*} satisfies
\begin{align*}
	g_0(x_{jN_2}) = y_{jN_2} ~~ \text{and} ~~g_0(x_{(j+1)N_2-1}) = y_{(j+1)N_2-1}\qquad \forall j\in \{0, \ldots, N_1-1\},
\end{align*} and $g_0(x)$ is linear on $[x_{jN_2}, x_{(j+1)N_2-1}]$ for all $j \in \{0, \ldots, N_1-1\}$. Moreover, the weight bound implies that
\begin{align*}
	\max_{j'}\{|w_{0,j'}|\} \lor |b_0| \le 2 \max_{j\in \{0,\ldots,N_1-1\}} \frac{|y_{(j+1)N_2-1}-y_{jN_2}|}{x_{(j+1)N_2-1}-x_{jN_2}} \lor \frac{|y_{(j+1)N_2}-y_{(j+1)N_2-1}|}{x_{(j+1)N_2}-x_{(j+1)N_2-1}} \le \frac{2}{\delta}
\end{align*} where the last inequality follows from the fact that all the pairwise distances between two different $x_i$'s are lower-bounded by $\delta$, and all the $y_i$'s are in $[0,1]$.

Moreover, $g_0(x) \in [0,1]$ for any $x\in [x_0,x_{m}]$ because $g_0(x)$ is a piecewise linear function that fits $\{(x_i, y_i)\}_{i\in \mathcal{A}}$, and all the points $y_i$ are in $[0,1]$.

\noindent {\sc Step 3. Construction of the Target Function.}
Now we set $s_0(x) = g_0(x)$, and consider to construct $s_k(x)$ for $k\in \{1,\ldots, N_2-1\}$ iteratively as 
\begin{align*}
	s_k(x) = s_{k-1}(x) + \Delta_k(x).
\end{align*} where $\Delta_k(x)$ is the function to be specified. We claim that the above construction of $s_k(x)$ satisfies
\begin{itemize}
\item[1.] (point fitting) For all $j\in \{0, \ldots, N_1-1\}$, \begin{align*}
	s_k(x_{jN_2+\ell}) &= y_{jN_2+\ell} \qquad &\text{ for any } \ell\in \{0,\ldots,k\}, \\
	s_k(x_{jN_2+\ell}) &= g_0(x_{jN_2+\ell}) \qquad &\text{ for any } \ell \in \{k+1, N_2-1\}.
\end{align*} 
\item[2.] (linearity) $s_k(x)$ is linear on the intervals
	\begin{align*}
		[x_{jN_2+\ell}, x_{jN_2+\ell+1}] \qquad \text{ and } \qquad [x_{jN_2+k+1}, x_{(j+1)N_2-1}]	
	\end{align*} for any $\ell \in \{0, \ldots, k\}$ and $j \in \{0,\ldots, N-1\}$.
\end{itemize}

Note $s_{N_2-2}(x)$ is the target function that we want. In this step, it remains to construct $\Delta_{k}(x)$ and show $s_k(x)$ satisfies the above condition by induction, we will show in the next steps that both $\Delta_k$ and $s_k$ can be implemented via ReLU networks. Firstly, the properties of $g_0(x)$ derived in Step 2 and the fact that $s_0(x)=g_0(x)$ yields that point fitting condition and linearity condition both hold in the case $k=0$. 

For $k\in \{1,\ldots, N_2-2\}$, we first try to specify our choice of $\Delta_k(x)$, and then prove by induction that the two conditions for $s_k(x)$ also hold if $s_{k-1}(x)$ satisfies the two conditions.

Intuitively, the role of $\Delta_k(x)$ is to let $s_k(x)$ fit one more point compared with $s_{k-1}(x)$, i.e., to fit the point $(x_{jN_2+k}, y_{jN_2+k})$, while still maintaining (1) other points unchanged, and (2) linearity condition between the interval $[x_{jN_2+k-1},x_{jN_2+k}]$, $[x_{jN_2+k}, x_{jN_2+k+1}]$, and $[x_{jN_2+k+1}, x_{(j+1)N_2-1}]$.

Let us formally implement the idea, let $\Delta_k(x)$ be
\begin{align*}
	\Delta_k(x) = \begin{cases}
 		0 & \qquad x \in [x_{jN_2}, x_{jN_2+k-1}] \\
 		\Big(\frac{y_{jN_2+k} - g_0(x_{jN_2+k})}{x_{jN_2+k}-x_{jN_2+k-1}} \Big) (x-x_{jN_2+k-1}) & \qquad x \in [x_{jN_2+k-1}, x_{jN_2+k}]\\
 		\Big(\frac{y_{jN_2+k} - g_0(x_{jN_2+k})}{x_{jN_2+k}-x_{jN_2+k+1}} \Big) (x-x_{jN_2+k+1})& \qquad x \in [x_{jN_2+k}, x_{jN_2+k+1}] \\
 		0 & \qquad x \in [x_{jT+k+1}, x_{j(T+1)-1}]
 	\end{cases},
\end{align*} on interval $[x_{jN_2}, x_{(j+1)N_2-1}]$ for $j\in \{0,\ldots, N_1-1\}$.

For the point fitting condition, we have $\Delta_k(x_{jN_2+\ell})=0$ as long as $\ell \le k-1$, this implies
\begin{align*}
	s_k(x_{jN_2+\ell}) = s_{k-1}(x_{jN_2+\ell}) = y_{jN_2+\ell}.
\end{align*} 

When $\ell=k$, 
\begin{align*}
	s_k(x_{jN_2+k}) &= s_{k-1}(x_{jN_2+k}) + \Delta_k(x_{jN_2+k})	\\
	&= g_0(x_{jN_2+k}) + \Big(\frac{y_{jN_2+k} - g_0(x_{jN_2+k})}{x_{jN_2+k}-x_{jN_2+k-1}} \Big) (x_{jN_2+k}-x_{jN_2+k-1}) \\
	&= y_{jN_2+k}.
\end{align*} 

When $\ell \ge k+1$, it follows from $\Delta_k(x_{jN_2+\ell})=0$ that 
\begin{align*}
	s_k(x_{jN_2+\ell}) = s_{k-1}(x_{jN_2+\ell}) = g_0(x_{jN_2+\ell}).
\end{align*} These pieces together verifies the point fitting condition.

For linearity, our construction implies that $\Delta_k(x)$ is a constant on $[x_{jN_2}, x_{jN_2+k-1}]$ and is linear on $[x_{jN_2+k-1}, x_{jN_2+k}]$, combined with the fact that $s_{k-1}(x)$ is linear on the interval $[x_{jN_2+\ell}, x_{jN_2+\ell+1}]$ for all $\ell \in \{0,\ldots,k-1\}$, the linearity property when $\ell \le k-1$ is preserved. Because $s_{k-1}(x)$ is linear on $[x_{jN_2+k}, x_{(j+1)N_2-1}]$, and our construction of $\Delta_k(x)$ satisfies $\Delta_k(x)$ is linear on the intervals $[x_{jN_2+k}, x_{jN_2+k+1}]$ and $[x_{jN_2+k+1}, x_{(j+1)N_2-1}]$, then $s_k(x)=s_{k-1}(x) + \Delta_k(x)$ is also linear on $[x_{jN_2+k}, x_{jN_2+k+1}]$ and $[x_{jN_2+k+1}, x_{(j+1)N_2-1}]$.

Therefore, by induction, our construction of $\Delta_k(x)$ and $s_k(x)$ satisfies the point fitting condition and linearity condition as claimed. It should be noted that in our construction of $s_k(x)$ and $\Delta_k(x)$, we are not interested in how these functions behave in the interval $[x_{(j+1)N_2-1},x_{(j+1)N_2}]$ for $j\in\{0,\ldots, N_1-1\}$.

\noindent {\sc Step 4. Implementing $\Delta_k(x)$ by Neural Network.} 

The function values of $\Delta_k(x)$ on the interval $[x_{jN_2+k-1}, x_{jN_2+k+1}]$ can be either non-negative (if $y_{jN_2+k}-g_0(x_{jN_2+k})\ge 0$) or non-positive (if $y_{jN_2+k}-g_0(x_{jN_2+k}) \le 0$), so we consider the two cases separately. That is, let $\Delta_k(x) = \Delta^+_k(x) - \Delta^-_k(x)$, where $\Delta^+_k(x)$, $\Delta^-_k(x)$ are all non-negative functions, and to implement the two separately by ReLU neural network based on our construction of first hidden layer. Without loss of generality, we consider the implementation of $\Delta^+_k(x)$, and the implementation of $\Delta^-_k(x)$ is the same if we alter the sign.

The key idea of implementing the function $\Delta^+_k(x)$ is to let $\Delta^+_k(x) = g^{+,l}_k(x) \land g^{+,u}_k(x)$. In order to simplify the notations, we omit the $+$ in the superscripts of $g$. 

We first try to implement $g^l_k(x)$ as a linear combination of the basis function in {\sc Step 1} followed by a ReLU activation. The linear combination is constructed via a point-fitting task using Lemma \ref{lemma:nn-piecewise-linear}. To be specific, we are going to fit the points
\begin{align}
\label{eq-tm41-1}
	\Bigg\{ \Big(x_{jN_2}, \tilde{y}_{jN_2}\Big), \Big(x_{(j+1)N_2-1},\tilde{y}_{(j+1)N_2-1} \Big) \Bigg\}_{j=0}^{N_1-1}.
\end{align} with 
\begin{align*}
	\tilde{y}_{jN_2} &= \frac{(y_{jN_2+k}-g_0(x_{jN_2+k}))_+}{x_{jN_2+k}-x_{jN_2+k-1}} (x_{jN_2}-x_{jN_2+k-1}) \\
	\tilde{y}_{(j+1)N_2-1} &= \frac{(y_{jN_2+k}-g_0(x_{jN_2+k}))_+}{x_{jN_2+k}-x_{jN_2+k-1}}(x_{(j+1)N_2-1}-x_{jN_2+k-1})
\end{align*}

Then it follows from Lemma \ref{lemma:nn-piecewise-linear} that there exists some $\tilde{g}^l_k(x)$ which is a linear combination of $h_j(x)$, i.e, 
\begin{align*}
	\tilde{g}^l_k(x) = \sum_{j'=1}^{2N_1-1} w_{k,j'}^l h_{j'}(x) + b_k^l,	
\end{align*} and it satisfies the point fitting condition defined in \eqref{eq-tm41-1}, and is pointwise linear.

Let us first see what such construction implies about $\tilde{g}_k^l(x)$. Given fixed $j$, if $y_{jT+k}-g_0(x_{jT+k}) \le 0$, then $g_k^l(x)$ is $0$ on the interval $[x_{jN_2}, x_{(j+1)N_2-1}]$. In the other case, that $y_{jN_2+k}-g_0(x_{jN_2+k})>0$, the constructed $\tilde{g}_k^l(x)$ is linear on $[x_{jN_2}, x_{(j+1)N_2-1}]$. Combining with its values at the points $x_{jN_2}$ and $x_{(j+1)N_2-1}$, we have
\begin{align*}
	\tilde{g}_k^l(x) &< 0 & \qquad \text{ if } x \in [x_{jN_2}, x_{jN_2+k-1}], \\
	\tilde{g}_k^l(x) &= 0 & \qquad \text{ if } x = x_{jN_2+k-1}, \\
	\tilde{g}_k^l(x) &= y_{jN_2+k}-g_0(x_{jN_2+k}) & \qquad \text{ if } x = x_{jN_2+k},
\end{align*} which implies,
\begin{align*}
	\sigma(\tilde{g}_k^l(x)) &= 0 & \qquad \text{ if } x \in [x_{jN_2}, x_{jN_2+k-1}], \\
	\sigma(\tilde{g}_k^l(x)) &= y_{jN_2+k}-g_0(x_{jN_2+k}) & \qquad \text{ if } x=x_{jN_2+k},
\end{align*} and $\sigma(\tilde{g}_k^l(x))$ is linear on $[x_{jN_2}, x_{jN_2+k-1}]$ and $[x_{jN_2+k-1}, x_{(j+1)N_2-1}]$

We then provide an upper bound on the weights $|w_{k,j'}^l|$ and $|b_k^l|$. Note that
\begin{align*}
	\frac{|\tilde{y}_{(j+1)N_2-1} - \tilde{y}_{jN_2}|}{x_{(j+1)N_2-1} - x_{jN_2}} = \frac{(y_{jN_2+k}-g_0(x_{jN_2+k}))_+}{x_{jN_2+k}-x_{jN_2+k-1}} \le \frac{1}{\delta},
\end{align*} and
\begin{align*}
	\frac{|\tilde{y}_{(j+1)T} - \tilde{y}_{(j+1)N_2-1}|}{x_{(j+1)N_2} - x_{(j+1)N_2-1}} \le \frac{1}{\delta} \Big(|\tilde{y}_{(j+1)N_2}| + |\tilde{y}_{(j+1)N_2-1}|\Big) \le \frac{1}{\delta}\Big(\frac{1}{\delta} + \frac{1}{\delta}\Big) = \frac{2}{\delta^2}.
\end{align*}
which implies
\begin{align*}
	\max_{j'}\{|w_{k,j'}^l|\} \lor |b_k^l| \le 2 \max_{j \in \{0,\ldots, N_1-1\}} \frac{|\tilde{y}_{(j+1)N_2-1} - \tilde{y}_{jN_2}|}{x_{(j+1)N_2-1} - x_{jN_2}} \lor \frac{|\tilde{y}_{(j+1)N_2} - \tilde{y}_{(j+1)N_2-1}|}{x_{(j+1)N_2} - x_{(j+1)N_2-1}} \le \frac{4}{\delta^2}.
\end{align*}

Let $g^l_k(x)=\sigma(\tilde{g}_k^l(x))$, the following facts about $g^l_k(x)$ holds,
\begin{itemize}
	\item[1.] All the weights are bounded by $4/\delta^2$.
	\item[2.] For any $j\in\{0,\ldots, N_1-1\}$, $g^l_k(x)$ is linear on the interval $[x_{jN_2}, x_{jN_2+k-1}]$ and $[x_{jN_2+k-1}, x_{(j+1)N_2-1}]$, and satisfies $g^l_k(x)=0$ on $[x_{jN_2+k-1}, x_{(j+1)N_2-1}]$ and $g^l_k(x_{jN_2+k}) = (y_{jN_2+k}-g_0(x_{jN_2+k}))_+$.
\end{itemize}

Following a similar way, we can also construct $g^u_k(x)=\sigma\left(\sum_{j'=1}^{2N_1-1} w_{k,j'}^u h_{j'}(x) + b_k^u\right)$ as a hidden node in the second layer satisfying
\begin{itemize}
	\item[1.] All the weights are bounded by $4/\delta^2$.
	\item[2.] For any $j\in\{0,\ldots, N-1\}$, $g^u_k(x)$ is linear on the interval $[x_{jN_2}, x_{jN_2+k+1}]$ and $[x_{jN_2+k+1}, x_{(j+1)N_2-1}]$, and satisfies $g^u_k(x)=0$ on $[x_{jN_2+k+1}, x_{(j+1)N_2-1}]$ and $g^u_k(x_{jN_2+k}) = (y_{jN_2+k}-g_0(x_{jN_2+k}))_+$.		
\end{itemize}

Combining the above claims about $g^{+,l}_k$ and $g^{+,u}_k$ together, we have
\begin{align*}
\Delta^+_k(x) = g^{+,l}_k(x) \land g^{+,u}_k(x) = \begin{cases}
 		0 & \qquad x \in [x_{jN_2}, x_{jN_2+k-1}] \\
 		\Big(\frac{(y_{jN_2+k} - g_0(x_{jN_2+k}))_+}{x_{jN_2+k}-x_{jN_2+k-1}} \Big) (x-x_{jN_2+k-1}) & \qquad x \in [x_{jN_2+k-1}, x_{jN_2+k}]\\
 		\Big(\frac{(y_{jN_2+k} - g_0(x_{jN_2+k}))_+}{x_{jN_2+k}-x_{jN_2+k+1}} \Big) (x-x_{jN_2+k+1})& \qquad x \in [x_{jN_2+k}, x_{jN_2+k+1}] \\
 		0 & \qquad x \in [x_{jN_2+k+1}, x_{(j+1)N_2-1}]
 	\end{cases}.
\end{align*}

It follows from a similar argument that we can implement $\Delta^-_k(x)$ with $\Delta^-_k(x) = g^{-,l}_k \land g^{-,u}_k$ satisfying
\begin{align*}
\Delta^-_k(x) = \begin{cases}
 		0 & \qquad x \in [x_{jN_2}, x_{jN_2+k-1}] \\
 		\Big(\frac{(g_0(x_{jN_2+k})-y_{jN_2+k})_+}{x_{jN_2+k}-x_{jN_2+k-1}} \Big) (x-x_{jN_2+k-1}) & \qquad x \in [x_{jN_2+k-1}, x_{jN_2+k}]\\
 		\Big(\frac{(g_0(x_{jN_2+k})-y_{jN_2+k})_+}{x_{jN_2+k}-x_{jN_2+k+1}} \Big) (x-x_{jN_2+k+1})& \qquad x \in [x_{jN_2+k}, x_{jN_2+k+1}] \\
 		0 & \qquad x \in [x_{jN_2+k+1}, x_{j(N_2+1)-1}]
 	\end{cases}
\end{align*} based on the basis in the first hidden layer.

Put these pieces together, we can conclude the constructed $\Delta_k(x) = \Delta^+_k(x) - \Delta^-_k(x)$ satisfies the condition we presented in {\sc Step 3}.

\noindent {\sc Step 5. Conclude the Neural Network Architecture and Weight Bound. }

In this step, we combine all the claims from the above steps together and conclude the proof, considering the function \begin{align*}
	g^\dagger(x) &= g_0(x) + \sum_{k=1}^{N_2-2} \Delta_k(x) \\
		 &= g_0(x) + \sum_{k=1}^{N_2-2} \Delta^+_k(x) - \Delta^-_k(x) \\
		 &= g_0(x) + \sum_{k=1}^{N_2-2} g^{+,l}_k(x) \land g^{+,u}_k(x) - g^{-,l}_k(x) \land g^{-,u}_k(x).
\end{align*} This is the function that satisfies conditions (1) and (2) by our arguments in {\sc Step 3} and {\sc Step 4}.

Recall that all the above $g$ with subscripts can be written as the form of $\sigma(\bw^\top \bh(x) + b)$ where the weights are bounded by $4/\delta^2$. Applying the neural network composition argument and Lemma \ref{lemma:nn-simple} with $\min(x,y)$ function and identity function, we can conclude that $g^\dagger$ can be realized as a 3-hidden layer ReLU neural network with hidden size $(2N_1-1,4(N_2-2)+2,8(N_2-2)+2)$. Through our construction, it is easy to verify that $\|\bW_1\|_{\max} \lor \|\bb_1\|_{\max} \lor \|\bW_4\|_{\max} \lor \|\bb_4\|_{\max} \le 1$.

\end{proof}

\subsection{Proof of Lemma \ref{lemma:point-fitting-dim-d}}

\begin{proof}[Proof of Lemma~\ref{lemma:point-fitting-dim-d}]

We first use Lemma \ref{lemma:point-fitting-dim-1} to approximate the step function. Particularly, we argue that for arbitrary $T \in \mathbb{N}^+$, there exists a deep ReLU network $\phi$ with depth $3$, width $16T$, weighted bounded by $4/\Delta^2$ such that 
\begin{align}
\label{eq:point-fitting-dim-d:step-func-approx}
    \phi(x) = \frac{t}{T^2} ~~~~ \text{if} ~~~~ \frac{t}{T^2} \le x\le \frac{t+1}{T^2} - 1{\{t\le T^2-1\}} \Delta
\end{align} for any $t\in \{0, \ldots, T^2-1\}$. It follows from Lemma \ref{lemma:point-fitting-dim-1} with $N_1 = T$ and $N_2 = 2T$, and the point set
\begin{align*}
    x_{2k} = \frac{k}{T^2}, x_{2k+1} = \frac{k+1}{T^2} - \Delta ~~~~ \text{and} ~~~~ y_{2k} = y_{2k+1} = \frac{k}{T^2} ~~~~~~~~ \text{for}~~k\in \{0,\ldots, T^2-1\}
\end{align*} that there exists a ReLU network $\phi$ with depth $3$ and width $8(2T-2)+2 \le 16T$, weights bounded by $4/\Delta^2$ such that $\phi(x_i) = y_i$ for $i \in \{0,\ldots, 2T^2-1\}$, and $\phi(x)$ is linear on $[x_i, x_{i+1}]$ for all the $i\in \{0,\ldots, 2T^2-1\} \setminus \{j\cdot 2T-1: j\in 1, \ldots, T-1\}$. This implies that $\phi(x)$ is linear on $[x_{2k}, x_{2k+1}]$ for all the $k\in \{0,\ldots, T^2-1\}$, which completes the proof of the claim \eqref{eq:point-fitting-dim-d:step-func-approx}.

With the help of the step function module, we are ready to construct the target function $\bg^\dagger(\bx)$. To be specific, for any $i\in \{1,\ldots, d\}$, we apply the above claim with $T=\tilde{N}$. Then there exists a ReLU network $\phi_i$ with $i\in \{1,\ldots, d\}$ with depth $3$, width $16\tilde{N}\le 16N$, and weights bounded by $4/\Delta^2$ such that
\begin{align*}
    \phi_i(x_i) = \frac{k}{K} ~~~~ \text{if} ~~~~ \frac{k}{K} \le x \le \frac{k+1}{K} - 1_{\{k\le K-1\}} \Delta
\end{align*} 

It follows from the parallelization argument that $\bg^\dagger = (\phi_1,\cdots, \phi_d) \in \mathcal{G}(3, d, d, 16dN, 4/\Delta^2,\infty)$ is the target function satisfying all the conditions in the statement of Lemam \ref{lemma:point-fitting-dim-d}.

\end{proof}

\section{Proof for the FANAM Estimator in Section \ref{sec:additive}}
\label{sec:proof-fanam}

We first introduce some notations. Define the function class
\begin{align*}
    \mathcal{G}_m = \left\{m(\bx) = m(\bx;\bW, \bV, \{{g}_j\}, \bbeta \in \mathbb{R}^p) \text{ of form in }\eqref{eq:fanam-pred} \right\},
\end{align*} 
and suppose that $g_j \in \mathcal{G}_j$.
Similar to that of the FAST-NN estimator, we adopt the notation of population $L_2$ norm and empirical $L_2$ norm \eqref{eq:l2-ln-norm-fast-def} that have unified definition over the function class $\mathcal{H} = \mathcal{H}_3 - \mathcal{H}_3$, where
\begin{align*}
    \mathcal{H}_3 = \left\{m(\bB\bbf + \bu): m\in \mathcal{G}_m\right\} \cup \{m^*(\bbf, \bu)\} \cup \{0\}
\end{align*}

\begin{proof}[Proof of Theorem \ref{thm:fanam}]

\noindent {\sc Step 1. Find An Approximation of $m^*$. } Let
\begin{align*}
    \delta_\mathtt{a} = \left(\frac{NL}{\log N \log L}\right)^{-4\gamma^*} ~~~~ \text{and} ~~~~ \delta_{\mathtt{f}} = \Bigg(\frac{\overline{r}\cdot r}{\nu_{\min}^2(\bH)\cdot p}\Bigg).
\end{align*}
The goal in this step is to find $\tilde{m} =\tilde{m}(\bx;\bW, \tilde{\bV}, \{\tilde{g}_j\}_{j=0}^p, \tilde{\bbeta}) \in \mathcal{G}_m$ with $\|\tilde{\bbeta}\|_1 = s-1$ such that
\begin{align}
\label{eq:fanam-proof-claim1}
    \|\tilde{m}(\bx) - m^*(\bbf, \bu)\|_2^2 = \int |\tilde{m}(\bx) - m^*(\bbf, \bu)|^2 \mu(d\bbf, d\bu) \lesssim s^2 \delta_{\mathtt{a}} + (s_u+1) s\delta_{\mathtt{f}}=s \delta_{\mathtt{a}+\mathtt{f}}.
\end{align} To this end, we will proceed in a similar way as the proof of Theorem \ref{thm:fast-oracle}. For each $m^*_j$ with $j \in \{0\}\cup \mathcal{J}_x \cup \mathcal{J}_u$, it follows from Proposition 3.4 of \cite{fan2022noise} such that there exists a ReLU network $h_j \in \mathcal{G}(L, 1+1_{\{j=0\}}(r-1), 1, N, M, \infty)$ such that
\begin{align*}
    \|h_j(\bx) - m^*_j(\bx)\|_\infty \lesssim \left(\frac{NL}{\log N \log L}\right)^{-2\gamma^*} = \sqrt{\delta_{\mathtt{a}}} ~~~~\text{for any}~ \bx\in [-2b, 2b]^{1+1_{\{j=0\}}(r-1)}.
\end{align*} 

\noindent \emph{Case 1. $r\ge 1$.} We first prove the claim in the case where $r\ge 1$. Recall the definition of the matrix $\bH$ in Definition \eqref{def:dpm}, we let $\tilde{m}$ be 
\begin{align*}
    \tilde{m}(\bx) = \underbrace{h_0\left(\bH^+ \tilde{\bbf}(\bx)\right)}_{\tilde{g}_0\left(\tilde{\bbf}(\bx)\right)} + \sum_{j\in \mathcal{J}_x} \underbrace{h_j(x_j)}_{\tilde{g}_j(x_j)} + \sum_{j\in \mathcal{J}_u} \underbrace{h_j\left(x_j - \bb_j^\top \bH^+ \tilde{\bbf}(\bx)\right)}_{\tilde{g}_j\left(x_j - \tilde{\bv}^\top \tilde{\bbf}(\bx)\right)}
\end{align*} 
then
\begin{align*}
    \|\tilde{m}(\bx) - m^*(\bx)\|_2^2 &= \mathbb{E} \Bigg|\left[h_0\left(\bH^+ \tilde{\bbf}(\bx)\right) - m^*_0(\bbf)\right] + \sum_{j\in \mathcal{J}_x} \left[h_j(x_j) - m^*_j(x_j) \right] \\
    &~~~~~~~~~~~~~ + \sum_{j\in \mathcal{J}_u} \left[h_j\left(x_j - \bb_j^\top \bH^+ \tilde{\bbf}(\bx)\right) - h_j(u_j)\right]\Bigg|^2 \\
    & \le s \left[\left\|h_0\left(\bH^+ \tilde{\bbf}(\bx)\right) - m^*_0(\bbf)\right\|_2^2 + \sum_{j\in \mathcal{J}_x} \|h_j - m^*_j\|_2^2 + \sum_{j\in \mathcal{J}_u} \left\|h_j\left(x_j - \bb_j^\top \bH^+ \tilde{\bbf}(\bx)\right) - m_j^*(u_j)\right\|_2^2\right]
\end{align*}
To derive upper bounds on $\left\|h_0\left(\bH^+ \tilde{\bbf}(\bx)\right) - m^*_0(\bbf)\right\|_2^2$ and $\left\|h_j\left(x_j - \bb_j^\top \bH^+ \tilde{\bbf}(\bx)\right) - h_j(u_j)\right\|_2^2$, we proceed in a similar way as the proof of Theorem \ref{thm:fast-oracle} via a truncation argument. In particular, we have
\begin{align*}
    \left\|h_0\left(\bH^+ \tilde{\bbf}(\bx)\right) - m^*_0(\bbf)\right\|_2^2 \lesssim \frac{\overline{r}}{\nu^2_{\min(\bH)}\cdot p} + \left(\frac{NL}{\log N \log L}\right)^{-4\gamma^*} = \delta_{\mathtt{a}} + \delta_{\mathtt{f}},
\end{align*} and
\begin{align*}
    \left\|h_j\left(x_j - \bb_j^\top \bH^+ \tilde{\bbf}(\bx)\right) - h_j(u_j)\right\|_2^2 \lesssim \frac{\overline{r} \cdot r}{\nu^2_{\min(\bH)}\cdot p} + \left(\frac{NL}{\log N \log L}\right)^{-4\gamma^*} = \delta_{\mathtt{a}} + \delta_{\mathtt{f}}.
\end{align*}
Putting these pieces together completes the proof of the claim \eqref{eq:fanam-proof-claim1}.

\noindent \emph{Case 2. $r=0$. } Notice that $\delta_{\mathtt{f}}=0$, and $\bx = \bu$. Letting
\begin{align*}
    \tilde{m}(\bx) = \sum_{j\in \mathcal{J}_x} \underbrace{h_j(x_j)}_{\tilde{g}_j(x_j)} + \sum_{j\in \mathcal{J}_u} \underbrace{h_j\left(x_j\right)}_{\tilde{g}_j(x_j)} = \sum_{j\in \mathcal{J}_x} h_j(x_j) + \sum_{j\in \mathcal{J}_u} h_j(u_j).
\end{align*}
It is easy to show that the bound \eqref{eq:fanam-proof-claim1} holds then.

\noindent {\sc Step 2. Derive Basic Inequality \& Concentration for Fixed Function.} Let $\hat{m}=\hat{m}(\bx;\bW, \hat{\bV}, \{\hat{g}_{j}\}_{j=1}^p, \hat{\bbeta})$ be the empirical risk minimizer of \eqref{eq:fanam-est}. By our construction of $\tilde{m}$,
\begin{align*}
    \frac{1}{n} \sum_{i=1}^n \left\{y_i - \hat{m}(\bx)\right\}^2 + \lambda \|\hat{\bbeta}\|_1 \le \frac{1}{n} \sum_{i=1}^n \left\{y_i - \tilde{m}(\bx)\right\}^2 + \lambda \|\tilde{\bbeta}\|_1.
\end{align*} Plugging $y=m^*(\bbf, \bu) + \varepsilon_i$ in and doing some simple algebra as that in the proof of Theorem \ref{thm:fast-oracle} gives
\begin{align*}
    \|\hat{m} - \tilde{m}\|_n^2 + 2\lambda \|\hat{\bbeta}\|_1 \le 4\|\tilde{m}-m^*\|_n^2 + \frac{4}{n} \sum_{i=1}^n \varepsilon_i \left(\hat{m}(\bx_i) - \tilde{m}(\bx_i)\right) + 2\lambda \|\tilde{\bbeta}\|_1.
\end{align*}
We next consider establishing a bound on $\|\tilde{m}-m^*\|_n^2$. Observe that $h_j - m^*_j$ is uniformly bounded by $M+M^*\lesssim 1$, this implies $|\tilde{m} - m^*|\lesssim s$. Hence
\begin{align*}
    \mathrm{Var} |\tilde{m} - m^*|^2 \le \mathbb{E} |\tilde{m} - m^*|^4 \le s^2 \mathbb{E} |\tilde{m} - m^*|^2 ~~~~\text{and}~~~~ |\tilde{m} - m^*|\lesssim s^2.
\end{align*} It then follows from the Bernstein inequality (Proposition 2.10 of \cite{wainwright2019high}) that the following event
\begin{align*}
    \|\tilde{m} - m^*\|_n^2 - \|\tilde{m} - m^*\|_2^2 \lesssim \sqrt{s^2 \|\tilde{m} - m^*\|_2^2\frac{t}{n}} + \frac{s^2 t}{n} \le \frac{1}{2} \|\tilde{m} - m^*\|_2^2 + \frac{3s^2 \cdot t}{2n}
    \lesssim s \left(\delta_{\mathtt{a}+\mathtt{f}} + \frac{s\cdot t}{n}\right)
\end{align*} occurs with probability $1-e^{-t}$ for arbitrary $t>0$. 
Define the event
\begin{align*}
    \mathcal{A}_t = \left\{\|\tilde{m} - m^*\|_n^2 \le C_1 s\left(\delta_{\mathtt{a}+\mathtt{f}} + \frac{s\cdot t}{n}\right)\right\}
\end{align*} for some universal constant $C_1>0$ and any $t>0$, we thus have $\mathbb{P}[\mathcal{A}_t] \ge 1-e^{-t}$. Moreover, under $\mathcal{A}_t$, the basic inequality can be written as
\begin{align}
\label{eq:proof:fanam:basic-ineq-step2}
    \|\hat{m} - \tilde{m}\|_n^2 + 2\lambda \|\hat{\bbeta}\|_1 \le 4C_1 s\left(\delta_{\mathtt{a}+\mathtt{f}} + \frac{s\cdot t}{n}\right) + \frac{4}{n} \sum_{i=1}^n \varepsilon_i \left(\hat{m}(\bx_i) - \tilde{m}(\bx_i)\right) + 2\lambda \|\tilde{\bbeta}\|_1.
\end{align}

\noindent {\sc Step 3. Concentration for Weighted Empirical Process.}
Denote $\hat{\beta}_0=\tilde{\beta}_0=1$. Observe that
\begin{align}
    \label{eq:fanam-weighted-decomp}
    \frac{1}{n} \sum_{i=1}^n \varepsilon_i \left(\hat{m}(\bx_i) - \tilde{m}(\bx_i)\right) = \sum_{j=0}^p \frac{1}{n} \sum_{i=1}^n (\hat{\beta}_j \hat{g}_j - \tilde{\beta}_j \tilde{g}_j) \varepsilon_i.
\end{align} 

Note that $\hat{g}_j$ is a single ReLU network with depth $L$ and the number of parameters $\lesssim LN^2 + \overline{r}N \lesssim LN^2$. Then it follows from Theorem 7 of \cite{bartlett2019nearly} that
\begin{align*}
   \frac{\mathrm{Pdim}(\mathcal{G}_j) \log n}{n} \asymp \frac{N^2L^2 \log NL \log n}{n} = v_n,
\end{align*}

Define the event $\mathcal{B}_t = \bigcap_{j=0}^p \mathcal{B}_{t,j}$, where 
\begin{align*}
    \mathcal{B}_{t,j} = \left\{\left|\frac{1}{n} \sum_{i=1}^n (\hat{\beta}_j \hat{g}_j - \tilde{\beta}_j \tilde{g}_j) \varepsilon_i \right| \le C_2 \left(\|\hat{\beta}_j \hat{g}_j - \tilde{\beta}_j \tilde{g}_j\|_n^2 + \frac{1}{p} \right)\sqrt{v_n + \frac{u}{n}}\right\}
\end{align*} with $u$ to be determined, where $C_2$ is the universal constant $c_1$ in Lemma \ref{lemma:fa-weight-empirical-process}.
Applying Lemma \ref{lemma:fa-weight-empirical-process} with $\epsilon=1/p$ gives $\mathbb{P}(\mathcal{B}_{t,j}^c) \le \log(p) e^{-u}$. Then it follows from the union bound that
\begin{align*}
    \mathbb{P}(\mathcal{B}_t^c) \le \sum_{j=0}^p \mathbb{P}(\mathcal{B}_{t,j}^c) \le p\log(np) e^{-u}
\end{align*} Let $u=\log [p\log (p)] + t$, we have $\mathbb{P}[B_t] \ge 1-e^{-t}$. Moreover, under $\mathcal{B}_t$, we have
\begin{align*}
    \forall j \in \{1,\ldots, p\} ~~~~ \left|\frac{1}{n} \sum_{i=1}^n (\hat{\beta}_j \hat{g}_j - \tilde{\beta}_j \tilde{g}_j) \varepsilon_i \right| \lesssim \left(\|\hat{\beta}_j \hat{g}_j - \tilde{\beta}_j \tilde{g}_j\|_n^2 + \frac{1}{p} \right)\sqrt{v_n + \frac{\log p}{n} + \frac{t}{n}}
\end{align*} by the fact that $\log [p(\log p)] \lesssim \log p$. Plugging it into \eqref{eq:fanam-weighted-decomp} yields
\begin{align}
\label{eq:fanam-bound-weighted}
    \left|\frac{1}{n} \sum_{i=1}^n \varepsilon_i \left(\hat{m}(\bx_i) - \tilde{m}(\bx_i)\right) \right|\lesssim \left(\| \hat{\bbeta}\|_1 + \| \hat{\bbeta}\|_1 + 3 \right) \sqrt{v_n + \frac{(\log p) + t}{n}}
\end{align} by triangle inequality. 

The remaining proof of Theorem \ref{thm:fanam} proceeds conditioned on $\mathcal{A}_t \cap \mathcal{B}_t$, in which the basic inequality in \eqref{eq:proof:fanam:basic-ineq-step2} turns to be
\begin{align}
\label{eq:proof:fanam:basic-ineq-step3}
\begin{split}
    &~\|\hat{m} - \tilde{m}\|_n^2 + 2\lambda \left(\|\hat{\bbeta}\|_1 - \|\tilde{\bbeta}\|_1\right) \\
    &~~~~~~~~\le C_3 \left\{ s\left(\delta_{\mathtt{a}+\mathtt{f}} + \frac{s\cdot t}{n}\right) + \left(\| \hat{\bbeta}\|_1 + \| \hat{\bbeta}\|_1 + 3 \right) \sqrt{v_n + \frac{(\log p) + t}{n}} \right\}
\end{split}
\end{align} for some universal constant $C_3>0$.

\noindent {\sc Step 4. Exploiting Low Complexity Structure.} If 
\begin{align}
\label{eq:fanam:proof:cond-lambda}
    \lambda \ge C_3 \sqrt{v_n + \frac{(\log p) + t}{n}},
\end{align} then it follows from the above inequality \eqref{eq:proof:fanam:basic-ineq-step3} that
\begin{align*}
   2\lambda \left(\|\hat{\bbeta}\|_1 - \|\tilde{\bbeta}\|_1\right) \le C_3 s\left(\delta_{\mathtt{a}+\mathtt{f}} + \frac{s\cdot t}{n}\right) + 3\lambda + \lambda(\|\hat{\bbeta}\|_1 + \|\tilde{\bbeta}\|_1).
\end{align*} which implies
\begin{align*}
    \|\hat{\bbeta}\|_1 \le 3\| \tilde{\bbeta} \|_1 + 3 + C_3 \frac{\delta_{\mathtt{a}+\mathtt{f},n}}{\lambda} = s \left(3+ C_3 \frac{\delta_{\mathtt{a}+\mathtt{f},n}}{\lambda}\right) ~~~~\text{for} ~~ \delta_{\mathtt{a}+\mathtt{f},n} = \delta_{\mathtt{a}+\mathtt{f}} + \frac{s\cdot t}{n} .
\end{align*} Therefore, we can conclude that \eqref{eq:proof:fanam:basic-ineq-step3} can be further written as
\begin{align}
\label{eq:proof:fanam:basic-ineq-step4}
    \|\hat{m} - \tilde{m}\|_n^2 + \lambda \|\hat{\bbeta}\|_1 \le C_3 \delta_{\mathtt{a}+\mathtt{f},n} + 3\lambda\|\tilde{\bbeta}\|_1 + 3\lambda \le s(3\lambda + C_3 \delta_{\mathtt{a}+\mathtt{f},n})
\end{align} and the estimator $\hat{m}$ lie in the function class
\begin{align*}
    \mathcal{G}_{m,s} = \left\{m(\bx; \bW, \bV, \{g_j\}_{j=0}^p, \bbeta) \in \mathcal{G}_m: \|\bbeta\|_1 \le s \left(3+ C_3 \lambda^{-1} \delta_{\mathtt{a}+\mathtt{f}, n}\right)\right\}.
\end{align*}

\noindent {\sc Step 5. Equivalence between $L_2$ norm and $L_n$ norm. } In this step, we aim to show that the event
\begin{align*}
    \mathcal{C}_t = \left\{\forall m\in \mathcal{G}_{m,s} ~~ \left|\|m - \tilde{m}\|_n^2 - \|m - \tilde{m}\|_2^2 \right|\le C_4 s^2\left(1 + \lambda^{-1} \delta_{\mathtt{a}+\mathtt{f},n} \right)^2 \sqrt{v_n + \frac{(\log p) + t}{n}} \right\}
\end{align*} occurs with probability at least $1-e^{-t}$ for any $0<t<s^{-2}(1 + \lambda^{-1} \delta_{\mathtt{a}+\mathtt{f},n})^{-2} \sqrt{n}$. 

To this end, we first establish an upper bound on $\mathbb{E} \left[\sup_{m\in \mathcal{G}_{m,s}} |\|m - \tilde{m}\|_n^2 - \|m-\tilde{m}\|_2^2|\right]$. Note that $\bx_1,\cdots, \bx_n$ are i.i.d. copies of $\bx \in \mu$, then it follows from the symmetrization and contraction arguments that
\begin{align*}
    \mathbb{E} \left[\sup_{m\in \mathcal{G}_{m,s}} |\|m - \tilde{m}\|_n^2 - \|m-\tilde{m}\|_2^2|\right] &= \mathbb{E} \left[ \sup_{m\in \mathcal{G}_{m,s}} \left| \frac{1}{n} \sum_{i=1}^n (m(\bx_i) - \tilde{m}(\bx_i))^2 - \|m-\tilde{m}\|_2^2\right| \right] \\
    &\le 2\mathbb{E} \left[ \sup_{m\in \mathcal{G}_{m,s}} \left| \frac{1}{n} \sum_{i=1}^n (m(\bx_i) - \tilde{m}(\bx_i))^2 r_i \right| \right]\\
    &\le 2\sup_{m\in \mathcal{G}_{m,s}} \|m - \tilde{m}\|_\infty \mathbb{E} \left[ \sup_{m\in \mathcal{G}_{m,s}} \left| \frac{1}{n} \sum_{i=1}^n (m(\bx_i) - \tilde{m}(\bx_i)) r_i \right| \right] \\
    &\lesssim s\left(1 + \lambda^{-1} \delta_{\mathtt{a}+\mathtt{f},n} \right) \mathbb{E} \left[ \sup_{m\in \mathcal{G}_{m,s}} \left| \frac{1}{n} \sum_{i=1}^n (m(\bx_i) - \tilde{m}(\bx_i)) r_i \right| \right],
\end{align*} for i.i.d. Rademacher random variables $r_1,\ldots, r_n$ that is also independent of $\bx_1,\ldots, \bx_n$. Following a similar argument as that of {\sc Step 3}, in which we derive a tail probability bound for the weighted empirical process, we obtain
\begin{align*}
    &\mathbb{E} \left[ \sup_{m\in \mathcal{G}_{m,s}} \left| \frac{1}{n} \sum_{i=1}^n (m(\bx_i) - \tilde{m}(\bx_i)) r_i \right| \right] \\
    &~~~~~~~~~~~~ \overset{(a)}{\le} \mathbb{E} \left[ \sup_{m\in \mathcal{G}_{m,s}} \left| \frac{1}{n} \sum_{i=1}^n m(\bx_i) r_i \right| \right] + \mathbb{E} \left[ \sup_{m\in \mathcal{G}_{m,s}} \left| \frac{1}{n} \sum_{i=1}^n \tilde{m}(\bx_i) r_i \right| \right] \\
    &~~~~~~~~~~~~ \overset{(b)}{\le} \sup_{m\in \mathcal{G}_{m,s}} \left(\|\bbeta\|_1 + 1\right) \mathbb{E} \left[\sup_{j\in \{0,1,\ldots,p\}, g_j\in \mathcal{G}_j} \left|\frac{1}{n} \sum_{i=1}^n  g_j r_i\right|\right] + \mathbb{E} \left[ \left| \frac{1}{n} \sum_{i=1}^n \tilde{m}(\bx_i) r_i \right| \right] \\
    &~~~~~~~~~~~~ \overset{(c)}{\le} s\left(1 + \lambda^{-1} \delta_{\mathtt{a}+\mathtt{f},n}\right) \sqrt{v_n + \frac{\log p}{n}} + s \sqrt{\frac{1}{n}} \\
    &~~~~~~~~~~~~ \lesssim s\left(1 + \lambda^{-1} \delta_{\mathtt{a}+\mathtt{f},n}\right) \sqrt{v_n + \frac{\log p}{n}}.
\end{align*} Here (a) follows from triangle inequality, (b) follows from the fact that $\tilde{m}$ is fixed and $|\bx^\top \by| \le \|\bx\|_1 \|\by\|_\infty$ for two vectors $\bx, \by \in \mathbb{R}^{p+1}$, (c) follows from the standard chaining result (e.g. Theorem 2.2.18 of \cite{talagrand2014upper}) to bound the expectation of the supremum of a sub-Gaussian process and the fact that $g_j r_i$ is $M$ sub-Gaussian. Substituting it back gives
\begin{align*}
    \mathbb{E} \left[\sup_{m\in \mathcal{G}_{m,s}} \left|\|m - \tilde{m}\|_n^2 - \|m-\tilde{m}\|_2^2\right|\right] \lesssim s^2 \left(1 + \lambda^{-1} \delta_{\mathtt{a}+\mathtt{f},n}\right)^2 \sqrt{v_n + \frac{\log p}{n}}.
\end{align*} 

Recall the uniform boundedness of $m - \tilde{m}$ in $m\in \mathcal{G}_{m,s}$ that $\|m - \tilde{m}\|_\infty \lesssim s(1 + \lambda^{-1} \delta_{\mathtt{a}+\mathtt{f}, n})$, applying the Talagrand concentration inequality for the empirical process (Eq (3.86) with $\epsilon=1$ in \cite{wainwright2019high}), we can conclude
\begin{align*}
    \mathbb{P}\left[\sup_{m\in \mathcal{G}_{m,s}}{\left|\|m - \tilde{m}\|_n^2 - \|m - \tilde{m}\|_2^2 \right|} \ge 2\mathbb{E}\left[\sup_{m\in \mathcal{G}_{m,s}}{\left|\|m - \tilde{m}\|_n^2 - \|m - \tilde{m}\|_2^2 \right|}\right] + \sqrt{R\frac{t}{n}} + R\frac{t}{n}\right] \le e^{-t}
\end{align*} with $R = s^4 (1 + \lambda^{-1} \delta_{\mathtt{a}+\mathtt{f},n})^4$. Combining with the upper bound on the expectation of the supremum of the empirical process, we conclude the proof of the claim that $\mathbb{P}[\mathcal{C}_t] \ge 1-e^{-t}$ as long as $0<t< \sqrt{n/R}$.

\noindent {\sc Step 6. Conclusion of the Proof.} In this step, we combine all the pieces together. Under the event $\mathcal{A}_t \cap \mathcal{B}_t \cap \mathcal{C}_t$, which occurs with probability $1-3e^{-t}$ by the union bound, we have
\begin{align*}
    \|\hat{m} - \tilde{m}\|_2^2 &\le \|\hat{m} - \tilde{m}\|_n^2 + C_4 \left(s + \lambda^{-1} \delta_{\mathtt{a}+\mathtt{f},n}\right)^2 \sqrt{v_n + \frac{(\log p)+t}{n}} \\
    &\le s (3\lambda + C_3 \delta_{\mathtt{a}+\mathtt{f},n}) + C_4 s^2 \left(1+\lambda^{-1} \delta_{\mathtt{a}+\mathtt{f},n}\right)^2 \sqrt{v_n + \frac{(\log p) + t}{n}},
\end{align*} as long as $t<s^{-2}(1 + \lambda^{-1} \delta_{\mathtt{a}+\mathtt{f},n})^{-2} \sqrt{n}$, where the last inequality follows from \eqref{eq:proof:fanam:basic-ineq-step4}. If we further choose
\begin{align*}
    t \asymp \sqrt{\frac{n}{s^4 \left(1+\lambda^{-1} \delta_{\mathtt{a}+\mathtt{f}}\right)^4}} \bigwedge \frac{n \delta_{\mathtt{a}+\mathtt{f}}}{s} \bigwedge (nv_n + \log p) = t^*
\end{align*} then $\delta_{\mathtt{a}+\mathtt{f}, n} \asymp \delta_{\mathtt{a}+\mathtt{f}}$, so we can conclude that
\begin{align*}
    \|\hat{m} - m^*\|_2 \lesssim \|\hat{m} - \tilde{m}\|_2^2 + \|\tilde{m} - m^*\|_2^2 \lesssim s\lambda + s\delta_{\mathtt{a}+\mathtt{f}} + s^2 \left(1+\lambda^{-1} \delta_{\mathtt{a}+\mathtt{f}} \right)^2 \delta_{\mathtt{s}}
\end{align*} with probability at least $1-3e^{-t^*}$. Finally, we optimize $\lambda$ and $NL$ to get the optimal rate of convergence. By some calculations, it is easy to check the optimal convergence rate is attained when
\begin{align*}
    NL \asymp n^{\frac{1}{2(4\gamma^*+1)}} (\log n)^{\frac{8\gamma^*-1}{4\gamma^*+1}} ~~~~\text{and}~~~~ \lambda \asymp s \left\{\sqrt{\frac{\log p}{n}} + \left(\frac{\log^6 n}{n}\right)^{\frac{2\gamma^*}{4\gamma^*+1}}\right\},
\end{align*} in which we have
\begin{align*}
    \|\hat{m} - m^*\|_2^2 \lesssim s^2 \left\{\sqrt{\frac{\log p}{n}} + \left(\frac{\log^6 n}{n}\right)^{\frac{2\gamma^*}{4\gamma^*+1}}\right\} + (s_u + 1) s \frac{\overline{r} \cdot r}{\nu_{\min}^2(\bH)\cdot p}.
\end{align*}
\end{proof}

\newpage
\bibliographystyle{apalike2}
\bibliography{main.bbl}

\end{document}